\documentclass[twoside,a4paper]{article}
\usepackage{amsmath,amssymb}
\usepackage{mathtools}
\usepackage{amsthm}
\usepackage{thmtools}
\usepackage[many]{tcolorbox}
\usepackage{graphicx}
\usepackage{stmaryrd}
\usepackage[utf8]{inputenc}
\usepackage{ascmac}
\usepackage{fancybox}
\usepackage{color}
\usepackage{mathrsfs}
\usepackage{url}
\usepackage{fancyhdr}
\usepackage{enumitem}
\usepackage{varwidth}
\usepackage{calc} 
\usepackage{appendix}
\usepackage{refcount}
\usepackage{xparse}
\usepackage[T1]{fontenc}

\usepackage[top=30mm, bottom=30mm, left=25mm, right=25mm]{geometry}
\allowdisplaybreaks[1]
\thepage

\numberwithin{equation}{section}

\usepackage{hyperref} 
\usepackage{cleveref}

\crefname{thm}{Theorem}{Theorems}
\crefname{dfn}{Definition}{Definitions}
\crefname{lem}{Lemma}{Lemmas}
\crefname{prp}{Proposition}{Propositions}
\crefname{cor}{Corollary}{Corollaries}
\crefname{claim}{claim}{claims}
\crefname{rmk}{Remark}{Remarks}
\crefname{ex}{Example}{Examples}
\crefname{assume}{Assumption}{Assumptions}
\crefname{cond}{Condition}{Conditions}
\crefname{fact}{Fact}{Facts}
\crefname{note}{Notation}{Notations}

\theoremstyle{plain}
\newtheorem{thm}{Theorem}[section]
\newtheorem{lem}[thm]{Lemma}
\newtheorem{prp}[thm]{Proposition}
\newtheorem{cor}[thm]{Corollary}

\theoremstyle{definition}

\newtheorem{assume}[thm]{Assumption}
\newtheorem{cond}[thm]{Condition}

\newtheorem{dfn}[thm]{Definition}
\newtheorem{rmk}[thm]{Remark}

\newcommand{\bbE}{\mathbb{E}}
\newcommand{\bbF}{\mathbb{F}}

\newcommand{\bbP}{\mathbb{P}}

\newcommand{\calA}{\mathcal{A}}

\newcommand{\calC}{\mathcal{C}}
\newcommand{\calD}{\mathcal{D}}
\newcommand{\calE}{\mathcal{E}}
\newcommand{\calF}{\mathcal{F}}

\newcommand{\calH}{\mathcal{H}}

\newcommand{\calM}{\mathcal{M}}

\newcommand{\calR}{\mathcal{R}}

\newcommand{\calT}{\mathcal{T}}

\newcommand{\calX}{\mathcal{X}}
\newcommand{\calY}{\mathcal{Y}}

\newcommand{\scrF}{\mathscr{F}}

\newcommand{\frK}{\mathfrak{K}}

\newcommand{\frM}{\mathfrak{M}}

\newcommand{\rBCM}{\mathbf{rBCM}}
\newcommand{\BCM}{\mathbf{BCM}}
\newcommand{\Met}{\mathbf{Met}}
\newcommand{\MTop}{\mathbf{MTop}}
\newcommand{\bcm}{\mathsf{bcm}}

\usepackage{xparse}
\usepackage{refcount}

\NewDocumentCommand{\g}{ m >{\SplitArgument{1}{,}} m }{
  \CInternal{#1} #2
}

\newcommand{\CInternal}[3]{
  #1_{\getrefnumber{#2}}
  \IfValueT{#3}{^{(#3)}}
}

\def\rm#1{\mathrm{ #1 }}

\def\floor#1{\lfloor #1 \rfloor}
\def\ceil#1{\lceil #1 \rceil}

\newcommand{\R}{\mathbb{R}}
\newcommand{\Z}{\mathbb{Z}}
\newcommand{\N}{\mathbb{N}}
 
\newcommand{\Q}{\mathbb{Q}}

\newcommand{\E}{\mathcal{E}}

\newcommand{\supp}{\mathrm{supp}}

\newcommand{\diam}{\mathrm{diam}}

\newcommand{\Lip}{\mathrm{Lip}}
\newcommand{\BL}{\mathrm{BL}}

\title{Convergence rate estimates for semigroups and heat kernels associated with resistance forms}
\date{}
\author{Koyo Oishi\thanks{Research Institute for Mathematical Sciences, Kyoto University, Kyoto, 606-8502,
JAPAN. E-mail:koishi@kurims.kyoto-u.ac.jp}}

\begin{document}
\maketitle
\begin{abstract}
In this paper, we derive quantitative convergence rates for stochastic processes associated with resistance forms.
While the qualitative convergence of heat kernels and semigroups under the Gromov-Hausdorff-vague convergence of
underlying measured resistance metric spaces has been investigated previously, 
their quantitative convergence rates have remained unexplored. 
We establish explicit convergence rates for the associated semigroups and heat kernels 
under the assumptions of measure regularity and lower resistance estimates. 
Furthermore, we introduce a new metric that induces 
the Gromov-Hausdorff-vague topology, and is convenient for evaluation.
As applications of our main results, 
we present two illustrative examples.
First, we derive first estimate on the convergence rate 
for the random walk approximation of Brownian motion on the Sierpinski gasket. 
Second, we apply our results to the one-dimensional Bouchaud trap model, 
successfully extending the previously known parameter regime to all cases 
where homogenization occurs
and improving the convergence rate estimates 
in the existing regime by at least a quadratic factor.
\end{abstract}
\tableofcontents
\section{Introduction}

In the study of limit theorems for stochastic processes, 
once convergence is established, 
obtaining a corresponding rate is a natural and important problem 
from both theoretical and applied perspectives.
For instance, the Berry-Esseen theorem \cite{Berry41, Ess42}, which provides a rate estimate for the 
central limit theorem, stands as a monumental achievement in probability theory.
Regarding the assessment of convergence rates for i.i.d. sums, 
the Kolm\'os-Major-Tusn\'ady approximation shown in \cite{KMT1, KMT2} is also widely recognized.
While these represent studies on the convergence rate for sums of i.i.d.\ random variables,
research has also been conducted on the convergence of more complex spaces and the associated random walks,
as well as their convergence rates. 
Examples are shown below:
\begin{itemize}
\item In \cite{BP88, Gol87, Kus87}, Brownian motion on the Sierpinski gasket was 
constructed as a limit of the random walk on the finite approximating graphs. 
Rate estimates for the associated generators and resolvents were obtained in \cite{PS18}.
\item Convergence of a process called slow bond random walk 
and a rate estimate 
were established in \cite{SNOB}.
\item In \cite{kolokoltsov2023rates}, 
a convergence rate estimate for a continuous time random walk converging to 
a fractional diffusion was provided. 
\end{itemize}

One of the powerful frameworks for dealing with the stochastic processes on complex, fractal-like 
is the theory of resistance forms.
We provide a brief introduction to resistance forms here, 
referring the reader to \Cref{resis} for a more details.  
Recall that a resistance form on a set $F$,
as introduced by Kigami \cite{Kig01, Kig03har}, 
 is a pair $(\calE, \calF)$ consisting of 
 a bilinear form $\calE$ and its domain $\calF$ 
 (a function space on $F$).
It is well known that a space $F$ equipped with a resistance form is naturally endowed 
with a metric known as the resistance metric (see \cref{resisform,corrresismet}).
Furthermore, if $F$ is regular in the sense of \cref{regresis}, 
any full-support Radon measure $\mu$ on $F$ induces a corresponding regular Dirichlet 
form $(\calE, \calD)$ on $L^2(F, \mu)$.
Consequently, a Hunt process can be naturally constructed 
on any regular resistance metric space equipped with a fully supported Radon measure.

Convergence results of the stochastic processes associated with resistance forms 
have been well investigated. For instance:
\begin{itemize}
\item In \cite{Crtree}, it was shown that the rescaled random walk on a random graph 
satisfying suitable assumptions converges to the Brownian motion on 
the random tree whose search-depth function is the Brownian
excursion.
\item It was shown in \cite{athreya2017invariance} that if a sequence of measured metric trees $(T_n, \nu_n)$ 
converges to a limiting measured metric tree $(T,\nu)$ in an appropriate sense 
(under certain additional assumptions), 
then the processes with speed measures $\nu_n$ on $T_n$ converge 
to the process with speed measure $\nu$ on $T$.
\item Following recent advancements, including \cite{Miller2,Miller},
 the scaling limit of the random walk on critical percolation on the triangular lattice 
 has been identified. In this paper, the theory of resistance forms is also utilized.
\item Also, there are general convergence results.
In \cite{CHK17, Cr18, Ndloc}, the convergence of the process on a measured 
resistance metric space and its local time were demonstrated.
\end{itemize}
Given the correspondence between measured resistance metric spaces and stochastic processes, 
the continuity of this correspondence, 
namely the convergence of the corresponding probabilistic objects in a suitable sense, 
has been extensively studied. 
For the Sierpinski gasket, 
which is a fundamental example of a resistance metric space, 
Goldstein originally constructed Brownian motion via finite graph approximations. 
For the Sierpinski gasket, which is a fundamental example of a resistance metric space, 
Goldstein’s construction of Brownian motion via finite graph approximations 
serves as a classical example of such convergence. 
In recent years, further qualitative results have been established.
Using the Gromov-Hausdorff-vague topology, several answers to this question have been provided in recent years.
As mentioned above,
Croydon, Hambly, and Kumagai \cite{CHK17} established 
the convergence of the processes and their local times under the uniform volume doubling (UVD) condition.
Relaxing the UVD condition, Croydon \cite{Cr18} proved the convergence of the 
stochastic processes under a non-explosion condition, while Noda \cite{Ndloc}
demonstrated the convergence of the local times under a metric entropy condition combined with the non-explosion condition.
So far, none of these results were quantitative, 
and providing such results is the goal of this article.

In research on stochastic processes on random spaces and fractals, 
including the aforementioned studies, 
the Gromov-Hausdorff type topology now serves as a standard framework.
Notable examples of research that established the foundation 
for applying the GH-type topology to probability theory 
are the work of \cite{abraham2013note, evans2006rayleigh, miermont2009tessellations}.
Subsequently, the framework for the convergence of metric spaces 
with additional structures, including the Gromov-Hausdorff-vague topology, 
was studied and underwent major development in \cite{Khe23,Nd24}.

This study also employs the theory of 
Gromov-Hausdorff-type topologies established 
in these prior works. However, the Prokhorov metric, which is commonly used to define the Gromov-Hausdorff-vague topology,
does not appear to capture the information necessary for the quantitative estimates
that are the goal of this paper. 
Therefore, we first define a distance between finite measures on a metric space $(M,d^M)$, 
which is expected to be better suited for quantitative evaluation, as follows:
\begin{equation}\label{11.1}
  d_{\BL^{\kappa}}^M(\mu,\nu):=\sup\left\{ \left| \int_M f(x)(\mu-\nu)(dx)\right|:\|f\|_{\BL^{\kappa}(M,d^M)}\le 1 \right\},
\end{equation} 
where $\kappa\in (0,1]$, $\mu$ and $\nu$ are finite Borel measures on $M$, and 
\begin{equation*}
  \|f\|_{\BL^{\kappa}(M,d^M)}:=\sup_{x\in M}|f(x)|+\sup_{\substack{x,y\in M\\x\neq y}}\frac{|f(x)-f(y)|}{d^M(x,y)^{\kappa}}\in [0,\infty].
\end{equation*}
Here, $\BL$ is an abbreviation of bounded Lipschitz.
Furthermore, we use this metric to define a Gromov-Hausdorff-type metric.
As shown in \Cref{ghtype}, this metric is complete, and the topology it induces indeed coincides with the Gromov-Hausdorff-vague topology. 
Consequently, this metric serves as a new tool for performing quantitative evaluations.

We also define spaces $\bbF,\bbF_r,$ and $\bbF_c$ as follows:
   \begin{equation*}
  \bbF := \left\{ (F,R,\mu,\rho) : 
  \begin{varwidth}{ 0.65\textwidth } 
    $(F,R)$ is a boundedly compact, and regular resistance metric space, 
    $\mu$ is a full-support Radon measure on $F$, and $\rho \in F$ is a marked point
  \end{varwidth}
  \right\},
\end{equation*}
\begin{equation}\label{bbfr}
  \bbF_r:=\{ (F,R,\mu,\rho)\in \bbF:\textrm{$(F,R)$ is recurrent} \}\subseteq \bbF,
\end{equation}
\begin{equation*}
  \bbF_c:=\{ (F,R,\mu,\rho)\in \bbF:\textrm{$(F,R)$ is compact} \}\subseteq \bbF_r.
\end{equation*}
Here, we say a metric space is boundedly compact ($\bcm$) if any bounded closed set is compact,
and a resistance metric space $(F,R)$ is recurrent
if the following condition is satisfied:
\begin{itemize}
    \item there exists an increasing sequence $(U_n)_{n\geq1}$ 
    of relatively compact open subsets of $F$ such that $\bigcup_{n\geq1}U_n=F$ and
    \begin{align*}
    \lim_{n \to \infty} R(\rho, U_n^c) = \infty  
    \end{align*}
    for some $\rho \in F$.
  \end{itemize}
Recall that the effective resistance between $\rho\in F$ and $A\subseteq F$ is defined using the corresponding resistance form $(\calE,\calF)$
\begin{align}
  R(\rho,A):=\inf\{\calE(f,f):f\in\calF,\;f(\rho)=1,\;f|_A=0\}^{-1}.
\end{align}
Furthermore, $(F,R)$ is said to be regular if $\calF\cap C_c(F)$ is dense in $C_c(F)$ with respect to the supremum norm, where $\calF$ is the domain of the corresponding resistance form. 
Also, we denote the set of all triplets $(F,R,\mu)$ by $\bbF^{\circ}$. Namely,
   \begin{equation}\label{bbfcirc}
  \bbF^{\circ} := \left\{ (F,R,\mu) : 
  \begin{varwidth}{ 0.65\textwidth } 
    $(F,R)$ is a boundedly compact, and regular resistance metric space, 
    and
    $\mu$ is a full-support Radon measure on $F$
  \end{varwidth}
  \right\}.
\end{equation}
Similarly, we define $\bbF_r^{\circ}$ and $\bbF_c^{\circ}$.

Let $(F,R,\mu),$ $(F_n,R_n,\mu_n)$ be elements of $\bbF^{\circ}$.
We denote the corresponding regular Dirichlet form on $L^2(F,\mu)$, semigroup,
Hunt process, and heat kernel by 
$(\calE,\calD),\;(P_t)_{t\ge 0},\;((X_t)_t,(P_x)_{x\in F}),$ and $p=(p(t,x,y))_{x,y\in F,t>0}$.
We also write the semigroup, Hunt process, and heat kernel 
corresponding to $(F_n,R_n,\mu_n)$ by $(P^n_t)_{t\ge 0},$
$((X_t^n)_t,(P^n_x)_{x\in F_n})$ and $p_n=(p_n(t,x_n,y_n))_{x_n,y_m\in F_n,t>0}$.

First, we state main results of this paper for the compact case. 
To this end, we introduce a Gromov-Hausdorff-type topology on $\bbF_c^{\circ}$ briefly. 
See \Cref{nrc} for more details.
\begin{itemize}
  \item For $\calF_i=(F_i,R_i,\mu_i)\in\bbF_c^{\circ},$ $i=1,2$, we define 
  \begin{align*}
  d^{\tau,\BL^{\kappa}}_{\frK}(\calF_1,\calF_2):=\inf_{f,g,Z}\Big\{ d^Z_{\mathrm{H}}(f(F_1),g(F_2))\vee d^Z_{\BL^{\kappa}}(\mu_1\circ f^{-1},\mu_2\circ g^{-1})\Big\},
\end{align*}
where the infimum is taken over all compact metric spaces $(Z,d^Z)$ and isometric embeddings $f:F_1\to Z$ and 
$g:F_2\to Z$, and $d^Z_{\mathrm{H}}$ denotes the Hausdorff distance on $(Z,d^Z)$.
Here we adopt the notation $d^{\tau,\BL^{\kappa}}_{\frK}$
from \cite{Nd24}, and $\tau$ refers to a suitable functor between appropriate categories.

Then $\calF_n=(F_n,R_n,\mu_n)\in\bbF_c^{\circ}$ 
converges to $\calF=(F,R,\mu)\in\bbF_c^{\circ}$ with respect to $d^{\tau,\BL^{\kappa}}_{\frK}$
if and only if there exist a common compact metric space $(M,d^M)$ and isometric embeddings $g_n:F_n\to M,$ $g:F\to M$
such that $g_n(F_n)\to g(F)$ in the Hausdorff topology, and $\mu_{n}\circ g_n^{-1}\to \mu\circ g^{-1}$ weakly (see \cref{nonrootcpt}).  
Therefore $d_{\frK}^{\tau,\BL^{\kappa}}$ induces the compact Gromov-Hausdorff-Prokhorov topology.

\end{itemize}
Note that, by this characterization of convergence, if $\calF_n=(F_n,R_n,\mu_n)\in\bbF_c^{\circ}$ converges to $\calF=(F,R,\mu)\in\bbF_c^{\circ}$ with respect to $d^{\tau,\BL^{\kappa}}_{\frK}$,
we may regard $F_n,$ $n\in\N,$ and $F$ as subsets of a common compact metric space $(M,d^M)$.

Also, writing $B(x,r):=\{y\in F:R(x,y)<r\}$, we consider the following three assumptions for the limiting space $(F,R,\mu)$.
\begin{assume}\label{as11}\quad\par
\begin{itemize}
    \item[\textbf{(A1)}] There exist constants $c_u,\;c_l,\;s_0,\;s_1\in(0,\infty)$ such that \[c_l r^{s_0} \le \mu(B(x,r)) \le c_u r^{s_1}\] for all $x\in F$ and $r<\diam F$.
    \item[\textbf{(A2)}] Assumption $\rm{(A1)}$ holds. Furthermore, there exists a constant $\theta>0$ satisfying $(1+s_0)\theta>1\vee s_1$ and a constant $c_{\rm{LR}}>0$ such that \[R(x,B(x,r)^c)\ge c_{\rm{LR}}r^{\theta}\] for all $B(x,r)\neq F$.
    \item[\textbf{(A3)}] Assumption $\rm{(A1)}$ holds with $s_0,\;s_1$ satisfying $(s_0-s_1)(2+s_0)<1$, the space $(F,R)$ is uniformly perfect, and the corresponding Dirichlet form $(\calE,\calD)$ is local.
\end{itemize}  
\end{assume}
Recall that $(F,R)$ is said to be uniformly perfect if 
there exists a constant $c_{\rm{UP}}\in (0,1)$ 
such that $B(x,r)\setminus B(x,c_{\rm{UP}}r)\neq \emptyset$ 
for any ball with $B(x,r)\neq F$.

Note that condition $\mathrm{(A3)}$ implies $\mathrm{(A2)}$ (see \cref{summ}(i)). 
Finally, we remark that, in this section, $A_n \lesssim B_n$ means that there exists a constant $C\in (0,\infty)$,
depending on the data of spaces such as supremum of their diameters and measures, 
such that $A_n \le CB_n$ for any $n$. 
Whether $C$ is allowed to depend on the spatial data 
or must be a universal constant varies by section. 
The specific convention adopted is explicitly stated 
at the beginning of each section.
With this preparation, the main result for the compact case is given as follows.
Recall the definition of $\bbF_c^{\circ}$ from \eqref{bbfcirc}.

\begin{thm}\label{maincptsg}
Suppose that $(F_n,R_n,\mu_n)\in \bbF_c^{\circ}$ converges to $(F,R,\mu)\in\bbF_c^{\circ}$ 
with respect to $d^{\tau,\BL^{\kappa}}_{\frK}$.
We regard $F_n,$ $n\in\N,$ and $F$ as subsets of a common compact metric space $(M,d^M)$.
Let $x\in F,\;x_n\in F_n,\;t>0$ and $f:M\to\R$, and set
\begin{align*}
  h:=d^M(x,x_n)+d_{\rm{H}}^M(F,F_n)^{\kappa}+d_{\BL^{\kappa}}^M(\mu,\mu_n).
\end{align*}
Here, $d^M_{\rm{H}}$ is the Hausdorff distance on $(M,d^M)$.
Then we have the following.
\begin{itemize}
\item [$\rm{(i)}$]If $(F,R,\mu)$ satisfies Assumption $\rm{(A1)}$ with $s_0-s_1<1$, and $h\leq 1/2$,
then there exists a constant $E=E(s_0,s_1,\kappa)$ (see \cref{sgrw} for an explicit formula), and it holds for $\kappa\in [1/2,2/3)$ that
\begin{align*}
 \Big|E_x[f(X_t)]-E^n_{x_n}[f(X^n_t)]\Big|\lesssim \|f\|_{\BL^{\kappa}(M,d^M)} h^{E},\qquad t>0.
\end{align*}
Moreover, the time-dependence of the constant in $\lesssim$ is $t^{-2}+t,$ $t>0$.
  \item [$\rm{(ii)}$]If $(F,R,\mu)$ satisfies Assumption $\rm{(A2)}$, and $h\leq 1/2$,
then there exists a constant $E=E(s_0,\theta,\kappa)$ (see \cref{gen}$\rm{(i)}$ for an explicit formula), and it holds for $\kappa\in [1/2,1]$ that
\begin{align*}
 \Big|E_x[f(X_t)]-E^n_{x_n}[f(X^n_t)]\Big|\lesssim \|f\|_{\BL^{\kappa}(M,d^M)} h^{E},\qquad t>0.
\end{align*}
Moreover, the time-dependence of the constant in $\lesssim$ is $t^{-[\kappa-\frac{s_1}{2(1+s_0)\theta}(2\kappa-1)]},\;t\in (0,1]$.
\item[$\rm{(iii)}$]If $(F,R,\mu)$ satisfies Assumption $\rm{(A3)}$, and $h$ is sufficiently small,
then there exist constants $E_1=E_1(s_0,s_1,\kappa)$ and $E_2=E_2(s_0,s_1,\kappa)$, and it holds for $\kappa\in [1/2,1]$ that
\begin{align*}
 \Big|E_x[f(X_t)]-E^n_{x_n}[f(X^n_t)]\Big|\lesssim \|f\|_{\BL^{\kappa}(M,d^M)} h^{E_1}(\log(1/h))^{E_2},\qquad t>0.
\end{align*}
Moreover, the time-dependence of the constant in $\lesssim$ is $t^{-[\kappa-\frac{s_1}{2(1+s_0)\Theta}(2\kappa-1)]},\;t\in (0,1]$
for some constant $\Theta=\Theta(s_0,s_1)$. 
See \cref{gen}$\rm{(ii)}$ for explicit formulae defining $E_1,\;E_2,$ and $\Theta$.
\end{itemize}
\end{thm}

\begin{thm}\label{mainhk}
  Suppose that $(F_n,R_n,\mu_n)\in \bbF_c^{\circ}$ converges to $(F,R,\mu)\in\bbF_c^{\circ}$
  with respect to $d^{\tau,\BL^{\kappa}}_{\frK}$.
We regard $F_n,n\in\N,$ and $F$ as subsets of a common compact metric space $(M,d^M)$.
  Fix $t>0,\;x,y\in F,$ and $x_n,y_n\in F_n$.
Suppose that 
\begin{align*}
  h:=d_{\rm{H}}^M(F,F_n)^{\kappa}+d_{\BL^{\kappa}}^M(\mu,\mu_n)+d^M(x,x_n)+d^M(y,y_n)\le \frac{1}{2}.
\end{align*}
Then we have the following.
\begin{itemize}
  \item [$\mathrm{(i)}$]If $(F,R,\mu)$ satisfies Assumption $\rm{(A1)}$ with $s_0-s_1<1$,
then there exists a constant $E=E(s_0,s_1,\kappa)$ (see \cref{hkest}$\rm{(i)}$ for an explicit formula), and it holds for $\kappa\in [1/2,2/3)$ that
\begin{align*}
  |p(t,x,y)-p_n(t,x_n,y_n)|\lesssim h^E.
\end{align*}
Moreover, the time-dependence of the constant in $\lesssim$ is $t^{-2},\;t\in (0,1]$.
\item[$\mathrm{(ii)}$]If $(F,R)$ satisfies Assumption $\mathrm{(A2)}$ with $s_0-s_1<1$, 
then there exists a constant $E=E(s_0,s_1,\kappa)$ (see \cref{hkest}$\rm{(ii)}$ for an explicit formula), and it holds for $\kappa\in [1/2,1]$ that
\begin{align*}
  |p(t,x,y)-p_n(t,x_n,y_n)|\lesssim h^E.
\end{align*}
Moreover, the time-dependence of the constant in $\lesssim$ is $t^{-2},\;t\in (0,1]$.
\item[$\rm{(iii)}$]If $(F,R)$ satisfies Assumption $\mathrm{(A3)}$, 
then there exist constants $E_1=E_1(s_0,s_1,\kappa)$ and $E_2=E_2(s_0,s_1,\kappa)$ (see \cref{hkest}$\rm{(iii)}$ for an explicit formula), 
and it holds for $\kappa\in [1/2,1]$ that
\begin{align*}
  |p(t,x,y)-p_n(t,x_n,y_n)|\lesssim h^{E_1}(\log(1/h))^{E_2}.
\end{align*}
Moreover, the time-dependence of the constant in $\lesssim$ is $t^{-2},\;t\in (0,1]$.
\end{itemize}
\end{thm}

\begin{rmk}
  Although these estimates are likely not optimal, it is novel that the convergence rate has been proven to be at least polynomial with respect to the distance of the space.
\end{rmk}

\begin{rmk}
  One can find the proof of \cref{maincptsg}(i) in \cref{sgrw}, and those of  
  \cref{maincptsg}(ii) and (iii) in \cref{gen}.
  Also, \cref{mainhk} is shown in \cref{hkest}.
\end{rmk}

\begin{rmk}
\cref{maincptsg} is established via an argument based on estimates for resolvents. 
On the other hand, an alternative estimate for semigroups can be obtained using heat kernel estimates (see \cref{sgrw}). 
However, since this approach generally yields a suboptimal convergence rate compared to the present theorem (see \cref{66}), 
we omit a detailed description of the latter estimate here.
\end{rmk}

Next, we state the results for the non-compact case.
Recall the definitions of $\bbF$ and $\bbF_r$ from \eqref{bbfr}.
For $(F,R,\mu,\rho)\in \bbF$, we define 
$F^{(r)}=\overline{B(\rho,r)},\;R^{(r)}=R|_{F^{(r)}\times F^{(r)}}$ and $\mu^{(r)}=\mu|_{F^{(r)}}$.
First we introduce the Gromov-Hausdorff-vague topology.
We say $(F_n,R_n,\mu_n,\rho_n)\in\bbF$ 
converges to $(F,R,\mu,\rho)\in\bbF$ if and only if the following holds.
\begin{itemize}
  \item There exists a boundedly compact metric space $(M,d^M)$
   and isometric embeddings $g_n:F_n\to M,\;g:F\to M$
  such that $g_n(F_n)\to g(F)$ in the Fell topology, $g_n(\rho_{n})\to g(\rho)$ in $M$,
  and $\mu_n\circ g_n^{-1}\to \mu\circ g^{-1}$ vaguely.
\end{itemize}
Here, recall that the convergence of $A_n\subseteq M,$ $n\in\N$ to $A\subseteq M$ in the Fell topology 
is equivalent to the following condition (see \cite[Theorem 3.9]{Nd24}). 
Also see \cref{fell} for the precise definition of the Fell topology. 
\begin{itemize}
  \item For any sequence $(x_n)\subseteq M$ converging to $x\in X$, it holds that $A_n\cap \overline{B(x_n,r)}\to A\cap \overline{B(x,r)}$
  in the Hausdorff topology for all but countably many $r>0$.
\end{itemize}

Note that, by this characterization of convergence, 
if $(F_n,R_n,\mu_n,\rho_n)\in\bbF$ converges to $(F,R,\mu,\rho)\in\bbF$ 
in the Gromov-Hausdorff-vague topology,
we may regard $F_n,$ $n\in\N,$ and $F$ as subsets of a common boundedly compact metric space $(M,d^M)$.
\begin{rmk}
  Although we omit the details here to avoid cumbersome notation, 
  the Gromov-Hausdorff-vague topology can also be metrized by appropriately modifying the distance defined in \eqref{11.1} for the non-compact case.
  See \cref{232} for details.
\end{rmk}
In non-compact case, we consider the following assumptions.

\begin{assume}\label{mainassump}\quad\par
 \begin{itemize}
  \item[1.]The sequence $(F_n,R_n,\mu_n,\rho_n)\in\bbF_r$ converges to $(F,R,\mu,\rho)\in\bbF_r$ in the Gromov-Hausdorff-vague topology. 
  \item[2.]There exists an $r_0>0$ such that the following hold.
 
\begin{enumerate}[align=left, labelwidth=\widthof{2-a.},
    labelsep=0.5em, 
    leftmargin=!]
  \item [2-a.] There exists an $s>0$ such that $\mu^{(r)}$ is uniformly $s$-Ahlfors regular for $r>r_0$. That is,
   there exist constants $c_u,c_l,s>0$ such that for any $r\in (r_0,\infty),$ $a\in [0,\diam F^{(r)})$ and $x\in F^{(r)}$,
   it holds that $c_l a^s\leq \mu^{(r)}(B(x,a))\leq c_u a^s$.
  \item [2-b.]The sequence $(F,R)$ satisfies the lower resistance estimate with an exponent $\theta>\frac{1\vee s}{1+s}$.
   That is, there exists a constant $c_{\rm{LR}}>0$ such that 
   \begin{equation}\label{qsc}
    c_{\rm{LR}}r^{\theta}\leq R(x,B(x,r)^c),\qquad \forall x\in F,\;r>0.
   \end{equation}
  \item [2-c.]There exists a constant $c_{\rm{LHK}}>0$ such that 
  \begin{align*}
    c_{\rm{LHK}}t^{-\frac{s}{(1+s)\theta}}\leq p^{(r)}(t,x,x),\quad \forall r>r_0,\;x\in F^{(r)},\;t>0,
  \end{align*}
  where $p^{(r)}$ is the heat kernel associated with $(F^{(r)},R^{(r)},\mu^{(r)})$.
  \end{enumerate}
\end{itemize}
\end{assume}

Under this assumption, the following theorem holds.
\begin{thm}\label{maincpt}
    Suppose that \cref{mainassump} is satisfied. 
  We regard $F_n,\;n\in\N,$ and $F$ as subsets of a common $\bcm$ space $(M,d^M)$.
  Let $x\in F$ and $x_n\in F_n$. Take $r_{\infty}\ge r_0$ satisfying 
  $x\in B(\rho,r_{\infty})$ and $x_n\in B(\rho_n,r_{\infty})$.
  Then there exist a constant $E=E(s,\theta,\kappa)$ and a polynomial $C_n(r)=C_n(\mu_n(F_n^{(r)}),\mu(F^{(r)}),\diam F^{(r)}\cup F_n^{(r)},\diam F^{(r)})$ (see \cref{noncpt} for explicit formulae),
  and it holds for any $f:M\to\R,\;\kappa\in [1/2,1],$ and $t>0$ that
  \begin{equation}\label{lolo}
    \begin{aligned}
   \MoveEqLeft \Big|E_{x}[f(X_t)]-E^n_{x_n}[f(X^n_t)]\Big|\\
    \lesssim&\|f\|_{\BL^{\kappa}(M,d^M)}
    \inf\left\{\Psi_1(n,r)+C_n(r)\Psi_2(n,r)^E : r\ge r_{\infty}, \Psi_2(n,r)\le \frac{1}{2}\right\},  
    \end{aligned}
  \end{equation}
  where
  \begin{equation*}
    \Psi_1(n,r)=P_{x}(\sigma_{B(\rho,r)^c}\le t)\vee P_{x_n}^n(\sigma^n_{B(\rho_n,r)^c}\le t),
  \end{equation*}
  \begin{equation*}
    \Psi_2(n,r)=d_{\rm{H}}^M(F^{(r)},F_n^{(r)})^{\kappa}+d_{\BL^{\kappa}}^M(\mu^{(r)},\mu^{(r)}_n)+d^M(x,x_n),
  \end{equation*}
  Moreover, the time-dependence of the constant in $\lesssim$ is $t^{-[\kappa-\frac{s}{2(1+s)\theta}(2\kappa-1)]},\;t\in (0,1]$.
\end{thm}

\begin{rmk}
  To make the right hand side of \eqref{lolo} small, we may choose $r$ in the following way.
  First, we take sufficiently large $r_{\ast}>0$ so that $\Psi_1(n,r_{\ast})$ is small uniformly in $n$.
  Then we may confirm that $(C_n(r_{\ast}))_n$ is bounded for this fixed $r_{\ast}$.
  Finally letting $n\to\infty$, $\Psi_2(n,r_{\ast})$ converges to $0$.
\end{rmk}

\begin{rmk}
Although \cref{maincpt} presents only the convergence rate of the semigroup, 
it is in principle possible to derive quantitative estimates for the heat kernel 
by applying the proof techniques from \Cref{esthk}. 
However, since the estimate terms become extremely complicated,
we prioritize the clarity of 
the main result and restrict our attention to the semigroup estimate in this paper.
\end{rmk}

\begin{rmk}
For sufficient conditions to satisfy \cref{mainassump} used here,
we refer the reader to \cref{cork} and \cref{7.3}.
\end{rmk}

\begin{rmk}
If $F_n,\;n\in\N,$ and $F$ satisfy the lower resistance estimate with an exponent $\theta$ (see \eqref{qsc}),
we have $\Psi_1(n,r)\lesssim r^{-\theta}$ by \cref{4.2-a}(ii).
\end{rmk}

\begin{rmk}
Since the constants in $\lesssim$ of \cref{maincptsg,mainhk,maincpt} 
do not depend on $x,\;x_n,\;y,$ and $y_n$,
it is not difficult to obtain the uniform bound over starting points.
For example, if $F$ is not compact, and $(x^{\lambda})_{\lambda\in\Lambda}\subseteq B(\rho,r_{\infty})\cap F$ and 
$(x^{\lambda}_n)_{\lambda\in\Lambda}\subseteq B(\rho_n,r_{\infty})\cap F_n$ satisfy $d^M(x^{\lambda},x_n^{\lambda})\leq \varepsilon$ for some constant $\varepsilon>0$,
then we have
\begin{align*}
  \sup_{\lambda\in\Lambda}\Big|E_{x^{\lambda}}[f(X_t)]-E^n_{x_n^{\lambda}}[f(X^n_t)]\Big|
    \lesssim\|f\|_{\BL^{\kappa}(M,d^M)}
  \inf\Bigl\{  \Tilde{\Psi}_1(n,r)+C_n(r)\Tilde{\Psi}_2(n,r)^E : r\ge r_{\infty}, \Tilde{\Psi}_2(n,r)\le \frac{1}{2}\Bigr\},
\end{align*}
  where
  \begin{equation*}
    \Tilde{\Psi}_1(n,r)=\sup_{\lambda\in\Lambda}P_{x^{\lambda}}(\sigma_{B(\rho,r)^c}\le t)\vee P_{x_n^{\lambda}}^n(\sigma^n_{B(\rho_n,r)^c}\le t),\quad
    \Tilde{\Psi}_2(n,r)=d_{\rm{H}}^M(F^{(r)},F_n^{(r)})^{\kappa}+d_{\BL^{\kappa}}^M(\mu^{(r)},\mu^{(r)}_n)+\varepsilon.
  \end{equation*}
\end{rmk}

The proof of the main theorem proceeds in the following steps. 
First, we establish error estimates for the resolvent (see \cref{estres}). 
The idea for this part is largely inspired by \cite{Cr18}, 
relying on the fact that the Green function of the process killed at a single point 
can be expressed using the resistance metric, and that the $\alpha$-resolvent 
of the original process can be written in terms of the associated Green operator.

Next, using these resolvent estimates, 
we derive error estimates for the semigroup acting on test functions belonging to the domain of the generator
(more precisely, the domain of its square; see \cref{sg2}). 
To approximate a general bounded H\"older continuous function $f$ by functions in the domain,
we consider the action of a short-time semigroup on $f$, 
which provides the necessary smoothing to obtain the desired semigroup estimates. 
For the heat kernel, the basic idea is fundamentally the same: 
we approximate the Dirac measure (an approximate identity) 
by acting a short-time semigroup on the indicator function of a small open ball.

Regarding the non-compact case, we obtain the results 
by approximating the space by an increasing sequence of compact sets.
\cref{mainassump} ensures that this approximating sequence satisfies the conditions required for the compact case.

To carry out these quantitative estimates rigorously,
precise H\"older continuity estimates for 
the heat kernel, $P_tf$, and the $\alpha$-potential density are indispensable.
The analysis related to these quantities is conducted in \Cref{cont}.
In the process, we also prove convergence of hitting times (see \cref{hitwkconv}),
a result we believe to be of independent interest for various applications.

As applications of our results, 
we present two illustrative examples in this paper: 
the Sierpinski gasket (SG) and the Bouchaud trap model (BTM). 

The Sierpinski gasket is one of the most well-investigated self-similar fractals 
(see \Cref{picsg} below).
While the convergence of random walks on approximating graphs to Brownian motion 
on the gasket is now a well-known fact \cite{BP88, Gol87, Kus87}, 
this paper is the first to rigorously derive an estimate on its quantitative convergence rate. 
As an immediate consequence of our results,
we may confirm that
\begin{align}\label{ohsugo}
  \Big|E_x[f(X_t)]-E_{x_n}[f(X^n_t)]\Big|\lesssim \|f\|_{\BL^{1}}\left( \frac{3}{5} \right)^{E_1 n}n^{E_2},\quad E_1=0.0956756\cdots,\quad E_2=0.86580748\cdots
\end{align}
for a function $f:\R^2\to \R$ which is Lipschitz continuous with respect to the Euclidean distance $d_E$ on $\R^2$,
where $(X_t)$ is a Brownian motion on SG, $(X^n_t)$ is a random walk on the $n$-th approximating graph $V_n$, 
and $x\in\rm{SG}$ and $x_n\in V_n$ are points with $d_E(x,x_n)\lesssim 2^{-n}$.
For more precise description, the reader is referred to \cref{sgrate}. 
\begin{figure}[h]
  \centering
  \begin{minipage}[b]{0.4\textwidth}
    \centering
    \includegraphics[width=40mm]{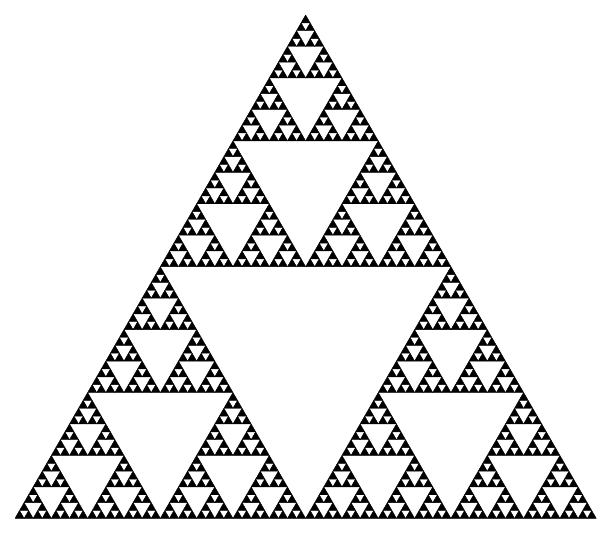}
    \par \vspace{1pt}
  \end{minipage}
  \caption{Sierpinski Gasket}
  \label{picsg}
\end{figure}
\begin{rmk}
We used the open set condition and 
the fact that the resistance metric and the Euclidean metric are comparable 
up to a power to verify \cref{as11} for the Sierpinski gasket.
Note that (A2) or (A3) are satisfied provided that the connectivity of the space or the continuity of the process is established.
Consequently, our results apply to a wider class of fractals satisfying these properties, 
such as nested fractals. 
We also conjecture that our results are applicable 
to the two-dimensional Sierpinski carpet 
under some modifications.
\end{rmk}

BTM is a model of aging, which refers to the phenomenon 
in which a system's dynamics progressively slow down without reaching equilibrium on laboratory time scales, 
so that its current state carries information about its age. 
A rough description of the model is as follows: 
First, for given $\alpha>0$, we introduce i.i.d.\ weights $(\tau_i)_{i\in\Z}$ with $P(\tau_0>u)=u^{-\alpha},\;u>1$ and construct a measure $\mu_n$ on $n^{-1}\Z$ by $\mu_n=\frac{1}{n}\sum_{i\in\Z}\tau_i\delta_{i/n}$.
We regard $(n^{-1}\Z,\mu_n)$ equipped with the usual Euclidean distance $d_E$ as a measured resistance metric space, and investigate the scaling limit of the space and the corresponding process $X^n=(X^n_t)$.
Note that, if $\alpha>1$, $(n^{-1}\Z,d_E,\mu_n)$ converges to $(\R,d_E,E[\tau_0]dx)$ in the Gromov-Hausdorff-vague topology almost surely,
and the limit process $X=(X_t)_t$ of $X^n$ is a Brownian motion scaled by a deterministic factor.

Quantitative estimates for the speed of convergence for the distribution function and the heat kernel were shown in \cite{BTM24}.
Since they used an arguments involving second moment,
the regime of $\alpha$ is restricted to $\alpha>2$.
As an application of our results, we extend the regime of $\alpha$ to $\alpha>1$ in \cref{btmrate}.
Moreover, by virtue of a near-optimal estimate for the $\BL^{\kappa}$ distance between $\mu_n$ and the limiting measure, our result significantly improves the previous estimates from \cite{BTM24},
at least doubling the exponent of the convergence rate.
Figure \ref{fig:exponent_comparison} compares the estimates for the distribution function in \cite{BTM24} and in this paper.
Although the range of $\alpha$ is somewhat restricted, we also obtain annealed estimates for the semigroup and the heat kernel,
which were unknown. See \cref{btmsgQ,btmsgA,annhk} for more details.

\begin{figure}[htbp]
  \centering
  
  \includegraphics[height=0.3\textheight,width=0.8\textwidth]{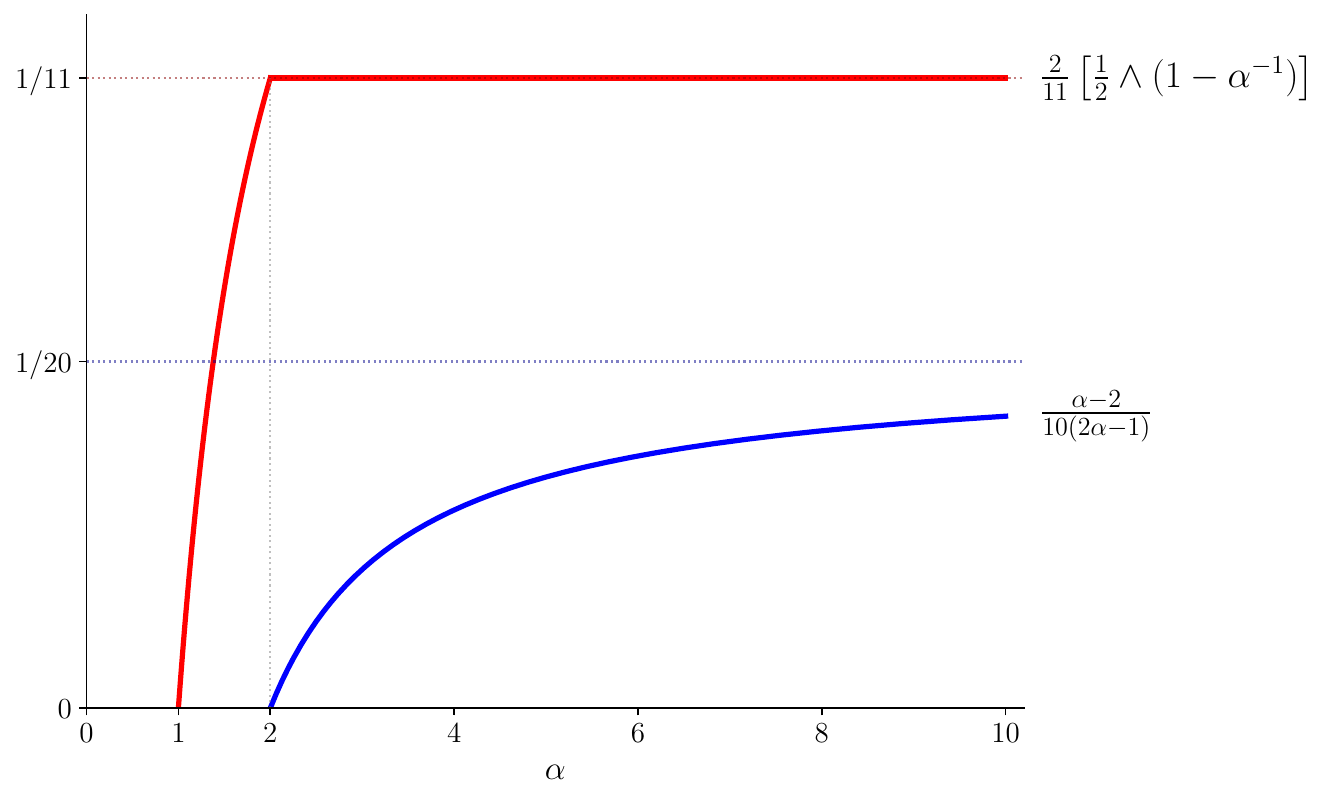}
  \caption{Graph of the exponent $E=E(\alpha)$ appearing in the quenched estimate $\sup_{x\in\R}|P_0(X_1\le x)-P_0(X^n_1\le x)|\lesssim n^{-E+\varepsilon}$. The blue line represents the exponent derived in \cite{BTM24}, and the red line represents the one derived in \cref{btmrate}.}
  \label{fig:exponent_comparison}
\end{figure}

\begin{rmk}
  Our main result is also applicable to the setting considered in \cite{SNOB}. 
  However, since the resulting estimates were weaker than those in \cite[Theorem 2.4]{SNOB}, 
  we omit the details in this paper.
\end{rmk}

The organization of this paper is as follows.
In \Cref{pre}, 
we introduce the metric defined in \eqref{11.1} and the corresponding Gromov-Hausdorff-type distance, 
along with the concept of resistance forms and related known results.
In \Cref{cont}, we investigate the H\"older continuity of 
the heat kernel, $P_tf$ for suitable functions $f$, and the $\alpha$-potential density.
\Cref{efr,efsg,esthk} are devoted to establishing the resolvent estimates,
the semigroup estimates, and the heat kernel estimates, respectively.
In \Cref{ncpt}, we consider the non-compact case.
Finally, in \Cref{app}, we present examples 
illustrating the applications of our results.

Lastly, we make some remarks on notation used in this paper.
\begin{itemize}
  \item We write $a\wedge b:=\min\{a,b\} $ and $a\vee b:=\max\{a,b\}$ for $a,b\in \R\cup\{\pm\infty\}$.
\item We also note that throughout the paper, 
a constant appearing in proposition, lemma, or theorem X 
is denoted by $c_X$.
\end{itemize}

\section{Preliminaries}\label{pre}

\subsection{Bounded Lipschitz metric}\label{undertop}
In this section, we introduce an appropriate metric on a space of measures. 
In \cite{Cr18}, the Gromov-Hausdorff-vague metric (or topology)
is used to formulate the convergence
of root-and-measured metric spaces. 
Although the metric also works in our setting, it is unsuitable for our purpose.
More precisely, we aim to establish estimates for semigroups and heat kernels
using the distance between spaces.
To acheive this, roughly speaking, we need to 
estimate the difference $|\int fd\mu-\int fd\nu|$ by using the metric between $\mu$ and $\nu$.
However, the (suitably adapted to non-compact settings) Prokhorov metric, which is used to define the Gromov-Hausdorff-vague metric,
does not seem to capture the necessary information for this.
Therefore, we first need to define an alternative metric that is better suited to this, and that is the aim of this subsection.   

In this paper, for a metric space $(X,d^X)$, $x\in X$, and $r>0$, we define $B_{d^X}(x,r)$ and $D_{d^X}(x,r)$
as follows.
\begin{align*}
  B_{d^X}(x,r):=\{x\in X:d^X(x,y)<r\},\;\mathrm{and}\;D_{d^X}(x,r):=\{y\in X:d^X(x,y)\leq r\}.
\end{align*}
If the metric $d^X$ is clear in the context, we omit $d^X$ from the notation.

\begin{dfn}
  Let $(X,d^X)$ be a metric space. For a function $f:X\to\R$, a non-empty subset
  $A\subseteq X$ and a real number $\kappa\in (0,1]$, we define $\|f\|_A,$ $\|f\|_{\Lip^{\kappa}(A)}\in [0,\infty],$ and $\|f\|_{\BL^{\kappa}(A)}\in [0,\infty]$ 
  as follows.
  \begin{equation*}
    \|f\|_A:=\sup_{a\in A}|f(a)|,\qquad \|f\|_{\Lip^{\kappa}(A)}:=\sup_{\substack{x,y\in A\\x\neq y}}\frac{|f(x)-f(y)|}{d^X(x,y)^{\kappa}},
  \end{equation*}
  and
  \begin{equation*}
    \|f\|_{\BL^{\kappa}(A)}:=\|f\|_A+\|f\|_{\Lip^{\kappa}(A)}.
  \end{equation*} 
  If $A$ is a singleton, we set $\|f\|_{\Lip^{\kappa}(A)}=0$. 
  The quantity $\|f\|_{\Lip^{\kappa}(A)}$ is the $\kappa$-H\"older constant of $f|_A$, 
  and we call $\|\cdot\|_{\BL^{\kappa}(A)}$ the $\BL^{\kappa}(A)$ norm 
  or simply the $\BL^{\kappa}$ norm if $A$ is clear from the context.
Also, we denote the set of functions with a finite $\BL^{\kappa}(A)$ norm by
$\BL^{\kappa}(A)$. Namely,
  \begin{equation*}
  \BL^{\kappa}(A):=\{f:X\to\R:\|f\|_{\BL^{\kappa}(A)}<\infty\}.
\end{equation*}
\end{dfn}

\begin{rmk}\label{submult}
One can confirm the submultiplicativity of $\|\cdot\|_{\BL^{\kappa}}$:
\begin{align*}
  \|fg\|_{\BL^{\kappa}(A)}
  =&\|fg\|_{A}+\|fg\|_{\Lip^{\kappa}(A)}\\
  \leq&\|f\|_A\|g\|_A+\|f\|_{\Lip^{\kappa}(A)}\|g\|_A+\|f\|_A\|g\|_{\Lip^{\kappa}(A)}\\
  \leq&\|f\|_{\BL^{\kappa}(A)}\|g\|_{\BL^{\kappa}(A)}
\end{align*}
Also, if we define $|f|_{\BL^{\kappa}(A)}:=\|f\|_A\vee \|f\|_{\Lip^{\kappa}(A)}$, then 
we can show $|fg|_{\BL^{\kappa}(A)}\le 3 |f|_{\BL^{\kappa}(A)}|g|_{\BL^{\kappa}(A)}$
by the same argument as above.
\end{rmk}

\begin{rmk}
 Although the notion of $\BL^{\kappa}$ depends on the metric $d^X$, we omit $d^X$ from the notation,
  since we always fix one metric for each space.
\end{rmk}

The following lemma plays an important role, and is known as the McShane's extension theorem.
The lemma is well known, but we provide a proof for the reader's convenience.
\begin{lem}[{\cite[Theorem 1]{MS34}}]\label{mcshane}
  Let $A$ be a non-empty subset of a metric space $(X,d^X)$. For any $f:A\to\R$ with 
  $\|f\|_{\BL^{\kappa}(A)}<\infty$, there exists an extension $\bar{f}:X\to\R$ such that
  $\|\bar{f}\|_X=\|f\|_A$ and $\|\bar{f}\|_{\Lip^{\kappa}(X)}=\|f\|_{\Lip^{\kappa}(A)}$.
\end{lem}

\begin{rmk}
Strictly speaking, it would be more precise to denote the above extension,
for instance, by $\bar{f}^{A}$ rather than $\bar{f}$,
since, for example, the extensions of
$1_{[0,1]}:[0,1]\to\R$ and $1_{[0,1]}:[-2,-1]\cup[0,1]\to\R$
may differ.
Here, we denote the indicator function of $A$ by $1_A$ for a set $X$ and its subset $A$.
Nevertheless, in what follows, we omit the superscript $A$ from the notation,
as no confusion will arise.
\end{rmk}

\begin{proof}
  Once we construct an extension $g:X\to \R$ of $f$ satisfying
   $\|g\|_{\Lip^{\kappa}(X)}=\|f\|_{\Lip^{\kappa}(A)}$, it is easy to obtain the desired extension $\bar{f}$.
   In fact, it is enough to set $\bar{f}:=((\sup f(A))\wedge g)\vee \inf f(A)$.
   Thus, we focus on showing the existence of such a function $g$.

   Let $L:=\|f\|_{\Lip^{\kappa}(A)}\in [0,\infty)$ and define $g:X\to (-\infty,\infty]$ by setting
   \begin{align*}
    g(x)=\sup_{a\in A}(f(a)-Ld^X(x,a)^{\kappa}),\qquad x\in X.
   \end{align*}
   
We first show that $g$ is an extension of $f$. Fix $x\in A$.
Since $f$ is $\kappa$-H\"older continuous with constant $L$, it holds that
\begin{align*}
  f(a)-Ld^X(a,x)^{\kappa}
  =&f(a)-f(x)+f(x)-Ld^X(a,x)^{\kappa}\\
  \leq&Ld^X(x,a)^{\kappa}+f(x)-Ld^X(a,x)^{\kappa}\\
  =&f(x)
\end{align*}
for any $a\in A$. This implies $g(x)\leq f(x)$, and the reverse inequality is obvious.
Thus, $g$ is an extension of $f$.
Next, we will show that $g(x)$ is finite for any $x\in X$.
Take $x_0\in X$ with $g(x_0)<\infty$.
Note that the existence of such an $x_0$ is guaranteed by the fact that $g|_A=f<\infty$ and $A\neq \emptyset$.
Then, for any $x_1\in X$ and $a\in A$, it holds that
\begin{equation}\label{lipineq}
  \begin{aligned}
      f(a)-Ld^X(x_1,a)^{\kappa}-g(x_0)
  =&f(a)-Ld^X(x_1,a)^{\kappa}-\sup_{b\in A}(f(b)-Ld^X(x_0,b)^{\kappa})\\
  \leq&Ld^X(x_0,a)^{\kappa}-Ld^X(x_1,a)^{\kappa}\\
  \leq&Ld^X(x_0,x_1)^{\kappa}.
  \end{aligned}
\end{equation}
Here, we take $b=a$ in the first inequality.
Taking a supremum over $a\in A$, 
we obtain $g(x_1)-g(x_0)\leq Ld^X(x_0,x_1)^{\kappa}$ and thus, 
$g$ is finite at any point of $X$. Finally, by \eqref{lipineq}, 
we can check the H\"older continuity.
\end{proof}

\begin{dfn}[McShane extension]\label{Mcext}
Let $A$ be a non-empty subset of a metric space $(X,d^X)$ and $f$ be an element in $\BL^{\kappa}(A)$.
We say a function $\bar{f}:X\to\R$ is a McShane extension of $f:A\to\R$ (with exponent $\kappa$) 
if $\bar{f}$ satisfies $\bar{f}|_A=f$, $\|\bar{f}\|_X=\|f\|_A$, and $\|\bar{f}\|_{\Lip^{\kappa}(X)}=\|f\|_{\Lip^{\kappa}(A)}$.
\end{dfn}

We now define a metric between measures that will be suitable for our goal.

\begin{dfn}[$\BL^{\kappa}$ metric]
Let $(X,d^X)$ be a metric space and define $\calM(X)_{\mathrm{fin}}$ by
\begin{equation*}
    \calM_{\mathrm{fin}}(X):=\{\mu:\mu\;\mathrm{is\;a\;finite\;Radon\;measure\;on\;}X\}.
\end{equation*}
For $\kappa\in (0,1]$, we define a function $d_{\BL^{\kappa}}^X$ on $\calM(X)_{\mathrm{fin}}\times \calM(X)_{\mathrm{fin}}$ as follows.
\begin{align}\label{metblk}
  d_{\BL^{\kappa}}^X(\mu,\nu):=\sup\left\{ \left| \int_X f d(\mu-\nu)\right|:\|f\|_{\BL^{\kappa}(X)}\leq 1 \right\}\in [0,\infty],\quad \mu,\nu\in \calM(X)_{\mathrm{fin}}.
\end{align}
\end{dfn}

If $\kappa=1$, this function coincides with the metric used in 
\cite[Theorem 2.4]{SNOB} to estimate the convergence rate of 
a scaled random walk on $\Z$ with a ``constant gap'' to 
the snapping out Brownian motion on $\R$.
Additionally, an analogous metric is employed in 
\cite[Theorem 1]{kolokoltsov2023rates} to provide 
convergence rate estimates for a scaled random walk converging to 
a stable subordinator.

\begin{lem}
  The function $d_{\BL^{\kappa}}^X:\calM(X)_{\mathrm{fin}}\times \calM(X)_{\mathrm{fin}}$
  defined in \eqref{metblk}
  is a metric. We call $d_{\BL^{\kappa}}^X$ the $\BL^{\kappa}$ metric. 
\end{lem}
\begin{proof}
  Although this was essentially shown in \cite[Proposition 11.3.2]{Dud02},
  we give a proof for the reader's convenience.
  The finiteness, symmetry and triangle inequality are obvious.
  It remains to show $\mu=\nu$ if $d_{\BL^{\kappa}}^X(\mu,\nu)=0$.
  Since the constant function $1:X\to\{1\}$ satisfies $\|1\|_{\BL^{\kappa}}(X)\le1$, 
  we have $\mu(X)=\nu(X)$. Thus, by $\pi$-$\lambda$ theorem, it is enough to show $\mu(A)=\nu(A)$ for any closed set $A\subseteq X$.
  Fix a closed subset $A\subseteq X$ and set $f_n(x)=\frac{1}{1+nd^X(x,A)^{\kappa}},\;x\in X$ for $n\in\N$.
  Then it is not difficult to check the following:
\begin{align*}
  0\le f_n\le f_{n+1},\qquad \lim_{n\to\infty}f_n(x)=1_A(x),\;x\in X,\qquad \|f_n\|_{\BL^{\kappa}(X)}<\infty.
\end{align*}
Thus we deduce$\int_{X}f_nd\mu=\int_{X}f_nd\nu$, and letting $n\to\infty$, we obtain
$\mu(A)=\nu(A)$. This completes the proof.
 \end{proof}

Recall that we say a metric space $(X,d^X)$ is a $\mathsf{bcm}$ space 
if and only if any closed bounded set of $(X,d^X)$ is compact.
The notation
$\bcm$ is the abbreviation for boundedly compact metric.

\begin{prp}\label{blproperties}
Let $(X,d^X)$ be a $\bcm$ space and $\kappa\in (0,1]$.
Then the following hold.
\begin{itemize}
  \item [$\mathrm{(i)}$]$(\calM(X)_{\mathrm{fin}},d_{\BL^{\kappa}}^X)$ is complete.
  \item [$\mathrm{(ii)}$]$(\calM(X)_{\mathrm{fin}},d_{\BL^{\kappa}}^X)$ is separable.
  \item [$\mathrm{(iii)}$]Convergence with respect to $d_{\BL^{\kappa}}^X$ is equivalent to weak convergence.
  \item [$\mathrm{(iv)}$]Let $(Y,d^Y)$ be a metric space and fix an isometric embedding $f:X\to Y$.
  Then $\calM(X)_{\mathrm{fin}}\ni \mu\mapsto \mu\circ f^{-1}\in \calM(Y)_{\mathrm{fin}}$ is isometric 
  with respect to $d^X_{\BL^{\kappa}}$ and $d^Y_{\BL^{\kappa}}$. 
\end{itemize}
\end{prp}
\begin{rmk}
  \cref{blproperties} also holds for Polish spaces regardless of the choice of metric, except for (i). 
  For instance, in an open interval $(0,1)$ equipped with the usual Euclidean metric,
  $(\delta_{1/n})_{n\in\N}$ is a Cauchy sequence that does not converge weakly in $\calM((0,1))_{\mathrm{fin}}$.
\end{rmk}

\begin{proof}[Proof of $\mathrm{(i)}$]
  Let $(\mu_j)_j$ be a Cauchy sequence in $(\calM(X)_{\mathrm{fin}},d_{\BL^{\kappa}}^X)$.
  To begin with, we will show the tightness of $(\mu_j)$.
  Fix $x_0\in X$. For $R>1$, let $\bar{f_R}:X\to [0,1]$ be a McShane extension of $1_{B(x_0,R)}:B(x_0,R)\cup B(x_0,2R)^c\to \R$
  with exponent $\kappa$.  Note that
  $\bar{f_R}$ satisfies 
  \begin{align*}
    \bar{f_R}|_{B(x_0,R)}=1,\quad \supp\bar{f_R}\subseteq D(x_0,2R),\;\;and\;\;\|\bar{f_R}\|_{\Lip^{\kappa}}=\|1_{B(x_0,R)}\|_{\Lip^{\kappa}(B(x_0,R)\cup B(x_0,2R)^c)}\leq R^{-\kappa}.
  \end{align*}
  Take $\varepsilon>0$ arbitrarily and choose $K\in \N$ such that 
  $d_{\BL^{\kappa}}^X(\mu_i,\mu_j)<\varepsilon$ holds for all $i,j\geq K$.
  Then, it holds from the finiteness of $\mu_K$ that
  \begin{align*}
    \lim_{R\to\infty}\limsup_{j\to\infty}\mu_j(X\setminus B(x_0,2R))
    \leq&\limsup_{R\to\infty}\limsup_{j\to\infty}\int_X \Big[1-\bar{f_R}\Big]d\mu_j\\
    \leq&\limsup_{R\to\infty}\limsup_{j\to\infty}\int_X \Big[1-\bar{f_R}\Big]d(\mu_j-\mu_K)+\limsup_{R\to\infty}\int_X \Big[1-\bar{f_R}\Big]d\mu_K\\
    \leq&\limsup_{R\to\infty}\limsup_{j\to\infty}(1+R^{-\kappa})d_{\BL^{\kappa}}^X(\mu_j,\mu_K)+\limsup_{R\to\infty}\mu_K(X\setminus B(x_0,R))\\
    \leq&\varepsilon.
  \end{align*}
 Now, the tightness of $(\mu_j)$ follows from the relative compactness of $B(x_0,2R)$.
 Also, we can easily check the boundedness of $(\mu_j(X))_j$ using 
 \begin{align*}
  |\mu_i(X)-\mu_j(X)|=\left| \int_X 1d(\mu_i-\mu_j)\right|\le d_{\BL^{\kappa}}^X(\mu_i,\mu_j),
 \end{align*}
 which shows that $(\mu_j(X))_j$ is a Cauchy sequence.
 Combining these facts, we obtain a weakly convergent sequence $(\mu_{j_k})_k$ and its limit 
 $\mu\in \calM(X)_{\mathrm{fin}}$. For the above $K\in\N$ and $i\geq K$, we have
 \begin{align*}
  d_{\BL^{\kappa}}(\mu,\mu_i)
  =&\sup_{\|f\|_{\BL^{\kappa}(X)}\leq 1}\left| \int_X fd(\mu-\mu_i) \right|\\
  =&\sup_{\|f\|_{\BL^{\kappa}(X)}\leq 1}\lim_{k\to\infty}\left| \int_X fd(\mu_{j_k}-\mu_i) \right|\\
  \leq&\limsup_{k\to\infty}d_{\BL^{\kappa}}^X(\mu_{j_k},\mu_i)\\
  \leq&\varepsilon.
 \end{align*}
Therefore $(\mu_j)$ converges to $\mu$ in $d_{\BL^{\kappa}}^X$.
\end{proof}
\begin{proof}[Proof of $\mathrm{(ii)}$]
  Note that $(X,d^X)$ is separable. Take a countable dense subset $(x_j)\subseteq X$
  and define a countable subset $\calA\subseteq \calM(X)_{\mathrm{fin}}$ by
  \begin{align*}
    \calA:=\left\{ \sum_{j=1}^{n}q_j\delta_{x_j}:q_j\in\Q_{\geq 0},n\in\N \right\}.
  \end{align*}
  We will show the denseness of $\calA$ in $(\calM(X)_{\mathrm{fin}},d_{\BL^{\kappa}}^X)$.\\
  Fix $\mu\in \calM(X)_{\mathrm{fin}}$ and $\varepsilon>0$. We may assume that $\mu\neq 0$ and, thus we can define
  $r=\left(\frac{\varepsilon}{\mu(X)}\right)^{1/\kappa}\in (0,\infty)$. Take a compact set $K\subseteq X$ with $\mu(X\setminus K)<\varepsilon$.
  We can choose $x_1,\dots,x_l$ such that $(B(x_j,r))_j$ covers $K$. 
  When we define $C_1,\dots,C_l$ by setting
  \begin{align*}
    C_1=K\cap B(x_1,r),\quad C_{j+1}=(B(x_{j+1},r)\cap K)\setminus C_j,\quad 1\leq j\leq l-1, 
  \end{align*} 
  then $(C_j)_{j=1}^l$ forms a partition of $K$ satisfying $\sup_{x\in C_j}d^X(x,x_j)\leq r$.
  Choosing $q_1,\dots,q_l\in \Q_{\geq 0}$ such that
  $\sum_{j=1}^l |\mu(C_i)-q_i|<\varepsilon$, we define $\nu\in\calA$ by
  $\nu:=\sum_{j=1}^{l}q_j\delta_{x_j}$. For any $f:X\to\R$ with $\|f\|_{\BL^{\kappa}(X)}\leq 1$, it holds that
  \begin{align*}
    \left| \int_Xfd(\mu-\nu) \right|
    \leq&\int_{X\setminus K}|f|d\mu+\left| \sum_{j=1}^{l}\int_{C_j} \Big[f(x)-f(x_j)\Big]\mu(dx) \right|+\left| \sum_{j=1}^{l}f(x_j)(\mu(C_j)-q_j) \right|\\
    \leq&\varepsilon+r^{\kappa}\mu(K)+\varepsilon\\
    \leq&3\varepsilon.
  \end{align*} 
  Therefore, $\calA$ is dense in $(\calM(X),d_{\BL^{\kappa}}^X)$. 
\end{proof}
\begin{proof}[Proof of $\mathrm{(iii)}$]
It suffices to show the case $\kappa=1$ since the two metrics $d^X(\cdot,\cdot)$ and $d^X(\cdot,\cdot)^{\kappa}$
define the same topology on $X$. For that case, the assertion was shown in \cite[Theorem 11.3.3]{Dud02}.
  
\end{proof}
\begin{proof}[Proof of $\mathrm{(iv)}$]
  Take $\mu,\nu\in \calM(X)_{\mathrm{fin}}$ arbitrarily. Noting $\|g\circ f\|_{\BL^{\kappa}(X)}\leq 1$ for any $g:Y\to\R$ with 
  $\|g\|_{\BL^{\kappa}(X)}\leq 1$, we obtain
  \begin{align*}
    \int_{Y} g d(\mu\circ f^{-1}-\mu\circ f^{-1})=\int_X g\circ fd(\mu-\nu)\leq d^X_{\BL^{\kappa}}(\mu,\nu).
  \end{align*}
 This implies $d^Y_{\BL^{\kappa}}(\mu\circ f^{-1},\nu\circ f^{-1})\leq d^X_{\BL^{\kappa}}(\mu,\nu)$.
 We next fix $h:X\to \R$ satisfying $\|h\|_{\BL^{\kappa}(X)}\leq 1$. Since $f$ is an injection onto its image,
 \begin{align*}
  g_0=h\circ f^{-1}:f(X)\ni f(x)\mapsto h(x)\in\R
 \end{align*}  
 is well defined and satisfies $\|g_0\|_{\BL^{\kappa}(f(X))}\leq 1$.
 By \cref{mcshane}, we can find $g:Y\to \R$ such that it satisfies $\|g\|_{\BL^{\kappa}(Y)}\leq 1$ and $g|_{f(X)}=g_0$.
 It is then the case that
 \begin{align*}
  \int_{X} hd(\mu-\nu)
  =&\int_{f(X)}g_0d(\mu\circ f^{-1}-\nu\circ f^{-1})\\
  =&\int_{Y}gd(\mu\circ f^{-1}-\nu\circ f^{-1})\\
  \leq&d^Y_{\BL^{\kappa}}(\mu\circ f^{-1},\nu\circ f^{-1}).
 \end{align*}
 This completes the proof.
\end{proof}
It is an immediate consequence of the definition of $d^X_{\BL^{\kappa}}$ that the metric 
can be used to estimate the difference between the integral of a function $f\in\BL^{\kappa}(X)$ with respect to two measures.
That is, the following holds.
\begin{align*}
  \int_{X} fd\mu-\int_{X}fd\nu\leq \|f\|_{\BL^{\kappa}(X)}d_{\BL^{\kappa}}^X(\mu,\nu),\quad f\in\BL^{\kappa}(X).
\end{align*}
The metric is also useful for estimating the difference between the measure of a ball with respect to two measures.
\begin{lem}\label{ballineq}
For any $r,\rho>0,\;x\in X$ and $\mu,\nu\in \calM(X)_{\mathrm{fin}}$, it holds that
\begin{align*}
  \nu(B(x,r))\leq\mu(B(x,r+\rho))+(1+\rho^{-\kappa})d_{\BL^{\kappa}}^X(\mu,\nu).
\end{align*}
\end{lem} 
\begin{proof}
  Take $\varepsilon\in (0,\rho)$ arbitrarily.
  Apply \cref{mcshane} to $1_{B(x,r)}:B(x,r)\cup B(x,r+\rho-\varepsilon)^c\to [0,1]$ 
  and denote the resulting extension by $f_{\varepsilon}:X\to[0,1]$.
  Then $f_{\varepsilon}$ satisfies $\|f_{\varepsilon}\|_{\BL^{\kappa}(X)}\leq 1+(\rho-\varepsilon)^{-\kappa},$ $f_{\varepsilon}|_{B(x,r)}=1$ 
  and $\supp f_{\varepsilon}\subseteq D(x,r+\rho-\varepsilon)$. Therefore, it holds that
  \begin{align*}
    \nu(B(x,r))-\mu(B(x,r+\rho))
    \leq\int_X f_{\varepsilon}d\nu-\int_X f_{\varepsilon}d\mu
    \leq(1+(\rho-\varepsilon)^{-\kappa})d_{\BL^{\kappa}}^X(\mu,\nu).
  \end{align*}
  Letting $\varepsilon\to 0$, we obtain the result.
\end{proof}

For $\bcm$ $(X,d^X)$, set
\begin{align*}
  \calM(X):=\{\mu:\mu\;\mathrm{is\;a\;Radon\;measure\;on\;}X.\}.
\end{align*} 
We next consider an analog of $d_{\BL^{\kappa}}^X$ for this space.
Take $\rho_X\in X$ arbitrarily and define $d^{X,\rho_X}_{\BL^{\kappa}}$ by setting
\begin{align}\label{rootedblkappa}
  d^{X,\rho_X}_{\BL^{\kappa}}(\mu,\nu):=\int_{0}^{\infty}e^{-r}(1\wedge d_{\BL^{\kappa}}^X(\mu^{(r)}_{\rho_X},\nu^{(r)}_{\rho_X}))dr,\quad \mu,\nu\in \calM(X).
\end{align}
Here, $\mu^{(r)}_{\rho_X}\in\calM(X)_{\mathrm{fin}}$ is defined by 
$\mu^{(r)}_{\rho_X}(\cdot)=\mu(\cdot\cap D(\rho_X,r))$ for $\mu\in\calM(X)$.
\begin{prp}\label{rootedblproperties}
  For a boundedly compact metric space $(X,d^X)$ and $\rho_X\in X$, we have the following.
  \begin{itemize}
    \item [$\mathrm{(i)}$]The function $d^{X,\rho_X}_{\BL^{\kappa}}$ is a well-defined complete, separable metric on $\calM(X)$.
    \item [$\mathrm{(ii)}$]The topology of $(\calM(X),d^{X,\rho_X}_{\BL^{\kappa}})$ is independent of $\rho_X$ 
    and the convergence with respect to this topology is equivalent to vague convergence.
    \item [$\mathrm{(iii)}$]Let $(Y,d^Y)$ be a $\bcm$ space and fix an isometric embedding $f:X\to Y$.
  Then $\calM(X)\ni \mu\mapsto \mu\circ f^{-1}\in \calM(Y)$ is isometric 
  with respect to $d^{X,\rho_X}_{\BL^{\kappa}}$ and $d^{Y,f(\rho_X)}_{\BL^{\kappa}}$. 
   \end{itemize}
\end{prp}
We refer the reader to \cite[Section 2]{Nd24} for the definitions of terminologies used in the following proof.
\begin{proof}
  We first show the claim (i) and the first assertion of (ii)
  by confirming the assumptions of \cite[Corollary 2.14]{Nd24}.
  Since it is shown in \cite[Lemmas A.5 and A.6]{Nd24} that a restriction system $R$ from $\calM(X)$ to $\calM_{\rm{fin}}(X)$
   defined by $R^{(r)}_{x}(\mu)(\cdot):=\mu(\cdot\cap D_{d^X}(x,r))$
  satisfies condition 2 in \cite[p8]{Nd24},
  the function $d^{X,\rho_X}_{\BL^{\kappa}}$ is a metric, and the topology induced by $d^{X,\rho_X}_{\BL^{\kappa}}$
  is independent of $\rho_X$ by \cite[Corollary 2.14]{Nd24}.
  Also, since $(X,d^X)$ is complete, and the restriction system $R$ 
  above satisfies condition 4 and is complete in the sense of \cite[Definition 2.8]{Nd24}
 (see \cite[Lemma 3.18]{Nd24}),   $d^{X,\rho_X}_{\BL^{\kappa}}$ is separable and complete by \cite[Corollary 2.14]{Nd24}.

 We next show the latter assertion of (ii).
 Since convergence with respect to $d^X_{\BL^{\kappa}}$ (in $\calM(X)$) is equivalent to weak convergence,
 using \cite[Theorem 2.11]{Nd24}, we can check that $d^{X,\rho_X}(\mu,\mu_n)\to 0$ is synonymous with the following:
 \begin{itemize}
  \item There exists a sequence $x_n\in X$ converging to an element $x\in X$ such that 
  $\mu_n(\cdot\cap D_{d^X}(x_n,r))$ converges to $\mu(\cdot\cap D_{d^X}(x,r))$ weakly
  for all but countably many $r>0$.
 \end{itemize} 
  Also, it follows from  \cite[Theorem 3.20]{Nd24} that the above condition is equivalent
  to vague convergence. Thus we have completed the proof of (ii).

  The claim (iii) is obvious from $(\mu\circ f^{-1})^{(r)}_{f(\rho_X)}=(\mu^{(r)}_{\rho_X})\circ f^{-1}$ and \cref{blproperties}(iv).
\end{proof}

\begin{rmk}\label{srmoftau}
By the above proposition, we can check that the space-rooted metrization $\tau^{\mathrm{srm}}:\rBCM\to\Met$ of the functor
$\tau:\BCM\to\MTop$ which is defined by:
\begin{itemize}
  \item $\tau(X)=\calM(X)$ equipped with the vague topology,
  \item $\tau_f$ is a pushforward map given by $f$,
\end{itemize}
is given by
\begin{itemize}
  \item $\tau^{\mathrm{srm}}((X,\rho_X))=(\calM(X),d^{X,\rho_X}_{\BL^{\kappa}})$.
\end{itemize}
For precise definitions of these terms, see \cite[Section 6.1]{Nd24}.
\end{rmk}
\subsection{Associated Gromov-Hausdorff-type topology}\label{ghtype}
We now consider a Gromov-Hausdorff type topology defined
using $\BL^{\kappa}$ metric.
The interested reader is referred to \cite{Khe23} or \cite{Nd24}
 for a more general theory of Gromov-Hausdorff-type topologies.
For later use, we introduce terminologies.
\begin{itemize}
  \item A measured metric space is a triplet $(X,d^X,\mu_X)$
such that $(X,d^X)$ is a metric space, and $\mu_X\in \calM(X)$.
  \item A rooted metric space is a triplet $(X,d^X,\rho_X)$
such that $(X,d^X)$ is a metric space, and $\rho_X\in X$ is a distinguished point called a root.
\item A root-and-measured metric space is a quadruple $(X,d^X,\rho_X,\mu_X)$
such that $(X,d^X,\rho_X)$ is a rooted metric space, 
and $\mu_X\in \calM(X)$. 
\item A root-preserving map is a map between two rooted metric spaces such that it maps 
the root of the domain to that of the codomain.
\item We say two measured metric spaces $(X,d^X,\mu_X)$ and $(Y,d^Y,\mu_Y)$
are equivalent if and only if there exists an isometry $f:X\to Y$ 
with $\mu_Y=\mu_X\circ f^{-1}$.
\item We say two rooted metric spaces $(X,d^X,\rho_X)$ and $(Y,d^Y,\rho_Y)$
are equivalent if and only if there exists a root-preserving isometry $f:X\to Y$.
\item We say two root-and-measured metric spaces $(X,d^X,\rho_X,\mu_X)$ and $(Y,d^Y,\rho_Y,\mu_Y)$
are equivalent if and only if there exists a root-preserving isometry $f:X\to Y$
with $\mu_Y=\mu_X\circ f^{-1}$.
\end{itemize}

\subsubsection{Rooted compact case}
We begin with the rooted compact case. 
Note that if $X$ is compact then it holds that $\calM(X)=\calM(X)_{\rm{fin}}$.
\begin{dfn}[Hausdorff distance, {\cite[Section 7.3.1]{BBI01}}]\label{haustop}
Let $(Z,d^Z)$ be a metric space and let $\calC_c(Z)$ be the set of all compact subsets in $Z$.
For $A,B\in\calC_c(Z)$, define the Hausdorff distance between the two $d_{\rm{H}}^Z(A,B)$
by setting
 \begin{equation*}
  d^Z_{\mathrm{H}}(A,B):=\inf\{\varepsilon>0:A\subseteq B^{\varepsilon},B\subseteq A^{\varepsilon}\},
 \end{equation*}
 and
 \begin{equation*}
  C^{\varepsilon}:=\{z\in Z:\exists c\in C,\;d^Z(z,c)\leq \varepsilon\},\quad C\subseteq Z.
 \end{equation*}
 Here, we use a convention $\inf\emptyset:=\infty$.
\end{dfn}
It is known that $d_{\rm{H}}^Z$ is indeed an extended metric on $\calC_c(Z)$, i.e., a metric that can take $\infty$ as its value.
The induced topology is called the Hausdorff topology, and it only depends on the topology of $Z$ (see \cite[Section 4.F]{Ke95}).

\begin{prp}[{\cite[Proposition 6.2]{Nd24}}]
There exists a set $\frK_{\bullet}(\tau)$ satisfying the following:
\begin{itemize}
  \item $\frK_{\bullet}(\tau)$ consists of root-and-measured compact metric spaces,
  \item for any root-and-measured compact metric space $(Y,d^Y,\rho_Y,\mu_Y)$, 
   there exists a unique element $(X,d^X,\rho_X,\mu_X)\in\frK_{\bullet}(\tau)$
   that is equivalent to $(Y,d^Y,\rho_Y,\mu_Y)$.
\end{itemize}
\end{prp}
\begin{dfn}\label{rcop}
For $\calX=(X,d^X,\rho_X,\mu_X),\;\calY=(Y,d^Y,\rho_Y,\mu_Y)\in\frK_{\bullet}(\tau)$, we define
\begin{align*}
  d^{\tau,\BL^{\kappa}}_{\frK_{\bullet}}(\calX,\calY):=\inf_{f,g,Z}\Big\{ d^Z_{\mathrm{H}}(f(X),g(Y))\vee d^Z(f(\rho_X),g(\rho_Y))\vee d^Z_{\BL^{\kappa}}(\mu_X\circ f^{-1},\mu_Y\circ g^{-1})\Big\},
\end{align*}
where the infimum is taken over all compact metric spaces $(Z,d^Z)$ and isometric embeddings $f:X\to Z$ and 
$g:Y\to Z$.
\end{dfn}
\begin{rmk}
By considering the disjoint union, one can check that the set of admissible triples $(f,g,Z)$ is not empty.
\end{rmk}

The following follows from 
\cite[Theorem 2.6]{Khe23}.
\begin{prp}
The function $d^{\tau,\BL^{\kappa}}_{\frK_{\bullet}}$ is a metric on $\frK_{\bullet}(\tau)$.
\end{prp}
Recall that the rooted compact Gromov-Hausdorff-Prokhorov topology
is the topology on $\frK_{\bullet}(\tau)$ induced by the following metric $d_{\frK}^{\tau,P}$:
\begin{align*}
  d^{\tau,P}_{\frK}(\calX,\calY):=\inf_{f,g,Z}\Big\{ d^Z_{\mathrm{H}}(f(X),g(Y))\vee d^Z(f(\rho_X),g(\rho_Y))\vee d^Z_{P}(\mu_X\circ f^{-1},\mu_Y\circ g^{-1})\Big\},
\end{align*}
where the infimum is taken over all compact metric spaces $(Z,d^Z)$ and isometric embeddings $f:X\to Z$ and 
$g:Y\to Z$, and $d^Z_{P}$ is the Prokhorov distance.
\begin{thm}
For $\calX_n=(X_n,d^{X_n},\rho_{X_n},\mu_{X_n}),\;\calX=(X,d^X,\rho_X,\mu_X)\in\frK_{\bullet}(\tau)$,
the following conditions are equivalent to each other.
\begin{itemize}
  \item [$\mathrm{(i)}$]There exist a compact metric space $(M,d^M,\rho_M)$ and root-preserving isometric embeddings $f_n:X_n\to M,\;f:X\to M$
  such that $f_n(X_n)\to f(X)$ in the Hausdorff topology, and $\mu_{X_n}\circ f_n^{-1}\to \mu_{X}\circ f^{-1}$ weakly.  
  \item [$\mathrm{(ii)}$]There exist a compact metric space $(M,d^M)$ and isometric embeddings $f_n:X_n\to M,\;f:X\to M$
  such that $f_n(X_n)\to f(X)$ in the Hausdorff topology, $f_n(\rho_{X_n})\to f(\rho_X)$ in $M$,
  and $\mu_{X_n}\circ f_n^{-1}\to \mu_{X}\circ f^{-1}$ weakly.
  \item [$\mathrm{(iii)}$]$d^{\tau,\BL^{\kappa}}_{\frK_{\bullet}}(\calX_n,\calX)\to 0$
\end{itemize}
In particular, $d^{\tau,\BL^{\kappa}}_{\frK_{\bullet}}$ induces the rooted compact Gromov-Hausdorff-Prokhorov topology.
\end{thm}
\begin{proof}
  The equivalence between (i) and (iii) follows from \cite[Lemma 2.5]{Khe23}, and
  that between (i) and (ii) follows from \cite[Lemma 3.17 and Theorem B.9]{Nd24}.
\end{proof}

\begin{lem}
For an isometric embedding between two compact metric spaces $f:X\to Z$, define $A_f\subseteq Z\times \calM(Z)_{\mathrm{fin}}$ by setting
\begin{align*}
  A_f:=\{(f(x),\mu\circ f^{-1}):x\in X,\;\mu\in \calM(X)_{\mathrm{fin}}\}.
\end{align*}
Let $f:X\to Z$ and $f_n:X_n\to Z,\;n\in\N,$ be isometric embeddings between compact metric spaces $(X,d^X),\;(X_n,d^{X_n})$ and $(Z,d^Z)$.
Suppose that $\lim_{n\to\infty}d^{Z}_{\mathrm{H}}(f(X),f_n(X_n))=0$. It is then the case that
\begin{align*}
  \lim_{n\to\infty}d_{\mathrm{H}}(A_f,A_{f_n})= 0.
\end{align*}
Here, we equip $Z\times \calM(Z)_{\mathrm{fin}}$ with the maximum metric $d$, which is defined by 
\begin{align*}
  d((z,\lambda),(w,\theta)):=d^Z(z,w)\vee d^Z_{\BL^{\kappa}}(\lambda,\theta),\qquad (z,\lambda),(w,\theta)\in Z\times \calM(Z)_{\mathrm{fin}},
\end{align*}
and we denote the Hausdorff distance between $A_f$ and $A_{f_n}$ with respect to $d$ by $d_{\mathrm{H}}(A_f,A_{f_n})$.
\end{lem}
\begin{proof}
  This follows from \cite[Lemma 2.13]{Khe23} and the fact that the Hausdorff topology on a metric space 
  only depends on the topology of the space (see \cite[Section 4.F]{Ke95}).
\end{proof}
The following theorem is an immediate consequence of above lemma, \cite[Theorem 2.12]{Khe23} and \cref{blproperties}.
\begin{thm}
The metric space $(\frK_{\bullet}(\tau),d^{\tau,\BL^{\kappa}}_{\frK_{\bullet}})$ is complete and separable.
\end{thm}
 
\subsubsection{Non-rooted compact case}\label{nrc}
We next consider the non-rooted compact case.

Similarly to \cite[Proposition 6.2]{Nd24}, we can show the following proposition.
\begin{prp}
  There exists a set $\frK(\tau)$ satisfying the following:
\begin{itemize}
  \item $\frK(\tau)$ consists of measured compact metric spaces,
  \item for any measured compact metric space $(Y,d^Y,\mu_Y)$, 
   there exists a unique element $(X,d^X,\mu_X)\in\frK(\tau)$
   such that equivalent to $(Y,d^Y,\mu_Y)$.
\end{itemize}
\end{prp}

\begin{dfn}\label{nrca}
For $\calX=(X,d^X,\mu_X),\;\calY=(Y,d^Y,\mu_Y)\in\frK(\tau)$, we define
\begin{align*}
  d^{\tau,\BL^{\kappa}}_{\frK}(\calX,\calY):=\inf_{f,g,Z}\Big\{ d^Z_{\mathrm{H}}(f(X),g(Y))\vee d^Z_{\BL^{\kappa}}(\mu_X\circ f^{-1},\mu_Y\circ g^{-1})\Big\},
\end{align*}
where the infimum is taken over all compact metric spaces $(Z,d^Z)$ and isometric embeddings $f:X\to Z$ and 
$g:Y\to Z$.
\end{dfn}

The following follows from 
\cite[Theorem 2.6]{Khe23}.
\begin{prp}
The function $d_{\frK}^{\tau,\BL^{\kappa}}$ is a metric on $\frK(\tau)$.
\end{prp}
It has similar properties to $d_{\frK_{\bullet}}^{\tau,\BL^{\kappa}}$.
Recall that the compact Gromov-Hausdorff-Prokhorov topology
is the topology on $\frK(\tau)$ induced by the following metric $d_{\frK}^{\tau,P}$:
\begin{align*}
  d^{\tau,P}_{\frK}(\calX,\calY):=\inf_{f,g,Z}\Big\{ d^Z_{\mathrm{H}}(f(X),g(Y))\vee d^Z_{P}(\mu_X\circ f^{-1},\mu_Y\circ g^{-1})\Big\},
\end{align*}
where the infimum is taken over all compact metric spaces $(Z,d^Z)$ and isometric embeddings $f:X\to Z$ and 
$g:Y\to Z$, and $d^Z_{P}$ is the Prokhorov distance.

\begin{thm}\label{nonrootcpt}
For $\calX_n=(X_n,d^{X_n},\mu_{X_n}),\;\calX=(X,d^X,\mu_X)\in\frK(\tau)$,
the following conditions are equivalent to each other.
\begin{itemize}
  \item [$\mathrm{(i)}$]There exist a compact metric space $(M,d^M)$ and isometric embeddings $f_n:X_n\to M,\;f:X\to M$
  such that $f_n(X_n)\to f(X)$ in the Hausdorff topology, and $\mu_{X_n}\circ f_n^{-1}\to \mu_{X}\circ f^{-1}$ weakly.  
  \item [$\mathrm{(ii)}$]$d^{\tau,\BL^{\kappa}}_{\frK}(\calX_n,\calX)\to 0$
\end{itemize}
In particular, $d_{\frK}^{\tau,\BL^{\kappa}}$ induces the compact Gromov-Hausdorff-Prokhorov topology.
\end{thm}
\begin{proof}
  This follows from \cite[Lemma 2.5]{Khe23}.
\end{proof}

\begin{lem}
For an isometric embedding between two compact metric spaces $f:X\to Z$, define $A_f\subseteq \calM(Z)_{\mathrm{fin}}$ by setting
\begin{align*}
  A_f:=\{\mu\circ f^{-1}:\mu\in \calM(X)_{\mathrm{fin}}\}.
\end{align*}
Let $f:X\to Z$ and $f_n:X_n\to Z,\;n\in\N,$ be isometric embeddings between compact metric spaces $(X,d^X),\;(X_n,d^{X_n})$ and $(Z,d^Z)$.
Suppose that $\lim_{n\to\infty}d^{Z}_{\mathrm{H}}(f(X),f_n(X_n))=0$. It is then the case that
\begin{align*}
  \lim_{n\to\infty}(d^{Z}_{\BL^{\kappa}})_{\mathrm{H}}(A_f,A_{f_n})= 0.
\end{align*}
Here, we denote the Hausdorff metric on $\calM(Z)_{\mathrm{fin}}$ with respect to $d^{Z}_{\BL^{\kappa}}$
by $(d^{Z}_{\BL^{\kappa}})_{\mathrm{H}}$.
\end{lem}

\begin{proof}
  This follows from \cite[Lemma 2.13]{Khe23} and the fact that the Hausdorff topology on a metric space 
  only depends on the topology of the space (see \cite[Section 4.F]{Ke95}).
\end{proof}
The following theorem is an immediate consequence of above lemma, \cite[Theorem 2.12]{Khe23} and \cref{blproperties}.
\begin{thm}
The metric space $(\frK(\tau),d^{\tau,\BL^{\kappa}}_{\frK})$ is complete and separable.
\end{thm}

\subsubsection{Rooted $\bcm$ case}
We finally consider the non-compact case.
The spirit is similar to the compact case, but 
we need to use $d_{\BL^{\kappa}}^{X,\rho_X}$ for the metric on $\calM(X)$ 
since $d_{\BL^{\kappa}}^X$ is no longer available.
\begin{dfn}[The local Hausdorff distance, {\cite[Section 3.1]{Nd24}}]\label{fell}
Let $(Z,d^Z,\rho_Z)$ be a rooted $\bcm$ space and $\calC(Z)$ be the set of all closed sets in $Z$.
We define the local Hausdorff distance $d^{Z,\rho_Z}_{\rm{H}}(A,B)$ for $A,B\in\calC(Z)$ by setting
\begin{equation*}
    d_{\mathrm{H}}^{Z,\rho_Z}(A,B):=\int_{0}^{\infty}e^{-r}(1\wedge d^{Z}_{\mathrm{H}}(A^{(r,\rho_Z)},B^{(r,\rho_Z)}))dr,
\end{equation*}
where
\begin{equation*}
      C^{(r,\rho_Z)}:=C\cap D(\rho_Z,r),\quad C\subseteq Z.
\end{equation*}
\end{dfn}
One can check that $d^{Z,\rho_Z}_{\rm{H}}$
is indeed a metric on $\calC(Z)$, and the induced topology, called 
Fell topology is independent of the root $\rho_Z\in Z$
 (see \cite[Section 3.1]{Nd24} for details).

\begin{prp}[{\cite[Proposition 6.2]{Nd24}}]
There exists a set $\frM_{\bullet}(\tau)$ satisfying the following:
\begin{itemize}
  \item $\frM_{\bullet}(\tau)$ consists of root-and-measured $\bcm$ spaces,
  \item for any root-and-measured $\bcm$ space $(Y,d^Y,\rho_Y,\mu_Y)$, 
   there exists a unique element $(X,d^X,\rho_X,\mu_X)\in\frM_{\bullet}(\tau)$
   that is equivalent to $(Y,d^Y,\rho_Y,\mu_Y)$.
\end{itemize}
\end{prp}

Recall the definition of $d^{Z,\rho_Z}_{\BL^{\kappa}}$ from \eqref{rootedblkappa}. 
\begin{dfn}
For $\calX=(X,d^X,\rho_X,\mu_X)$ and $\calY=(Y,d^Y,\rho_Y,\mu_Y)\in\frM_{\bullet}(\tau)$,
set
\begin{align*}
  d^{\tau,\BL^{\kappa}}_{\frM_{\bullet}}(\calX,\calY):=\inf_{f,g,Z}\Big\{ d_{\mathrm{H}}^{Z,\rho_Z}(f(X),g(Y))\vee d^{Z,\rho_Z}_{\BL^{\kappa}}(\mu_X\circ f^{-1},\mu_Y\circ g^{-1})\Big\},
\end{align*}
where the infimum is taken over all $(Z,d^Z,\rho_Z)\in\frM_{\bullet}$ and root-preserving isometric embeddings $f:X\to Z$ and 
$g:Y\to Z$.
\end{dfn}

\begin{thm}\label{232}
For $\calX_n=(X_n,d^{X_n},\rho_{X_n},\mu_{X_n}),\calX=(X,d^X,\rho_X,\mu_X)\in\frM_{\bullet}(\tau)$,
the following conditions are equivalent to each other.
\begin{itemize}
  \item [$\mathrm{(i)}$]There exist a $\bcm$ space $(M,d^M)$ and root-preserving isometric embeddings $f_n:X_n\to M,\;f:X\to M$
  such that $f_n(X_n)\to f(X)$ in the Fell topology, and $\mu_{X_n}\circ f_n^{-1}\to \mu_{X}\circ f^{-1}$ vaguely.  
  \item [$\mathrm{(ii)}$]There exist a $\bcm$ space $(M,d^M)$ and isometric embeddings $f_n:X_n\to M,\;f:X\to M$
  such that $f_n(X_n)\to f(X)$ in the Fell topology, $f_n(\rho_{X_n})\to f(\rho_X)$ in $M$,
  and $\mu_{X_n}\circ f_n^{-1}\to \mu_{X}\circ f^{-1}$ vaguely.
  \item [$\mathrm{(iii)}$]$d^{\tau,\BL^{\kappa}}_{\frM_{\bullet}}(\calX_n,\calX)\to 0$.
  \item [$\mathrm{(iv)}$]The sequence $(\calX_n)_n$ converges to $\calX$ in the Gromov-Hausdorff-vague topology.
\end{itemize}
\end{thm}
\begin{proof}
  The equivalence between (i), (ii) and (iv) was shown in \cite[Theorem 8.10]{Nd24},
  and that between (i) and (iii) follows from 
  \cref{srmoftau} and \cite[Theorem 6.18]{Nd24}.
\end{proof}
\begin{thm}
The metric space $(\frM_{\bullet}(\tau),d^{\tau}_{\frM_{\bullet}})$ is complete and separable.
\end{thm}
\begin{proof}
  By \cite[Theorem 8.9]{Nd24} the functor $\tau$ defined in \cref{srmoftau} is continuous
  (see \cite[Definitions 6.15, 6.37, and 6.39]{Nd24} for the definitions), and $\tau^{\rm{srm}}$
  defined in \cref{srmoftau} is complete in the sense of \cite[Definition 5.8]{Nd24} by \cref{blproperties}.
  Thus the assertion follows from \cite[Corollary 6.42]{Nd24}.
\end{proof}

\subsection{Resistance forms and associated basic notions}\label{resis}
This section is devoted to the definition of 
resistance forms and their basic properties.
The interested reader is referred to \cite{Kig12}
for details on resistance forms.
Although the term ``resistance form'' does not explicitly appear,
\cite[Chapter 9]{LP17} is also a valuable reference for the graph case.
\begin{dfn}[Resistance form, {\cite[Definition 3.1]{Kig12}}]\label{resisform}
Let $F$ be a non-empty set.
A pair $(\calE,\calF)$ is called a resistance form if it satisfies the following conditions.
\begin{itemize}[leftmargin=3.5em, labelsep=0.3em]
  \item [(RF1)]$\calF$ is a linear subspace of $\R^F$ containing the constant functions.
  $\calE$ is a non-negative symmetric quadratic form on $\calF$ satisfying
  \begin{align*}
    \calE(u,u)=0\iff \text{$u$ is a constant.}
  \end{align*}
  \item[(RF2)]Let $\sim$ be an equivalence relation on $\calF$ defined by $u\sim v$ if and only if $u-v$ is constant.
  Then $(\calF/\sim,\calE)$ is a Hilbert space.
  \item[(RF3)] If $x\neq y\in F$, then there exists $u\in\calF$ such that $u(x)\neq u(y)$.
  \item[(RF4)]For any $x,y\in F$,
  \begin{align*}
    R(x,y)
    :=&\sup\left\{ \frac{|u(x)-u(y)|^2}{\calE(u,u)}:u\in\calF,\calE(u,u)>0 \right\}\\
    =&\inf\{\calE(u,u):u(x)=1,u(y)=0,u\in\calF\}^{-1}
  \end{align*} 
  is finite.
  \item[(RF5)]For $u\in \calF$, set $\bar{u}:=0\vee(u\wedge 1)$. 
  Then $\bar{u}\in\calF$ and $\calE(\bar{u},\bar{u})\leq \calE(u,u)$ hold. 
\end{itemize}
The function $R:F\times F\to[0,\infty)$ defined by (RF4) is a metric on $F$ (see \cite[Theorem 2.3.4]{Kig01}) and called the resistance metric associated with $(\calE,\calF)$.
\end{dfn}
More generally, we define an effective resistance $R(A,B)$ between subsets $A$ and $B$ of $F$ by setting
\begin{align*}
  R(A,B)=\inf\{\calE(u,u):u|_A=1,u|_B=0,u\in\calF\}^{-1},\quad \inf\emptyset:=\infty.
\end{align*}
We write $R(x, B)$ for $R(\{x\}, B)$ and $R(x, y) = R(\{x\}, \{y\})$.

Note that, for any $u\in\calF$ and $x,y\in \calF$, it holds that 
\begin{align}\label{resishol}
  |u(x)-u(y)|^2\leq\calE(u,u)R(x,y)
\end{align}
by (RF4). 
In particular, $u\in\calF$ is $1/2$-H\"older continuous
with respect to the resistance metric $R$,
and satisfies $\|u\|_{\Lip^{1/2}}\leq \calE(u,u)^{1/2}$.

\bigskip

For the following definition, recall the effective resistance on an electrical network with a finite vertex
set from \cite[Section 9.4]{LP17}. See also \cite[Section 2.1]{Kig01}
for the relation between Laplacians and resistance forms on finite sets.

\begin{dfn}[Resistance metric, {\cite[Definition 2.3.2]{Kig01}}]\label{resismet}
Let $F$ be a set. A function $R:F\times F\to[0,\infty)$ is called a resistance metric if and only if, 
for any non-empty finite set $V\subseteq F$, 
there exists an electrical network
$G$ with the vertex set $V$ such that the effective resistance on $G$ coincides with $R|_{V\times V}$.
\end{dfn}
A resistance metric is indeed a metric on $F$ (see \cite[Theorem 2.3.4]{Kig01}).
As explained by the following theorem, the term ``resistance metric'' in the above definition is compatible with 
the term ``resistance metric'' (associated with a resistance form) in \cref{resisform}.
\begin{thm}[{\cite[Theorem 2.3.6]{Kig01}}]\label{corrresismet}
Let $R$ be a resistance metric on a non-empty set $F$ in the sense of \cref{resismet}.
Then there exists a unique resistance form $(\calE,\calF)$ on $F$ such that $R=R_{(\calE,\calF)}$, 
where $R_{(\calE,\calF)}$ is a resistance metric associated with $(\calE,\calF)$. 
\end{thm}

In what follows, we assume that a non-empty set $F$ is equipped with a resistance form $(\calE,\calF)$,
 and the metric space $(F,R)$ is complete, separable and locally compact, where $R$ is the resistance metric corresponding to $(\calE,\calF)$.
\begin{dfn}[Regular resistance form, {\cite[Definition 6.2]{Kig12}}]\label{regresis}
The resistance form $(\calE,\calF)$ is called regular if and only if $\calF\cap C_c(F)$ is dense in $C_c(F)$ 
with respect to the supremum norm. Here, we denote the set of compactly supported continuous functions on $F$ by $C_c(F)$.
\end{dfn}

We next introduce related Dirichlet forms and stochastic processes.
Henceforth, we assume the regularity of a resistance form $(\calE,\calF)$ on $F$.
Note that, since all functions lie in $\calF$ is continuous by \eqref{resishol}, 
any compact resistance metric space is regular. 
Let $\mu$ be a (positive) Radon measure on $F$ such that $\supp\mu=F$,
and 
$L^2(F,\mu)$ be the set of (equivalence classes of) Borel measurable functions
$f:F\to[-\infty,\infty]$ such that $\int_F f^2d\mu<\infty$. We denote an inner product on $L^2(F,\mu)$ by 
$(\cdot,\cdot)_{L^2(F,\mu)}$.
Note that, since $\mu$ is of full support, each continuous function on $F$ can be naturally regarded as an element of $L^2(F,\mu)$.
Then $(L^2(F,\mu)\cap\calF,\calE_1),\;\calE_1(u,v):=\calE(u,v)+(u,v)_{L^2(F,\mu)}$ is a Hilbert space (see \cite[Lemma 9.2]{Kig12}).
Define $\calD\subseteq L^2(F,\mu)$ by setting
$\calD:=\overline{\calF\cap C_c(F)}^{\calE_1}\subseteq L^2(F,\mu)\cap\calF$.
By virtue of the regularity of $(\calE,\calF)$, we have from \cite[Theorem 9.4]{Kig12} that 
$(\calE,\calD)$ is a regular Dirichlet form on $L^2(F,\mu)$. (See \cite[Chapter 1.1]{FOT94} for the definition of regular Dirichlet forms.)
Moreover, standard theory gives us the existence of an associated Hunt process $((X_t)_{t\geq 0},(P_x)_{x\in F})$
(e.g. \cite[Theorem 7.2.1]{FOT94}). Note that such a process is, in general, only specified uniquely for starting
points outside a set of zero capacity. However, in this setting, every point has strictly positive capacity
(see \cite[Theorem 9.9]{Kig12}), and so the process is defined uniquely everywhere.
In the rest of this section, we fix a full support Radon measure $\mu$ on a regular resistance metric space $(F,R)$ 
and corresponding Hunt process $(X=(X_t)_{t\geq 0},(P_x)_{x\in F})$.
We denote the expectation with respect to $P_x$ by $E_x[\cdot]$. 

\bigskip

\begin{dfn}[Recurrence of resistance forms, {\cite[Definition 3.13]{Nd25}}]\label{rec}
Let $(\calE, \calF)$ be a resistance form on 
$F$ and write $R$ for the corresponding resistance metric.
We say that $(\calE, \calF)$ and $R$ are recurrent 
if and only if the following condition is satisfied:
\begin{enumerate}[align=left, labelwidth=\widthof{(RRF)},
    labelsep=1em, 
    leftmargin=!]
    \item[(RRF)] there exists an increasing sequence $(U_n)_{n\geq1}$ 
    of relatively compact open subsets of $F$ such that $\bigcup_{n\geq1}U_n=F$ and
    \begin{align*}
    \lim_{n \to \infty} R(\rho, U_n^c) = \infty  
    \end{align*}
    for some $\rho \in F$.
  \end{enumerate}
\end{dfn}
Note that this condition is equivalent to the following:
\begin{itemize}
  \item For any increasing sequence $(U_n)_{n\geq1}$ 
    of relatively compact open subsets of $F$ such that $\bigcup_{n\geq1}U_n=F$ and any $\rho\in F$,
    it holds that
    \begin{align*}
    \lim_{n \to \infty} R(\rho, U_n^c) = \infty.
    \end{align*}
\end{itemize}
In particular, if $(F,R)$ is a $\bcm$ space, then recurrence is equivalent to 
\begin{align}\label{reccc}
\lim_{r\to\infty}R(\rho,B(\rho,r)^c)=\infty  
\end{align}
 for some (or equivalently any) $\rho\in F$.
\begin{rmk}
  The recurrence of resistance forms implies their regularity.
  See \cite[Corollary 3.26]{Nd25} for details.
\end{rmk} 
\begin{rmk}
The recurrence of the resistance form $(\calE,\calF)$ is equivalent to that of the Dirichlet form $(\calE,\calD)$
associated with a measure $\mu$ of full support (see \cite[p55]{FOT94} for the definition of the recurrence of Dirichlet forms).
See \cite[Lemma 2.3]{Cr18} for a proof.
\end{rmk}

Throughout this article, we denote the hitting time of a subset $A\subseteq F$ by $\sigma_A$.
Namely,
\begin{align*}
  \sigma_A:=\inf\{t>0:X_t\in A\}.
\end{align*}
We abbreviate $\sigma_x:=\sigma_{\{x\}}$.
We also write the commute time between $x,y\in F$ by $\sigma_{x,y}:=\inf\{t>0:X_t=x,y\in X_{[0,t]}\}$,
where $X_{[0,t]}:=\{X_s:s\in[0,t]\}$.
(NB. The quantity $\sigma_{x,y}$ is not symmetric in $x,y$.)
\begin{lem}[Commute time identity,{\cite[Proof of Lemma 2.9]{CHK17}}]\label{cti}
If $F$ is compact, then
\begin{align*}
 E_{x}[\sigma_{x,y}]=E_x[\sigma_y]+E_y[\sigma_x]=R(x,y)\mu(F).
\end{align*}
\end{lem}
\begin{lem}[{\cite[Lemma 4.2]{Cr18}}]\label{4.2-a}\quad\par
  \begin{itemize}
    \item[$\mathrm{(i)}$] If $x\in F$, $A$ is a non-empty closed subset of $F\setminus\{x\}$ and $\delta\in (0,R(x,A))$, then
\begin{align*}
  P_x(\sigma_A\leq t)\leq 2\left[ 1-\left(\frac{R(x,A)-\delta}{R(x,A)+\delta}\right)e^{-\frac{ 2t }{\mu(B(x,\delta))(R(x,A)-\delta)}} \right],\qquad \forall t\geq 0.
\end{align*}
In particular, if $x,y\in F$, $\varepsilon\in (0,R(x,y)/2),\delta\in (0,R(x,y)-2\varepsilon)$, then
\begin{align*}
  P_x\left(\sigma_{\overline{B(y,\varepsilon)}}\leq t\right)\leq 4\left[\frac{\delta}{R(x,y)-2\varepsilon}+\frac{t}{\mu(B(x,r))(R(x,y)-2\varepsilon-\delta)}\right],\qquad \forall t\geq 0.
\end{align*}
\item [$\mathrm{(ii)}$]If $\rho\in F$ and $\delta\in (0,R(\rho,B(\rho,r)^c))$, then 
\begin{align*}
  P_{\rho}\left( \sigma_{B(\rho,r)^c}\leq t \right)\leq 4\left[ \frac{\delta}{R(\rho,B(\rho,r)^c)}+\frac{t}{\mu(B(\rho,\delta))(R(\rho,B(\rho,r)^c)-\delta)} \right],\quad t\geq 0.
\end{align*}
NB. If $B(\rho,r)^c=\emptyset$, then we interpret the right-hand side as being equal to zero.
  \end{itemize}
\end{lem}
\begin{rmk}
  In the original paper \cite{Cr18}, compactness of $F$ is assumed to show the above lemma.
  However, it is only used to ensure $R(x,A)>0$, but this the positivity follows from the regularity of $F$ (see \cite[Theorems 4.3 and 6.3]{Kig12}).  
\end{rmk}

Finally, we introduce a killed process and its heat kernel. 
For any non-empty open set $U\subseteq F$, 
we can define a process $X^U$ that behaves like $X$ until it exits $U$,
at which point it is killed and sent to a cemetery state. 
Precisely, $X^U$ is defined as follows, and is also a Hunt process (see \cite[Section 3.3]{ChF12}).
\begin{align*}
  X^U_t=
  \begin{dcases}
  X_t & t<\tau_{U},\\
  \partial & t\geq   \tau_{U},
  \end{dcases}
  \qquad  \tau_U:=\inf\{t>0:X_t\not\in U\},
\end{align*}
where $\partial$ is a cemetery state.
Note that if $U=F$ then $X^U=X$.
As shown in the following theorem, if $(F,R)$ is boundedly compact, 
then $X^U$ admits a jointly continuous heat kernel 
$p_U:(0,\infty)\times F\times F\to[0,\infty)$.
\begin{thm}[{\cite[Theorem 10.4]{Kig12}}]\label{104}
  Suppose that $(F,R)$ is boundedly compact and let $U$ be a non-empty open subset of $F$.
  Then there exists $p_U:(0,\infty)\times F\times F\to[0,\infty)$ which satisfies the following conditions:
  \begin{itemize}
    \item The function $p_U$ is continuous on $(0,\infty)\times F\times F$. Define $p^{t,x}_U(y):=p_U(t,x,y)$. Then $p^{t,x}$ lies in the domain of the Dirichlet form corresponding to $X^U$.
    \item It satisfies $p_U(t,x,y)=p_U(t,y,x)$ for any $(t,x,y)\in (0,\infty)\times F\times F$.
    \item For any Borel measurable function $u$ on $F$ and any $x\in X$, it holds that
    \begin{align*}
      E_x[u(X^U_t)]=\int_{X}p_U(t,x,y)u(y)\mu(dy).
    \end{align*}
    \item For any $t,\;s>0$ and $x,\;y\in X$, it holds that 
    \begin{align}\label{chap-kol}
      p_U(t+s,x,y)=\int_{X}p_U(t,x,z)p_U(s,y,z)\mu(dz).
    \end{align}
  \end{itemize}
  Furthermore, let $A$ be a Borel subset of $X$ which satisfies $0<\mu(A)<\infty$. 
  Then
  \begin{align}\label{kig104}
    p_U(t,x,x)\leq \frac{2\sup_{y\in A}R(x,y)}{t}+\frac{\sqrt{2}}{\mu(A)},
  \end{align}
for any $x\in X$ and any $t>0$.
\end{thm}
Finally, we summarize the definitions of the operators associated with the processes $X$
corresponding to $(F,R,\mu)\in\bbF^{\circ}$ 
(recall the definition of $\bbF^{\circ}$ from \eqref{bbfcirc}). 
We denote $(\calE,\calF)$ and $(\calE,\calD)$ the resistance form and the Dirichlet form corresponding to $(F,R,\mu)$, respectively.
\begin{itemize}
  \item The semigroup $(P_t)$ is defined by 
  \begin{align}\label{defsg}
    P_tf(x)=E_{x}[f(X_t)]=\int_{F}p(t,x,y)f(y)\mu(dy) ,\qquad x\in F,\;f\in L^2(F,\mu).
  \end{align}
  \item The resolvent $(G^{\alpha})$ is defined by 
  \begin{align*}
    G^{\alpha}f(x)=\int_{0}^{\infty}e^{-\alpha t}P_tf(x)dt,\qquad x\in F,\;f\in L^2(F,\mu).
  \end{align*}
  \item The generator $\Delta$ is characterized by 
  \begin{align*}
  \calE(f,g)=-(\Delta f,g),\qquad f\in\rm{Dom}(\Delta),\;g\in\calD,  
  \end{align*}
  where $\rm{Dom}(\Delta)\subseteq L^2(F,\mu)$ is the domain of $\Delta$.
  \item Since $-\Delta$ is a non-negative self adjoint operator, there exists a projection-valued measure $E$
   satisfying $A=\int_{0}^{\infty}\lambda E(d\lambda)$. For $u,\;v\in L^2(F,\mu)$, the spectral measure $E_{u,v}(d\lambda)$ of $-\delta$ is a measure on $[0,\infty)$
  defined by $E_{u,v}(A)=(E(A)u,v)_{L^2(F,\mu)}$. We often denote the spectral measure by $E_{\cdot,\cdot}(d\lambda)$. 
  \item The heat kernel is the function described at \cref{104},
   and $\alpha$-potential density $g^{\alpha}:F\times F\to[0,\infty)$ is defined by 
   \begin{align}\label{reshk}
    g^{\alpha}(x,y)=\int_{0}^{\infty}e^{-\alpha t}p(t,x,y)dt.
   \end{align}
   \item For a non-empty open set $U\subsetneq F$ and its complement $A:=F\setminus U$,
   the Green function $g_A:F\times F\to [0,\infty)$ of $X^U$ is characterized as the unique function satisfying the following 
   (see \cite[Theorem 4.1]{Kig12}):
   \begin{itemize}
    \item Set $\calF(A):=\{u\in\calF: u|_A=0\}$. For any $x\in X$, $g_A^x:=g_A(x,\cdot)\in\calF(A)$ and
    \begin{align}\label{repro}
    \calE(g^x_A,u)=u(x),  \qquad u\in\calF(A).
    \end{align}
   \end{itemize} 
\end{itemize}

\section{Continuity of heat kernels and $\alpha$-potential densities}\label{cont}
To establish estimates for semigroups and heat kernels, it is necessary to bound
$|\int f d\mu - \int f d\mu_n|$ for suitable functions $f$ (using $d_{\BL^{\kappa}}(\mu,\mu_n)$).
Typical examples of such $f$ are functions in the image of semigroups or resolvents.
Therefore, it is important to provide sharp bounds for the $\BL^{\kappa}$ norm of such functions.
To this end, this section is devoted to estimating the H\"older regularity of the heat kernel, 
its time derivative, and the $\alpha$-potential density.
While interesting in themselves, 
these estimates are crucial, as their optimality affect the exponents in our main result.

In this section, we fix an element $(F,R,\mu)\in\bbF^{\circ}$ (recall the definition of $\bbF^{\circ}$ from \eqref{bbfcirc}), 
and the Hunt process $(X=(X_t)_{t\geq 0},(P_x)_{x\in F})$ corresponding to the triplet $(F,R,\mu)$.
We denote the corresponding resistance form, generator, semigroup, resolvent, heat kernel and $\alpha$-potential density by
$(\calE,\calF),$ $\Delta,$ $(P_t)_{t},$ $(G^{\alpha})_{\alpha},$ $p,$ and $g^{\alpha}$, respectively (recall the definition of these objects from the end of \Cref{resis}). 
We abbreviate $p(t,x,\cdot)$ to $p^{t,x}$.
For a non-empty open set $U\subsetneq F$ and its complement $A:=F\setminus U$, we also write $\Delta_U,\;p_U,\;G^{\alpha}_A,\;G_A,\;g_A^{\alpha}$, 
and $g_A$
for the generator, heat kernel, resolvent, Green operator, $\alpha$-potential density, and Green function associated with the killed process $X^U$.
If $A$ is a singleton $\{x\}$, we abbreviate the index $\{x\}$ to $x$.
For instance, we write $g_{x}$ for $g_{\{x\}}$.
Finally, we assume throughout this paper that
every metric space under consideration consists of 
at least two distinct points (except for \Cref{undertop}).

In this section, we seek explicit formulas for the constants 
involved in the inequalities that we will prove for the heat kernels and potential densities,
specifically written in terms of the underlying space data. 
Such dependence is of particular importance when considering 
convergence rate estimates in non-compact settings.
We denote constants appearing in lemma (or proposition/theorem) X by $c_X$ or $c_X^{(n)}$, except for \cref{summ}.
Also, we write $A\lesssim B$ if there exists a universal constant $C$ such that $A\leq CB$.

For later use, we introduce various terms and assumptions.
We consider the following for $s_0,s_1,\theta>0$ and a Borel measure $\nu$ on $F$.
\begin{enumerate}[align=left, labelwidth=\widthof{$\rm{LRES}(\theta)$}, 
    labelsep=0.5em, 
    leftmargin=!]
  \item[$\rm{Reg}_{s_0,s_1}$]We say $\nu$ satisfies $\rm{Reg}_{s_0,s_1}$ if and only if there exist constants $\exists c_{u},c_l>0$ such that 
  \begin{align*}
  c_l r^{s_0}\leq \nu(B(x,r))\leq  c_u r^{s_1},  
  \end{align*} 
  for any $ x\in F,r\in [0,\diam F)$.
  \item [$\rm{LRES}(\theta)$]There exists $c_{\rm{LR}}>0$ such that 
  \begin{align*}
  R(x,B(x,r)^c)\geq c_{\rm{LR}}r^{\theta},
  \end{align*}
   for any $B(x,r)\neq F$.
  \item [$\rm{UP}$]There exists $c_{\rm{UP}}\in (0,1)$ such that 
  \begin{align*}
  B(x,r)\setminus B(x,c_{\rm{UP}}r)\neq \emptyset,  
  \end{align*} 
  for any $B(x,r)\neq F$.
\end{enumerate}
The properties in $\rm{Reg}_{s,s}$ and UP
are called $s$-Ahlfors regularity and
uniform perfectness, respectively.
LRES is an abbreviation of 
lower resistance estimate.
\begin{rmk}
Note that every connected metric space is uniformly perfect with any $c_{\rm{UP}}\in (0,1)$.
\end{rmk}

Recall assumptions (A1), (A2), and (A3) from \cref{as11}.
We summarize the relation between the conditions and 
the consequences of the assumptions.
\begin{prp}\label{summ}\quad\par
  \begin{itemize}
    \item [$\rm{(i)}$]Suppose that there exists a measure $\nu$ satisfying $\rm{Reg}_{s_0,s_1}$.
    Then $\rm{UP}$ implies $\rm{LRES}(1+s_0-s_1)$ with 
    $c_{\rm{LR}}=\left[\frac{4c_u(\frac{1}{4}c_{\rm{UP}}+1)^{s_1}}{c_l(\frac{1}{16}c_{\rm{UP}})^{s_0}c_{\rm{UP}}}\right]^{-1}$.
     In particular, $\rm{(A3)}$ implies $\rm{(A2)}$.
\item [$\rm{(ii)}$]The assumption $\rm{(A1)}$ implies 
\begin{align}\label{uhka1}
  p(t,x,x)\lesssim c_{\rm{A1}} t^{-\frac{s_0}{1+s_0}},\qquad t<(\diam F)^{1+s_0},\;x\in F,
\end{align}
where $c_{\rm{A1}}=1+c_l^{-1}$.
\item[$\rm{(iii)}$] The assumption $\rm{(A2)}$ implies 
\begin{align}\label{lhka2}
  c_{\rm{A2}}t^{-\frac{s_1}{(1+s_0)\theta}}\lesssim p(t,x,x),\qquad t>0,\;x\in F.
\end{align}
Moreover, if $t<\frac{D_F^{(1+s_0)\theta}}{5c_{\rm{Exit}}}$, we can take $c_{\rm{A2}}=\frac{1}{100}\left[c_u(5c_{\rm{Exit}})^{\frac{s_1}{(1+s_0)\theta}} \right]^{-1}$.
Here, $D_F$ and $c_{\rm{Exit}}$ are given by
\begin{align*}
  D_F=\frac{1}{2}\diam F\wedge (2c_{\rm{LR}}^{-1}\diam F)^{1/\theta},\qquad c_{\rm{Exit}}=\frac{4^{2+s_0}}{3c_l c_{\rm{LR}}^{1+s_0}}.
\end{align*} 
\item[$\rm{(iv)}$]The assumption $\rm{(A3)}$ implies 
\begin{align}\label{lhka3}
  c_{\rm{A3}}t^{-\frac{\beta s_1}{1+\beta}}\lesssim p(t,x,x),\qquad t>0,\;x\in F,\;\beta=\frac{1-(2+s_0)(s_0-s_1)}{s_1+2(2+s_0)(s_0-s_1)}.
\end{align}
Moreover, if $t<2\left[  \frac{C (\frac{1}{2}\diam F)^{1+\beta}}{\log 4}  \right]^{1/\beta} $, we can take $c_{\rm{A3}}=\frac{1}{4c_u}\left[ \frac{\log 4}{2^{\beta}C}\right]^{\frac{\beta s_1}{1+\beta}}$.
Here, $C$ is given by $C=\frac{c_l(c_{\rm{LR}}/4)^{1+s_0}}{4c_u c_{\rm{UR}}} \left( \frac{1}{256}\frac{[c_l(c_{\rm{LR}}/4)^{1+s_0}]^2}{c_u c_{\rm{UR}}} \right)^{\beta}$. 
\item[$\rm{(v)}$]Assume $\rm{(A3)}$. Then, for any $r<\frac{1}{2}\diam F$,
\begin{align*}
  P_x(\sigma_{B(x,r)^c}\leq t)\leq 2 \exp\left[ -\g{c}{summ}\frac{r^{1+\beta}}{t^{\beta}}\right],
\end{align*}
where $\beta=\frac{1-(2+s_0)(s_0-s_1)}{s_1+2(2+s_0)(s_0-s_1)}>0$ and $\g{c}{summ}=\frac{c_l(c_{\rm{LR}}/4)^{1+s_0}}{4c_u c_{\rm{UR}}} \left( \frac{1}{256}\frac{[c_l(c_{\rm{LR}}/4)^{1+s_0}]^2}{c_u c_{\rm{UR}}} \right)^{\beta}$.
\end{itemize}
\end{prp}
\begin{proof}
  Except for the explicit expressions of the constants, 
  these are shown in \cite{Cr07,Kum04}. See \Cref{appB} for the proof.
\end{proof}
\begin{rmk}\label{le1}
Note that $\frac{\beta s_1}{1+\beta}\leq\frac{s_1}{(1+s_0)(1+s_0-s_1)}<1$, and 
equality holds if and only if $s_0=s_1$.
Also, we have $\frac{\beta s_1}{1+\beta}=\frac{s_1}{(1+s_0)\Theta}$ for $\Theta:=\frac{1+s_0-s_1}{1-(s_0-s_1)(2+s_0)}$.
\end{rmk}

We say a function $f:X\to\R$ on a metric space $(X,d)$ is 
$L$-Lipschitz if it satisfies $\|f\|_{\Lip^1}\leq L$. 
Recall the definitions of heat kernels, potential densities, 
and Green functions from the end of \Cref{resis}. 
\begin{thm}\label{continuity}The following hold.
  \begin{itemize}
    \item [$\mathrm{(i)}$]For any $x\in F$, a non-empty open set $U\subsetneq F$ and its complement $A:=F\setminus U$, 
    $g_A^{\alpha}(x,\cdot)$ is $2$-Lipschitz. 
    \item [$\mathrm{(ii)}$]For any $x\in F$, $g^{\alpha}(x,\cdot)$ is a Lipschitz continuous function.
    Moreover, there exists $\g{c}{continuity,1}\in (0,\infty)$ such that
    \begin{align}\label{lipres}
      \|g^{\alpha}(x,\cdot)\|_{\Lip^1}\lesssim 
      \g{c}{continuity,1}(1+\alpha^{-1}),\qquad x\in F,\;\alpha\in (0,\infty),
    \end{align}
    and $\g{c}{continuity,1}$ is given by 
    \begin{align}\label{rescont}
    \g{c}{continuity,1}=\frac{1}{1-E_{x_0}[e^{-\sigma_{x_0,x_1}}]}  
    \end{align} 
    for any distinct points $x_0,x_1\in F$.
    \item [$\mathrm{(iii)}$]Let $t_0\in (0,\infty]$. Suppose that the heat kernel $p$ satisfies
    \begin{equation*}
      \begin{aligned}
        c_{\rm{LHK}}t^{-\theta_l}\leq p(t,x,x),\quad t\in (0,\infty),\;
        \rm{and}\;\;
         p(t,x,x)\leq c_{\rm{UHK}}t^{-\theta_u},\quad t\in (0,t_0).
      \end{aligned}
    \end{equation*}
    with constants $c_{\rm{LHK}},c_{\rm{UHK}},\theta_u>0$, and $\theta_l\in (0,1)$.
    Then $p(t,x,\cdot)$ is Lipschitz continuous for $t<t_0$, and it holds that
    \begin{align}\label{liphk}
       \|p(t,x,\cdot)\|_{\Lip^1}\lesssim \g{c}{continuity,2}(1+t)t^{-\left(\frac{\theta_u-\theta_l}{2}+1\right)},\qquad x\in F,\;t\in (0,t_0),
    \end{align}
    where $\g{c}{continuity,2}=\g{c}{continuity,1}\left( \frac{c_{\rm{UHK}}}{c_{\rm{LHK}}\Gamma(1-\theta_l)} \right)^{1/2}$.
  \end{itemize} 
\end{thm}
\begin{rmk}
The estimate in \cref{continuity}(iii) can not be improved in general for small $t$.
For instance,
if $p(t,x,y)=\frac{1}{\sqrt{2\pi t}}e^{-\frac{|x-y|^2}{2t}}$ is 
the heat kernel of Brownian motion on $\R$,
it holds that $\|p(t,x,\cdot)\|_{\Lip^1}=\sup_{y\in \R}|\partial_y p(t,x,y)|=ct^{-1}$. 
Note that Brownian motion is the process corresponding to $\R$ equipped with 
usual Euclidean metric and half the Lebesgue measure.
\end{rmk}
\begin{rmk}
Since we are typically concerned with the short-time behavior 
of the heat kernel, the term $1+t$ can be regarded 
as a bounded factor. For the large-time case ($t\gg1$), 
the estimate becomes $\|p(t,x,\cdot)\|_{\Lip^1}\lesssim\g{c}{continuity,2}t^{-\theta_u/2}$.
See the proof for details.
\end{rmk}

\begin{proof}[Proof of $\mathrm{(i)}$]
  We first deal with the case where $U$ is relatively compact.
  Write $\calD_U$ for the domain of the Dirichlet form 
  associated with the killed process $X^U$.
  Then it holds that $\calD_U=\calF(A)$ (see the line below \cite[Definition 10.2]{Kig12}).
  Note that we have $g^A(x,\cdot)\in \calD_U$ and 
  $\calE(G^{\alpha}_Af,g_A(x,\cdot))=G^{\alpha}_Af(x)$ for any $x\in F$ 
  and $f\in L^2(U,\mu|_U)$
  (see \eqref{repro}). 
  Combining these observations and $G^{\alpha}_A=(\alpha-\Delta_U)^{-1}$ (see \cite[Lemma 1.3.2]{FOT94}),
  it holds that
  \begin{equation}\label{kill}
    \begin{aligned}
      G^{\alpha}_Af(x)
    =&\E(G_A^{\alpha}f,g_A(x,\cdot))\\
      =&(-\Delta_UG_A^{\alpha}f,g_A(x,\cdot))\\
      =&((I-\alpha G_A^{\alpha})f,g_A(x,\cdot))\\
      =&\int_U g_A(x,z)f(z)\mu(dz)-\alpha \int_U G_A^{\alpha}(g_A(x,\cdot))(z)f(z)\mu(dz)
    \end{aligned}
  \end{equation}
Fix $y\in F$. 
Let $(f_n)$ be a non-negative sequence in $C_c(F)$ 
such that $f_n d\mu $ converges weakly to $\delta_y$. 
Substituting $f_n$ into $f$ in \eqref{kill},
we deduce that 
\begin{align*}
  g_A^{\alpha}(x,y)
      =g_A(x,y)-\alpha G_A^{\alpha}(g_A(x,\cdot))(y)
      =g_A(x,y)-\alpha G_A^{\alpha}(g_A(y,\cdot))(x)
\end{align*}
from the symmetry of $g_A$ and $g_A^{\alpha}$.
Therefore, it follows that
\begin{align*}
  |g_A^{\alpha}(x,y)-g_A^{\alpha}(x,z)|
      \leq&|g_A(x,y)-g_A(x,z)|+\alpha \int_F g_A^{\alpha}(x,w)|g_A(y,w)-g_A(z,w)|\mu(dw)\\
      \leq&R(y,z)\left(1+\alpha\int_{F} g_A^{\alpha}(x,w)\mu(dw)\right)\\
      \leq&2R(y,z).
\end{align*}
Here we used $1$-Lipschitz continuity of $g_A(x,\cdot)$ (see \cite[Theorem 4.1]{Kig12}).

We next consider the general case. Take $x_0\in U$ and set $U_k:=U\cap B(x_0,k)$.
Then $p_{U_k}$ is increasing in $k$ and enjoys pointwise convergence to $p_U$ 
(see \cite[Proof of Theorem 10.4]{Kig12}).
Then the result follows from a formula $g_{F\setminus U_k}^{\alpha}(x,y)=\int_{0}^{\infty}e^{-\alpha t}p_{U_k}(t,x,y)dt$ (see \eqref{reshk}) and monotone convergence theorem.
\end{proof}

\begin{proof}[Proof of $\mathrm{(ii)}$]
  Fix distinct points $x_0,x_1\in F$, and take an arbitrary $x\in F$.
  Then \cite[p1956]{Cr18} and the strong Markov property yield that
  \begin{align*}
    G^{\alpha}f(y)
    =G^{\alpha}_{x_0}f(y)
    +E_y[e^{-\alpha \sigma_0}]\Big\{ G^{\alpha}_{x_1}f(x_0)+ E_{x_0}[e^{-\alpha \sigma_{x_1}}] G^{\alpha}_{x_0}f(x_1) \Big\}\sum_{i\geq 0}E_{x_0}[e^{-\alpha \sigma_1}]^i,
  \end{align*}
  for any $y\in F$,
  where
  \begin{align*}
  \sigma_0=\inf\{t>0:X_t=x_0\},\qquad \sigma_{i+1}=\inf\{ t>\sigma_i:X_t=x_0,\;x_1\in X_{[\sigma_i,t]} \}.
\end{align*}
Similarly to (i), we obtain 
\begin{equation}\label{lipresolvent}
  \begin{aligned}
  g^{\alpha}(x,y)
  =&g_{x_0}^{\alpha}(x,y)
  +E_y[e^{-\alpha \sigma_0}]\Big\{ g^{\alpha}_{x_1}(x,x_0)+ E_{x_0}[e^{-\alpha \sigma_{x_1}}] g^{\alpha}_{x_0}(x,x_1) \Big\}\sum_{i\geq 0}E_{x_0}[e^{-\alpha \sigma_1}]^i\\
  =&g_{x_0}^{\alpha}(x,y)
  +E_x[e^{-\alpha \sigma_0}]\Big\{ g^{\alpha}_{x_1}(y,x_0)+ E_{x_0}[e^{-\alpha \sigma_{x_1}}] g^{\alpha}_{x_0}(y,x_1) \Big\}\sum_{i\geq 0}E_{x_0}[e^{-\alpha \sigma_i}]^i.  
  \end{aligned}
\end{equation}
Here we used the symmetry of Green's functions.
Therefore, by (i), $g^{\alpha}(x,\cdot)$ is Lipschitz continuous.
We next bound its Lipschitz constant.
Noting that $(0,\infty)\ni\alpha\mapsto \frac{1-e^{-\alpha \sigma_{x_0,x_1}}}{\alpha}$ is decreasing
(recall the definition of $\sigma_{x_0,x_1}$ from \Cref{resis}.),
we obtain that 
\begin{align*}
  \frac{1}{E_{x_0}\left[1-e^{-\alpha \sigma_{x_0,x_1}}\right]}
  \leq&
  \begin{dcases}
  \frac{1}{\alpha E_{x_0}[1-e^{-\sigma_{x_0,x_1}}]},  & \alpha\in (0,1]\\
  \frac{1}{E_{x_0}[1-e^{-\sigma_{x_0,x_1}}]},& \alpha\in(1,\infty)   
  \end{dcases}\\
  \leq&\frac{1}{1-E_{x_0}[e^{-\sigma_{x_0,x_1}}]}\left(1+\frac{1}{\alpha}\right).
\end{align*}
By the above inequality, (i), and \eqref{lipresolvent}, it follows that
\begin{align*}
  \|g^{\alpha}(x,\cdot)\|_{\Lip^1}
  \lesssim&1+\sum_{i\geq 0}E_{x_0}[e^{-\alpha \sigma_{x_0,x_1}}]^i\\
=&1+\frac{1}{1-E_{x_0}[e^{-\alpha \sigma_{x_0,x_1}}]}\\
\lesssim&\frac{1}{1-E_{x_0}[e^{-\sigma_{x_0,x_1}}]}\left(1+\frac{1}{\alpha}\right).
\end{align*}
This completes the proof.
\end{proof}

\begin{proof}[Proof of $\mathrm{(iii)}$]
The idea of this proof is mainly based on \cite[Theorem 3.40]{DoF} and \cite[Lemma 2.3]{Zero}.

  Firstly, note that we have  
\begin{equation*}
  E_x[e^{-\alpha \sigma_y}]=\frac{g^{\alpha}(x,y)}{g^{\alpha}(y,y)},\qquad x,y\in F,\;\alpha\in (0,\infty),
\end{equation*}
(see \cite[Theorem 3.6.5]{MR06}).
Hence, if we put $\sigma_{x,y}:=\inf\{t>\sigma_x:X_t=y\}$ for $x,y\in F$, then
we have 
\begin{equation*}
  E_z[e^{-\alpha\sigma_y}]
  \geq E_z[e^{-\alpha\tau_{x,y}}]
  =E_z[e^{-\alpha\sigma_x}]E_x[e^{-\alpha\sigma_y}]
  =E_z[e^{-\alpha\sigma_x}]\frac{g^{\alpha}(x,y)}{g^{\alpha}(y,y)}.
\end{equation*}
This and symmetry yield that
\begin{align}\label{ub0}
  g^{\alpha}(y,z)
  =g^{\alpha}(z,y)
  =E_{z}[e^{-\alpha\sigma_y}]g^{\alpha}(y,y)
  \geq E_z[e^{-\alpha\sigma_x}]g^{\alpha}(x,y),\quad x,y,z\in F.
\end{align}
Therefore, we deduce from \eqref{ub0} and (ii) that
\begin{align*}
  g^{\alpha}(x,z)-g^{\alpha}(y,z)\leq E_z[e^{-\alpha \sigma_x}](g^{\alpha}(x,x)-g^{\alpha}(x,y))
  \leq \g{c}{continuity,1} (1+\alpha^{-1})E_z[e^{-\alpha \sigma_x}]R(x,y).
\end{align*}
Thus, by symmetry, it holds that
\begin{equation}\label{ub1}
  \begin{aligned}
  |g^{\alpha}(x,z)-g^{\alpha}(y,z)|
  \leq&\g{c}{continuity,1}(1+\alpha^{-1})(E_z[e^{-\alpha \sigma_x}]+E_z[e^{-\alpha \sigma_y}])R(x,y)\\
  =&\g{c}{continuity,1}(1+\alpha^{-1}) \left(\frac{g^{\alpha}(z,x)}{g^{\alpha}(x,x)}+\frac{g^{\alpha}(z,y)}{g^{\alpha}(y,y)}\right) R(x,y).  
  \end{aligned}
\end{equation}
Using our heat kernel estimate, we obtain that
\begin{align*}
  \inf_{x\in F}g^{\alpha}(x,x)
  =\inf_{x\in F}\int_{0}^{\infty}e^{-\alpha t}p(t,x,x)
  \geq c_{\rm{LHK}}\int_{0}^{\infty}e^{-\alpha t}t^{-\theta_l}dt
  =c_{\rm{LHK}}\Gamma(1-\theta_l)\alpha^{\theta_l-1}.
\end{align*}
Note that, since $\theta_l<1$, the integral $\Gamma(1-\theta_l)$ is finite.
Combining these observations and  
$\int_{F}g^{\alpha}(x,z)\mu(dz)\leq \alpha^{-1}$ (see \eqref{reshk}), we conclude that
\begin{align}\label{ub2}
  \int_{F}|g^{\alpha}(x,z)-g^{\alpha}(y,z)|\mu(dz)\leq \g{c}{continuity,1}(c_{\rm{LHK}}\Gamma(1-\theta_l))^{-1} (1+\alpha^{-1})\alpha^{-\theta_l}R(x,y).
\end{align}
Then it immediately follows from \eqref{ub1} and \eqref{ub2} that
\begin{align}\label{ub3}
  \int_{F}|g^{\alpha}(x,z)-g^{\alpha}(y,z)|^2\mu(dz)\leq 2(\g{c}{continuity,1})^2(c_{\rm{LHK}}\Gamma(1-\theta_l))^{-1}(1+\alpha^{-1})^2\alpha^{-\theta_l}R(x,y)^2.
\end{align}
Now it should be noted that for $f\in L^2(F,\mu)$ and the spectral measure $E_{\cdot,\cdot}(d\lambda)$ of $-\Delta$ (recall the definition from the end of \Cref{resis}),
\begin{equation}\label{norm}
  \begin{aligned}
 \|(\alpha-\Delta)P_tf\|_{L^2(F,\mu)}^2
  =&\int_{[0,\infty)}(\alpha+\lambda)^2e^{-2t\lambda}E_{f,f}(d\lambda)\\
  \leq&\left[ \sup_{\lambda\ge 0}(\alpha+\lambda)^2e^{-2t\lambda}\right]E_{f,f}([0,\infty))\\
  \leq&t^{-2} e^{2(\alpha t-1)}\|f\|_{L^2(F,\mu)}^2.    
  \end{aligned}
\end{equation}
In particular, we have $\|(\alpha-\Delta)P_t\|^2\le t^{-2} e^{2(\alpha t-1)}$.
Setting 
$h_0=(\alpha-\Delta)P_th$ for $h\in L^2(F,\mu)$,
we get the following by \eqref{ub3}:
\begin{equation}\label{p=2}
\begin{aligned}
  |P_th(x)-P_th(y)|^2
  =&|G^{\alpha}h_0(x)-G^{\alpha}h_0(y)|^2\\
  =&\left|\int_{F}(g^{\alpha}(x,z)-g^{\alpha}(y,z))h_0(z)\mu(dz)\right|^2\\
  \lesssim&(\g{c}{continuity,1})^2(c_{\rm{LHK}}\Gamma(1-\theta_l))^{-1}(1+\alpha^{-1})^2\alpha^{-\theta_l}R(x,y)^2\| (\alpha-\Delta)P_t \|^2\|h\|_{L^2(F,\mu)}^2\\
  \lesssim&(\g{c}{continuity,1})^2(c_{\rm{LHK}}\Gamma(1-\theta_l))^{-1}(1+\alpha^{-1})^2\alpha^{-\theta_l}R(x,y)^2t^{-2}e^{2(\alpha t-1)}\|h\|_{L^2(F,\mu)}^2.
\end{aligned}
\end{equation}
Here we used Cauchy-Schwarz and \eqref{ub3} in the first inequality.
In particular, taking $h=p^{t,z}=p(t,z,\cdot),\;t<t_0/2$, we obtain that
\begin{align*}
  |p(2t,x,z)-p(2t,y,z)|
  =&|P_th(x)-P_th(y)|\\
  \lesssim&\g{c}{continuity,1}(c_{\rm{LHK}}\Gamma(1-\theta_l))^{-1/2}\alpha^{-\theta_l/2}R(x,y)t^{-1}e^{\alpha t-1}\|p^{t,z}\|_{L^2(F,\mu)}\\
  =&\g{c}{continuity,1}(c_{\rm{LHK}}\Gamma(1-\theta_l))^{-1/2}(1+\alpha^{-1})\alpha^{-\theta_l/2}R(x,y)t^{-1}e^{\alpha t-1}\sqrt{p(2t,z,z)}\\
  \lesssim&\g{c}{continuity,1}\left( \frac{c_{\rm{UHK}}}{c_{\rm{LHK}}\Gamma(1-\theta_l)} \right)^{1/2}(1+\alpha^{-1})\alpha^{-\theta_l/2}t^{-\theta_u/2}t^{-1}e^{\alpha t-1}R(x,y).
\end{align*}
Here, we used the Chapman-Kolmogorov equation \eqref{chap-kol} in the second equality.
Substituting $\alpha=1/t$, we obtain the result.
\end{proof}

For $(F_n,R_n,\mu_n)\in \bbF^{\circ}$, 
we denote the process corresponding to $(F_n,R_n,\mu_n)$ by $(X^n=(X^n_t)_t,(P^n_x)_{x\in F_n})$,
 the commute time of $X^n$ between $x_0^n,x_1^n\in F_n$ by $\sigma^n_{x_0^n,x_1^n}:=\inf\{t>0:X_t^n=x_0^n,x_1^n\in X^n_{[0,t]}\}$,
  and the hitting time for $A\subseteq F_n$ by $\sigma^n_A:=\inf\{t>0:X_t^n\in A\}$.
We also abbreviate $\sigma^n_x:=\sigma^n_{\{x\}}$. For $r>0$ and $(F,R,\mu,\rho)\in\bbF$, we define 
$(F^{(r)},R^{(r)},\mu^{(r)},\rho^{(r)})\in\bbF_c$ by setting
\begin{align}\label{rversion}
  F^{(r)}=\overline{B(\rho,r)},\quad R^{(r)}=R|_{F^{(r)}\times F^{(r)}},\quad \mu^{(r)}=\mu|_{F^{(r)}},\;\;\textrm{and}\;\;\rho^{(r)}=\rho.
\end{align} 
Note that, since $\mu_n^{(r)}$ is of full support as a measure on $F_n^{(r)}$, and $F^{(r)}$ is compact,
$(F^{(r)},R^{(r)},\mu^{(r)},\rho^{(r)})$ lies in $\bbF_{c}$.
Also, we denote the process corresponding to 
$(F^{(r)},R^{(r)},\mu^{(r)})\in \bbF_c^{\circ}$ 
by $(X^{(r)}=(X^{(r)}_t)_t,(P_x)_{x\in F^{{(r)}}})$.
\begin{cor}\label{1_1.1}
Suppose that $(F_n,R_n,\mu_n,\rho_n)\in\bbF_r$ converges to $(F,R,\mu,\rho)\in\bbF$ in the Gromov-Hausdorff-vague topology.
Then there exists a constant $\g{c}{1_1.1}\in (0,\infty)$ such that 
\begin{align}\label{unipot}
  \sup_{\substack{n\in\N\\x\in F_n}}\|g^{\alpha,n}(x,\cdot)\|_{\Lip^1}\vee \sup_{x\in F}\|g^{\alpha}(x,\cdot)\|_{\Lip^1}\lesssim \g{c}{1_1.1} \left(1+\frac{1}{\alpha}\right),\quad \alpha\in (0,\infty),
\end{align} 
where $g^{\alpha,n}$ is the $\alpha$-potential density associated with $(F_n,R_n,\mu_n)$.
\end{cor}
\begin{rmk}
It is difficult to give an explicit formula using the data of spaces
for the constant $\g{c}{1_1.1}$. However, the proof gives us some information about it.
\end{rmk}

Towards showing this corollary, we start by presenting the following theorem.
The proof is left to \Cref{appA}.
Recall the definition 
of the Fell topology from \cref{fell}
and a characterization of convergence with respect to the Gromov-Hausdorff-vague topology from \cref{232}.
\begin{thm}\label{hitwkconv}
Let $(F_n,R_n,\mu_n,\rho_n)\in \bbF_r$ converge to $(F,R,\mu,\rho)\in \bbF_r$ in the Gromov-Hausdorff-vague topology.
We regard $F_n,\;n\in\N,$ and $F$ as subsets of a common $\bcm$ space $(M,d^M)$.
Suppose that a non-empty closed subset $A_n\subseteq F_n$ 
converges to a non-empty closed subset $A\subseteq F$ with respect to the Fell topology in $\calC (M)$, 
and that the following non-explosion condition \eqref{non-exp} holds. 
\begin{align}\label{non-exp}
 \lim_{r\to\infty}\liminf_{n\to\infty}R_n(\rho_n,B_{R_n}(\rho_n,r)^c)=\infty 
\end{align}
Then
 $P^n_{\rho_n}(\sigma_{A_n}^n\in \cdot)$ converges weakly to  $P_{\rho}(\sigma_{A}\in \cdot)$.
\end{thm}

\begin{proof}[Proof of \cref{1_1.1}]
  Take distinct points $x_0,x_1\in F$ and $x_0^n,x_1^n\in F_n$ such that $x_k^n$ converges to $x_k$ for $k=0,1$.
  By \cref{continuity}, it is enough to show that $\sup_{n\in \N}E_{x_0^n}^n[e^{-\sigma^n_{x_0^n,x_1^n}}]<1$.
  
  We first consider the case where all $F_n$ and $F$ are compact. 
  In this case, it follows from \cref{hitwkconv} and $\sigma_{x_0,x_1}>0$ a.s.  that 
  \begin{align*}
    \lim_{n\to\infty}E_{x_0^n}^n\left[e^{-\sigma^n_{x_0^n,x_1^n}}\right]=E_{x_0}[e^{-\sigma_{x_0,x_1}}]<1.
  \end{align*}
  Thus we obtain the result.

  We next show the general case. By our assumption,
  we may find $r\in (0,\infty)$ such that $(F_n^{(r)},R_n^{(r)},\mu_n^{(r)},\rho_n^{(r)})$
  converges to $(F^{(r)},R^{(r)},\mu^{(r)},\rho^{(r)})$ 
  in the Gromov-Hausdorff-vague topology with $x_0^n,x_1^n\in F_n^{(r)}$ and $x_0,x_1\in F^{(r)}$.
  Let $X^{n,(r)}=(X^{n,(r)}_t)_t$ be the process corresponding to 
  $(F_n^{(r)},R_n^{(r)},\mu_n^{(r)})$.
  Then, by recurrence, $X^{n,(r)}$ is the time change of $X^n$ with respect to $\mu_n^{(r)}$ (see \cite[Lemma 2.6]{CHK17}).
  That is,  $X^{n,(r)}_t$ is given by
  \begin{equation*}
    A^{n,(r)}_t:=\int_{0}^t 1_{\{X^n_s\in F^{(r)}_n\}}ds,\quad \tau^{n,(r)}_t:=\inf\{s>0:A^{n,(r)}_s>t\},\quad X^{n,(r)}_t:=X^n_{\tau^{n,(r)}_t}.
  \end{equation*}
Then we observe that $\sigma^{n,(r)}_{x_0^n,x_1^n}:=\inf\{ t>0:X^{n,(r)}_t=x_0^n,x_1^n\in X^{n,(r)}_{[0,t]} \}\leq \sigma_{x_0^n,x_1^n}^n$.
Therefore, we deduce that
\begin{align*}
  \limsup_{n\to\infty}E_{x_0^n}^n\left[e^{-\sigma^n_{x_0^n,x_1^n}}\right]\leq \limsup_{n\to\infty}E_{x_0^n}^n\left[e^{-\sigma^{n,(r)}_{x_0^n,x_1^n}}\right]<1
\end{align*}
from the compact case. This yields the result.
\end{proof}

While the following lemma is basic, it is used repeatedly in the rest of this section, and plays an important role.
\begin{lem}\label{mix}
  Let $f:X\to\R$ be a function on a metric space $(X,d)$
  with $\|f\|_{\Lip^{1/2}},\|f\|_{\Lip^{1}}<\infty$.
  Then it holds that
  \begin{align*}
    \|f\|_{\Lip^{\kappa}}\leq \|f\|_{\Lip^{1/2}}^{2(1-\kappa)}\|f\|_{\Lip^{1}}^{2\kappa-1}
  \end{align*}
  for $\kappa\in [1/2,1]$.
\end{lem}
\begin{proof}
  This follows from $2(1-\kappa),\;2\kappa-1\ge 0$, and following inequality:
  \begin{align*}
  |f(x)-f(y)|
  =&|f(x)-f(y)|^{2(1-\kappa)}\times |f(x)-f(y)|^{2\kappa-1}\\
  \leq&\Bigl\{ \|f\|_{\Lip^{1/2}}d(x,y)^{1/2} \Bigr\}^{2(1-\kappa)}\times \Bigl\{ \|f\|_{\Lip^{1}}d(x,y) \Bigr\}^{2\kappa -1}\\
  =&\|f\|_{\Lip^{1/2}}^{2(1-\kappa)}\|f\|_{\Lip^{1}}^{2\kappa-1} d(x,y)^{\kappa}.  
  \end{align*}
\end{proof}

The following proposition ensures 
the H\"older continuity of the heat kernel on 
a compact resistance metric space $F$, without 
requiring any regularity of the measure.
\begin{prp}\label{unihk}
If $(F,R,\mu)\in \bbF_c^{\circ}$, then the associated heat kernel $p(t,x,\cdot)$ is $\kappa$-H\"older continuous for any $t>0,\;\kappa\in [1/2,1]$, and it holds that
\begin{align*}
  \sup_{x\in F}\|p(t,x,\cdot)\|_{\Lip^\kappa}\lesssim \g{c}{unihk,1} (t^{-(1-\kappa)}+t^{-2\kappa}) ,\qquad t\in (0,\infty),\;\kappa\in [1/2,1],
\end{align*}
where $\g{c}{unihk,1}=(\g{c}{continuity,1})^{2\kappa-1}(1+\mu(F)\diam(F))^{2\kappa-1}(\mu(F)^{-1/2}+\diam(F)^{1/2})^{2(1-\kappa)}$.\\
In particular, if
$(F_n,R_n,\mu_n,\rho_n)\in \bbF_c$ converges to $(F,R,\mu,\rho)\in \bbF_c$ in the Gromov-Hausdorff-vague topology,
then there exists constant $\g{c}{unihk,2}\in (0,\infty)$ such that
\begin{align}
  \sup_{\substack{n\in\N\\x\in F_n}}\|p_n(t,x,\cdot)\|_{\Lip^{\kappa}}\vee\sup_{x\in F}\|p(t,x,\cdot)\|_{\Lip^\kappa} \leq \g{c}{unihk,2}\left(t^{-(1-\kappa)}+t^{-2\kappa}\right),\quad t\in (0,\infty),\;\kappa\in [1/2,1],
\end{align}
where $p_n$ is the heat kernel associated with $(F_n,R_n,\mu_n)$.
\end{prp}

\begin{proof}
  We write the spectral measure associated with $-\Delta$ as $E_{\cdot,\cdot}(d\lambda)$ (recall the definition from the end of \Cref{resis}). 
  We first deal with the case where $\kappa=1$. 
  By \cref{continuity}, we have 
  \begin{equation}\label{1_1.1ii1}
    \begin{aligned}
    \sup_{x\in F}\|p(t,x,\cdot)\|_{\Lip^1}
    =&\sup_{\substack{x\in F}}\|G^{1}(1-\Delta)p^{t,x}\|_{\Lip^1}\\
    \lesssim&\g{c}{continuity,1}\sup_{\substack{x\in F}}\left\| \left(1-\frac{\partial}{\partial t}\right)p^{t,x} \right\|_{L^1(F,\mu)}\\
    \leq&\g{c}{continuity,1}+\g{c}{continuity,1}\mu(F)\sup_{\substack{x,y\in F}}\left|\frac{\partial}{\partial t} p(t,x,y) \right|.  
    \end{aligned}
  \end{equation}
On the other hand, for any fixed $t_0\in (0,t/2)$, we observe that
\begin{equation}\label{1_1.1ii3}
  \begin{aligned}
    \sup_{\substack{x,y\in F}}\left| \frac{\partial}{\partial t}p(t,x,y) \right|
    =&\sup_{\substack{x,y\in F}}\left| \frac{\partial}{\partial t}(P_{t/2-t_0}p^{t_0,x},P_{t/2-t_0}p^{t_0,y}) \right|\\
    =&\sup_{\substack{x,y\in F}}\left| \frac{\partial}{\partial t}\int_{\R_{\geq 0}} e^{-\lambda (t-2t_0)}E_{p^{t_0,x},p^{t_0,y}}(d\lambda)\right|\\
    =&\sup_{\substack{x,y\in F}}\left|\int_{\R_{\geq 0}} \lambda e^{-\lambda (t-2t_0)}E^n_{p^{t_0,x},p^{t_0,y}}(d\lambda)\right|\\
    \leq&\sup_{\lambda\geq 0}\lambda e^{-\lambda (t-2t_0)}\times \sup_{\substack{x,y\in F}}|E_{p^{t_0,x},p^{t_0,y}}|(\R_{\geq 0})\\
    \lesssim&\frac{1}{t-2t_0}\sup_{\substack{x\in F}}\|p^{t_0,x}\|_{L^2(F,\mu)}^2\\
    =&\frac{1}{t-2t_0}\sup_{\substack{x\in F}}p(2t_0,x,x).
  \end{aligned}
\end{equation}
Here, we used the Chapman-Kolmogorov equation \eqref{chap-kol} in the last equality.
Substituting $t_0=t/4$ and using the estimate $\sup_{x\in F}p(t,x,x)\lesssim \mu(F)^{-1}+\diam(F)t^{-1}$ (see \eqref{kig104}),
we obtain the result for $\kappa=1$.

Next we consider the case where $\kappa=1/2$.
Note that from \eqref{chap-kol} and \eqref{defsg},
it holds that $p^{t,x}(y)=P_{t/2}p^{t/2,x}(y)$.
Thus, by \eqref{resishol}, we have  
that
\begin{equation}\label{hfholhk}
\begin{aligned}
  \sup_{\substack{x\in F}}\|p^{t,x}\|_{\Lip^{1/2}}^2
  \leq&\sup_{\substack{x\in F}}\calE(P_{t/2}p^{t/2,x},P_{t/2}p^{t/2,x})\\
  =&\sup_{\substack{x\in F}}(-\Delta P_{t/2}p^{t/2,x},P_{t/2}p^{t/2,x})_{L^2(F,\mu)}\\
  =&\sup_{\substack{x\in F}}\int_{\R\geq 0}\lambda e^{-t\lambda}E_{p^{t/2,x},p^{t/2,x}}(d\lambda)\\
  \lesssim&t^{-1}\sup_{\substack{x\in F}}p(t,x,x)\\
  \lesssim&t^{-1}(\mu(F)^{-1}+\diam(F)t^{-1}).
\end{aligned}  
\end{equation} 
This implies the result for $\kappa=1/2$.
The rest is clear from \cref{mix}.

For the latter assertion, it is enough to replace $\g{c}{continuity,1}$
in the above argument with $\g{c}{1_1.1}$.
\end{proof}

\begin{rmk}\label{derhk}
  An argument similar to \eqref{1_1.1ii3} yields the following estimate:
  \begin{align}\label{derihk}
    \sup_{x,y\in F}\left|\frac{\partial^k}{\partial t^k}p(t,x,y)\right|\lesssim k^k \left(\frac{2}{e}\right)^k t^{-k}\sup_{x\in F}p(t/2,x,x),\qquad k\in\N.
  \end{align} 
  In fact, it holds for any fixed $t_0\in (0,t)$ that
  \begin{align*}
    \sup_{\substack{x,y\in F}}\left| \frac{\partial^k}{\partial t^k}p(t,x,y) \right|
    =&\sup_{\substack{x,y\in F}}\left| \frac{\partial^k}{\partial t^k}(P_{t/2-t_0}p^{t_0,x},P_{t/2-t_0}p^{t_0,y}) \right|\\
    =&\sup_{\substack{x,y\in F}}\left| \frac{\partial^k}{\partial t^k}\int_{\R_{\geq 0}} e^{-\lambda (t-2t_0)}E_{p^{t_0,x},p^{t_0,y}}(d\lambda)\right|\\
    =&\sup_{\substack{x,y\in F}}\left|\int_{\R_{\geq 0}} \lambda^k e^{-\lambda (t-2t_0)}E^n_{p^{t_0,x},p^{t_0,y}}(d\lambda)\right|\\
    \leq&\sup_{\lambda\geq 0}\lambda^k e^{-\lambda (t-2t_0)}\times \sup_{\substack{x,y\in F}}|E_{p^{t_0,x},p^{t_0,y}}|(\R_{\geq 0})\\
    \leq&\frac{k^k e^{-k}}{(t-2t_0)^k}\sup_{\substack{x\in F}}\|p^{t_0,x}\|_{L^2(F,\mu)}^2\\
    =&\frac{k^k e^{-k}}{(t-2t_0)^k}\sup_{\substack{x\in F}}p(2t_0,x,x).
  \end{align*}
  Taking $t_0=t/4$, we obtain the desired inequality \eqref{derihk}.
\end{rmk}

We now extend the first assertion of \cref{unihk} to the non-campact setting,
providing improved estimates under the assumption that $\rm{Reg}_{s_0,s_1}$ holds.

\begin{prp}\label{c_1}Suppose that $\mu$ satisfies $\rm{Reg}_{s_0,s_1}$.
  \begin{itemize}
    \item [$\mathrm{(i)}$]If $F$ is compact, then there exists a constant $\g{c}{c_1,1}=c_{\rm{A1}}^{2(1-\kappa)}(\g{c}{unihk,1})^{2\kappa-1}\in (0,\infty)$ such that
  \begin{align*}
    |p(t,x,y)-p(t,x,z)|\lesssim \g{c}{c_1,1} t^{-\frac{2\kappa s_0 + 3\kappa - 1}{1+s_0}}R(y,z)^{\kappa},\quad x,y,z\in F,\;t\in (0,1],\;\kappa\in [1/2,1].
  \end{align*}
    \item [$\mathrm{(ii)}$]If $(F,R,\mu)$ satisfies $\mathrm{(A2)}$, then there exists a constant 
    $\g{c}{c_1,2}=c_{\rm{A1}}^{2(1-\kappa)}(\g{c}{continuity,2})^{2\kappa-1}\in (0,\infty)$ such that
  \begin{align*}
    |p(t,x,y)-p(t,x,z)|\lesssim \g{c}{c_1,2} t^{-\frac{2[\theta(1+s_0)-s_1]\kappa+s_0\theta+s_1}{2\theta(1+s_0)}}R(y,z)^{\kappa},\quad x,y,z\in F,\;t\in (0,1],\;\kappa\in [1/2,1].
  \end{align*}
    \item [$\mathrm{(iii)}$]If $(F,R,\mu)$ satisfies $\mathrm{(A3)}$, then there exists a constant $\g{c}{c_1,3}=c_{\rm{A1}}^{2(1-\kappa)}(\g{c}{continuity,2})^{2\kappa-1}\in (0,\infty)$ such that
  \begin{align*}
    |p(t,x,y)-p(t,x,z)|\lesssim \g{c}{c_1,3} t^{-\frac{2[(1+s_0)\Theta-s_1]\kappa+s_0\Theta+s_1}{2(1+s_0)\Theta}}R(y,z)^{\kappa},\quad x,y,z\in F,\;t\in (0,1],\;\kappa\in [1/2,1],
  \end{align*}
  where $\Theta=\frac{1+s_0-s_1}{1-(s_0-s_1)(2+s_0)}$
  \end{itemize}
\begin{rmk}
  Since $\kappa\le 1$, the estimate in \cref{c_1}(i) is better that that in \cref{unihk}.
\end{rmk}
  \begin{rmk}
  The constant $\g{c}{continuity,2}$ in $\g{c}{c_1,2}$
  does not necessarily coincide with
  that in $\g{c}{c_1,3}$ since the constants and exponent in the heat kernel estimate may differ.
  We shall continue to use this abuse of notation throughout this section. 
  \end{rmk}
\end{prp}
\begin{proof}[Proof of $\rm{(i)}$]
  By the computation in \eqref{hfholhk} and the heat kernel estimate \eqref{uhka1},
  we deduce that 
  \begin{align}\label{a1kap}
    |p(t,x,y)-p(t,x,z)|
    \leq \|p^{t,x}\|_{\Lip^{1/2}}R(y,z)^{1/2}
    \lesssim c_{\rm{A1}}t^{-\left(\frac{1}{2}+\frac{s_0}{2(1+s_0)}\right)}R(y,z)^{1/2}.
  \end{align}
   Combining this and \cref{unihk} with $\kappa=1$,
   and arguing as in \cref{mix}, we obtain the result.
\end{proof}
\begin{proof}[Proof of $\rm{(ii)}$ and $\rm{(iii)}$]
  We first show (ii).
  Note that, by \cref{summ}(ii), the estimate \eqref{a1kap} holds.
  It remains to obtain an estimate for $\kappa=1$. In fact, if we have that estimate, we may apply \cref{mix} to 
  complete the proof. 
  We can confirm from $(1+s_0)\theta>s_1$ and \cref{summ}(ii) 
  that the assumptions of \cref{continuity}(iii) are satisfied.
  Thus we deduce that
  \begin{align*}
    \|p^{t,x}\|_{\Lip^{1}}\lesssim \g{c}{continuity,2}t^{-\left(\frac{\theta_u-\theta_l}{2}+1\right)},\qquad t\le 1,
  \end{align*} 
  where $\theta_u=\frac{s_0}{1+s_0}$ and $\theta_l=\frac{s_1}{(1+s_0)\theta}$.

  Regarding (iii), the same argument follows from \cref{summ}(iv), \cref{le1}, and \cref{continuity}(iii).
\end{proof}

Let $(F,R,\mu)$ be an element in $\bbF^{\circ}_c$. 
For each $\kappa\in [1/2,1]$ and $p,q\in [1,\infty]$,
consider the following conditions $1_{\kappa,p}$ and $2_{\kappa,p,q}$,
which play an crucial role in later sections.
\begin{cond}\label{cond12}\quad\par
\begin{enumerate}[align=left, labelwidth=\widthof{$2_{\kappa,p,q}$},
    labelsep=0.5em, 
    leftmargin=!]
  \item [$1_{\kappa,p}$]There exists a constant $c_1=c_1(t,\kappa,p)\in (0,\infty)$ such that, 
  for any $t>0$ and a bounded measurable function $h:F\to\R$, 
  \begin{equation*}
    \|P_th\|_{\Lip^{\kappa}}\leq c_1\|h\|_{L^p(F,\mu)},
  \end{equation*}
  and
  \begin{equation*}
    \forall t\in(0,\infty),\quad c_1([0,t],\kappa,p):=\int_{0}^t c_1(s,\kappa,p)ds<\infty. 
  \end{equation*}
  \item[$2_{\kappa,p,q}$]There exist constants $c_2=c_2(\alpha,t,\kappa,p),c_3=c_3(\alpha,t,\kappa,q)\in (0,\infty)$ such that, 
  for any $\alpha,t\in (0,1]$ and a bounded measurable function $h:F\to\R$, 
  \begin{equation*}
    \|(\alpha-\Delta)^2 P_th\|_{F}\leq c_2\|h\|_{L^p(F,\mu)},\quad \mathrm{and}\quad\|(\alpha-\Delta)^2 P_th\|_{\Lip^{\kappa}}\leq c_3 \|h\|_{L^q(F,\mu)}.
  \end{equation*}
\end{enumerate}
\end{cond}
\begin{rmk}
  These conditions are always satisfied for appropriate $\kappa, p, q$ (see \cref{suff}(i)). 
  Therefore, we refer to them as conditions rather than assumptions. 
\end{rmk}
The parameters $p$ and $q$ are typically assumed to be $1$ or $2$.

The following lemma is used to show \cref{suffcond1}.
\begin{lem}\label{uuu}
Let $(F,R,\mu)$ be in $\bbF^{\circ}$.
Fix $\alpha\in (0,1]$. Suppose that there exists a constant $C=C(\alpha)\in (0,\infty)$ such that $\int_F |g^{\alpha}(x,z)-g^{\alpha}(y,z)|^2\mu(dz)\leq CR(x,y)^2$ for any $x,y,z\in F$.
Then it holds that
\begin{align}\label{uqw}
  \|(\alpha-\Delta)^2P_th\|_{\Lip^{1}}\leq C^{1/2}t^{-3}\|h\|_{L^2(F,\mu)}, \qquad t\in (0,1],\;h\in L^2(F,\mu).
\end{align}
In particular, we have the following:
\begin{itemize}
  \item [$\rm{(i)}$]If $F$ is compact, then it holds that
  \begin{align}\label{uqw1}
     \|(\alpha-\Delta)^2P_th\|_{\Lip^{1}}\lesssim \mu(F)^{1/2}\g{c}{continuity,1}\alpha^{-1} t^{-3}\|h\|_{L^2(F,\mu)},
  \end{align}
  for $\alpha,\;t\in (0,1]$ and $h\in L^2(F,\mu)$.
  \item [$\rm{(ii)}$]If $(F,R,\mu)$ satisfies $\rm{(A2)}$, then it holds that
  \begin{align}\label{uqw2}
    \|(\alpha-\Delta)^2P_th\|_{\Lip^{1}}\lesssim \g{c}{continuity,1} \left[c_{\rm{A2}}\Gamma\left(1-\frac{s_1}{(1+s_0)\theta}\right)\right]^{-1/2}\alpha^{-(1+\frac{s_1}{(1+s_0)\theta})}t^{-3}\|h\|_{L^2(F,\mu)},
  \end{align}
  for $\alpha,\;t\in (0,1]$ and $h\in L^2(F,\mu)$.
  \item [$\rm{(iii)}$]If $(F,R,\mu)$ satisfies $\rm{(A3)}$, then it holds that
  \begin{align}\label{uqw3}
     \|(\alpha-\Delta)^2P_th\|_{\Lip^{1}}\lesssim \g{c}{continuity,1} \left[c_{\rm{A3}}\Gamma\left(1-\frac{s_1}{(1+s_0)\theta}\right)\right]^{-1/2}\alpha^{-(1+\frac{s_1}{(1+s_0)\Theta})}t^{-3}\|h\|_{L^2(F,\mu)}, 
  \end{align}
  for $\alpha,\;t\in (0,1]$ and $h\in L^2(F,\mu)$. Here, $\Theta=\frac{1+s_0-s_1}{1-(s_0-s_1)(2+s_0)}$.
\end{itemize}
\end{lem}

\begin{proof}
  Set $h_0=(\alpha-\Delta)^3P_th$. Then, similarly as in \eqref{p=2}, we have the following from $\alpha,t\le 1$:
  \begin{align*}
    |(\alpha-\Delta)^2P_th(x)-(\alpha-\Delta)^2P_th(y)|^2
    =&|G^{\alpha}h_0(x)-G^{\alpha}h_0(y)|^2\\
    =&\left| \int_F (g^{\alpha}(x,z)-g^{\alpha}(y,z))h_0(z) \mu(dz)\right|^2\\
    \leq&CR(x,y)^2\left[\sup_{\lambda\ge 0}(\alpha+\lambda)^6e^{-2t\lambda}\right]\|h\|_{L^2(F,\mu)}^2\\
    \leq&CR(x,y)^2t^{-6}e^{2(\alpha t-1)}\|h\|_{L^2(F,\mu)}^2\\
    \leq&CR(x,y)^2t^{-6}\|h\|_{L^2(F,\mu)}^2.
  \end{align*}
  This shows \eqref{uqw}.

  We next prove (i). Note that the inequality \eqref{ub1} holds without any assumptions.
  By the finiteness of $\mu$, $\alpha\le 1$, and \eqref{ub1}, we observe that
  \begin{align*}
    \int_F |g^{\alpha}(x,z)-g^{\alpha}(y,z)|^2\mu(dz)
    \mu(F) (4\g{c}{continuity,1}\alpha^{-1})^2R(x,y)^2.
  \end{align*}
  Now the result follows from \eqref{uqw}.

  Finally we show (ii) and (iii). 
  Under assumption $\rm{(A2)}$, we may take $\theta_l=\frac{s_1}{(1+s_0)\theta}$ in the assumptions of \cref{continuity}(iii) from \cref{summ}(iii).
  Thus we obtain from \eqref{ub3} that
  \begin{align*}
     \int_F |g^{\alpha}(x,z)-g^{\alpha}(y,z)|^2\mu(dz)
   \leq 8(\g{c}{continuity,1})^2 \left[c_{\rm{A2}}\Gamma\left(1-\frac{s_1}{(1+s_0)\theta}\right)\right]^{-1}\alpha^{-2-\frac{s_1}{(1+s_0)\theta}}R(x,y)^2
  \end{align*}
  for $\alpha,t\in (0,1]$.
  Applying \eqref{uqw}, we complete the proof for (ii).
  The final assertion (iii) can be shown by replacing 
  $\theta$ and $c_{\rm{A2}}$ with $\Theta$ and $c_{\rm{A3}}$, respectively.
\end{proof}

The following estimates will be used in \Cref{esthk} 
to derive the convergence rate for the heat kernel.
\begin{prp}\label{suffcond1}
  Let $(F,R,\mu)$ be in $\bbF^{\circ}$.
  Set $\g{c}{suffcond1,1}=\mu(F)^{-1}+\diam F$ and $\g{c}{suffcond1,2}=\left[\mu(F)^{1/2}\g{c}{continuity,1}\right]^{2\kappa-1}$.
\begin{itemize}
  \item [$\mathrm{(i)}$]Suppose that $F$ is compact. Then it holds that
  \begin{equation*}
     \|(\alpha-\Delta)^2P_th\|_{F}\lesssim \g{c}{suffcond1,1}t^{-3}\|h\|_{L^1(F,\mu)},
    \quad \|(\alpha-\Delta)^2P_th\|_{F}\lesssim (\g{c}{suffcond1,1})^{1/2}t^{-5/2}\|h\|_{L^2(F,\mu)},
  \end{equation*}
  \begin{equation*}
  \|(\alpha-\Delta)^2P_th\|_{\Lip^{\kappa}}\lesssim \g{c}{suffcond1,2}\alpha^{-( 2\kappa-1 )} t^{-(\kappa+2)}\|h\|_{L^1(F,\mu)},
  \end{equation*}
  and
  \begin{align}\label{alw}
  \|(\alpha-\Delta)^2 P_t h \|_{\Lip^{\kappa}}\lesssim (\g{c}{suffcond1,1})^{1/2} \g{c}{suffcond1,2} \alpha^{-( 2\kappa-1 )} t^{-(\kappa+5/2)}\|h\|_{L^2(F,\mu)},
  \end{align}
  for any $\alpha,\;t\in (0,1],\;\kappa\in [1/2,1],$ and $h\in L^2(F,\mu)$. 
     \item [$\mathrm{(ii)}$]
     If $(F,R,\mu)$ satisfies assumption $\rm{(A1)}$, and is compact, 
     then it holds that
  \begin{equation*}
     \|(\alpha-\Delta)^2P_th\|_{F}\lesssim c_{\rm{A1}}t^{-\left(2+\frac{s_0}{1+s_0}\right)}\|h\|_{L^1(F,\mu)},
    \quad \|(\alpha-\Delta)^2P_th\|_{F}\lesssim c_{\rm{A1}}^{1/2}t^{-\left(2+\frac{s_0}{2(1+s_0)}\right)}\|h\|_{L^2(F,\mu)},
  \end{equation*}
  and
  \begin{equation*}
  \|(\alpha-\Delta)^2P_th\|_{\Lip^{\kappa}}\lesssim c_{\rm{A1}}^{1/2}\g{c}{suffcond1,2}\alpha^{-( 2\kappa-1 )} t^{-(\kappa+2+\frac{s_0}{2(1+s_0)})}\|h\|_{L^1(F,\mu)},
  \end{equation*}
  for any $\alpha,\;t\in (0,1],\;\kappa\in [1/2,1],$ and $h\in L^2(F,\mu)$.

   \item [$\mathrm{(iii)}$]If $(F,R,\mu)$ satisfies assumption $\rm{(A2)}$, 
     then it holds that
  \begin{equation*}
  \|(\alpha-\Delta)^2P_th\|_{\Lip^{\kappa}}\lesssim c_{\rm{A1}}^{\kappa}\g{c}{suffcond1,3}\alpha^{-(1+\frac{s_1}{2(1+s_0)\theta})( 2\kappa-1 )} t^{-(\kappa+2+\frac{s_0}{2(1+s_0)})}\|h\|_{L^1(F,\mu)},
  \end{equation*}
  and
  \begin{equation*}
  \|(\alpha-\Delta)^2P_th\|_{\Lip^{\kappa}}\lesssim \g{c}{suffcond1,3}\alpha^{-(1+\frac{s_1}{2(1+s_0)\theta})(2\kappa-1)}t^{-(\kappa+2)}\|h\|_{L^2(F,\mu)},
  \end{equation*}
  for any $\alpha,\;t\in (0,1],\;\kappa\in [1/2,1],$ and $h\in L^2(F,\mu)$, where
  \begin{align*}
     \g{c}{suffcond1,3}= \left\{ \g{c}{continuity,1}\left[ c_{\rm{A2}}\Gamma\left(1-\frac{s_1}{(1+s_0)\theta}\right) \right]^{-1/2} \right\}^{2\kappa-1}.
  \end{align*}

     \item [$\mathrm{(iv)}$]If $(F,R,\mu)$ satisfies assumption $\rm{(A3)}$, 
     then it holds that
  \begin{equation*}
  \|(\alpha-\Delta)^2P_th\|_{\Lip^{\kappa}}\lesssim c_{\rm{A1}}^{\kappa}\g{c}{suffcond1,4}\alpha^{-(1+\frac{s_1}{2(1+s_0)\Theta})( 2\kappa-1 )} t^{-(\kappa+2+\frac{s_0}{2(1+s_0)})}\|h\|_{L^1(F,\mu)},
  \end{equation*}
  and
  \begin{equation*}
  \|(\alpha-\Delta)^2P_th\|_{\Lip^{\kappa}}\lesssim \g{c}{suffcond1,4}\alpha^{-(1+\frac{s_1}{2(1+s_0)\Theta})(2\kappa-1)}t^{-(\kappa+2)}\|h\|_{L^2(F,\mu)},
  \end{equation*}
  for any $\alpha,t\in (0,1],\;\kappa\in [1/2,1],$ and $h\in L^2(F,\mu)$, where
  \begin{align*}
     \g{c}{suffcond1,4}= \left\{ \g{c}{continuity,1}\left[ c_{\rm{A3}}\Gamma\left(1-\frac{s_1}{(1+s_0)\Theta}\right) \right]^{-1/2} \right\}^{2\kappa-1},\quad \Theta=\frac{1+s_0-s_1}{1-(s_0-s_1)(2+s_0)}.
  \end{align*}
\end{itemize}
\end{prp}
\begin{rmk}
The inequality \eqref{alw} with $\kappa=1/2$ holds without any assumptions.
\end{rmk}

\begin{proof}[Proof of $\rm{(i)}$ and $\rm{(ii)}$]
  We first derive the estimates for the $\sup$ norm. 
  These follow 
  from  \cref{derhk}, 
  and we obtain the following:
    \begin{align*}
      \|(\alpha-\Delta)^2P_th\|_{F}
      \leq \|h\|_{L^1(F,\mu)}\sup_{x,y\in F}\left|\left( \alpha-\frac{\partial}{\partial t}\right)^2p(t,x,y) \right|
      \lesssim t^{-2}\sup_{x\in F}p(t/2,x,x)\|h\|_{L^1(F,\mu)},
    \end{align*}
    \begin{align*}
      \|(\alpha-\Delta)^2P_th\|_{F}
      \leq \|h\|_{L^2(F,\mu)}\sup_{x\in F}\|(\alpha-\Delta)^2p^{t,x}\|_{L^2(F,\mu)}
      \lesssim t^{-2}\sup_{x\in F}\sqrt{p(t,x,x)}\|h\|_{L^2(F,\mu)}
    \end{align*}
    
We next show the inequalities for $\kappa$ H\"older constants with the $L^2$ norm.
An estimate for $\kappa=1$ is shown in \cref{uuu}(i).
Also, using $\|f\|_{\Lip^{1/2}}^2\leq \calE(f,f),\;f\in\calF$ (see \eqref{resishol}),
we obtain 
that
\begin{equation}\label{pui}
  \begin{aligned}
  \|(\alpha-\Delta)^2P_t h\|_{\Lip^{1/2}}
  \leq&\calE((\alpha-\Delta)^2P_t h,(\alpha-\Delta)^2P_t h)^{1/2}\\
  =&(-\Delta (\alpha-\Delta)^2P_t h,(\alpha-\Delta)^2P_t h)_{L^2(F,\mu)}^{1/2}\\
  \leq&\|(-\Delta)^{1/2}(\alpha-\Delta)^2P_th\|_{L^2(F,\mu)}\\
  \lesssim&t^{-5/2}\|h\|_{L^2(F,\mu)}.
  \end{aligned}
\end{equation}
Here, the last inequality follows from a similar calculation as in \eqref{norm}.
Thus we now have estimates for $\kappa=1/2$ and $\kappa=1$. By \cref{mix}, we obtain the desired inequality.

Finally, we prove the inequalities for $\kappa$ H\"older constants with the $L^1$ norm.
Taking $h=p^{t,x}$ in \cref{uuu}(i), we deduce that
\begin{align*}
  \|(\alpha-\Delta)^2p^{t,x}\|_{\Lip^1}\lesssim \mu(F)^{1/2}\g{c}{continuity,1}\alpha^{-1}t^{-3}\sqrt{p(t,x,x)}.
\end{align*}
On the other hand, similar as in \eqref{pui}, we have
\begin{align}\label{w3e}
   \|(\alpha-\Delta)^2p^{t,x}\|_{\Lip^{1/2}}\leq 
   \|(-\Delta)^{1/2}(\alpha-\Delta)^2P_{t/2}p^{t/2,x}\|_{L^2(F,\mu)}^{1/2}
   \lesssim t^{-5/2}\sqrt{p(t,x,x)}.
\end{align}
Combining above arguments, $p(t,x,x)\lesssim \g{c}{suffcond1,1} t^{-1},\;t\leq 1$ (see \eqref{kig104} and take $A=F$), and 
\cref{summ}(ii), we obtain the result. 
\end{proof}

\begin{proof}[Proof of $\rm{(iii)}$]
  We first deal with the inequality with the $L^2$ norm.
  An estimate for $\kappa=1$ is shown in \cref{uuu}(ii).
  On the other hand, we have the estimate \eqref{pui} for $\kappa=1/2$.
  By these two and \cref{mix}, we conclude the result.

  We next show the inequality with the $L^1$ norm.
  We observe that 
  \begin{align*}
    \|(\alpha-\Delta)^2p^{t,x}\|_{\Lip^{1/2}}\lesssim c_{\rm{A1}}^{1/2} t^{-(\frac{5}{2}+\frac{s_0}{2(1+s_0)})}
  \end{align*}
  from \eqref{w3e} and \cref{summ}(ii).
  Also, by substituting $h=p^{t,x}$ in \cref{uuu}(ii) and 
  applying Chapman-Kolmogorov equation \eqref{chap-kol}, it follows that
  \begin{align*}
    \|(\alpha-\Delta)^2p^{t,x}\|_{\Lip^{1}}\lesssim c_{\rm{A1}}^{1/2}\g{c}{continuity,1}\left[c_{\rm{A2}}\Gamma\left(1-\frac{s_1}{(1+s_0)\theta}\right)\right]^{-1/2}\alpha^{-(1+\frac{s_1}{2(1+s_0)\theta})}t^{-(3+\frac{s_0}{2(1+s_0)})}.
  \end{align*}
  To deduce the inequality for general $\kappa$, we may use \cref{mix}.
\end{proof}
\begin{proof}[Proof of $\rm{(iv)}$]
  It suffices to replace $c_{\rm{A2}}$ and $\theta$ in the above proof with $c_{\rm{A3}}$ and $\Theta$, respectively.
\end{proof}

The following estimates are used in \Cref{esthk} to derive the convergence rate estimates for the semigroup.

\begin{prp}\label{suffcond2}
  Let $(F,R,\mu)$ be in $\bbF^{\circ}$.
  Set $\g{c}{suffcond2,1}=\left[\mu(F)^{1/2}\g{c}{unihk,1}\right]^{2\kappa-1}$ and $\g{c}{suffcond2,2}=\left[\mu(F)^{1/2}\g{c}{c_1,1}\right]^{2\kappa-1}$.
\begin{itemize}
  \item [$\mathrm{(i)}$]If $F$ is compact, then it holds that
  \begin{align}\label{aa}
    \|P_th\|_{\Lip^{\kappa}}\lesssim \g{c}{unihk,1} t^{-2\kappa}\|h\|_{L^1(F,\mu)},
  \end{align}
  and
  \begin{align}\label{bb}
    \|P_th\|_{\Lip^{\kappa}}\lesssim \g{c}{suffcond2,1} t^{-(3\kappa-1)}\|h\|_{L^2(F,\mu)},
  \end{align}
  for any $t\in(0,1],\;\kappa\in [1/2,1],$ and $h\in L^2(F,\mu)$.

  \item [$\mathrm{(ii)}$]If $F$ is compact, and satisfies assumption $\rm{(A1)}$,
  then it holds that
  \begin{equation}\label{cc}
    \|P_th\|_{\Lip^{\kappa}}\lesssim\g{c}{c_1,1}t^{-\frac{2\kappa s_0+3\kappa-1}{1+s_0}}\|h\|_{L^1(F,\mu)},
  \end{equation}
  and
  \begin{equation}\label{dd}
    \|P_th\|_{\Lip^{\kappa}}\lesssim \g{c}{suffcond2,2} t^{-(3\kappa-1)}\|h\|_{L^2(F,\mu)},
  \end{equation}
  for any $t\in(0,1],\;\kappa\in [1/2,1],$ and $h\in L^2(F,\mu)$.
  \item [$\mathrm{(iii)}$]If $(F,R,\mu)$ satisfies assumption $\rm{(A2)}$,
  then it holds that
  \begin{equation}\label{ee}
    \|P_th\|_{\Lip^{\kappa}}\lesssim \g{c}{c_1,2}t^{-\frac{2[(1+s_0)\theta-s_1]\kappa+s_0\theta+s_1}{2(1+s_0)\theta}}\|h\|_{L^1(F,\mu)},
  \end{equation}
  and
  \begin{equation}\label{ff}
     \|P_th\|_{\Lip^{\kappa}}\lesssim  \g{c}{suffcond1,3} t^{-[\kappa-\frac{s_1}{2(1+s_0)\theta}(2\kappa-1)]}\|h\|_{L^2(F,\mu)},
  \end{equation}
  for any $t\in(0,1],\;\kappa\in [1/2,1],$ and $h\in L^2(F,\mu)$.
  \item [$\mathrm{(iv)}$]If $(F,R,\mu)$ satisfies assumption $\rm{(A3)}$,
  then it holds that
  \begin{equation}\label{gg}
    \|P_th\|_{\Lip^{\kappa}}\lesssim \g{c}{c_1,3}t^{-\frac{2[(1+s_0)\Theta-s_1]\kappa+s_0\Theta+s_1}{2(1+s_0)\Theta}}\|h\|_{L^1(F,\mu)},
  \end{equation}
  and
  \begin{equation}\label{hh}
     \|P_th\|_{\Lip^{\kappa}}\lesssim  \g{c}{suffcond1,4}t^{-[\kappa-\frac{s_1}{2(1+s_0)\Theta}(2\kappa-1)]} \|h\|_{L^2(F,\mu)},
  \end{equation}
  for any $t\in(0,1],\;\kappa\in [1/2,1],$ and $h\in L^2(F,\mu)$. Here, $\Theta=\frac{1+s_0-s_1}{1-(s_0-s_1)(2+s_0)}$.
\end{itemize}
\end{prp}

\begin{rmk}\label{arigato}
The inequality \eqref{bb} with $\kappa=1/2$ holds without any assumptions.
\end{rmk}

\begin{proof}
  The inequalities \eqref{cc}, \eqref{ee}, and \eqref{gg} 
  are immediate consequences of \cref{c_1}.
  Also, \eqref{aa} follows from \cref{unihk}. 
  We next show \eqref{bb}. If $\kappa=1/2$, similar as in \eqref{pui}, we can check this from
  \begin{equation}\label{1098}
    \begin{aligned}
  \|P_th\|_{\Lip^{1/2}}
    \leq&\calE(P_th,P_th)^{1/2}\\
    =&\|(-\Delta)^{1/2}P_th\|_{L^2(F,\mu)}\\
    \lesssim&t^{-1/2}\|h\|_{L^2(F,\mu)}.    
    \end{aligned}
  \end{equation}
  And, the inequality for $\kappa=1$ follows from \eqref{aa} and $\|h\|_{L^1(F,\mu)}\leq \mu(F)^{1/2}\|h\|_{L^2(F,\mu)}$.
  For general $\kappa\in [1/2,1]$, we deduce the inequality from \cref{mix}.
  The inequality \eqref{dd} also follows from a similar argument.
  In fact, if $\kappa=1/2$, we have shown the desired inequality at \eqref{1098}, and 
  if $\kappa=1$, we may use \eqref{cc} and $\|h\|_{L^1(F,\mu)}\leq \mu(F)^{1/2}\|h\|_{L^2(F,\mu)}$ to derive
  \begin{align*}
    \|P_th\|_{\Lip^1}\lesssim \g{c}{c_1,1}t^{-2}\|h\|_{L^1(F,\mu)}\leq \mu(F)^{1/2}\g{c}{c_1,1}t^{-2}\|h\|_{L^2(F,\mu)}.
  \end{align*}
  We next deal with \eqref{ff}. 
  By taking $\alpha=t^{-1}$
  in \eqref{p=2}, we obtain
  \begin{align*}
    \|P_t h\|_{\Lip^{1}}\lesssim \g{c}{continuity,1}\left[ c_{\rm{A2}}\Gamma\left(1-\frac{s_1}{(1+s_0)\theta}\right) \right]^{-1/2}t^{-\left( 1-\frac{s_1}{2(1+s_0)\theta}\right)}\|h\|_{L^2(F,\mu)},
  \end{align*}
  which is the desired inequality for $\kappa=1$.
  Using \cref{mix} and \eqref{bb} with $\kappa=1/2$,
  we deduce \eqref{ff}.
  Finally, \eqref{hh} can be shown by replacing $\theta$ with $\Theta$.
\end{proof}

Finally we summarize results in this section.
Since 
these are direct consequences of \cref{suffcond1,suffcond2},
we omit proofs.

\begin{cor}\label{suff}
Suppose that $F$ is compact.
\begin{itemize}
  \item [$\rm{(i)}$]The condition $1_{\kappa,2}$ holds for $\kappa\in \left[ \frac{1}{2},\frac{2}{3}\right)$ with 
    \begin{align*}
      c_1(t,\kappa,2)=\g{c}{suffcond2,1}t^{-(3\kappa-1)},\quad t\in (0,1],
    \end{align*}
    and $2_{\kappa,p,q}$ holds for $p,q=1,2,$ and $\kappa\in [1/2,1]$ with the following constants:
    \begin{align*}
      c_2(\alpha,t,\kappa,1)=\g{c}{suffcond1,1}t^{-3},\quad c_2(\alpha,t,\kappa,2)=(\g{c}{suffcond1,1})^{1/2}t^{-5/2},
    \end{align*}
     \begin{align*}
       c_3(\alpha,t,\kappa,1)= \g{c}{suffcond1,2} \alpha^{-(2\kappa-1)}t^{-(\kappa+2)},
      \quad c_3(\alpha,t,\kappa,2)= (\g{c}{suffcond1,1})^{1/2}\g{c}{suffcond1,2} \alpha^{-(2\kappa-1)}t^{-(\kappa+5/2)}.
     \end{align*}

  \item [$\rm{(ii)}$]Assume that the assumption $\rm{(A1)}$ is satisfied. 
  Then the condition $1_{\kappa,p}$ holds for $p=1,2$ with the following constants:
   \begin{align*}
    c_1(t,\kappa,1)=\g{c}{c_1,1}t^{-\frac{2\kappa s_0+3\kappa-1}{1+s_0}},\quad t\in (0,1],\kappa\in \left[\frac{1}{2},\frac{s_0+1}{2s_0+3}\right),
   \end{align*}
      \begin{align*}
    c_1(t,\kappa,2)=\g{c}{suffcond2,2}t^{-(3\kappa-1)},\quad t\in (0,1],\kappa\in \left[\frac{1}{2},\frac{2}{3}\right).
   \end{align*}
   Also, the condition $2_{\kappa,p,q}$ is satisfied for $p,q=1,2,$ and $\kappa\in [1/2,1]$ with the following.
     \begin{equation*}
     c_2(\alpha,t,\kappa,1)=c_{\rm{A1}}t^{-\left(2+\frac{s_0}{1+s_0}\right)},\quad 
    c_2(\alpha,t,\kappa,2)=c_{\rm{A1}}^{1/2}t^{-\left(2+\frac{s_0}{2(1+s_0)}\right)},
  \end{equation*}
  \begin{equation*}
  c_3(\alpha,t,\kappa,1)= c_{\rm{A1}}^{1/2}\g{c}{suffcond1,2}\alpha^{-( 2\kappa-1 )} t^{-(\kappa+2+\frac{s_0}{2(1+s_0)})}.
  \end{equation*}
  The constant $c_3(\alpha,t,\kappa,2)$ is same as in $\mathrm{(i)}$.

  \item [$\rm{(iii)}$]
  Assume that the assumption $\rm{(A2)}$ is satisfied. 
  Then the condition $1_{\kappa,p}$ holds for $p=1,2$ with the following constants:
   \begin{align*}
    c_1(t,\kappa,1)=\g{c}{c_1,2}t^{-\frac{2[(1+s_0)\theta-s_1]\kappa+s_0\theta+s_1}{2(1+s_0)\theta}},\quad t\in (0,1],\kappa\in \left[\frac{1}{2},  \frac{(2+s_0)\theta-s_1}{2[(1+s_0)\theta-s_1]}\right),
   \end{align*}

   \begin{align*}
    c_1(t,\kappa,2)=\g{c}{suffcond1,3} t^{-[\kappa-\frac{s_1}{2(1+s_0)\theta}(2\kappa-1)]} \quad t\in (0,1],\kappa\in\left[\frac{1}{2},1\right].
   \end{align*}
   Also, the condition $2_{\kappa,p,q}$ is satisfied for $p,q=1,2$ and $\kappa\in [1/2,1]$ with the following.
     \begin{equation*}
     c_3(\alpha,t,\kappa,1)=c_{\rm{A1}}^{\kappa}\g{c}{suffcond1,3}\alpha^{-(1+\frac{s_1}{2(1+s_0)\theta})( 2\kappa-1 )} t^{-(\kappa+2+\frac{s_0}{2(1+s_0)})},
  \end{equation*}
  \begin{equation*}
    c_3(\alpha,t,\kappa,2)=\g{c}{suffcond1,3}\alpha^{-(1+\frac{s_1}{2(1+s_0)\theta})(2\kappa-1)}t^{-(\kappa+2)}.
  \end{equation*}
  The constants $c_2(\alpha,t,\kappa,1)$ and $c_2(\alpha,t,\kappa,2)$ are same as in $\rm{(ii)}$.

  \item [$\rm{(iv)}$]
    Assume that the assumption $\rm{(A3)}$ is satisfied. 
  Then the condition $1_{\kappa,p}$ holds for $p=1,2$ with the following constants:
   \begin{align*}
    c_1(t,\kappa,1)=\g{c}{c_1,3}t^{-\frac{2[(1+s_0)\Theta-s_1]\kappa+s_0\Theta+s_1}{2(1+s_0)\Theta}},\quad t\in (0,1],\kappa\in \left[\frac{1}{2},  \frac{(2+s_0)\Theta-s_1}{2[(1+s_0)\Theta-s_1]}\right),
   \end{align*}

   \begin{align*}
    c_1(t,\kappa,2)=\g{c}{suffcond1,4} t^{-[\kappa-\frac{s_1}{2(1+s_0)\Theta}(2\kappa-1)]} \quad t\in (0,1],\kappa\in\left[\frac{1}{2},1\right].
   \end{align*}
   Also, the condition $2_{\kappa,p,q}$ is satisfied for $p,q=1,2$ and $\kappa\in [1/2,1]$ with the following constants:  
     \begin{equation*}
     c_3(\alpha,t,\kappa,1)=c_{\rm{A1}}^{\kappa}\g{c}{suffcond1,4}\alpha^{-(1+\frac{s_1}{2(1+s_0)\Theta})( 2\kappa-1 )} t^{-(\kappa+2+\frac{s_0}{2(1+s_0)})},
  \end{equation*}
  \begin{equation*}
    c_3(\alpha,t,\kappa,2)=\g{c}{suffcond1,4}\alpha^{-(1+\frac{s_1}{2(1+s_0)\Theta})(2\kappa-1)}t^{-(\kappa+2)}.
  \end{equation*}
  Here, $\Theta=\frac{1+s_0-s_1}{1-(s_0-s_1)(2+s_0)}$.
  The constants $c_2(\alpha,t,\kappa,1)$ and $c_2(\alpha,t,\kappa,2)$ are same as in $\rm{(ii)}$.
\end{itemize}
\end{cor}

\section{Estimates for resolvents}\label{efr}
The aim of this section is to establish estimates for convergence rates of the resolvents.
To this end, we assume that $(F_n,R_n,\mu_n)\in \bbF_c^{\circ}$ converges to 
$(F,R,\mu)\in \bbF_c^{\circ}$ as described in \cref{nonrootcpt}. 
Let $(M,d)$ be a compact metric space in which all $F_n$ and $F$ are isometrically embedded.
In this setting, our main result is the following \cref{estres},
which gives a bound the difference $|G^{\alpha}f(x)-G^{\alpha,n}f(x_n)|$ in terms of the Hausdorff distance,
$\BL^{\kappa}$ metric, and $d(x,x_n)$ for $f\in\BL^{\kappa}(M)$.
 Note that we assume that any metric space
contains at least two distinct points. Also, we write $A\lesssim B$ if there exists a universal constant $C$ 
such that $A\leq CB$.

We fix $x_0,x_1\in F$ and $x_0^n,x_1^n\in F_n$ satisfying the following:
    \begin{itemize}
      \item [(i)]$x_0\neq x_1,\;x_0^n\neq x_1^n$,
      \item [(ii)]$d(x_0,x_0^n),d(x_1,x_1^n)\leq d_{\mathrm{H}}(F,F_n)$.
    \end{itemize}
    We also set 
    \begin{align}\label{iroiro}
    \delta_n:=d(x_0,x_1)\wedge d(x_0^n,x_1^n),\quad 
    \quad D_n=(\mu(F)\diam F)\vee (\mu_n(F_n)\diam F_n),\quad \alpha_n=\frac{1}{2}\wedge\frac{1}{2D_n^2},
    \end{align}
    and $K_n=F\cup F_n$. Note that $D_n$, $n\in\N$, is a bounded sequence.
    In the statement of \cref{estres}, we refer to constants $\g{c}{itkill}$ and $\g{c}{sle,n}(\delta_n)$,
    which are defined below in the statement of \cref{itkill} and \cref{sle}, respectively.
    More generally, as in the previous section, 
    we denote constants appearing in proposition, lemma, or theorem X by $c_X$ or $c^{(n)}_X$.  

    Then the following holds. Note that we do not need convergence here.
\begin{thm}\label{estres}
For $\alpha\in (0,\alpha_n),\;\kappa\in (0,1],\;y\in F,\;y_n\in F_n$ and $f\in\BL^{\kappa}(M)$, we have that
\begin{equation*}
  \begin{aligned}
    |G^{\alpha}f(y)-G^{\alpha,n}f(y_n)|
  \lesssim&\g{c}{estres,n}\Bigl\{ \|f\|_{K_n}  (d(y,y_n)+d_{\mathrm{H}}(F,F_n))+\|f\|_{\BL^{\kappa}(K_n)}d_{\BL^{\kappa}}(\mu,\mu_n)\Bigr\}\\
  \leq&\g{c}{estres,n}\|f\|_{\BL^{\kappa}(K_n)}(d(y,y_n)+d_{\mathrm{H}}(F,F_n)+d_{\BL^{\kappa}}(\mu,\mu_n)).
  \end{aligned}
\end{equation*}
Here, 
$\g{c}{estres,n}=\g{c}{estres,n}(\alpha,\delta_n)$ is defined by setting
\begin{align*}
  \g{c}{estres,n}(\alpha,\delta_n)=\g{c}{itkill}\left( 1+\frac{\mu(F)\diam F}{\g{c}{sle,n}(\delta_n)} \right)\left(1+\frac{1}{\alpha \g{c}{sle,n}(\delta_n)}\right).
\end{align*}
\end{thm}
\begin{rmk}\label{estresrmk}
If $x_0,x_1,x_0^n,x_1^n\in B(\rho,r)$
for some $\rho\in M,\;r<\infty$, then we can replace $m_n(\gamma)$ in the definition of $\g{c}{sle,n}(\delta)$ by 
$\displaystyle \inf_{b\in B(\rho,r)\cap F}\mu(B(b,\gamma\delta))\wedge \inf_{b\in B(\rho,r)\cap F_n}\mu_n(B(b,\gamma\delta))$.
Since the spaces considering in this section are all compact, this remark is unimportant.
However, it is useful when we consider the non-compact case.
See \cref{slermk} and the proof of \cref{estres}.
\end{rmk}

\begin{rmk}
For applications to the non-compact case, all coefficients of the estimates in this section depend on $n$. 
However, it is possible to establish the $n$-independent inequalities under our assumption that $(F_n,R_n,\mu_n)\to (F,R,\mu)$.
In fact, most of constants appearing in this section are written in terms of $\mu(F_n)$ or $\diam F_n$, and this is bounded in $n$.   
\end{rmk}

Here, we use the following notation.
\begin{itemize}
  \item As in \Cref{cont}, we denote the corresponding Hunt process, resistance form, generator, semigroup, resolvent, heat kernel and $\alpha$-potential density by
$(X=(X_t)_t,(P_x)_{x\in F}),\;(\calE,\calF),\;\Delta,\;(P_t)_{t},\;(G^{\alpha})_{\alpha},\;p,$ and $g^{\alpha}$, respectively. 
We abbreviate $p(t,x,\cdot)$ to $p^{t,x}$.
\item We also denote the $n$-version by 
$(X^n=(X^n_t)_t,(P^n_x)_{x\in F_n}),\;(\calE_n,\calF_n),\;\Delta_n,\;(P_t^n)_{t},\;(G^{\alpha,n})_{\alpha},\;p_n,$ and 
$g^{\alpha,n}$, respectively. 
\end{itemize}

We mainly adapt ideas from \cite[Section 5]{Cr18}.

To begin with, we introduce the basic properties of a killed Green operator.
For a non-empty closed subset $A\subseteq F$ and 
measurable function $f:F\to\R_{\geq 0}$, define $G_Af:F\to\R$ by setting
\begin{align*}
  G_Af(y)=E_y\left[ \int_{0}^{\sigma_A}f(X_t)dt \right].
\end{align*} 
Recall that $\sigma_A$ is a hitting time of $A$, and so
$G_A$ is the Green operator
of the process $X$ killed on hitting $A$. 
For $x\in F$, we will use the abbreviation $G_x:=G_{\{x\}}$.
\begin{lem}[{\cite[Lemma 3.1]{Cr18}}]\label{killedgreen}
Let $A$ be a non-empty open subset of $F$. 
For any measurable $f:F\to\mathbb{R}_{\geq 0}$, it holds that
\begin{equation*}
    G_A f(y)=\int_F g_A(y,z)f(z)\mu(dz), \quad\forall y\in F,
\end{equation*}
where
\begin{equation*}
    g_A(y,z):=\frac{R(y,A)+R(z,A)-R_A(y,z)}{2},
\end{equation*}
with
\begin{equation*}
    R_A(y,z):= \sup\{\calE(f,f)^{-1} : f \in \calF, f(y)=1, f(z)=0, f|_{A}\;\;\rm{constant} \}
\end{equation*}
being the resistance between $y$ and $z$ in the network with $A$  `fused'. 
Note that the function $g_A$ satisfies $0 \le g_A(y,z)=g_A(z,y) \le g_A(y,y) = R(y,A)$ and
\begin{equation}\label{1lip}
    |g_A(y,z)-g_A(y,w)|\leq R(w,z), \quad \forall y, z, w \in F.
\end{equation}
Furthermore, $g_{x}$ is given by
\begin{equation*}
    g_{x}(y,z):=\frac{R(x,y)+R(x,z)-R(y,z)}{2}.
\end{equation*}
\end{lem}

Owing to this lemma, we can naturally extend the domain of $G_x,x\in F$ and $G_xf$ 
by setting 
\begin{align*}
  G_xf(y)=\int_F \Lambda^{x,y}(z)f(z)\mu(dz),\quad\mathrm{and}\quad \Lambda^{x,y}(z)=\frac{d(x,y)+d(x,z)-d(y,z)}{2},\quad x,y,z\in M
\end{align*}
for bounded measurable function $f:M\to\R$.
Similarly, we extend the domain of $G_x^n$.
Note that $\Lambda^{x,y}$ satisfies $0\leq \Lambda^{x,y}\leq d(x,y)$ and $\|\Lambda^{x,y}\|_{\Lip^1}\leq 1$,
and thus, it holds that 
\begin{align}\label{4-6}
  |\Lambda^{x,y}(z)-\Lambda^{x,y}(w)|
  =|\Lambda^{x,y}(z)-\Lambda^{x,y}(w)|^{1-\kappa}|\Lambda^{x,y}(z)-\Lambda^{x,y}(w)|^{\kappa}
  \leq 2d(x,y)^{1-\kappa}d(z,w)^{\kappa}
\end{align}
for $z,w\in M$ and $\kappa\in (0,1]$.

In this section, we fix $\kappa\in (0,1]$, 
and denote the Hausdorff distance on $\calC_c(M)$ and $\BL^{\kappa}$ metric on $\calM(M)_{\rm{fin}}=\calM(M)$
by $d_{\mathrm{H}}$ and $d_{\BL^{\kappa}}$, respectively (recall the definitions from \cref{haustop} and \eqref{metblk}).
\begin{lem}\label{killres}
For any $f,\psi,\varphi\in\BL^{\kappa},x,y,x_n,y_n\in M$ and a non-empty subset $A\subseteq M$,
the following hold.
  \begin{align}\label{4-1}
     |G_x\psi(y)-G_{x_n}^n\varphi(y_n)|
  \leq&\|\varphi\|_{L^1(F_n,\mu_n)}(d(x,x_n)+d(y,y_n))\notag\\
    &+(2d(x,y)^{1-\kappa}+d(x,y))\|\varphi\|_{\BL^{\kappa}(K_n)}d_{\BL^{\kappa}}(\mu,\mu_n)\\
    &+d(x,y)\|\varphi-\psi\|_{L^1(F,\mu)}\notag
  \end{align}
  \begin{equation}\label{4-4}
    \begin{aligned}
    \|G_xf-G^n_{x_n}f\|_F
    \leq&\|f\|_{L^1(F_n,\mu_n)} d(x,x_n)\\
    &+3(1+\diam(F))\|f\|_{\BL^{\kappa}(K_n)}d_{\BL^{\kappa}}(\mu,\mu_n)  
    \end{aligned}
  \end{equation}
  \begin{equation}\label{4-2}
   \|G_xf\|_{A} \leq\sup_{a\in A}d(x,a)\|f\|_{L^1(F,\mu)}
  \end{equation}
  \begin{equation}\label{4-3}
    \|G_xf\|_{\Lip^{\kappa}(A)}\leq\diam(A)^{1-\kappa}\|f\|_{L^1(F,\mu)}
  \end{equation}
\end{lem}
\begin{proof}
  The inequality \eqref{4-4} follows from \eqref{4-1}.
  Also, we can verify \eqref{4-2} and \eqref{4-3} from
  the definition of $G_xf$ and \eqref{1lip}.
  It remains to show \eqref{4-1}.
  Note that, by considering McShane extension of $g|_{\supp\nu\cup\supp\lambda}$ (recall its definition from \cref{Mcext}),
  we obtain 
  \begin{align*}
    \left| \int_M gd(\nu-\lambda) \right|\leq \|g\|_{\BL^{\kappa}(\supp\nu\cup\supp\lambda)}d_{\BL^{\kappa}}(\nu,\lambda)
  \end{align*}
  for $g\in\BL^{\kappa}(M)$ and $\nu,\lambda\in \calM(M)$.
   Using this, the submultiplicativity of $\|\cdot\|_{\BL^{\kappa}}$ (see \cref{submult}), and \eqref{4-6},
   we conclude that
      \begin{align*}
    |G_x\psi(y)-G_{x_n}^n\varphi(y_n)|
    =&\left|\int_M \Lambda^{x,y}\psi d\mu-\int_M\Lambda^{x_n,y_n}\varphi d\mu_n  \right|\\
    \leq&\left|\int_{F_n} (\Lambda^{x,y}-\Lambda^{x_n,y_n})\varphi d\mu_n\right|
    +\left| \int_{K}\Lambda^{x,y}\varphi d(\mu-\mu_n) \right|
    +\left| \int_{F}\Lambda^{x,y}(\varphi-\psi)d\mu \right|\\
    \leq&\|\varphi\|_{L^1(F_n,\mu_n)}(d(x,x_n)+d(y,y_n))\\
    &+(2d(x,y)^{1-\kappa}+d(x,y))\|\varphi\|_{\BL^{\kappa}(K_n)}d_{\BL^{\kappa}}(\mu,\mu_n)\\
    &+d(x,y)\|\varphi-\psi\|_{L^1(F,\mu)}.
  \end{align*}
\end{proof}

We next prove an inequality for the iteration of the killed Green operator.

\begin{lem}\label{itkill}
For any $x,y\in F,\;x_n,y_n\in F_n,\;f\in \BL^{\kappa}(M)$ and $i\in\N$, we have that
\begin{align*}
  |G^{\circ i}_xf(y)-(G^n_{x_n})^{\circ i}f(y_n)|
  \leq &\g{c}{itkill}\left( 1+D_n^i +D_n^{i-1}\sum_{j=0}^{i-2}D_n^j\right)\\
  &\qquad\qquad\qquad\times \Bigl\{\|f\|_{K_n}(d(x,x_n)+d(y,y_n))+\|f\|_{\BL^{\kappa}(K_n)}d_{\BL^{\kappa}}(\mu,\mu_n)\Bigr\}, 
\end{align*}
where $\g{c}{itkill}=\diam K_n\mu_n(F_n)^2+(1+\diam K_n)^2\mu_n(F_n)+\diam F+1$.
\end{lem}

\begin{proof}
  We may assume $i\geq 2$.
  We first observe that
  \begin{equation}\label{1st}
    \begin{aligned}
    |G^{\circ i}_xf(y)-(G^n_{x_n})^{\circ i}f(y_n)|
     \leq&\mu_n(F_n)\|(G^n_{x_n})^{\circ (i-1)}f\|_{F_n}(d(x,x_n)+d(y,y_n))\\
    &+2(1+d(x,y))\|(G^n_{x_n})^{\circ (i-1)}f\|_{\BL^{\kappa}(K_n)}d_{\BL^{\kappa}}(\mu,\mu_n)\\
    &+\mu(F)d(x,y)\|G^{\circ (i-1)}_xf-(G^n_{x_n})^{\circ (i-1)}f\|_{F}  
    \end{aligned}
  \end{equation}
  by \eqref{4-1}.
  Iterating \eqref{4-2} and \eqref{4-3}, we obtain 
\begin{equation}\label{8}
  \begin{aligned}
    \|(G_{x_n}^n)^{\circ (i-1)}f\|_{K_n}\leq& \diam K_n \mu_n(F_n)D_n^{i-2}\|f\|_{F_n},\\
    \|(G_{x_n}^n)^{\circ (i-1)}f\|_{\Lip^{\kappa}(K_n)}\leq& (\diam K_n)^{1-\kappa} \mu_n(F_n)D_n^{i-2}\|f\|_{F_n}.
  \end{aligned}
\end{equation}
Substituting \eqref{8} into \eqref{1st} and setting $a_n=d(x,x_n)+d(y,y_n)+d_{\BL^{\kappa}}(\mu,\mu_n)$, it follows that
\begin{equation}\label{11}
  \begin{aligned}
    |G^{\circ i}_xf(y)-(G^n_{x_n})^{\circ i}f(y_n)|
  \leq& \diam K_n \mu_n(F_n)^2 D_n^{i-2}\|f\|_{F_n} (d(x,x_n)+d(y,y_n))\\
    &+2(1+\diam K_n)^2\mu_n(F_n)D_n^{i-2}\|f\|_{F_n}d_{\BL^{\kappa}}(\mu,\mu_n)\\
    &+D_n \|G^{\circ (i-1)}_xf-(G^n_{x_n})^{\circ (i-1)}f\|_{F}\\
    \leq&4 (\diam K_n \mu_n(F_n)^2+(1+\diam K_n)^2\mu_n(F_n) )D_n^{i-2}\|f\|_{F_n}a_n\\
    &+D_n \|G^{\circ (i-1)}_xf-(G^n_{x_n})^{\circ (i-1)}f\|_{F}.
  \end{aligned}
\end{equation}
We put $C_n=4 (\diam K_n \mu_n(F_n)^2+(1+\diam K_n)^2\mu_n(F_n) )$ and $A_n=C_n D_n^{i-2}\|f\|_{F_n}a_n$.
Hence, it suffices to bound $\|G^{\circ (i-1)}_xf-(G^n_{x_n})^{\circ (i-1)}f\|_F$ for $i\geq 2$.
A similar argument yields that
\begin{align}\label{9}
  \|G^{\circ i}_xf-(G^n_{x_n})^{\circ i}f\|_{F}
  \leq A_n
  +D_n \|G^{\circ (i-1)}_xf-(G^n_{x_n})^{\circ (i-1)}f\|_{F}
\end{align}
for $i\geq 2$.
A repeated application of \eqref{9} and \eqref{4-4} yields that
\begin{equation}\label{10}
  \begin{aligned}
   \|G^{\circ i}_xf-(G^n_{x_n})^{\circ i}f\|_{F}
  \leq&A_n\sum_{j=0}^{i-2}D_n^{j}
  +D_n^{i-1}\|G_xf-G^n_{x_n}f\|_{F}\\
  \leq &A_n\sum_{j=0}^{i-2}D_n^{j}
  +D_n^{i-1}[\mu_n(F_n) \|f\|_{F}d(x,x_n)\\
  &\qquad\qquad\qquad +2(1+\diam F) \|f\|_{\BL^{\kappa}(K_n)}d_{\BL^{\kappa}}(\mu,\mu_n)]. 
  \end{aligned}
\end{equation}
We write $B_n$ for $\mu_n(F_n) \|f\|_{F}d(x,x_n)+2(1+\diam F) \|f\|_{\BL^{\kappa}(K_n)}d_{\BL^{\kappa}}(\mu,\mu_n)$.
Note that this inequality is still valid for $i=1$.
Combining \eqref{11} and \eqref{10}, we deduce that
\begin{align*}
  |G^{\circ i}_xf(y)-(G^n_{x_n})^{\circ i}f(y_n)|
  \leq & A_n+A_n D_n \sum_{j=0}^{i-2}D_n^{j}+D_n^i B_n\\
  =&C_n D_n^{i-2}\|f\|_{F_n}a_n+C_n D_n^{i-2}\|f\|_{F_n}a_nD_n \sum_{j=0}^{i-2}D_n^{j}+D_n^i B_n\\
  \leq&\left( 1+D_n^i +D_n^{i-1}\sum_{j=0}^{i-2}D_n^j\right)(C_n\|f\|_{K_n}a_n+B_n)\\
  \lesssim&\g{c}{itkill}\left( 1+D_n^i +D_n^{i-1}\sum_{j=0}^{i-2}D_n^j\right)\\
  &\qquad\times\left[ \|f\|_{K_n}(d(x,x_n)+d(y,y_n))+\|f\|_{\BL^{\kappa}(K_n)}d_{\BL^{\kappa}}(\mu,\mu_n) \right],
\end{align*}
where $\g{c}{itkill}=\diam K_n\mu_n(F_n)^2+(1+\diam K_n)^2\mu(F_n)+\diam F+1$.
This completes the proof.
\end{proof}

To prove the following lemma, 
we recall the following identity (see \cite[p1955]{Cr18}).
 For any bounded measurable function $f:M \to \R$ and $\alpha \in \left(0, \frac{1}{\text{diam}(F)\mu(F)}\right)$, it holds that
\begin{align} \label{killalpha}
G^{\alpha}_xf(y)=\sum_{i\in\N}(-\alpha)^{i-1}G_x^{\circ i}f(y)
\end{align}
for all $x,y\in F$.

Recall the definition of $\alpha_n$ from 
\eqref{iroiro}.

\begin{lem}\label{killedalpha}
For all $\alpha\in (0,\alpha_n),\;x,y\in F,\;x_n,y_n\in F_n$, and $f\in\BL^{\kappa}(M)$,
it holds that
\begin{align*}
  |G^{\alpha}_xf(y)-G^{\alpha,n}_{x_n}f(y_n)|\lesssim \g{c}{itkill}\Bigl\{ \|f\|_{K_n}(d(x,x_n)+d(y,y_n))+\|f\|_{\BL^{\kappa}(K_n)}d_{\BL^{\kappa}}(\mu,\mu_n) \Bigr\}.
\end{align*}
\end{lem}
\begin{proof}
  This easily follows from \eqref{killalpha} and \cref{itkill}.
\end{proof}

We prepare an elementary lemma for later use.
Recall the definition of the Hausdorff topology from\cref{haustop}.
\begin{lem}\label{eee}
  Let $A,A_n,$ and $B$ be compact subsets of a metric space
  $(X,d^X)$ consisting of at least two elements.
  \begin{itemize}
    \item [$\mathrm{(i)}$]There exist $a_1,a_2\in A$ and $b_1,b_2\in B$ satisfying the following.
    \begin{itemize}
      \item $a_1\neq a_2,b_1\neq b_2$
      \item $d^X(a_1,b_1),d^X(a_2,b_2)\leq d_{\mathrm{H}}^X(A,B)$
    \end{itemize}
    \item [$\mathrm{(ii)}$]If $A_n$ converges to $A$ in the Hausdorff topology,
    then there exists $a_1^n,a_2^n\in A$ and $b_1^n,b_2^n\in A_n$ satisfying the following.
    \begin{itemize}
      \item $a_1^n\neq a_2^n,b_1^n\neq b_2^n$
      \item $d^X(a_1^n,b_1^n),d^X(a_2^n,b_2^n)\leq d_{\mathrm{H}}^X(A,A_n)$
      \item $\inf_n d^X(a_1^n,a_2^n),\inf_n d^X(b_1^n,b_2^n)>0$
    \end{itemize}
  \end{itemize} 
\end{lem}
\begin{proof}[Proof of $\mathrm{(i)}$]
  Since $B$ is not a singleton, we may take distinct points $y_1,y_2\in B$.
  Using the formula $d^X_{\mathrm{H}}(A,B)=\max\{ \max_{a\in A}d^X(a,B),\max_{b\in B}d^X(b,A)\}$,
  we can find $x_1,x_2\in A$ with $d^X(x_i,y_i)\leq d^X_{\mathrm{H}}(A,B)$ for $i=1,2$.
  If $x_1\neq x_2$, we are done. 
  Suppose that $x_1=x_2=:x$. Then there exists $x'\in A\setminus\{x\}$, and 
  it holds that $d^X(x',y')\le d_{\rm{H}}^X(A,B)$ for some $y'\in B$. Since $y_1$ and $y_2$ are distinct,
  we may assume that $y'\neq y_1$. Then $(x,y_1)$ and $(x',y')$ are what we desired.
\end{proof}
\begin{proof}[Proof of $\mathrm{(ii)}$]
  Fix distinct points $x,y\in A$ and put $\varepsilon:=d^X(x,y)>0$.
  By assumption, there exists $N\in\N$ such that $d^X_{\mathrm{H}}(A,A_n)<\varepsilon/3$ for all $n\geq N$.
  If we take $b_1^n,b_2^n\in A_n$ such that $d^X(x,b_1^n),d^X(y,b_2^n)\leq d^X_{\mathrm{H}}(A,A_n)$,
  then it is the case that
  \begin{align*}
    d^X(b_1^n,b_2^n)\geq d^X(x,y)-d^X(x,b_1^n)-d^X(y,b_2^n)\geq \varepsilon/3>0,\quad n\geq N.
  \end{align*}
  Thus, we may take $a_1^n=x,a_2^n=y$ for $n\geq N$.
  The first claim (i) ensures the desired pairs for $n<N$. 
\end{proof}

In the proof of \cref{estres}, the following lemma also plays an important role.
\begin{lem}\label{sle}
  For any $\delta>0$ and $n\in\N$,
there exists a constant $\g{c}{sle,n}=\g{c}{sle,n}(\delta)\in (0,\infty)$ such that
\begin{align*}
  \g{c}{sle,n}\alpha\leq \inf_{\substack{x,y\in F\\d(x,y)\geq \delta}}E_{x}[1-e^{-\alpha \sigma_{x,y}}]\wedge \inf_{\substack{z,w\in F_n\\d(z,w)\geq \delta}}E_{z}^n[1-e^{-\alpha \sigma^n_{z,w}}],\quad \alpha\in (0,\alpha_n).
\end{align*} 
Here, $\sigma_{x,y}$ and $\sigma^n_{z,w}$ are the commute time of $X$ between $x$ and $y$, 
and that of $X^n$ between $z$ and $w$ (recall its definition from \Cref{resis}, respectively).
Moreover, $\g{c}{sle,n}$ is given by
\begin{align*}
  \g{c}{sle,n}=\int_{0}^{T}(Ae^{-Bs}-1)^2e^{-\alpha_n s}ds.
\end{align*}
The definitions of $A,\;B,$ and $T$ are given below at \eqref{abtdef}.
\end{lem}

\begin{rmk}
By the proof of this lemma and \cite[Corollary 5.7]{ALW16}, the constant $\g{c}{sle,n}$ is bounded below uniformly in $n$ for a convergent sequence $(F_n,R_n,\mu_n),\;n\in\N$.
\end{rmk}
\begin{proof}
  Define $m(\gamma),\;T,\;A$, and $B$ for $\gamma\in (1/3,1)$ by setting as follows:
  \begin{equation}\label{abtdef}
     \begin{aligned}
  m_n(\gamma)=&\inf_{b\in F}\mu(B(b,\gamma \delta))\wedge \inf_{\substack{b\in F_n}}\mu_n(B(b,\gamma \delta))\\
  T=&(1-\gamma)m_n(\gamma)\delta\log\left(\frac{1+\gamma}{2(1-\gamma)}\right)\\
  A=&\frac{2(1-\gamma)}{1+\gamma}\\
  B=&\frac{1}{(1-\gamma)m_n(\gamma)\delta}
\end{aligned}    
  \end{equation}
By compactness (or \cite[Corollary 5.7]{ALW16}), these are well-defined and positive.
It follows from \cref{4.2-a} that, whenever $d(x,y)\geq \delta$,
\begin{align*}
  P_{x}(\sigma_y\leq t)
  \leq&2\left\{ 1-\frac{1-\gamma}{1+\gamma}\exp\left(-\frac{2t}{ (1-\gamma)\mu(B_{F}(x,\gamma d(x,y)))d(x,y) }\right) \right\}\\
  \leq&2\left\{ 1-\frac{1-\gamma}{1+\gamma}\exp\left(-\frac{2t}{ (1-\gamma)m_n(\gamma)\delta}\right) \right\}.
\end{align*}
Thus we obtain
\begin{align*}
  P_x(\sigma_y\geq t/2)
  \geq&\left\{\frac{2(1-\gamma)}{1+\gamma}\exp\left(-\frac{t}{ (1-\gamma)m_n(\gamma)\delta}\right)-1\right\}\vee 0\\
  =&\left( Ae^{-Bt}-1\right)\vee 0.
\end{align*}
Noting that $Ae^{-BT}-1=0$, we deduce that
\begin{align*}
  E_{x}[1-e^{-\alpha \sigma_{x,y}}]
  =&\alpha\int_{0}^{\infty}P_{x}(\sigma_{x,y}>s )e^{-\alpha s}ds\\
  \geq&\alpha\int_{0}^{\infty}P_{x}(\sigma_{y}>s/2)P_{y}(\sigma_{x}>s/2)e^{-\alpha s}ds\\
  \geq&\alpha\int_{0}^{T}(Ae^{-Bs}-1)^2e^{-\alpha s}ds\\
  \geq&\alpha\int_{0}^{T}(Ae^{-Bs}-1)^2e^{-\alpha_n s}ds.
\end{align*}
The same inequality holds for the $n$-version.
\end{proof}

\begin{rmk}\label{slermk}
If one is interested in the points $x,y\in F$ and $x_n,y_n\in F_n$ with $x,y,x_n,y_n\in B(\rho,r)$ 
for some $\rho\in M,r<\infty$, then one can replace $m_n(\gamma)$ below by 
$\displaystyle \inf_{b\in B(\rho,r)\cap F}\mu(B(b,\gamma\delta))\wedge \inf_{b\in B(\rho,r)\cap F_n}\mu_n(B(b,\gamma\delta))$.
Since the spaces considered in this section are all compact, this remark is not important here.
However, it is useful when we consider the non-compact case.
\end{rmk}

We can now complete the proof of \cref{estres}.

\begin{proof}[Proof of \cref{estres}]
  We first fix $x_0,x_1\in F$ and $x_0^n,x_1^n\in F_n$ satisfying the following.
    \begin{itemize}
      \item [(i)]$x_0\neq x_1,x_0^n\neq x_1^n$
      \item [(ii)]$d(x_0,x_0^n),d(x_1,x_1^n)\leq d_{\mathrm{H}}(F,F_n)$
    \end{itemize}
    We also set $\delta:=d(x_0,x_1)\wedge d(x_0^n,x_1^n)$.
    The existence of these points are guaranteed by \cref{eee}.
    We define $\sigma_i,\sigma^n_i$ by setting
    \begin{align*}
      \sigma_0=\inf\{t>0:X_t=x_0\},\quad \sigma_{i+1}=\inf\{t>\sigma_i:X_t=x_0,\;x_1\in X_{[\sigma_i,t]}\},\quad i\geq 0,
    \end{align*}
    and
    \begin{align*}
      \sigma^n_0=\inf\{t>0:X_t^n=x_0^n\},\quad &\sigma_{i+1}^n=\inf\{t>\sigma^n_i:X^n_t=x_0^n,x_1^n\in X_{[\sigma^n_i,t]}\},\quad i\geq 0.
    \end{align*}
    From the strong Markov property and \cite[p1956]{Cr18}, it follows that
    \begin{equation*}
  G^{\alpha}f(y)
  =G^{\alpha}_{x_0}f(y)+\sum_{i\geq 0}E_y[e^{-\alpha\sigma_0}]E_{x_0}[e^{-\alpha \sigma_i}]E_{x_0}\left[ \int_{0}^{\sigma_1}e^{-\alpha s}f(X_s)ds \right],
\end{equation*}
\begin{equation}\label{hit}
  E_y[e^{-\alpha\sigma_0}]=1-\alpha G^{\alpha}_{x_0}1(y),
\end{equation}
\begin{equation*}
    E_{x_0}[e^{-\alpha \sigma_i}]=\left\{ (1-\alpha G^{\alpha}_{x_1}1(x_0))(1-\alpha G^{\alpha}_{x_0}1(x_1)) \right\}^i,
\end{equation*}
and 
\begin{equation}\label{decompp}
  E_{x_0}\left[ \int_{0}^{\sigma_1}e^{-\alpha s}f(X_s)ds \right]=G^{\alpha}_{x_1}f(x_0)+(1-\alpha G^{\alpha}_{x_1}1(x_0))G^{\alpha}_{x_0}f(x_1)
\end{equation}
for any bounded measurable function $f:M\to\R$.
Hence, we have that
\begin{equation}\label{a}
  \begin{aligned}
    |G^{\alpha}f(y)-G^{\alpha,n}f(y_n)|
  \leq&|G^{\alpha}_{x_0}f(y)-G^{\alpha,n}_{x_0^n}f(y_n)|\\
  &+\left| \sum_{i\geq 0}E_y[e^{-\alpha \sigma_0}]E_{x_0}[e^{-\alpha\sigma_i}]E_{x_0}\left[ \int_{0}^{\sigma_1}f(X_s) ds\right]\right. \\
  &\qquad\qquad\qquad\qquad\qquad\left.-E_{y_n}^n[e^{-\alpha \sigma^n_0}]E_{x_0^n}[e^{-\alpha    \sigma^n_i}]E_{x_0^n}\left[ \int_{0}^{\sigma^n_1}f(X_s^n) ds\right]  \right|.
  \end{aligned}
\end{equation}
We abbreviate the expectations as follows:
\begin{equation*}
  A=E_y[e^{-\alpha \sigma_0}]E_{x_0}\left[ \int_{0}^{\sigma_1}f(X_s) ds \right],\qquad A_n=E_{y_n}^n[e^{-\alpha \sigma_0^n}]E_{x_0^n}^n\left[ \int_{0}^{\sigma^n_1}f(X^n_s) ds \right],
\end{equation*}
\begin{equation*}
  T(z,w)=
  E_z[e^{-\alpha \sigma_w}],\quad z,w\in F,
  \qquad
  T_n(z,w)=
  E_z^n[e^{-\alpha \sigma_w^n}],\quad z,w\in F_n.
\end{equation*}
It is then the case that
\begin{align*}
  \MoveEqLeft \left| \sum_{i\geq 0}E_y[e^{-\alpha \sigma_0}]E_{x_0}[e^{-\alpha\sigma_i}]E_{x_0}\left[ \int_{0}^{\sigma_1}f(X_s) ds\right] 
  -E_{y_n}^n[e^{-\alpha \sigma^n_0}]E_{x_0^n}[e^{-\alpha    \sigma^n_i}]E_{x_0^n}\left[ \int_{0}^{\sigma^n_1}f(X_s^n) ds\right]  \right|\\
  =&\left| \sum_{i\geq 0}A(T(x_0,x_1)T(x_1,x_0))^i-A_n(T_n(x_0^n,x_1^n)T_n(x_1^n,x_0^n))^i \right|\\
  \leq&\left| \sum_{i\geq 0}A\Bigl\{ (T(x_0,x_1)T(x_1,x_0))^i-(T_n(x_0^n,x_1^n)T_n(x_1^n,x_0^n))^i \Bigr\} \right|
  +\left| \sum_{i\geq 0}(A-A_n)(T_n(x_0^n,x_1^n)T_n(x_1^n,x_0^n))^i \right|\\
  \leq&|A|\frac{|T(x_0,x_1)T(x_1,x_0)-T_n(x_0^n,x_1^n)T_n(x_1^n,x_0^n)|}{(1-T(x_0,x_1)T(x_1,x_0))(1-T_n(x_0^n,x_1^n)T_n(x_1^n,x_0^n))}
  +\frac{|A-A_n|}{1-T_n(x_0^n,x_1^n)T_n(x_1^n,x_0^n)}\\
  \leq&|A|\frac{|T(x_0,x_1)-T_n(x_0^n,x_1^n)|+|T(x_1,x_0)-T_n(x_1^n,x_0^n)|}{(1-T(x_0,x_1)T(x_1,x_0))(1-T_n(x_0^n,x_1^n)T_n(x_1^n,x_0^n))}
  +\frac{|A-A_n|}{1-T_n(x_0^n,x_1^n)T_n(x_1^n,x_0^n)}.
\end{align*}
It is necessary to address the following three points:
\begin{itemize}
  \item [1.]bound $|T(x_0,x_1)-T_n(x_0^n,x_1^n)|$;
  \item [2.]estimate $1-T(x_0,x_1)T(x_1,x_0)$;
  \item [3.]bound $|A-A_n|$.
\end{itemize}
Using \eqref{hit}, \cref{killedalpha}, and (ii), $1$ is done as follows:
\begin{align}\label{hitting}
  |T(x_0,x_1)-T_n(x_0^n,x_1^n)|
  =\alpha |G^{\alpha}_{x_0}1(x_1)-G^{\alpha,n}_{x_0^n}1(x_1^n)|
  \lesssim \alpha \g{c}{itkill} (d_{\mathrm{H}}(F,F_n)+d_{\BL^{\kappa}}(\mu,\mu_n))
\end{align}
We next address $2$. Noting the definition of $\delta$, 
we may apply \cref{eee} to obtain
\begin{align}\label{b}
  \g{c}{sle,n}(\delta)\alpha\leq 1-T(x_0,x_1)T(x_1,x_0),1-T_n(x_0^n,x_1^n)T_n(x_1^n,x_0^n).
\end{align} 
Lastly, we deal with $3$.
By \eqref{hit} and \eqref{decompp},
it holds that
\begin{align*}
  E_{x_0}\left[ \int_{0}^{\sigma_1}f(X_s)ds \right]=G^{\alpha}_{x_1}f(x_0)+T( x_0,x_1 )G^{\alpha}_{x_0}f(x_1).
\end{align*}
Thus, we deduce from \cref{killedalpha} and 
\eqref{hitting}
that
\begin{align*}
  \MoveEqLeft \left| E_{x_0}\left[ \int_{0}^{\sigma_1}f(X_s)ds \right]-E_{x_0^n}^n\left[ \int_{0}^{\sigma^n_1}f(X_s^n)ds \right] \right|\\
\leq&|G^{\alpha}_{x_1}f(x_0)-G^{\alpha,n}_{x_1^n}f(x_0^n)|
+T_n(x_1^n,x_0^n)|G^{\alpha}_{x_0}f(x_1)-G^{\alpha,n}_{x_0^n}f(x_1^n)|
+|G^{\alpha}_{x_0}f(x_1)|\times |T(x_0,x_1)-T_n(x_0^n,x_1^n)|\\
\lesssim&\g{c}{itkill}\Bigl\{ \|f\|_{K_n}d_{\mathrm{H}}(F,F_n)+\|f\|_{\BL^{\kappa}(K_n)}d_{\BL^{\kappa}}(\mu,\mu_n)\Bigr\}
+\alpha \g{c}{itkill}\mu(F)\diam F\|f\|_F (d_{\mathrm{H}}(F,F_n)+d_{\BL^{\kappa}}(\mu,\mu_n)) .
\end{align*}
Therefore, it is the case that
\begin{equation}\label{c}
\begin{aligned}
  |A-A_n|
  \leq&E_{y_n}^n[e^{-\alpha \sigma_0^n}]\left| E_{x_0}\left[\int_{0}^{\sigma_1}e^{-\alpha s}f(X_s)ds\right]-E_{x_0^n}^n\left[\int_{0}^{\sigma_1^n}e^{-\alpha s}f(X_s^n)ds\right] \right|\\
  &+\left|E_{x_0}^n\left[\int_{0}^{\sigma_1}e^{-\alpha s}f(X_s)ds\right] \right|\left|E_y[e^{-\alpha \sigma_0}]-E_{y_n}^n[e^{-\alpha \sigma_0^n}]\right|\\
  \lesssim&\g{c}{itkill}\Bigl\{ \|f\|_{K_n}d_{\mathrm{H}}(F,F_n)+\|f\|_{\BL^{\kappa}(K_n)}d_{\BL^{\kappa}}(\mu,\mu_n)\Bigr\}\\
&+\alpha \g{c}{itkill}\mu(F)\diam F \|f\|_F (d(y,y_n)+d_{\mathrm{H}}(F,F_n)+d_{\BL^{\kappa}}(\mu,\mu_n)).
\end{aligned}
\end{equation}
Combining \eqref{a},\eqref{hitting},\eqref{b}, and \eqref{c}, 
we obtain the result.
\end{proof}

\section{Estimates for semigroups}\label{efsg}
The aim of this section is to transfer our previous estimates on the resolvent 
to corresponding ones on the semigroup. Our main result is
\cref{gen} below. 
We use the same notation and assumptions as in \Cref{efr}.
In particular, we assume that $(F_n,R_n,\mu_n)\in \bbF_c^{\circ}$ converges to $(F,R,\mu)\in\bbF_c^{\circ}$ as described in 
\cref{nonrootcpt} (recall the definition of $\bbF_c^{\circ}$ from \eqref{bbfcirc}), and 
regard $F_n,$ $n\in\N,$ and $F$ as subsets of a common compact metric space $(M,d)$. 
In the statement of \cref{gen}, for instance, we refer to constant $\g{c}{12q2}$,
    which is defined below in the statement of \cref{12q2}.
    More generally, as in the previous section, 
    we denote constants appearing in proposition, lemma, or theorem X by $c_X$ or $c^{(n)}_X$.  
Note that we assume that any metric space contains at least two distinct points. 
Also, we write $A\lesssim B$ if there exists a universal constant $C$ 
such that $A\leq CB$

Recall the assumptions $\rm{(A2)},$ $\rm{(A3)}$ and $\alpha_n$ from \cref{as11} and \eqref{iroiro}.
Then our main result in this section is the following.
\begin{thm}\label{gen}
Let $x\in F,\;x_n\in F_n,\;\kappa\in [1/2,1],\;t>0$ and $f\in \BL^{\kappa}(M)$ and set
\[
h:=d(x,x_n)+d_{\rm{H}}(F,F_n)^{\kappa}+d_{\BL^{\kappa}}(\mu,\mu_n).
\]
\begin{itemize}
  \item [$\rm{(i)}$]If $(F,R,\mu)$ satisfies the assumption $\rm{(A2)}$, and $h\leq 1/2$,
then it holds that
\begin{align*}
 \Big|E_x[f(X_t)]-E^n_{x_n}[f(X^n_t)]\Big|\lesssim \g{c}{gen,n}\|f\|_{\BL^{\kappa}} h^{E},
\end{align*}
where
\begin{align*}
  E=\frac{\kappa}{\kappa+(1+s_0)(2+\kappa)(\kappa+\theta)},
\end{align*}
and
\begin{align*}
  \g{c}{gen,n}=\frac{1}{1-(1/2)^E}\left[ \g{c}{12q2}
    +\g{c}{sg2,n}\mu(F)^{1/2}\Bigl\{ c_{\rm{A1}} +\g{c}{suffcond1,3}\left(\frac{\alpha_n}{2}\right)^{-(1+\frac{s_1}{2(1+s_0)\theta})(2\kappa-1)}\Bigr\}
    +\g{c}{suffcond1,3}\mu(F)^{1/2} \right],
\end{align*}
and the $p$ in $\g{c}{sg2,n}$ is $2$.
Moreover, the time-dependence of $\g{c}{gen,n}$ is $t+t^{-[\kappa-\frac{s_1}{2(1+s_0)\theta}(2\kappa-1)]},\;t\in (0,1]$.
\item[$\rm{(ii)}$]If $(F,R,\mu)$ satisfies the assumption $\rm{(A3)}$, and $h$ is sufficiently small,
then it holds that
\begin{align*}
 \Big|E_x[f(X_t)]-E^n_{x_n}[f(X^n_t)]\Big|\lesssim \g{C}{gen,n}\|f\|_{\BL^{\kappa}} h^{E_1}(\log(1/h))^{E_2},
\end{align*}
where
\begin{equation*}
  E_1=\frac{\kappa}{\kappa+(1+s_0)(2+\kappa)\Theta},\qquad E_2=\frac{\kappa (2+\kappa)[ (1+s_0)\Theta-1 ]}{\kappa+(1+s_0)(2+\kappa)\Theta},
  \qquad
  \Theta=\frac{1+s_0-s_1}{1-(s_0-s_1)(2+s_0)},
\end{equation*}
\begin{equation*}
  \begin{aligned}
    \g{C}{gen,n}=&\left\{ 1+\mu(F)^{1/2}\g{c}{suffcond1,4}+\g{c}{sg2,n}\mu(F)^{1/2}\Bigl[ c_{\rm{A1}}^{1/2} +\g{c}{suffcond1,4}\left(\frac{\alpha_n}{2}\right)^{-(1+\frac{s_1}{2(1+s_0)\Theta})(2\kappa-1)}\Bigr] \right\}\\
  &\times\left\{ 1+\left[\g{c}{summ}^{1/\beta}(1\wedge\frac{1}{2}\diam F)^{1+1/\beta}\left(\frac{\kappa}{\kappa+(1+s_0)(2+\kappa)\Theta}\right)^{-\frac{1}{\beta}}\right]^{-(\kappa+2)}\right\},
  \end{aligned}
\end{equation*}
and $\beta=\frac{1-(2+s_0)(s_0-s_1)}{s_1+2(2+s_0)(s_0-s_1)}$,
and the $p$ in $\g{c}{sg2,n}$ is $2$.
Moreover, the time-dependence of $\g{C}{gen,n}$ is $t+t^{-[\kappa-\frac{s_1}{2(1+s_0)\Theta}(2\kappa-1)]},\;t\in (0,1]$.
\end{itemize}
\end{thm}

\begin{rmk}\label{su}
The word ``sufficiently small'' in the assertion of \cref{gen}(ii) means
that the $\varepsilon$ in \eqref{ep} is less than $1$.
\end{rmk}

In estimating the resolvent, roughly speaking,
since we have an explicit formula to
write $G_{\alpha}$ in terms of $G_x$, 
it is enough to estimate $|G_xf(y)-G_{x_n}^{n}f(y_n)|$.
This can be done by extending the domain of $g_x$ in a natural way.
On the other hand, the heat kernel does not have such a natural 
extension. Although we have a McShane extension,
it has a shortcoming that the operation of extension is not linear.
To overcome these obstacles, we introduce the following
operator $J_n^r$.

Fix $r>d_{\mathrm{H}}(F,F_n)$. For a bounded measurable function $f:F\to\R$,
we define $J_n^rf:F_n\to\R$ by setting
\begin{align}\label{defjn}
  J_n^rf(x)=\frac{1}{\mu(B(x,r))}\int_{B(x,r)}fd\mu,\qquad x\in F_n.
\end{align}
Note that this is well-defined.
In fact, for each $x\in F_n$, we can take $y\in F$ such that $d(x,y)\leq d_{\mathrm{H}}(F,F_n)$, 
and it holds that $0<\mu(B(y,r-d_{\mathrm{H}}(F,F_n)))\leq\mu(B(x,r))$. 
While $J^r_nf$ is no longer an extension of $f$, 
we may estimate the difference between $f$ and $J_n^rf$ for suitable $f$, 
and the operator $J_n^r$ is linear. 
Note that $J_n^rf$ is bounded. In particular $J_n^rf\in L^2(F_n,\mu_n)$ since $F_n$ is compact.

\begin{rmk}
For applications to the non-compact case, all coefficients of the estimates in this section depend on $n$. 
However, it is immediate to establish $n$-independent inequalities 
in the case when we are considering a convergent sequence of compact space.
\end{rmk}

The following lemma plays an essential role 
in connecting the resolvent estimates 
with those for the semigroup.

\begin{lem}\label{deriv}
For any bounded measurable function $f:M\to\R,\;t>0$ and $x\in F_n$, define $\varphi:[0,t]\to\R$ by setting
\begin{align*}
 \varphi(s)=P^n_{t-s}G^{\alpha,n}J_n^rP_sG^{\alpha}f(x),\quad s\in [0,t]. 
\end{align*}
Then $\varphi$ is of class $C^1$ and satisfies
\begin{align*}
  \varphi'(s)=P^n_{t-s}(J_n^rG^{\alpha}-G^{\alpha,n}J_n^r)P_sf(x).
\end{align*} 
\end{lem}
\begin{proof}
  Put $h=G^{\alpha}f$. Then a direct computation yields that
  \begin{align*}
    &\varphi(s)\\
    =&\int_{0}^{\infty}\int_F\int_F\int_{F_n}\int_{F_n}e^{-\alpha u}\frac{1_{B(z,r)}(w)}{\mu(B(z,r))}p_n(t-s,x,y)p_n(u,y,z)p(s,w,v)h(v)\mu_n(dy)\mu_n(dz)\mu(dw)\mu(dv)du.
  \end{align*}
Thus, if we define 
\[\psi(s)=e^{-\alpha u}\frac{1_{B(z,r)}(w)}{\mu(B(z,r))}p_n(t-s,x,y)p_n(u,y,z)p(s,w,v)h(v)\]
 for each $(u,v,w,z,y)\in \R_{\geq 0}\times F^2\times F_n^2$,
it suffices to find a dominating function $\Psi(u,v,w,z,y)$ of $\varphi'(s)$ on a neighborhood of each $s\in (0,t)$ to show differentiability.
Note that, since $F$ is compact, it holds from \eqref{kig104} and \cref{derhk} that
\begin{align*}
  \sup_{\substack{t\geq \delta\\x,y\in F}}p(t,x,y)<\infty,\quad \textrm{and}\quad \sup_{\substack{t\geq \delta\\x,y\in F}}\left|\frac{\partial}{\partial t}p(t,x,y)\right|<\infty
\end{align*}
for each $\delta>0$. Hence,
we may take $\Psi=C e^{-\alpha u}$.
Therefore we deduce that
\begin{equation*}
  \varphi'(s)=\int_{0}^{\infty}\int_F\int_F\int_{F_n}\int_{F_n}(\psi_1-\psi_2)\mu_n(dy)\mu_n(dz)\mu(dw)\mu(dv)du.
\end{equation*} 
Here, $\psi_k=\psi_k(u,v,w,z,y)$ is defined by
\begin{equation*}
  \psi_1=e^{-\alpha u}\frac{1_{B(z,r)}(w)}{\mu(B(z,r))}p_n(t-s,x,y)p_n(u,y,z)p'(s,w,v)h(v),
\end{equation*}
and
\begin{equation*}
  \psi_2=e^{-\alpha u}\frac{1_{B(z,r)}(w)}{\mu(B(z,r))}p_n'(t-s,x,y)p_n(u,y,z)p(s,w,v)h(v).
\end{equation*}
By the formula $\Delta p^{t,x}=\frac{\partial}{\partial t}p^{t,x}$ and direct computation,
we observe that
\begin{align}\label{psi1}
  \int_{0}^{\infty}\int_F\int_F\int_{F_n}\int_{F_n}\psi_1\mu_n(dy)\mu_n(dz)\mu(dw)\mu(dv)du
  =P^n_{t-s}G^{\alpha,n}J_n^r\Delta P_sh(x).
\end{align}
Similarly, it holds that
\begin{align}\label{psi2}
  \int_{0}^{\infty}\int_F\int_F\int_{F_n}\int_{F_n}\psi_2\mu_n(dy)\mu_n(dz)\mu(dw)\mu(dv)du
  =P^n_{t-s}\Delta_n G^{\alpha,n}J_n^r P_sh(x).
\end{align}
Combining \eqref{psi1}, \eqref{psi2}, and $\Delta G^{\alpha}=\alpha\Delta-I$ (see \cite[Lemma 1.3.2]{FOT94}), we conclude that
\begin{align*}
  \psi'(s)
  =&P^n_{t-s}G^{\alpha,n}J_n^r\Delta P_sh(x)-P^n_{t-s}\Delta_n G^{\alpha,n}J_n^r P_sh(x)\\
  =&P^n_{t-s}G^{\alpha,n}J_n^r\Delta G^{\alpha}P_sf(x)-P^n_{t-s}\Delta_n G^{\alpha,n}J_n^r G^{\alpha}P_sf(x)\\
  =&P^n_{t-s}G^{\alpha,n}J_n^r(\alpha G^{\alpha}-I)P_sf(x)-P^n_{t-s}(\alpha G^{\alpha,n}-I)J_n^r G^{\alpha}P_sf(x)\\
  =&P^n_{t-s}G^{\alpha,n}J_n^rP_sf(x)-P^n_{t-s}J_n^r G^{\alpha}P_sf(x)\\
  =&P^n_{t-s}(G^{\alpha,n}J_n^r-J_n^r G^{\alpha})P_sf(x),
\end{align*}
which completes the proof.
\end{proof}

We recall the condition $1_{\kappa,p}$
from \cref{cond12}.
We write $\overline{g}$ for 
McShane extension of $g\in\BL^{\kappa}(A),A\subseteq M$ 
(recall the definition of McShane extension from \cref{Mcext}).

While the class of test functions is limited, 
the following theorem provides an estimate 
for the semigroup without any assumptions.
Also,
it is expected that the heat kernel estimates and the results in later examples could be improved 
if the test functions in the following theorem 
could be extended to, for instance, those of the form $g=G^{\alpha}f$. 
Recall the definition of
$K_n$ and $\alpha_n$ from \eqref{iroiro},
and that of $c_1(t,\kappa,p)$ and $c_1([0,t],\kappa,p)$
from \cref{cond12}.
\begin{thm}\label{sg1}
Let $x\in F,x_n\in F_n,\kappa\in [1/2,1],p\in [1,\infty],\alpha\in (0,\alpha_n),t>0$ and $f\in\BL^{\kappa}(M)$.
Then, under the condition $1_{\kappa,p}$,
we have that
\begin{align*}
  |P_t(G^{\alpha})^{\circ 2}f(x)-P_t^n(G^{\alpha,n})^{\circ 2}f(x_n)|
  \lesssim&\g{c}{sg1,n}
    \Bigl\{ \|f\|_{K_n}(d_{\rm{H}}(F,F_n)+d_{\rm{H}}(F,F_n)^{\kappa}+d(x,x_n))\\
    &\quad\quad\quad\quad\quad\qquad\qquad\qquad\qquad +\|f\|_{\BL^{\kappa}(K_n)}d_{\BL^{\kappa}}(\mu,\mu_n)\Bigr\}\\
    \leq&\g{c}{sg1,n}\|f\|_{\BL^{\kappa}(K_n)}( d_{\rm{H}}(F,F_n)+d_{\rm{H}}(F,F_n)^{\kappa}+d(x,x_n)+d_{\BL^{\kappa}}(\mu,\mu_n) ),
\end{align*}
where 
\begin{align*}
  \g{c}{sg1,n}:=\g{c}{estres,n}(t+\alpha^{-1}+\alpha^{-1}c_1(t,\kappa,p)\mu(F)^{1/p})+\alpha^{-1}\mu(F)^{1/p}( c_1(t,\kappa,p)+c_1([0,t],\kappa,p)).
\end{align*}
\end{thm}

\begin{proof}We fix $r>d_{\rm{H}}(F,F_n)$ and set $J_n^r$ as in \eqref{defjn}.
  Note that we have that
   \begin{align*}
    \MoveEqLeft |P_t^n(G^{\alpha,n})^{\circ 2}f(x_n)-P_t(G^{\alpha})^{\circ 2}f(x_n)|\\
    \leq&|P_t^n(G^{\alpha,n})^{\circ 2}f(x_n)-P_t^nG^{\alpha,n}J_n^rG^{\alpha}f(x_n)|
     +|P_t^nG^{\alpha,n}J_n^rG^{\alpha}f(x_n)-G^{\alpha,n}J_n^rP_tG^{\alpha}f(x_n)|\\
     &+|G^{\alpha,n}J_n^rP_tG^{\alpha}f(x_n)-G^{\alpha,n}\overline{P_tG^{\alpha}f}(x_n)|
    +|G^{\alpha,n}\overline{P_tG^{\alpha}f}(x_n)- G^{\alpha}P_tG^{\alpha}f(x)|\\
    =&:I+II+III+IV.
  \end{align*}
  Here,  $\overline{P_tG^{\alpha}f}:M\to\R$ is a McShane extension of $P_tG^{\alpha}f:F\to\R$ with exponent $\kappa$.
We estimate each term separately.

First, using \cref{estres}, $I$ can be bounded
as following: 
\begin{equation}\label{I}
  \begin{aligned}
  I
  \leq&\sup_{y\in F_n}| (G^{\alpha,n})^{\circ 2}f(y)-G^{\alpha,n}J_n^rG^{\alpha}f(y)|\\
  \leq&\alpha^{-1}\sup_{y\in F_n}| G^{\alpha,n}f(y)-J_n^rG^{\alpha}f(y)|\\
  \leq&\alpha^{-1}\sup_{y\in F_n}\frac{1}{\mu(B(y,r))}\int_{B(y,r)}|G^{\alpha,n}f(y)-G^{\alpha}f(z)|\mu(dz)\\
  \lesssim&\alpha^{-1}\g{c}{estres,n}\{\|f\|_{K_n}(d_{\mathrm{H}}(F,F_n)+r)+\|f\|_{\BL^{\kappa}(K_n)}d_{\BL^{\kappa}}(\mu,\mu_n)\}  .
  \end{aligned}
\end{equation}
We next estimate $II$.
Since, for the function $\varphi$ defined in \cref{deriv},
we observe $II=|\varphi(t)-\varphi(0)|\leq \int_{0}^{t}|\varphi'(s)|ds$,
it suffices to estimate $|\varphi'(s)|$.
We write $\overline{P_sf}$ for a McShane extension of $P_sf$ with exponent $\kappa$. 
It follows from condition $1_{\kappa,p}$ and \cref{estres} that
\begin{align*}
  |\varphi'(s)|
  =&|P_{t-s}^n(J_n^rG^{\alpha}-G^{\alpha,n}J_n^r)P_sf(x_n)|\\
  \leq&\sup_{y\in F_n}|(J_n^rG^{\alpha}-G^{\alpha,n}J_n^r)P_sf(y)|\\
  \leq&\sup_{y\in F_n}\frac{1}{\mu(B(y,r))}\int_{B(y,r)}|G^{\alpha}P_sf(z)-G^{\alpha,n}J_n^rP_sf(y)|\mu(dz)\\
  \leq&\sup_{y\in F_n}\frac{1}{\mu(B(y,r))}\int_{B(y,r)}|G^{\alpha}P_sf(z)-G^{\alpha,n}\overline{P_sf}(y)|\mu(dz)\\
  &+\sup_{y\in F_n}\frac{1}{\mu(B(y,r))}\int_{B(y,r)}|G^{\alpha,n}\overline{P_sf}(y)-G^{\alpha,n}J_n^rP_sf(y)|\mu(dz)\\
  \lesssim&\g{c}{estres,n}\{ \|P_sf\|_F(d_{\mathrm{H}}(F,F_n)+r)+\|P_sf\|_{\BL^{\kappa}(F)}d_{\BL^{\kappa}}(\mu,\mu_n) \}\\
    &+\alpha^{-1}\sup_{w\in F_n}|J_n^rP_sf(w)-\overline{P_sf}(w)|\\
    \leq& \g{c}{estres,n}\{ \|f\|_F(d_{\mathrm{H}}(F,F_n)+r)+(\|f\|_F+c_1(s,\kappa,p)\|f\|_{L^p(F,\mu)})d_{\BL^{\kappa}}(\mu,\mu_n) \}\\
    &+\alpha^{-1}c_1(s,\kappa,p)\|f\|_{L^p(F,\mu)}r^{\kappa}.
\end{align*}
Hence we deduce that
\begin{equation}\label{II5}
  \begin{aligned}
    II
  \lesssim& 
  \g{c}{estres,n}\{ \|f\|_F(d_{\mathrm{H}}(F,F_n)+r)t+(t\|f\|_F+c_1([0,t],\kappa,p)\|f\|_{L^p(F,\mu)})d_{\BL^{\kappa}}(\mu,\mu_n) \}\\
    &+\alpha^{-1}c_1([0,t],\kappa,p)\|f\|_{L^p(F,\mu)}r^{\kappa}.
  \end{aligned}
\end{equation}
See \cref{cond12} for the definition if $c_1([0,t],\kappa,p)$. 

Using \cref{continuity}(ii) and the fact that $\|G^{\alpha}f\|_{L^p(F,\mu)}\leq \alpha^{-1}\|f\|_{L^p(F,\mu)}$, $III$ is bounded as follows:
\begin{equation}
  \begin{aligned}\label{III}
    III
  \leq&\alpha^{-1}\sup_{y\in F_n}|J_n^rP_tG^{\alpha}f(y)-\overline{P_tG^{\alpha}f}(y)|\\
  \leq&\alpha^{-1}\sup_{y\in F_n}\frac{1}{\mu(B(y,r))}\int_{B(y,r)}|P_tG^{\alpha}f(z)-\overline{P_tG^{\alpha}f}(y)|\mu(dz)\\
  \leq&\alpha^{-1}c_1(t,\kappa,p) \|f\|_{L^p(F,\mu)}r^{\kappa}.
  \end{aligned}
\end{equation}
Finally, we deal with $IV$. 
By \cref{estres} and \cref{continuity}(ii), we have that
\begin{equation}\label{IV}
  \begin{aligned}
    IV
  \lesssim&\g{c}{estres,n}\{\|P_t G^{\alpha}f\|_{F}(d_{\mathrm{H}}(F,F_n)+d(x,x_n))+\|P_tG^{\alpha}f\|_{\BL^{\kappa}(F)}d_{\BL^{\kappa}}(\mu,\mu_n)\}\\
  \lesssim&\alpha^{-1}\g{c}{estres,n}\{\|f\|_{F}(d_{\mathrm{H}}(F,F_n)+d(x,x_n))+(\|f\|_{F}+c_1(t,\kappa,p)\|f\|_{L^p(F,\mu)})d_{\BL^{\kappa}}(\mu,\mu_n)\}.
  \end{aligned}
\end{equation}
Combining \eqref{I}, \eqref{II5}, \eqref{III}, and \eqref{IV}, using $\|f\|_{L^p(F,\mu)}\leq \mu(F)^{1/p}\|f\|_{F}$, and letting $r\to d_{\rm{H}}(F,F_n)$,
we obtain the desired inequality.
\end{proof}

Note that, by \cref{continuity}(ii), $G^{\alpha}\psi$ is a Lipschitz continuous 
function with 
\[
\|G^{\alpha}\psi\|_{\Lip^1}\lesssim \alpha^{-1} \g{c}{continuity,1}\|\psi\|_{L^1(F,\mu)}\leq \alpha^{-1} \g{c}{continuity,1}\mu(F)\|\psi\|_{F}
\]
for any bounded measurable function $\psi:F\to\R$ and $\alpha\le 1$.
Also, we write the constant $\g{c}{continuity,1}$ for the $n$-version as
$\g{c}{continuity,1.n}$.

While the semigroup has been denoted by $P_t$ above,
we denote it using expectation in statements 
for the sake of intuitive clarity.

\begin{lem}\label{sg2}
For $f\in \BL^{\kappa}(M)$, define $g:M\to\R$ by $g=\overline{(G^{\alpha})^{\circ 2}f}$ 
with the exponent of the McShane extension being $1$.
Then, for any 
$x\in F,\;x_n\in F_n,\;\kappa\in [1/2,1],\;p\in [1,\infty],\;\alpha\in (0,\alpha_n),$ and $t>0$,
under the condition $1_{\kappa,p}$,
it holds that
\begin{align*}
  \Big|E_x[g(X_t)]-E_{x_n}^n[g(X^n_t)]\Big|
  \lesssim &\g{c}{sg2,n}
    \Bigl\{ \|f\|_{K_n}(d_{\rm{H}}(F,F_n)+d_{\rm{H}}(F,F_n)^{\kappa}+d(x,x_n))+\|f\|_{\BL^{\kappa}(K_n)}d_{\BL^{\kappa}}(\mu,\mu_n)\Bigr\}\\
    \leq&\g{c}{sg2,n}\|f\|_{\BL^{\kappa}(K_n)}
    \Bigl\{d_{\rm{H}}(F,F_n)+d_{\rm{H}}(F,F_n)^{\kappa}+d(x,x_n)+d_{\BL^{\kappa}}(\mu,\mu_n)\Bigr\},
\end{align*}
where 
\begin{align*}
  \g{c}{sg2,n}:=\g{c}{sg1,n}+\alpha^{-1}\g{c}{estres,n}(1+\mu_n(F_n)\g{c}{continuity,1.n})+\alpha^{-2}\Big[\mu(F)\g{c}{continuity,1}+\mu_n(F_n)\g{c}{continuity,1.n}\Big].
\end{align*}
\end{lem}
\begin{proof}
  All McShane extensions in the following proof are with exponent $1$.

  We first observe that
  \begin{align*}
    \Big|E_x[g(X_t)]-E_{x_n}^n[g(X^n_t)]\Big|
    \leq&| P_t(G^{\alpha})^{\circ 2}f(x)-P_t^n(G^{\alpha,n})^{\circ 2}f(x_n)|
    +\sup_{y\in F_n}| (G^{\alpha,n})^{\circ 2}f(y)-\overline{(G^{\alpha,n})^{\circ 2}f}(y)|\\
    =&:I+II.
  \end{align*}
  The term $I$ can be estimated by \cref{sg1}, and we have
  \begin{align*}
    &|P_t(G^{\alpha})^{\circ 2}f(x)-P_t^n(G^{\alpha,n})^{\circ 2}f(x_n)|\\
  \lesssim&\g{c}{sg1,n}
    \Bigl\{ \|f\|_{K_n}(d_{\rm{H}}(F,F_n)+d_{\rm{H}}(F,F_n)^{\kappa}+d(x,x_n))+\|f\|_{\BL^{\kappa}(K_n)}d_{\BL^{\kappa}}(\mu,\mu_n)\Bigr\}.
  \end{align*}
  We next bound $II$. Fix $y\in F_n$ and take $z\in F$ such that
  $d(y,z)\leq d_{\rm{H}}(F,F_n)$.
  Then we observe from \cref{continuity}(ii) and \cref{estres} that
  \begin{equation*}\label{2II}
    \begin{aligned}
    II
    \leq& | \overline{(G^{\alpha})^{\circ 2}f}(y)-(G^{\alpha})^{\circ 2}f(z)|
    +|(G^{\alpha})^{\circ 2}f(z)-G^{\alpha}\overline{G^{\alpha,n}f}(z)|
    +| G^{\alpha}\overline{G^{\alpha,n}f}(z)-(G^{\alpha,n})^{\circ 2}f(y)|\\
    \lesssim&\|(G^{\alpha})^{\circ 2}f\|_{\Lip^1}d_{\rm{H}}(F,F_n)
    +\alpha^{-1}\sup_{w\in F}| G^{\alpha}f(w)-\overline{G^{\alpha,n}f}(w)|\\
    &+\g{c}{estres,n}\Bigl\{\|G^{\alpha,n}f\|_{F_n}d_{\mathrm{H}}(F,F_n)+(\|G^{\alpha,n}f\|_{F_n}+\|G^{\alpha,n}f\|_{\Lip^{\kappa}})d_{\BL^{\kappa}}(\mu,\mu_n)\Bigr\}\\
    \leq&\alpha^{-2}\mu(F)\g{c}{continuity,1}\|f\|_{F}d_{\mathrm{H}}(F,F_n)
    +\alpha^{-1}\sup_{w\in F}| G^{\alpha}f(w)-\overline{G^{\alpha,n}f}(w)|\\
    & +\g{c}{estres,n}\Bigl\{\alpha^{-1}\|f\|_{F_n}d_{\mathrm{H}}(F,F_n)+(\alpha^{-1}+\alpha^{-1}\mu_n(F_n)\g{c}{continuity,1.n})\|f\|_{F_n}d_{\BL^{\kappa}}(\mu,\mu_n)\Bigr\}.
  \end{aligned}
  \end{equation*}
  It remains to bound $\sup_{w\in F}| G^{\alpha}f(w)-\overline{G^{\alpha,n}f}(w)|$.
  Fix $w\in F$. We may find $w_n\in F_n$ with $d(w,w_n)\leq d_{\rm{H}}(F,F_n)$.
  It is then the case that
  \begin{equation*}\label{wwn}
    \begin{aligned}
      \MoveEqLeft | G^{\alpha}f(w)-\overline{G^{\alpha,n}f}(w)|\\
    \leq&|G^{\alpha}f(w)-G^{\alpha,n}f(w_n)|+|G^{\alpha,n}f(w_n)-\overline{G^{\alpha,n}f}(w)|\\
    \lesssim&\g{c}{estres,n}(\|f\|_{K_n}d_{\mathrm{H}}(F,F_n)+\|f\|_{\BL^{\kappa}(K_n)}d_{\BL^{\kappa}}(\mu,\mu_n))
    +\alpha^{-1}\g{c}{continuity,1.n}\mu_n(F_n)\|f\|_{F_n} d_{\rm{H}}(F,F_n).
    \end{aligned}
  \end{equation*}
Combining the above estimates, we deduce the result.
\end{proof}

The following lemma provides an estimate for the short-time error of the semigroup.
This is useful for understanding the influence of the smoothing effect in the proof of \cref{gen}.
Recall the assumptions $\rm{(A2)},\;\rm{(A3)}$ and parameters $s_0,\;s_1,$ and $\theta$ 
from \cref{as11}.
\begin{lem}\quad\par\label{12q2}
  \begin{itemize}
    \item [$\rm{(i)}$]Suppose that $(F,R,\mu)$ satisfies assumption $\rm{(A2)}$.
  Then
for any $f\in\BL^{\kappa},\;\delta\in (0,1],\;\varepsilon\in (0,1]$, and $\eta\in (0,\varepsilon)$, we have 
\begin{align*}
  \sup_{\substack{t\geq 0\\x\in F}}|P_{t+\delta}f(x)-P_tf(x)|\lesssim
  \g{c}{12q2}\left\{\|f\|_{\Lip^{\kappa}}\varepsilon^{\kappa}
    +\|f\|_{F}\left[\frac{\eta^{\theta}}{\varepsilon^{\theta} }+\frac{\delta}{\eta^{s_0 \theta}(\varepsilon^{\theta}-\eta^{\theta})}\right]\right\} ,
\end{align*}
where
\begin{align*}
  \g{c}{12q2}=1+\frac{1}{c_l \left[ \frac{c_{\rm{LR}}(1\wedge \frac{1}{2}\diam F)^{\theta}}{1+(1\vee \frac{1}{\diam F})c_{\rm{LR}}(1\wedge \frac{1}{2}\diam F)^{\theta}} \right]^{1+s_0}}.
\end{align*}
\item[$\rm{(ii)}$]Suppose that $(F,R,\mu)$ satisfies assumption $\rm{(A3)}$.
  Then
for any $f\in\BL^{\kappa},\;\delta>0,\;\varepsilon\in (0,1]$, and $\eta\in (0,\varepsilon)$, we have 
\begin{align*}
  \sup_{\substack{t\geq 0\\x\in F}}|P_{t+\delta}f(x)-P_tf(x)|\lesssim
  \|f\|_{\Lip^{\kappa}}\varepsilon^{\kappa}
    +\|f\|_{F}\exp \left[ -\g{c}{summ}(1\wedge \frac{1}{2}\diam F)^{1+\beta}\frac{\varepsilon^{1+\beta}}{\delta^{\beta}}\right],
\end{align*}
where $\beta=\frac{1-(2+s_0)(s_0-s_1)}{s_1+2(2+s_0)(s_0-s_1)}>0$.
\end{itemize}
\end{lem}

\begin{proof}
  It is enough to consider the case $t=0$ since the semigroup is a contraction with respect to the supremum norm.
  Fix $0<\eta<\varepsilon\leq 1$ and set 
  \[
  \varepsilon_0=(1\wedge\frac{1}{2}\diam F)\varepsilon,\qquad \eta_0=\frac{c_{\rm{LR}}(1\wedge \frac{1}{2}\diam F)^{\theta}}{1+(1\vee \frac{1}{\diam F})c_{\rm{LR}}(1\wedge \frac{1}{2}\diam F)^{\theta}}\eta^{\theta}.
  \]
  Then we have that $B(x,\varepsilon_0)\neq F$ and $\eta_0< R(x,B(x,\varepsilon_0)^c)\wedge\diam F$.
  We first confirm $\eta_0< R(x,B(x,\varepsilon_0)^c)$.
  Since it holds that $B(x,\varepsilon_0)\neq F$ from $\varepsilon_0<\frac{1}{2}\diam F$, we have 
  $c_{\rm{LR}}\varepsilon_0^{\theta}$. Thus it is enough to prove 
  \begin{align*}
    \frac{1}{1+(1\vee \frac{1}{\diam F})c_{\rm{LR}}(1\wedge \frac{1}{2}\diam F)^{\theta}}\eta^{\theta}
    <\varepsilon^{\theta},
  \end{align*}
  to confirm $\eta_0< R(x,B(x,\varepsilon_0)^c)$, and this follows from $\eta<\varepsilon$ and $\theta>0$.
  Next we show $\eta_0<\diam F$. Using $\frac{x}{1+ax}<\frac{1}{a}$ for $a,x>0$, we observe that
  \begin{align*}
    \eta_0<\frac{1}{1\vee \frac{1}{\diam F}}\eta^{\theta}<1\vee \diam F\le\diam F.
  \end{align*}
  By \cref{4.2-a}(ii), we have that
  \begin{align}
    |P_{\delta}f(x)-f(x)|
    \leq&E_x[|f(X_{\delta})-f(x)|:R(x,X_{\delta})\leq \varepsilon_0]+E_x[|f(X_{\delta})-f(x)|:R(x,X_\delta)>\varepsilon_0]\notag\\
    \lesssim&\|f\|_{\Lip^{\kappa}}\varepsilon_0^{\kappa}+\|f\|_F P_x(R(x,X_{\delta})>\varepsilon_0)\notag\\
    \lesssim&\|f\|_{\Lip^{\kappa}}\varepsilon_0^{\kappa}+\|f\|_F P_x(\sigma_{B(x,\varepsilon_0)^c}\leq \delta)\label{expa}\\
    \lesssim&\|f\|_{\Lip^{\kappa}}\varepsilon_0^{\kappa}+\|f\|_F\left[ \frac{\eta_0}{R(x,B(x,\varepsilon_0)^c)}+\frac{\delta}{\mu(B(x,\eta_0))(R(x,B(x,\varepsilon_0)^c)-\eta_0)} \right]\notag\\
    \leq&\|f\|_{\Lip^{\kappa}}\varepsilon^{\kappa}\notag\\
    &+\|f\|_F\left[\frac{\eta^{\theta}}{\varepsilon^{\theta} }+\frac{1}{c_l \left[ \frac{c_{\rm{LR}}(1\wedge \frac{1}{2}\diam F)^{\theta}}{1+(1\vee \frac{1}{\diam F})c_{\rm{LR}}(1\wedge \frac{1}{2}\diam F)^{\theta}} \right]^{1+s_0}}\frac{\delta}{\eta^{s_0 \theta}(\varepsilon^{\theta}-\eta^{\theta})} \right]\notag,
  \end{align}
  which implies the result of (i).
  That of (ii) follows from \eqref{expa} and \cref{summ}(v).
\end{proof}

Finally, we are now ready to prove the main result of this section, 
which provides an estimate for semigroups with H\"older continuous 
test functions under suitable assumptions.

\begin{proof}[Proof of \cref{gen}$\rm{(i)}$]
  Take $\alpha:=\frac{1}{2}\alpha_n\in (0,\alpha_n)$ and $\delta\in (0,1]$.
  To begin with, we observe that 
  \begin{equation}\label{IIIIII}
  \begin{aligned}
    |P_tf(x)-P_{t}^nf(x_n)|
    \leq& |P_tf(x)-P_{t+\delta}f(x)|\\
    &+|P_t(G^{\alpha})^{\circ 2}(\alpha-\Delta)^2P_{\delta}f(x)-P_t^n \overline{(G^{\alpha})^{\circ 2}(\alpha-\Delta)^2P_{\delta}f}(x_n)|\\
    &+| P_t^n \overline{P_{\delta}f}(x_n)-P^n_tJ_n^rP_{\delta}f(x_n)|
    +|P^n_tJ_n^rP_{\delta}f(x_n)-P^n_tf(x_n)|\\
    =:&I+II+III+IV.
  \end{aligned}
  \end{equation}
  Here, $\overline{(G^{\alpha})^{\circ 2}(\alpha-\Delta)^2P_{\delta}f}$ is a MacShane extension with exponent $1$.
  By \cref{12q2}, we have that
  \begin{align*}
    I\lesssim  \g{c}{12q2}\|f\|_{\BL^{\kappa}}\left[\varepsilon^{\kappa}
    +\frac{\eta^{\theta}}{\varepsilon^{\theta} }+\frac{\delta}{\eta^{s_0 \theta}(\varepsilon^{\theta}-\eta^{\theta})} \right],
  \end{align*}
  for any $\varepsilon\in (0,1]$ and $\eta\in (0,\varepsilon)$.
  Next, it follows from \cref{sg2} and \cref{suff}(iii) with $p=q=2$ that
  \begin{align*}
    II
    \lesssim&\g{c}{sg2,n}\|(\alpha-\Delta)^2 P_{\delta}f\|_{\BL^{\kappa}}h\\
    \leq& \g{c}{sg2,n}\mu(F)^{1/2}\|f\|_F\Bigl\{ c_{\rm{A1}}^{1/2} +\g{c}{suffcond1,3}\alpha^{-(1+\frac{s_1}{2(1+s_0)\theta})(2\kappa-1)}\Bigr\}\delta^{-\left(2+\kappa\right)}h.
  \end{align*}
  Here we used that $2+\frac{s_0}{2(1+s_0)}<2+\frac{1}{2}\leq 2+\kappa$.
  Using \cref{suff}(iii), we deduce that
  \begin{align*}
    III
    \leq \sup_{y\in F_n}\frac{1}{\mu(B(y,r))}\int_{B(y,r)}| P_{\delta}f(z)-\overline{P_{\delta}f}(y)|\mu(dz)
    \leq\|P_{\delta}f\|_{\Lip^1}r^{\kappa}.
  \end{align*}
  Taking the limit as $r\downarrow d_{\rm{H}}(F,F_n)$,
  we obtain  
\begin{align*}
  \limsup_{r\to d_{\rm{H}}(F,F_n)}III\leq
  \g{c}{suffcond1,3}\mu(F)^{1/2}\|f\|_{F}\delta^{-[1-\frac{s_1}{2(1+s_0)\theta}]}d_{\rm{H}}(F,F_n)
    \leq\g{c}{suffcond1,3}\mu(F)^{1/2}\|f\|_{F}\delta^{-[1-\frac{s_1}{2(1+s_0)\theta}]}h^{1/\kappa}.
\end{align*}
  Finally, by \cref{12q2}, we deduce that
  \begin{align*}
    IV
    \leq&\sup_{y\in F_n}\frac{1}{\mu(B(y,r))}\int_{B(y,r)}|P_{\delta}f(z)-P_{\delta}f(y)|+|P_{\delta}f(y)-f(y)|\mu(dz)\\
    \lesssim&\g{c}{suffcond1,3}\mu(F)^{1/2}\|f\|_{F}\delta^{-[\kappa-\frac{s_1}{2(1+s_0)\theta}(2\kappa-1)]}r
    +\g{c}{12q2}\|f\|_{\BL^{\kappa}}\left[\varepsilon^{\kappa}
    +\frac{\eta^{\theta}}{\varepsilon^{\theta} }+\frac{\delta}{\eta^{s_0 \theta}(\varepsilon^{\theta}-\eta^{\theta})} \right].
  \end{align*}
Similarly for $III$, we obtain
\begin{align*}
  \limsup_{r\to d_{\rm{H}}(F,F_n)}IV\leq
  \g{c}{suffcond1,3}\mu(F)^{1/2}\|f\|_{F}\delta^{-[\kappa-\frac{s_1}{2(1+s_0)\theta}(2\kappa-1)]}h
    +\g{c}{12q2}\|f\|_{\BL^{\kappa}}\left[\varepsilon^{\kappa}
    +\frac{\eta^{\theta}}{\varepsilon^{\theta} }+\frac{\delta}{\eta^{s_0 \theta}(\varepsilon^{\theta}-\eta^{\theta})} \right].
\end{align*}
  Altogether, we conclude that
  \begin{equation}\label{kkkkk}
    \begin{aligned}
  |P_tf(x)-P_{t}^nf(x_n)|
    \lesssim\g{C'}{gen}\|f\|_{\BL^{\kappa}}\left[\varepsilon^{\kappa}
    +\frac{\eta^{\theta}}{\varepsilon^{\theta} }+\frac{\delta}{\eta^{s_0 \theta}(\varepsilon^{\theta}-\eta^{\theta})} 
    +\delta^{-\left( 2+\kappa \right)}h+\delta^{-[1-\frac{s_1}{2(1+s_0)\theta}]}h^{1/\kappa}\right],    
    \end{aligned}
  \end{equation}
  where
  \begin{align*}
    \g{C'}{gen}=\g{c}{12q2}
    +\g{c}{sg2,n}\mu(F)^{1/2}\Bigl\{ c_{\rm{A1}} +\g{c}{suffcond1,3}\alpha^{-(1+\frac{s_1}{2(1+s_0)\theta})(2\kappa-1)}\Bigr\}
    +\g{c}{suffcond1,3}\mu(F)^{1/2}.
  \end{align*}
  Substituting $\varepsilon=h^a,\;\delta=h^b,$ and $\eta=h^{a+c}$, where $a,b,c\in (0,\infty)$ into \eqref{kkkkk}, 
  we obtain that
  \begin{align*}
    |P_tf(x)-P_{t}^nf(x_n)|\lesssim\frac{\g{C'}{gen}}{1-h^{\theta c}}\|f\|_{\BL^{\kappa}}\sum_{k=1}^4 E_k.
  \end{align*}
  Here, $E_k=E_k(a,b,c),k=1,\dots ,5$ are defined by
  \begin{align*}
    E_1=\kappa a,\quad E_2=\theta c,\quad E_3=-(1+s_0)\theta a+b-s_0\theta c,
    \;\;\rm{and}\;\;E_4=-(2+\kappa) b+1.
  \end{align*}
  Thus, it is enough to compute
  $E:=\sup_{a,b,c>0}\min_{1\leq k\leq 4}E_k$.
  By direct computation, we can verify that $E=\frac{\kappa}{\kappa+(1+s_0)(2+\kappa)(\kappa+\theta)}$, and the supremum is attainable.
  This completes the proof.
\end{proof}
\begin{proof}[Proof of \cref{gen}$\rm{(ii)}$]
  The strategy is the same and we bound the four terms in \eqref{IIIIII}.
  By \cref{suff}(iv), \cref{12q2}(ii), and the same argument as in the proof of (i),
  we deduce the following:
  \begin{equation*}
    I+IV\lesssim\left(1+\g{c}{suffcond1,4}\mu(F)^{1/2}\right)\|f\|_{\BL^{\kappa}}\left[\delta^{-[\kappa-\frac{s_1}{2(1+s_0)\Theta}(2\kappa-1)]}h+\varepsilon^{\kappa}+\exp\left(-A\frac{\varepsilon^{1+\beta}}{\delta^{\beta}}\right)\right],
  \end{equation*}
  \begin{equation*}
    II\lesssim\g{c}{sg2,n}\mu(F)^{1/2}\|f\|_F\Bigl\{ c_{\rm{A1}}^{1/2} +\g{c}{suffcond1,4}\alpha^{-(1+\frac{s_1}{2(1+s_0)\Theta})(2\kappa-1)}\Bigr\}\delta^{-\left(2+\kappa\right)}h,
  \end{equation*}
  \begin{equation*}
    III\lesssim\g{c}{suffcond1,4}\mu(F)^{1/2}\|f\|_{F}\delta^{-[1-\frac{s_1}{2(1+s_0)\Theta}]}h^{1/\kappa},
  \end{equation*}
  where $A=\g{c}{summ}(1\wedge \frac{1}{2}\diam F)^{1+\beta},\beta=\frac{1-(2+s_0)(s_0-s_1)}{s_1+2(2+s_0)(s_0-s_1)},$ and $\Theta=\frac{1+s_0-s_1}{1-(s_0-s_1)(2+s_0)}$.
  Combining these estimates, we deduce that
  \begin{align*}
    |P_tf(x)-P^n_tf(x_n)|
    \lesssim \g{C''}{gen}\|f\|_{\BL^{\kappa}}\left[\varepsilon^{\kappa}+\exp\left(-A\frac{\varepsilon^{1+\beta}}{\delta^{\beta}}\right)+ \delta^{-\left(2+\kappa\right)}h \right],
  \end{align*}
  where
  \begin{align*}
    \g{C''}{gen}=1+\mu(F)^{1/2}\g{c}{suffcond1,4}+\g{c}{sg2,n}\mu(F)^{1/2}\Bigl[ c_{\rm{A1}}^{1/2} +\g{c}{suffcond1,4}\left(\frac{\alpha_n}{2}\right)^{-(1+\frac{s_1}{2(1+s_0)\Theta})(2\kappa-1)}\Bigr].
  \end{align*}
  Taking 
  \begin{equation}\label{ep}
    \varepsilon=h^{\frac{1}{\kappa+(1+s_0)(2+\kappa)\Theta}}(\log(1/h))^{\frac{2+\kappa}{ (1+\beta)(2+\kappa)+\kappa\beta }},
  \end{equation}
  \begin{equation*}
    \delta=A^{1/\beta}\left(\frac{\kappa}{\kappa+(1+s_0)(2+\kappa)\Theta}\right)^{-\frac{1}{\beta}}h^{\frac{(1+s_0)\Theta}{ \kappa+(1+s_0)(2+\kappa)\Theta }}(\log(1/h))^{-\frac{\kappa}{(1+\beta)(2+\kappa)+\kappa\beta }},
  \end{equation*}
  and using the relation $1+\beta^{-1}=(1+s_0)\Theta$, we obtain the result.
\end{proof}

\section{Estimates for heat kernels}\label{esthk}
In this section, 
we derive the convergence rate for heat kernels by utilizing the corresponding estimates for semigroups. 
As a corollary, we also establish a semigroup estimate under assumption $\rm{(A1)}$. 
See \cref{sgrw} for more details.
We use the same notation and assumptions as in \Cref{efsg}.
In particular, we assume that $(F_n,R_n,\mu_n)\in \bbF_c^{\circ}$ converges to $(F,R,\mu)\in\bbF_c^{\circ}$ 
as described in 
\cref{nonrootcpt} (recall the definition of $\bbF_c^{\circ}$ from \eqref{bbfcirc}), and 
regard $F_n,$ $n\in\N,$ and $F$ as subsets of a common compact metric space $(M,d)$. 
Also, while we have pursued the constants to handle the non-compact cases,
we do not do so in this section. 
Thus we write $A\lesssim B$ if there exists a time-independent constant (not necessarily universal) $C\in (0,\infty)$ such that $A\leq CB$.

Recall the assumptions $\rm{(A1)},$ $\rm{(A2)},$ $\rm{(A3)}$ and parameters $s_0,\;s_1,\;\theta$ from \cref{as11}.
Then our main results in this section are the following \cref{hkest} and \cref{sgrw}.

\begin{thm}\label{hkest}
  Fix $t>0,\;x,y\in F,$ and $x_n,y_n\in F_n$.
Suppose that $(F,R,\mu)$ satisfies $\rm{(A1)}$
for $s_0,s_1>0$ with $s_0<s_1+1$,
and that
\begin{align*}
  h:=d_{\rm{H}}(F,F_n)^{\kappa}+d_{\BL^{\kappa}}(\mu,\mu_n)+d(x,x_n)+d(y,y_n)
\end{align*}
is less than $1/2$.
\begin{itemize}
  \item [$\mathrm{(i)}$]For each $\kappa\in [1/2,2/3)$, it holds that
\begin{align*}
  |p(t,x,y)-p_n(t,x_n,y_n)|\lesssim h^E,
\end{align*}
where
\begin{align*}
  E=\frac{2(s_1-s_0+1)}{\left\{ (2\kappa+5)s_0+2\kappa+4 \right\}(s_1+2)+2(s_1-s_0+1)}.
\end{align*}
Moreover, the time-dependence of the constant in $\lesssim$ is $t^{-2},\;t\in (0,1]$.
\item[$\mathrm{(ii)}$]If $(F,R,\mu)$ satisfies the assumption $\mathrm{(A2)}$, then it holds for $\kappa\in [1/2,1]$ that
\begin{align*}
  |p(t,x,y)-p_n(t,x_n,y_n)|\lesssim h^E,
\end{align*}
where
\begin{align*}
  E=\frac{s_1-s_0+1}{(\kappa+2)(1+s_0)(s_1+2)+\frac{s_0}{2}+(s_1-s_0+1)}.
\end{align*}
Moreover, the time-dependence of the constant in $\lesssim$ is $t^{-2},\;t\in (0,1]$.
\item[$\rm{(iii)}$]If $(F,R,\mu)$ satisfies the assumption $\mathrm{(A3)}$, then it holds for $\kappa\in [1/2,1]$ that
\begin{align*}
  |p(t,x,y)-p_n(t,x_n,y_n)|\lesssim  h^{E_1}(\log(1/h))^{E_2},
\end{align*}
where
\begin{align*}
  E_1=\frac{s_1-s_0+1}{(2+\kappa)(1+s_0)\Theta+\frac{s_0}{2}+(s_1-s_0+1)},
  \quad E_2=\frac{(2+\kappa)\{s_1+2(2+s_0)(s_0-s_1)\}}{1-(2+s_0)(s_0-s_1)},
\end{align*}
and $\Theta=\frac{1+s_0-s_1}{1-(s_0-s_1)(2+s_0)}$.
Moreover, the time-dependence of the constant in $\lesssim$ is $t^{-2},\;t\in (0,1]$.
\end{itemize}
\end{thm}

\begin{rmk}
  Since $\frac{1}{2}[(2\kappa+5)s_0+2\kappa+4]
  =(\kappa+2)(1+s_0)+\frac{s_0}{2}$,
  the estimate in (ii) is better than that in (i).
  Similarly, by $\Theta< 2+s_1$ if and only if $s_0-s_1<\frac{1+s_1}{1+(2+s_0)(2+s_1)}$, 
  the estimate in (iii) is sharper than that in (ii) if $s_0$ is sufficiently close to $s_1$. 
\end{rmk}

In \cref{gen}, we obtained the semigroup estimate only under $\rm{(A2)}$ and $\rm{(A3)}$, 
as $\rm{(A1)}$ alone was insufficient to derive the evaluation. 
However, by virtue of \cref{hkest}, we are now able to establish 
the estimate under the weaker assumption $\rm{(A1)}$.

\begin{cor}\label{sgrw}
Let $\kappa,\;E,\;E_1,$ and $E_2$ be in \cref{hkest}, and $x\in F,\;x_n\in F_n$ be points satisfying
\begin{align*}
  h:=d_{\rm{H}}(F,F_n)^{\kappa}+d_{\BL^{\kappa}}(\mu,\mu_n)+d(x,x_n)\leq \frac{1}{4}.
\end{align*}
Then, for any $t>0$ and $f\in \BL^{\kappa}$,
it holds that
\begin{align*}
  |P_tf(x)-P_t^nf(x_n)|\lesssim 
  \begin{dcases}
    \|f\|_{\BL^{\kappa}}h^E, & \text{if the assumption $\rm{(A1)}$ or $\rm{(A2)}$ holds.}\\
    \|f\|_{\BL^{\kappa}}h^{E_1}(\log (1/h))^{E_2}, & \text{if the assumption $\rm{(A3)}$ holds.}
  \end{dcases}
\end{align*}
\end{cor}

\begin{rmk}\label{66}
Under the assumption $\rm{(A2)}$, the estimate in \cref{gen} is sharper than that of \cref{sgrw} if and only if 
\begin{align*}
  \theta\leq \frac{\kappa}{s_1-s_0+1}\left( 1+s_0+\frac{s_0}{2(1+s_0)(2+\kappa)} \right)=:\theta_0.
\end{align*} 
Recall that we may take $\theta=1+s_0-s_1$ if $(F,R)$ is uniformly perfect (see \cref{summ}(i)).
We do not give general formula (for $\kappa$) when $1+s_0-s_1\leq \theta_0$ since it will be complicated.
However, if $\kappa=1$, we always have $1+s_0-s_1\leq \theta_0$. In fact, 
if we set $\delta=s_1-s_0+1\in (0,1]$, we have 
\begin{align*}
  1+s_0-s_1=\frac{\delta (2-\delta)}{\delta}\leq \frac{1}{\delta}\leq\theta_0.
\end{align*}
Moreover, if $s_0\geq 1$, we also have $1+s_0-s_1\leq \theta_0$ from
$1+s_0-s_1\leq \delta^{-1}$ and $1\leq \frac{1}{2}\times 2\leq \kappa (1+s_0)$.
Similarly, we may confirm that, under the assumption $\rm{(A3)}$,
the estimate in \cref{gen} is better than that of \cref{sgrw} if $\kappa=1$.
\end{rmk}

Note that, since $p_n(t,\cdot,\cdot):F_n\times F_n\to\R$ is Lipschitz continuous 
from \cref{unihk}
(with respect to the metric $d_1$ on $M\times M$ defined by $d_1((x_1,x_2),(y_1,y_2))=d(x_1,y_1)+d(x_2,y_2)$),
there exists a McShane extension of $p_n(t,\cdot,\cdot)$
with exponent $1$ denoted by $\overline{p_n}(t,\cdot,\cdot)$. 

Set $L(t),M(t)<\infty$ and $\alpha_0>0$ as follows:
\begin{equation*}
  L(t):=\sup_{x\in F}\|p^{t,x}\|_{\Lip^1}\vee \sup_{\substack{n\in\N\\x\in F_n}}\|p_n^{t,x}\|_{\Lip^1}
  \lesssim 1+t^{-2},
\end{equation*}
\begin{equation*}
  M(t):=\sup_{x,y\in F}|p(t,x,y)|\vee \sup_{\substack{n\in\N\\x,y\in F_n}}|p_n(t,x,y)|
  \lesssim 1+t^{-1},
\end{equation*}
\begin{equation*}
  \alpha_0=\inf_{n\in\N}\alpha_n=\frac{1}{2}\wedge \frac{1}{2}\left[\sup_n\mu(F_n)\diam F_n\right]^{-2}.
\end{equation*}
See \cref{unihk} and \eqref{kig104} for the bound of $L(t)$ and $M(t)$.

Our fundamental strategy for proving \cref{hkest} is based on the approximation of the delta function.
The first step is the following lemma.
Recall the definitions of conditions $1_{\kappa,p}$ and $2_{\kappa,p,q},$ and
the constants $c_1(t,\kappa,p),c_2(\alpha,t,\kappa,p),$ and $c_3(\alpha,t,\kappa,q)$
from \cref{cond12}.

\begin{lem}\label{totyuu}
Fix $x,y\in F,\;y_n\in F_n,\;t>0,\;\kappa\in [1/2,1],\;p_0,p,q\in [1,\infty],$ and $\alpha\in (0,\alpha_0)$.
\begin{itemize}
  \item [$\rm{(i)}$]Under the conditions $1_{\kappa,p_0}$ and $2_{\kappa,p,q}$, we have that
\begin{align*}
  \MoveEqLeft\left| \left(\int P_{\varepsilon}f_{r}d\mu\right)p(t,x,y)-\left( \int \overline{P_{\varepsilon}f_r}d\mu_n \right)\overline{p_n}(t,x,y_n) \right|\\
  \lesssim&L(t)\mu\left( B\left(x,4r\right)\right)r+M(t)\left\{ \frac{\delta}{r}+\frac{\varepsilon}{m(\delta)(r-\delta)} \right\}\\
  &+\g{c}{sg2,n}\Bigl\{c_2(\alpha,\varepsilon,\kappa,p)\mu(B(x,r))^{1/p} +c_3(\alpha,\varepsilon,\kappa,q) \mu(B(x,r))^{1/q}\Bigr\}\\
  &\quad\times \Bigl\{ d_{\rm{H}}(F,F_n)+d_{\rm{H}}(F,F_n)^{\kappa}+d(y,y_n)+d_{\BL^{\kappa}}(\mu,\mu_n)\Bigr\}\\
  &+L(t)\Bigl\{\mu(B(x,4r+\rho))+(1+\rho^{-\kappa})d_{\BL^{\kappa}}(\mu,\mu_n)\Bigr\}r\\
  &+M(t)c_1(\varepsilon,\kappa,p_0)\mu(B(x,r))^{1/p_0}d_{\rm{H}}(F,F_n)^{\kappa},
\end{align*}
where $f_r:=1_{B(x,r)}:F\to [0,1]$ and $\overline{P_{\varepsilon}f_r}$ is a McShane extension with exponent $\kappa$, and 
  $\varepsilon\in (0,1],\;\rho\in (0,\infty),\;r\in [d_{\rm{H}}(F,F_n),\infty),$ and $\delta\in (0,r)$.
  \item [$\rm{(ii)}$]Moreover, if the assumption $\rm{(A3)}$ is additionally satisfied, then it holds 
  for $\varepsilon,\rho\in (0,1),\;\xi\ge d_{\rm{H}}(F,F_n),$ and $0<r<\eta<1$ that
\begin{align*}
  \MoveEqLeft\left| \left(\int P_{\varepsilon}f_{r}d\mu\right)p(t,x,y)-\left( \int \overline{P_{\varepsilon}f_r}d\mu_n \right)\overline{p_n}(t,x,y_n) \right|\\
  \lesssim&L(t)\mu\left( B\left(x,\eta\right)\right)\eta+M(t)\exp[-c(\eta-r)^{1+\beta}/\varepsilon^{\beta}]\\
  &+\g{c}{sg2,n}\Bigl\{c_2(\alpha,\varepsilon,\kappa,p)\mu(B(x,r))^{1/p} +c_3(\alpha,\varepsilon,\kappa,q) \mu(B(x,r))^{1/q}\Bigr\}\\
  &\quad\times \Bigl\{ d_{\rm{H}}(F,F_n)+d_{\rm{H}}(F,F_n)^{\kappa}+d(y,y_n)+d_{\BL^{\kappa}}(\mu,\mu_n)\Bigr\}\\
  &+L(t)\Bigl\{\mu(B(x,\eta+\xi+\rho))+(1+\rho^{-\kappa})d_{\BL^{\kappa}}(\mu,\mu_n)\Bigr\}(\eta+\xi)\\
  &+M(t)c_1(\varepsilon,\kappa,p_0)\mu(B(x,r))^{1/p_0}d_{\rm{H}}(F,F_n)^{\kappa}.
\end{align*}
\end{itemize}
\end{lem}

\begin{proof}[Proof of $\rm{(i)}$]
  We first have for fixed $\alpha\in (0,\alpha_0)$ that
  \begin{equation}\label{decomp}
    \begin{aligned}
      \MoveEqLeft\left|\left(\int P_{\varepsilon}f_{r}d\mu\right)p(t,x,y)-\left(\int \overline{P_{\varepsilon}f_{r}}d\mu_n\right)\overline{p_n}(t,x,y_n)\right|\\
  \leq&\left|\left(\int P_{\varepsilon}f_{r}d\mu\right)p(t,x,y)-P_t(G^{\alpha})^{\circ 2}(\alpha-\Delta)^2P_{\varepsilon}f_{r}(y)\right|\\
  +&|P_t(G^{\alpha})^{\circ 2}(\alpha-\Delta)^2P_{\varepsilon}f_{r}(y)-P^n_t\overline{(G^{\alpha})^{\circ 2}(\alpha-\Delta)^2P_{\varepsilon}f_{r}}(y_n)|\\
  &+\left|P^n_t\overline{(G^{\alpha})^{\circ 2}(\alpha-\Delta)^2P_{\varepsilon}f_{r}}(y_n)-\left(\int \overline{P_{\varepsilon}f_{r}}d\mu_n\right)\overline{p_n}(t,x,y_n)\right|\\
  =:&I+II+III.
    \end{aligned}
  \end{equation}
Using $2_{\kappa,p,q}$ and \cref{sg2} with $f=(\alpha-\Delta)^2P_{\varepsilon}f_{r}$, we observe
\begin{equation}\label{II}
  \begin{aligned}
    II
  \lesssim&\g{c}{sg2,n}\|(\alpha-\Delta)^2P_{\varepsilon}f_{r}\|_{\BL^{\kappa}}
  \Bigl\{ d_{\rm{H}}(F,F_n)+d_{\rm{H}}(F,F_n)^{\kappa}+d(y,y_n)+d_{\BL^{\kappa}}(\mu,\mu_n)\Bigr\}\\
  \leq&\g{c}{sg2,n}\Bigl\{ c_2(\alpha,\varepsilon,\kappa,p)\|f_{r}\|_{L^p(F,\mu)}+c_3(\alpha,\varepsilon,\kappa,q) \|f_r\|_{L^q(F,\mu)} \Bigr\}\\
  &\times \Bigl\{ d_{\rm{H}}(F,F_n)+d_{\rm{H}}(F,F_n)^{\kappa}+d(y,y_n)+d_{\BL^{\kappa}}(\mu,\mu_n)\Bigr\}\\
  =&\g{c}{sg2,n}\Bigl\{c_2(\alpha,\varepsilon,\kappa,p)\mu(B(x,r))^{1/p} +c_3(\alpha,\varepsilon,\kappa,q) \mu(B(x,r))^{1/q}\Bigr\}\\
  &\times \Bigl\{ d_{\rm{H}}(F,F_n)+d_{\rm{H}}(F,F_n)^{\kappa}+d(y,y_n)+d_{\BL^{\kappa}}(\mu,\mu_n)\Bigr\}.
  \end{aligned}
\end{equation} 

We next estimate $I$.
For $z\in F$ with $d(x,z)>3r$ and $\delta\in (0,r)\subseteq (0,d(x,z)-2r)$,
it follows from \cref{4.2-a} that
\begin{equation}\label{exitprob}
  \begin{aligned}
  P_z(X_{\varepsilon}\in B(x,r))
  \leq &P_z(\sigma_{\overline{B(x,r)}}\leq\varepsilon)\\
  \leq &4\left[ \frac{\delta}{d(x,z)-2r}+\frac{\varepsilon}{\mu(B(x,\delta))(d(x,z)-2r-\delta)} \right]\\
  \leq &4\left[ \frac{\delta}{r}+\frac{\varepsilon}{m(\delta)(r-\delta)} \right],
  \end{aligned}
\end{equation}
where $m(\delta):=\inf_{w\in F}\mu(B(w,\delta))\in (0,\infty)$. Note that, since $\mu$ is of full support, and $F$ is compact, $m(\delta)$ is positive.
Thus we deduce that
\begin{align*}
    I
    \leq&\int_{\{d(x,z)\leq 3r\}} |p(t,x,y)-p(t,z,y)|P_{z}\left(X_{\varepsilon}\in B\left(x,r\right)\right)\mu(dz) \\
    &+\int_{\{d(x,z)>3r\}} |p(t,x,y)-p(t,z,y)|P_{z}\left(X_{\varepsilon}\in B\left(x,r\right)\right)\mu(dz) \\
    \lesssim&L(t)\mu\left( B\left(x,4r\right)\right)r+M(t)\left\{ \frac{\delta}{r}+\frac{\varepsilon}{m(\delta)(r-\delta)} \right\}.
\end{align*}

Finally we bound $III$. To begin with, we show the following inequality.
\begin{align}\label{barineq}
\sup_{z\in F_n\setminus D(x,4r)}\overline{P_{\varepsilon}f_{r}}(z)
\leq c_1(\varepsilon,\kappa,p_0)\mu(B(x,r))^{1/p_0}d_{\rm{H}}(F,F_n)^{\kappa}+4\left[ \frac{\delta}{r}+\frac{\varepsilon}{m(\delta)(r-\delta)} \right]
\end{align}
Fix $z\in F_n\setminus D(x,4r)$ and take $w\in F$ satisfying $d(z,w)\leq d_{\rm{H}}(F,F_n)\leq r$.
Since it holds that $d(x,w)\geq d(x,z)-d(z,w)>3r$,
it follows from \eqref{exitprob} that
\begin{align*}
  0
  \leq&\overline{P_{\varepsilon}f_{r}}(z)\\
  \leq&|\overline{P_{\varepsilon}f_{r}}(z)-P_{\varepsilon}f_{r}(w)|+P_{\varepsilon}f_{r}(w)\\
  \leq&c_1(\varepsilon,\kappa,p_0)\mu(B(x,r))^{1/p_0}d_{\rm{H}}(F,F_n)^{\kappa}+4\left[ \frac{\delta}{r}+\frac{\varepsilon}{m(\delta)(r-\delta)} \right].
\end{align*}
Thus we obtain \eqref{barineq}.
Therefore, similarly for $I$, we deduce that
\begin{align*}
  III
  \leq&\int_{\{d(x,z)\leq 4r\}}|p_n(t,z,y_n)-\overline{p_n}(t,x,y_n)|\overline{P_{\varepsilon}f_{r}}(z)\mu_n(dz)\\
  &+ \int_{\{d(x,z)>4r\}}|p_n(t,z,y_n)-\overline{p_n}(t,x,y_n)|\overline{P_{\varepsilon}f_{r}}(z)\mu_n(dz)\\
  \lesssim&L(t)\mu_n(B(x,4r))r\\
  &+M(t)\left\{ c_1(\varepsilon,\kappa,p_0)\mu(B(x,r))^{1/p_0}d_{\rm{H}}(F,F_n)^{\kappa}+4\left[ \frac{\delta}{r}+\frac{\varepsilon}{m(\delta)(r-\delta)} \right] \right\}\\
  \lesssim&L(t)\Bigl\{\mu(B(x,4r+\rho))+(1+\rho^{-\kappa})d_{\BL^{\kappa}}(\mu,\mu_n)\Bigr\}r\\
  &+M(t)\left\{ c_1(\varepsilon,\kappa,p_0)\mu(B(x,r))^{1/p_0}d_{\rm{H}}(F,F_n)^{\kappa}+4\left[ \frac{\delta}{r}+\frac{\varepsilon}{m(\delta)(r-\delta)} \right] \right\}.
\end{align*}
Here, we use \cref{ballineq} in the last inequality.
Combining the above estimates, we obtain the desired inequality.
\end{proof}

\begin{proof}[Proof of $\rm{(ii)}$]
  Similarly as in the proof of (i), we have \eqref{decomp} and \eqref{II}.
  We first estimate $I$. Since $P_z(X_{\varepsilon}\in B(x,r))\leq P_z(X_{\varepsilon}\not\in B(z,\eta-r))$ for $z\in F$ with $R(x,z)\ge\eta$,
  we deduce from \cref{summ}(v) that
  \begin{align*}
    I
    \leq&\int_{\{R(x,z)<\eta\}} |p(t,x,y)-p(t,z,y)|P_{z}\left(X_{\varepsilon}\in B\left(x,r\right)\right)\mu(dz)\\
    &+\int_{\{R(x,z)\ge\eta\}} |p(t,x,y)-p(t,z,y)|P_{z}\left(X_{\varepsilon}\in B\left(x,r\right)\right)\mu(dz)\\
    \lesssim&L(t)\mu\left( B\left(x,\eta\right)\right)\eta+M(t)\sup_{z\in B(x,\eta)^c}P_z(X_{\varepsilon}\not\in B(z,\eta-r))\\
    \lesssim&L(t)\mu\left( B\left(x,\eta\right)\right)\eta+M(t)\exp[-c(\eta-r)^{1+\beta}/\varepsilon^{\beta}]. 
  \end{align*}
  It follows from the argument in \eqref{barineq} that
  \begin{align*}
    \sup_{\substack{z\in F_n\\R(x,z)\ge \eta+\xi}}\overline{P_{\varepsilon}f_r}(z)\leq c_1(\varepsilon,\kappa,p_0)\mu(B(x,r))^{1/p_0}d_{\rm{H}}(F,F_n)^{\kappa}+ \exp[-c(\eta-r)^{1+\beta}/\varepsilon^{\beta}].
  \end{align*}
  Thus we obtain
  \begin{align*}
    III
  \leq&\int_{\{R(x,z)< \eta+\xi\}}|p_n(t,z,y_n)-\overline{p_n}(t,x,y_n)|\overline{P_{\varepsilon}f_{r}}(z)\mu_n(dz)\\
  &+ \int_{\{R(x,z)\ge\eta+\xi\}}|p_n(t,z,y_n)-\overline{p_n}(t,x,y_n)|\overline{P_{\varepsilon}f_{r}}(z)\mu_n(dz)\\
  \lesssim&L(t)\mu_n(B(x,\eta+\xi))(\eta+\xi)\\
  &+M(t)\left\{ c_1(\varepsilon,\kappa,p_0)\mu(B(x,r))^{1/p_0}d_{\rm{H}}(F,F_n)^{\kappa}+ \exp[-c(\eta-r)^{1+\beta}/\varepsilon^{\beta}] \right\}\\
  \lesssim&L(t)\Bigl\{\mu(B(x,\eta+\xi+\rho))+(1+\rho^{-\kappa})d_{\BL^{\kappa}}(\mu,\mu_n)\Bigr\}(\eta+\xi)\\
  &+M(t)\left\{ c_1(\varepsilon,\kappa,p_0)\mu(B(x,r))^{1/p_0}d_{\rm{H}}(F,F_n)^{\kappa}+ \exp[-c(\eta-r)^{1+\beta}/\varepsilon^{\beta}] \right\}.
  \end{align*}
  Altogether, we obtain the result.
\end{proof}

As an immediate consequence of the lemma above, 
we obtain the following result, 
which provides a convergence rate estimate for the heat kernel with parameters. 
We denote the right-hand sides of \cref{totyuu}(i) and (ii) 
by $\g{I}{totyuu}$ and $\g{II}{totyuu}$, respectively.

\begin{lem}\label{hkpara}
Fix $x,y\in F,\;y_n\in F_n,\;t>0,\;\kappa\in [1/2,1],\;p_0,p,q\in [1,\infty],$ and $\alpha\in (0,\alpha_0)$.
\begin{itemize}
  \item [$\rm{(i)}$]Under the conditions $1_{\kappa,p_0}$ and $2_{\kappa,p,q}$, we have for 
  $\varepsilon\in (0,1],\;\rho\in (0,\infty),\;r\in [d_{\rm{H}}(F,F_n),\infty),\;\delta\in (0,r)$ that
\begin{align*}
  |p(t,x,y)-p_n(t,x_n,y_n)|
  \lesssim&
  \frac{1}{\mu(B(x,r))}\g{I}{totyuu}\\
  &+M(t)(c_1(\varepsilon,\kappa,p_0)\mu(B(x,r))^{\frac{1}{p_0}-1}+\mu(B(x,r))^{-1})d_{\BL^{\kappa}}(\mu,\mu_n)\\
  &+L(t)d(x,x_n).
\end{align*}
\item[$\rm{(ii)}$]Moreover, if the assumption $\rm{(A3)}$ is additionally satisfied, then it holds 
  for $\varepsilon,\rho\in (0,1),\;\xi\ge d_{\rm{H}}(F,F_n),$ and $0<r<\eta<1$ that
  \begin{align*}
   |p(t,x,y)-p_n(t,x_n,y_n)|
  \lesssim&
  \frac{1}{\mu(B(x,r))}\g{II}{totyuu}\\
  &+M(t)(c_1(\varepsilon,\kappa,p_0)\mu(B(x,r))^{\frac{1}{p_0}-1}+\mu(B(x,r))^{-1})d_{\BL^{\kappa}}(\mu,\mu_n)\\
  &+L(t)d(x,x_n).
  \end{align*}
\end{itemize}
\end{lem}

\begin{proof}
  We observe that
  \begin{align*}
    |p(t,x,y)-p_n(t,x_n,y_n)|
    \leq& |p(t,x,y)-\overline{p_n}(t,x,y_n)|+|\overline{p_n}(t,x,y_n)-p_n(t,x_n,y_n)|\\
    \leq& \frac{1}{\int P_{\varepsilon}f_{r}d\mu}\left|\left(\int P_{\varepsilon}f_{r}d\mu\right)p(t,x,y)-\left(\int \overline{P_{\varepsilon}f_{r}}d\mu_n\right)\overline{p_n}(t,x,y_n)\right|\\
    &+\frac{1}{\int P_{\varepsilon}f_{r}d\mu}\left| \int P_{\varepsilon}f_{r}d\mu-\int \overline{P_{\varepsilon}f_{r}}d\mu_n \right||\overline{p_n}(t,x,y_n)|+L(t)d(x,x_n).
  \end{align*}
  We have estimated the first term in \cref{totyuu}.
  The second term is bounded by
  \begin{align*}
    \left| \int P_{\varepsilon}f_{r}d\mu-\int \overline{P_{\varepsilon}f_{r}}d\mu_n \right||\overline{p_n}(t,x,y_n)|
    \leq& M(t)\|P_{\varepsilon}f\|_{\BL^{\kappa}}d_{\BL^{\kappa}}(\mu,\mu_n)\\
    \leq& M(t)(c_1(\varepsilon,\kappa,p_0)\mu(B(x,r))^{1/p_0}+1)d_{\BL^{\kappa}}(\mu,\mu_n).
  \end{align*}
  Combining these estimates and $\int P_{\varepsilon}f_rd\mu=\mu(B(x,r))$ using invariance of $\mu$, we obtain the result.
\end{proof}

Now, we can complete the proof of \cref{hkest} by optimizing parameters in \cref{hkpara}.
In the following proofs, we write $A\lesssim_t B$ if there exists a time-dependent constant (not necessarily universal) 
$C=C(t)\in (0,\infty)$ such that $A\leq CB$.

\begin{proof}[Proof of \cref{hkest}$\mathrm{(i)}$]
  Using our assumptions, \cref{suff}(ii), and 
  \cref{hkpara} with $\kappa\in [1/2,2/3),\;p=q=1,\;p_0=2,\;\varepsilon,\rho\in (0,1],\;r\in [h,1],$
  and $\delta\in (0,r)$, we obtain
  \begin{align*}
    |p(t,x,y)-p_n(t,x_n,y_n)|
    \lesssim_t&\Bigl\{ r^{s_1}+\rho^{s_1}+\rho^{-\kappa}h \Bigr\}r^{1-s_0}\\
    &+\left\{ \frac{\delta}{r}+\frac{\varepsilon}{\delta^{s_0}(r-\delta)} \right\}r^{-s_0}
    +\varepsilon^{-\vartheta_1}h+\varepsilon^{-\left(2+\frac{s_0}{1+s_0}\right)}h\\
    &+\varepsilon^{-(3\kappa-1)}r^{-\frac{s_0}{2}}h
    +r^{-s_0}h,
  \end{align*}
  where $\vartheta_1:=\kappa+2+\frac{s_0}{2(1+s_0)}=\frac{(2\kappa+5)s_0+2\kappa+4}{2(1+s_0)}$, and 
  the time-dependent constant in the inequality is $t^{-2},\;t\in (0,1]$.
  Note that, $t^{-2}$ is from $L(t)$.
  Substituting $\varepsilon=h^a,\;r=h^b,\;\delta=h^{b+c},$ and $\rho=h^d,$ where $a,\;c,\;d\in (0,\infty)$ and $b\in (0,1)$, 
  into the above inequality,
  we deduce that 
  \begin{align*}
    |p(t,x,y)-p_n(t,x_n,y_n)|\lesssim_t\frac{1}{1-h^c}\sum_{k=1}^{9}h^{E_k},
  \end{align*}
  where $E_k=E_k(a,b,c,d)$ are defined by the following.
  \begin{equation*}
    E_1=(s_1-s_0+1)b,
    \qquad E_2=(1-s_0)b+s_1d,
    \qquad E_3=(1-s_0)b-\kappa d+1,
    \qquad E_4=-s_0b+c,
  \end{equation*}
  \begin{equation*}
    E_5=a-(2s_0+1)b-s_0c,
    \qquad E_6=-\vartheta_1 a+1,
    \qquad E_7=-(3\kappa-1)a-\frac{s_0}{2}b+1,
    \qquad E_8=-s_0b+1,
  \end{equation*}
  and $E_9=-(2+\frac{s_0}{1+s_0})a+1$.
  Now, it suffices to compute 
  \begin{align*}
   E:= \sup_{\substack{a,c,d\in (0,\infty)\\b\in (0,1)}}\min_{1\leq k\leq 9}E_k.
  \end{align*}
  Define $D,\;E_0,\;a_0,\;b_0,\;c_0,$ and $d_0$ by 
  \begin{align*}
    D=\Bigl\{ (2\kappa+5)s_0+2\kappa+4 \Bigr\}(s_1+2)+2(s_1-s_0+1)=2\vartheta_1 (1+s_0)(s_1+2)+2(s_1-s_0+1),
  \end{align*} 
  \begin{align*}
    E_0=\frac{2(s_1-s_0+1)}{D},
  \quad  a_0=\frac{2(1+s_0)(s_1+2)}{D},\quad b_0=\frac{2}{D},\quad c_0=\frac{2(s_1+1)}{D},\;\;\mathrm{and}\;\;d_0=\frac{1}{s_1+\kappa}.
  \end{align*}
  Note that, since $D\geq 8$, $b_0$ lies in $(0,1)$.
  We will show that $E=E_0$.
  First, we observe that substituting $(a,b,c,d)=(a_0,b_0,c_0,d_0)$ into $E_k$ 
  yields that
  \begin{align*}
    E_1=E_4=E_5=E_6=E_0,\;\;\mathrm{and}\;\; E_2,E_3,E_7,E_8,E_9\geq E_0.
  \end{align*}
  Thus it follows that $E_0\leq E$.
  We next show the converse.
  Set $p_1=\frac{\vartheta_1 (s_0+1)^2}{s_1-s_0+1},\;p_4=s_0\vartheta_1,\;p_5=\vartheta_1,\;p_6=1,$ and $p_k=0,\;k=2,3,7,8,9$.
  Then, by direct computation, we have  $\sum_k p_k=1/E_0,$ and $\sum p_kE_k=1$ for an arbitrary $(a,b,c,d)$.
  Therefore it holds for any $(a,b,c,d)$ that
  \begin{align*}
    \min_{1\leq k\leq 9}E_k\leq \sum_{k=1}^{9}\frac{p_k}{\sum_{j=1}^{9}p_j}E_k=E_0.
  \end{align*}
  This implies $E\leq E_0$.
\end{proof}
\begin{proof}[Proof of \cref{hkest}$\mathrm{(ii)}$]
  Similarly to above, 
  using \cref{suff}(iii) with $p_0=p=q=2$, we deduce that
  \begin{align*}
    |p(t,x,y)-p_n(t,x_n,y_n)|\lesssim_t\frac{1}{1-h^c}\sum_{k=1}^{9}h^{E_k},
  \end{align*}
  where $E_k=E_k(a,b,c,d)$ are defined by the following.
  \begin{equation*}
    E_1=(s_1-s_0+1)b,
    \qquad E_2=(1-s_0)b+s_1d,
    \qquad E_3=(1-s_0)b-\kappa d+1,
    \qquad E_4=-s_0b+c,
  \end{equation*}
  \begin{equation*}
    E_5=a-(2s_0+1)b-s_0c,
    \qquad E_6=-(\kappa+2)a-\frac{s_0}{2}b+1,
    \qquad E_7=-\left[\kappa-\frac{s_1}{2(1+s_0)\theta}(2\kappa-1)\right] a-\frac{s_0}{2}b+1,
  \end{equation*}
  \begin{align*}
    E_8=-s_0b+1,\qquad \mathrm{and}\qquad E_9=-\left[2+\frac{s_0}{2(1+s_0)}\right]a-\frac{s_0}{2}b+1.
  \end{align*}
  Thus it is enough to compute
  \begin{align*}
    E:= \sup_{\substack{a,c,d\in (0,\infty)\\b\in (0,1)}}\min_{1\leq k\leq 9}E_k.
  \end{align*}
  Define $D,E_0,a_0,b_0,c_0,$ and $d_0$ by 
  \begin{align*}
    D=[(\kappa+2)(1+s_0)(2+s_1)+s_0/2]+(s_1-s_0+1),
  \end{align*} 
  \begin{align*}
    E_0=\frac{s_1-s_0+1}{D},
  \quad  a_0=\frac{(1+s_0)(s_1+2)}{D},\quad b_0=\frac{1}{D},\quad c_0=\frac{s_1+1}{D},\;\;\mathrm{and}\;\;d_0=\frac{1}{s_1+\kappa}.
  \end{align*}
  Note that, since $D\geq 8$, $b_0$ lies in $(0,1)$.
  We will show that $E=E_0$.
  First, we observe that substituting $(a,b,c,d)=(a_0,b_0,c_0,d_0)$ into $E_k$ 
  yields that
  \begin{align*}
    E_1=E_4=E_5=E_6=E_0,\;\;\mathrm{and}\;\; E_2,E_3,E_7,E_8,E_9\geq E_0.
  \end{align*}
  Thus it follows that $E_0\leq E$.
  We next show the converse.
  Set $p_1=\frac{(\kappa+2)(s_0+1)^2+s_0/2}{s_1-s_0+1},\;p_4=s_0(\kappa+2),\;p_5=\kappa+2,\;p_6=1$ and $p_k=0,\;k=2,3,7,8,9$.
  Then, by direct computation, we have  $\sum_k p_k=1/E_0,$ and $\sum p_kE_k=1$ for an arbitrary $(a,b,c,d)$.
  Therefore it holds for any $(a,b,c,d)$ that
  \begin{align*}
    \min_{1\leq k\leq 9}E_k\leq \sum_{k=1}^{9}\frac{p_k}{\sum_{j=1}^{9}p_j}E_k=E_0.
  \end{align*}
  This implies $E\leq E_0$.
\end{proof}
\begin{proof}[Proof of $\rm{(iii)}$]
  Similarly as in the proof of (i), by \cref{hkpara} and \cref{summ}(iv) with $p_0=p=q=2$, we first have that
  \begin{align*}
    \MoveEqLeft |p(t,x,y)-p_n(t,x_n,y_n)|\\
    \lesssim_t 
    &r^{-s_0}\eta^{1+s_1}+r^{-s_0}\eta\xi^{s_1}+r^{-s_0}\eta\rho^{s_1}
    +r^{-s_0}\eta^{s_1}\xi+r^{-s_0}\xi^{1+s_1}+r^{-s_0}\xi\rho^{s_1}\\
    &+r^{-s_0}\eta\rho^{-\kappa}h+r^{-s_0}\xi\rho^{-\kappa}h
    +r^{-s_0}\exp\left[-c\frac{(\eta-r)^{1+\beta}}{\varepsilon^{\beta}}\right]
    +\varepsilon^{-(2+\kappa)}r^{-s_0/2}h
    +r^{-s_0}h
  \end{align*}
  for $\varepsilon,\rho\in (0,1),\;\xi\in [h,1),$ and $0<r<\eta<1$.
  Here,
  the time-dependent constant in the inequality is $t^{-2}+t$.
  Substituting
  \begin{equation*}
    \varepsilon=h^{a(1+\beta^{-1})}(M\log(1/h))^{-1/\beta},\quad M=\frac{2^{1+\beta}}{c}a(1+s_1),\qquad a=\frac{1}{1+s_1-\frac{s_0}{2}+(2+\kappa)(1+\beta^{-1})},
  \end{equation*}
  \begin{equation*}
    \eta=2r=\xi=\rho=r^a,
  \end{equation*}
  and using the formula $1+\beta^{-1}=(1+s_0)\Theta,\;\Theta=\frac{1+s_0-s_1}{1-(s_0-s_1)(2+s_0)}$,
  we obtain the result.
\end{proof}

\begin{proof}[Proof of \cref{sgrw}]
  We only provide the proof under the assumption $\rm{(A1)}$ or $\rm{(A2)}$. 
  The proof under the assumption $\rm{(A3)}$ is entirely analogous.
  Recall the definition of $J_n^r$ from \eqref{defjn}. 
  Then we have
  \begin{align*}
    |P_tf(x)-P^n_tf(x_n)|
    \leq&\left| \int p(t,x,y)f(y)\mu(dy)-\int \overline{p}(t,x,y)f(y)\mu_n(dy)\right|\\
    &+\left| \int (\overline{p}(t,x,y)-J_n^rp^{t,x}(y))f(y)\mu_n(dy) \right|\\
    &+\left| \int (J_n^rp^{t,x}(y)-p_n(t,x_n,y))f(y) \mu_n(dy)\right|\\
    =&:I+II+III.
  \end{align*}
  Here, $\overline{p}$ is the function on $M$ defined at the beginning of this section. 

  By the submultiplicativity of the $\BL^{\kappa}$ norm (see \cref{submult}), we have that
\begin{align*}
  I\leq \|p^{t,x}f\|_{\BL^{\kappa}}d_{\BL^{\kappa}}(\mu,\mu_n)\lesssim_t \|f\|_{\BL^{\kappa}}d_{\BL^{\kappa}}(\mu,\mu_n).
\end{align*}
Since $\overline{p}(t,\cdot,\cdot)$ is Lipschitz continuous,
we deduce that
\begin{align*}
  |\overline{p}(t,x,y)-J_n^rp^{t,x}(y)|
  \leq \frac{1}{\mu(B(y,r))}\int_{B(y,r)}|\overline{p}(t,x,y)-p(t,x,z)|\mu(dz)\lesssim_t r.
\end{align*}
Thus it holds that
\begin{align*}
  II\lesssim \|f\|_{\BL^{\kappa}}r.
\end{align*}
Finally, by \cref{hkest}, we observe that
\begin{align*}
  |J_n^rp^{t,x}(y)-p_n(t,x_n,y)|
  \leq& 
  \frac{1}{\mu(B(y,r))}\int_{B(y,r)}|p(t,x,z)-p_n(t,x_n,y)|\mu(dy)\\
  \lesssim&(d_{\rm{H}}(F,F_n)^{\kappa}+d_{\BL^{\kappa}}(\mu,\mu_n)+d(x,x_n))^E
\end{align*}
and, so we may conclude that
\begin{align*}
  III\lesssim \|f\|_{\BL^{\kappa}}(d_{\rm{H}}(F,F_n)^{\kappa}+d_{\BL^{\kappa}}(\mu,\mu_n)+d(x,x_n))^E.
\end{align*}
Altogether, we obtain the result from $E<1$.
\end{proof}

\section{Non-compact case}\label{ncpt}
In this section, we consider convergence rate estimates for
 non-compact cases as an application of the compact cases. 
 The main result in this section is the following.

\begin{thm}\label{noncpt}
  Suppose that \cref{ncptas} is satisfied. 
  We regard $F$ and $F_n$ as subsets of a common $\bcm$ space $(M,d)$.
  Let $x\in F$ and $x_n\in F_n$. Take $r_{\infty}\ge r_0$ satisfying 
  $x\in B(\rho,r_{\infty})$ and $x_n\in B(\rho_n,r_{\infty})$.
  Then, for any $f\in\BL^{\kappa}(M),\kappa\in [1/2,1]$ and $t>0$
  \begin{align*}
    |P_{t}f(x)-P^n_tf(x_n)|
    \lesssim\|f\|_{\BL^{\kappa}}
    \inf\Bigl\{  \Psi_1(n,r)+C_n(r)\Psi_2(n,r)^E : r\in [r_{\infty},\infty), \Psi_2(n,r)\le \frac{1}{2}\Bigr\},
  \end{align*}
  where
  \begin{equation*}
    \Psi_1(n,r)=P_{x}^n(\sigma_{B(\rho,r)^c}\le t)\vee P_{x_n}^n(\sigma^n_{B(\rho_n,r)^c}\le t),
  \end{equation*}
  \begin{equation*}
    \Psi_2(n,r)=d_{\rm{H}}(F^{(r)},F_n^{(r)})^{\kappa}+d_{\BL^{\kappa}}(\mu^{(r)},\mu^{(r)}_n)+d(x,x_n),
  \end{equation*}
  \begin{equation*}
    \begin{aligned}
      C_n(r)=&(D_n^{(r)})^{4+2(1+\frac{s}{2(1+s)\theta})(2\kappa-1)}\\
    &\quad\times \left[(\diam K_n^{(r)}+\mu_n(F_n^{(r)}))\diam K_n^{(r)}\mu_n(F_n^{(r)})\mu(F^{(r)})^{3/2}\diam F^{(r)}(\mu(F^{(r)})^{1/2}+\mu_n(F_n^{(r)}))\right]\\
    \lesssim& r^{6+\frac{3}{2}s+2(1+\frac{s}{2(1+s)\theta})(2\kappa-1)}(r+\mu_n(F^{(r)}))(r^{s/2}+\mu_n(F^{(r)}))(r^s+\mu_n(F^{(r)}))^{4+2(1+\frac{s}{2(1+s)\theta})(2\kappa-1)},
    \end{aligned}
  \end{equation*}
  \begin{equation*}
    K_n^{(r)}=F^{(r)}\cup F_n^{(r)},\qquad D_n^{(r)}=\mu(F^{(r)})\diam F^{(r)}\vee \mu_n(F^{(r)}_n)\diam F^{(r)}_n,
  \end{equation*}
  and
  \begin{equation*}
    E=\frac{\kappa}{\kappa+(1+s)(2+\kappa)(\kappa+\theta)}.
  \end{equation*}
  Moreover, the time-dependence of the constant in $\lesssim$ is $t^{[\kappa-\frac{s}{2(1+s)\theta}(2\kappa-1)]},\;t\in (0,1]$.
\end{thm}

Recall the definition of $\bbF_r$ from \eqref{bbfr}.
 Let $(F,R,\mu,\rho),\;(F_n,R_n,\mu_n,\rho_n)\in \bbF_r$.
 We assume that $F$ is unbounded, and each $F_n$ has at least two distinct points.
 As in the previous sections, we define $F^{(r)}:=\overline{B(\rho,r)},\;R^{(r)}:=R|_{F^{(r)}\times F^{(r)}},$ and $\mu^{(r)}:=\mu|_{F^{(r)}\times F^{(r)}}$ for $r>0$.
 We also define $F^{(r)}_n,\;R^{(r)}_n,$ and $\mu_n^{(r)}$ similarly.

 We consider 
 the following assumptions:
 \begin{assume}\label{ncptas}\quad\par
 \begin{itemize}
  \item[1.] The sequence $(F_n,R_n,\mu_n,\rho_n)$ converges to $(F,R,\mu,\rho)$ in the Gromov-Hausdorff-vague topology.
  \item[2.]There exists an $r_0> 0$ such that the following hold.
 
\begin{enumerate}[align=left, labelwidth=\widthof{2-a.}, 
    labelsep=0.5em, 
    leftmargin=!]
  \item [2-a.] There exists an $s>0$ such that $\mu^{(r)}$ is uniformly $s$-Ahlfors regular for $r>r_0$. That is,
   there exist $c_u,c_l,s>0$ such that for any $r\in (r_0,\infty),a\in [0,\diam F^{(r)})$ and $x\in F^{(r)}$ 
   it holds that $c_l a^s\leq \mu^{(r)}(B(x,a))\leq c_u a^s$.
  \item [2-b.]The sequence $(F,R)$ satisfies the lower resistance estimate with an exponent $\theta>\frac{1\vee s}{1+s}$.
   That is, there exists a constant $c_{\rm{LR}}>0$ such that 
   \begin{equation}
    c_{\rm{LR}}r^{\theta}\leq R(x,B(x,r)^c),\qquad \forall x\in F,\;r>0.
   \end{equation}
  \item [2-c.]There exists a constant $c_{\rm{LHK}}>0$ such that 
  \begin{align*}
    c_{\rm{LHK}}t^{-\frac{s}{(1+s)\theta}}\leq p^{(r)}(t,x,x),\quad \forall r>r_0,\;x\in F^{(r)},\;t>0,
  \end{align*}
  where $p^{(r)}$ is the heat kernel associated with $(F^{(r)},R^{(r)},\mu^{(r)})$.
  \end{enumerate}
\end{itemize}
 \end{assume}

 \begin{rmk}\label{cork}
  Suppose that $\mu$ is $s$-Ahlfors regular and let $r>0$.
  Then one sufficient condition that makes $\mu^{(r)}$ $s$-Ahlfors regular is the following:
  \begin{itemize}
    \item [$\rm{(C)}$]There
exists an $\varepsilon>0$ such that for all $x\in F^{(r)}$ and  $a\in (0,\diam F^{(r)}]$,
the set  $B(x,a)\cap B(\rho,r)$ contains a ball with radius $\varepsilon a$.
  \end{itemize}
  This requirement $\rm{(C)}$ is known as the corkscrew condition.
  We note that it is known that every $A$-uniform domain (see \cite[Definition 2.3]{Mur24} for the definition) 
  satisfies the corkscrew condition with $\varepsilon=1/3A$ (see \cite[Lemma 3.4]{Mur24}).
  Therefore, one sufficient condition that the Assumption $\rm{2}$-$\rm{a}$ is satisfied
  is that there exists a constant $A$ such that $B(\rho,r)$ is an $A$-uniform domain for $r\ge r_0$.
 \end{rmk}

\begin{lem}\label{7.3}\quad\par 
\begin{itemize}
  \item [$\rm{(i)}$]If $(F,R)$ is uniformly perfect, 
 then Assumption $\rm{2}$-$\rm{a}$ implies Assumption $\rm{2}$-$\rm{b}$ with $\theta=1$.
  \item [$\rm{(ii)}$]Under Assumption $\rm{2}$-$\rm{b}$, for any $r>0$, $F^{(r)}$ satisfies $\rm{LRES}(\theta)$ with the same constant $c_{\rm{LR}}$.
  \item [$\rm{(iii)}$]Suppose that there exists a constant $C\in (0,\infty)$ such that $Cr\leq \diam F^{(r)}$ for $r>r_0$.
  If Assumptions $\rm{2}$-$\rm{a}$ and $\rm{2}$-$\rm{b}$ are satisfied with $\theta=1$, then Assumption $\rm{2}$-$\rm{c}$ also holds.
\end{itemize}
\end{lem}
\begin{proof}[Proof of $\rm{(i)}$]
  This is an immediate consequence of \cref{summ}(i).
\end{proof}
\begin{proof}[Proof of $\rm{(ii)}$]
  Let $(\calE,\calF)$ and $(\calE^{(r)},\calF^{(r)})$ be the resistance forms corresponding to $(F,R),(F^{(r)},R^{(r)})$.
  Note that $\calF^{(r)}$ is given by $\calF^{(r)}:=\{u|_{F^{(r)}}:u\in\calF\}$ (see \cite[Theorem 8.4]{Kig12}).
  We write $h^{(r)}:\calF^{(r)}\to \calF$ for the harmonic extension map i.e., 
  $h^{(r)}u$ is the unique element in $\calF$ satisfying $(h^{(r)}u)|_{F^{(r)}}=u$ and 
  \begin{align*}
    \calE(h^{(r)}u,h^{(r)}u)=\inf\{\calE(v,v):v\in \calF,v|_{F^{(r)}}=u\},
  \end{align*}
  for any $u\in\calF^{(r)}$ (see \cite[Lemma 8.2 and Definition 8.3]{Kig12}). Then it holds that $\calE^{(r)}(u,u)=\calE(h^{(r)}u,h^{(r)}u)$.
  Therefore, we deduce for $x\in F^{(r)}$ and $B(x,a)\neq F^{(r)}$ that
  \begin{align*}
    R^{(r)}(x,F^{(r)}\setminus B(x,a))^{-1}
   =&\inf\{\calE^{(r)}(u,u):u\in\calF^{(r)},u(x)=1,u|_{F^{(r)}\setminus B(x,a)}=0\}\\
   =&\inf\{\calE(h^{(r)}u,h^{(r)}u):u\in\calF^{(r)},u(x)=1,u|_{F^{(r)}\setminus B(x,a)}=0\}\\
   =&\inf\{\calE(v,v):v\in\calF,v(x)=1,v|_{F^{(r)}\setminus B(x,a)}=0\}\\
   \leq&\inf\{\calE(v,v):v\in\calF,v(x)=1,v|_{F\setminus B(x,a)}=0\}\\
   =&R(x,F\setminus B(x,a))^{-1}\\
   \leq&\frac{1}{c_{\rm{LR}}a^{\theta}},
  \end{align*}
  which completes the proof.
\end{proof}
\begin{proof}[Proof of $\rm{(iii)}$]
  By our assumptions and \cref{summ}(ii), there exist constants $c_0,c_1\in (0,\infty)$ such that
  \begin{align*}
    c_0 t^{-\frac{s}{1+s}}\leq p^{(r)}(t,x,x),\qquad r>r_0,\;x\in F^{(r)},\;t<t_0^{(r)},
  \end{align*}
  where $t_0^{(r)}:=c_1r^{1+s}$.
  Let $\lambda_i^{(r)}$ and $\varphi^{(r)}_i,i\geq 0$ be the eigenvalues and eigenfunctions of
  $-\Delta^{(r)}$ satisfying $-\Delta^{(r)}\varphi^{(r)}_i=\lambda_i^{(r)}\varphi^{(r)}_i$ and $0\leq \lambda^{(r)}_i\leq \lambda^{(r)}_{i+1}$.
  Here, $\Delta^{(r)}$ is the generator corresponding to $(F^{(r)},R^{(r)},\mu^{(r)})$. 
  Note that the existence of such eigenvalues and eigenfunctions is guaranteed by 
  the fact that $-\Delta^{(r)}$ has a compact resolvent (see \cite[Lemma 9.7]{Kig12}).
  Then we have $\lambda^{(r)}_0=0$ and $\varphi_0^{(r)}=\frac{1}{\sqrt{\mu(F^{(r)})}}$.
  Therefore, it holds that, for $t\geq t_0^{(r)}$ and $r>r_0$,
  \begin{align*}
    p^{(r)}(t,x,x)
    =\sum_{i\geq 0}e^{-\lambda_i^{(r)}t}(\varphi_i^{(r)}(x))^2
    \geq\frac{1}{\mu(F^{(r)})}
    \geq c_l^{-1}r^{-s}
    =\frac{c_1^{\frac{s}{1+s}}}{c_l}t_0^{-\frac{s}{1+s}}
    \geq \frac{c_1^{\frac{s}{1+s}}}{c_l}t^{-\frac{s}{1+s}}.
  \end{align*}
  See \cite[Proof of Lemma 10.7]{Kig12} for the first equality.
  This implies the result.
\end{proof}

To complete the estimate for the non-compact case, we prepare an elementary lemma. 
\begin{lem}\label{bdd}\quad\par
\begin{itemize}
  \item [$\rm{(i)}$]Let $A,\;B$ be closed subsets of a metric space $(M,d)$.
Suppose that $A$ is unbounded and $B$ has at least two distinct elements.
Taking $a\in A,\;b\in B$, set $A^{(r)}=A\cap \overline{B(a,r)},\;B^{(r)}=B\cap \overline{B(b,r)}$ for $r>0$.
Then there exists an $\tilde{r}_1>0$ such that there exist
$a_0^{r},a_1^{r}\in A^{(r)},\;b_0^{r},b_1^{r}\in B^{(r)}$ for any $r>\tilde{r}_1$ satisfying
the following:
\begin{itemize}
  \item $d(a_0^{r},b_0^{r}),\;d(a_1^{r},b_1^{r})\leq d_{\rm{H}}(A^{(r)},B^{(r)})$;
  \item $\displaystyle  \inf_{r>\tilde{r}_1}d(a_0^{r},a_1^{r})\wedge d(b_0^{r},b_1^{r})>0$;
  \item the set $\{a_0^r,a_1^r,b_0^r,b_1^r:r>\tilde{r}_1 \}$ is bounded.
\end{itemize}
\item[$\rm{(ii)}$]Suppose Assumption 1, and we regard $F,\;F_n,\;n\in\N$ as subsets of a common $\bcm$ space $(M,d)$.
Then there exists an $r_1>0$ such that, for any $r>r_1$ and $n\in\N$ there exist
$a_0^{n,r},a_1^{n,r}\in F^{(r)},\;b_0^{n,r},b_1^{n,r}\in F_n^{(r)}$ satisfying
the following:
\begin{itemize}
  \item $d(a_0^{n,r},b_0^{n,r}),\;d(a_1^{n,r},b_1^{n,r})\leq d_{\rm{H}}(F^{(r)},F_n^{(r)})$;
  \item $\displaystyle \delta:= \inf_{\substack{n\in\N\\r>r_1}}d(a_0^{n,r},a_1^{n,r})\wedge d(b_0^{n,r},b_1^{n,r})>0$;
  \item there exists an $r_{\ast}$ such that the set $\{a_0^{n,r},a_1^{n,r},b_0^{n,r},b_1^{n,r}:r>r_1,n\in\N \}$ is a subset of $B(\rho,r_{\ast})$ and $B(\rho_n,r_{\ast}),\;n\in\N$.
\end{itemize}
\end{itemize}
\end{lem}

\begin{proof}[Proof of $\rm{(i)}$]
  Take distinct points $a_0,a_1\in A,\;b_0,b_1\in B$ and set $\delta_0=d(a_0,a_1)\wedge d(b_0,b_1)>0,\delta_1=d(a_0,b_0)\vee d(a_1,b_1)\in [0,\infty)$.
  Since $A$ is unbounded, we may find $a_2\in A$ such that $d(a,a_2)>2\delta_1+\delta_0$.
  Define $\tilde{r}_1=2d(a,a_2)$ and take arbitrary $r>\tilde{r}_1$.
  If $d_{\rm{H}}(A^{(r)},B^{(r)})\geq \delta_1$, it is enough to set $a_i^{r}=a_i,b_i^{r}=b_i$.
  Suppose that $d_{\rm{H}}(A^{(r)},B^{(r)})<\delta_1$. 
  Then there exist $b_0^r,b_1^r\in B^{(r)}$ such that $d(a,b^{r}_0),d(a_2,b_1^{r})\leq d_{\rm{H}}(A^{(r)},B^{(r)})<\delta_1$,
  and it holds that
\begin{align*}
  d(a,b^{r}_0)\geq d(a,a_2)-d(a,b^{r}_0)-d(a_2,b^r_1)\geq (2\delta_1+\delta_0)-2d_{\rm{H}}(A^{(r)},B^{(r)})\geq \delta_0.
\end{align*}
Therefore we may take $a_0^r=a,\;a_1^r=a_2$.
\end{proof}

\begin{proof}[Proof of $\rm{(ii)}$]
  Take $r'>0$ such that $F^{(r')}$ contains at least two distinct elements $a_0$ and $a_1$. 
  Set $r'_1=2(r'+\sup_{n\in\N}d(\rho,\rho_n))$.
  Then there exists $N\in\N$ such that $\frac{1}{2}d(a_0,a_1)>2d_{\rm{H}}(F^{(r'_1)},F^{(r'_1)}_n)$ for any $n>N$. 
  First we deal with $n>N$ and fix such an $n\in\N$. 
  Let $b_0^n,b_1^n\in F^{(r'_1)}$ be points satisfying
  $d(b_i^n,a_i)=\inf_{x\in F^{(r'_1)}_n}d(x,a_i)\leq d_{\rm{H}}(F^{(r'_1)},F^{(r'_1)}_n)$ for $i=0,1$.
   Then we have
   \begin{align*}
    d(b_0^n,b_1^n)\geq d(a_0,a_1)-d(a_0,b_0^r)-d(a_1,b_1^r)\geq d(a_0,a_1)-2d_{\rm{H}}(F^{(r'_1)},F^{(r'_1)}_n)>\frac{1}{2}d(a_0,a_1).
   \end{align*}
   Also, it holds that 
   \begin{align}\label{kg}
   d(a_i,b^n_i)\leq d_{\rm{H}}(F^{(r)},F^{(r)}_n),\qquad r>r'_1. 
   \end{align}
   We prove this.
   To begin with, we observe that
   \begin{equation}\label{kg1}
   \begin{aligned}
    d_{\rm{H}}(F^{(r)},F^{(r)}_n)
    \geq&\inf_{x\in F_n^{(r)}}d(x,a_i)\\
    =&\inf_{x\in F_n^{(r'_1)}}d(x,a_i)\wedge \inf_{x\in F_n^{(r)}\setminus  F_n^{(r'_1)}}d(x,a_i)\\
    =&d(b_i^n,a_i)\wedge \inf_{x\in F_n^{(r)}\setminus F_n^{(r'_1)}}d(x,a_i).
   \end{aligned}
   \end{equation}
   On the other hand, we have 
   \begin{align}\label{kg2}
    d(x,a_i)
    \geq d(\rho_n,x)-d(\rho_n,\rho)-d(\rho,a_i)
    \geq r'_1-d(\rho,\rho_n)-r'\geq r'+d(\rho,\rho_n)
   \end{align}
   for $x\in F_n^{(r)}\setminus x\in F_n^{(r'_1)}$.
   Also, by our choice of $b_i^n$, it follows that 
   \begin{align}\label{kg3}
    d(b_i^n,a_i)\leq d(\rho_n,a_i)\leq d(\rho,\rho_n)+d(\rho,a_i)\le d(\rho,\rho_n)+r'.
   \end{align}
   Combining \eqref{kg1},\eqref{kg2}, and \eqref{kg3}, we deduce \eqref{kg}.
   Thus, we may take $a_i^{n,r}=a_i,\;b_i^{n,r}=b_i^n$ for $r>r'_1$ and $n>N$.

   We next consider $n=1,\dots,N$. For such $n$, (i) implies the desired points.
\end{proof}

\begin{proof}
  We denote the processes corresponding to 
  $(F^{(r)},R^{(r)},\mu^{(r)},\rho)$ and $(F^{(r)}_n,R^{(r)}_n,\mu^{(r)}_n,\rho_n)$
  by $X^{(r)}$ and $X^{n,(r)}$, respectively. 
  Since $(F,R)$ is recurrent, $X^{(r)}$ is the time change of the original process $X$
  with respect to $\mu^{(r)}$. In particular, the two processes can be constructed on same probability space.
  The same thing holds for $X^{n,(r)}$.
  Thus it holds
  \begin{align*}
    |P_{t}f(x)-P^{(r)}_tf(x)|
    =\Big|E_x[f(X_t)-f(X^{(r)}_t):\sigma_{B(\rho,r)^c}\le t]\Big|
    \lesssim \|f\|_M P_{x}( \sigma_{B(\rho,r)^c}\le t ),
  \end{align*}
  for $r\in [r_{\infty},\infty)$ satisfying $\Psi_2(n,r)\le 1/2$.
  Hereafter we fix this $r$.
  Then we have
 \begin{align*}
  |P_{t}f(x)-P^n_tf(x_n)|
  =&|P_{t}f(x)-P^{(r)}_tf(x)|
  +|P_{t}^{(r)}f(x)-P^{n,(r)}_tf(x_n)|
  +|P_{t}^{n,(r)}f(x_n)-P^n_tf(x_n)|\\
  \lesssim&\|f\|_{M}P_{x}(\sigma_{B(\rho,r)^c}\le t)
  +|P_{t}^{(r)}f(x)-P^{n,(r)}_tf(x_n)|
  +\|f\|_{M}P_{x_n}^n(\sigma^n_{B(\rho_n,r)^c}\le t).
 \end{align*}
 By \cref{gen}(i), \cref{7.3}(i) and our assumptions, it holds that
 \begin{align*}
  |P_{t}^{(r)}f(x)-P^{n,(r)}_tf(x_n)|\lesssim \g{c}{gen,n}\|f\|_{\BL^{\kappa}}\Psi_2(n,r)^{E},
 \end{align*}
 where $E=\frac{\kappa}{\kappa+(1+s)(2+\kappa)(\kappa+\theta)}$.
 Combining the arguments in the previous sections (in particular, see \cref{estresrmk}),
  we may check the following
 (recall the definition of $\delta$ and $r_{\ast}$ from \cref{bdd}).
 \begin{align*}
  \g{c}{gen,n}\lesssim &
   \g{c}{12q2}
    +\g{c}{sg2,n}\mu(F^{(r)})^{1/2}\Bigl\{ c_{\rm{A1}} +\g{c}{suffcond1,3}(D_n^{(r)})^{2(1+\frac{s}{2(1+s)\theta})(2\kappa-1)}\Bigr\}
    +\g{c}{suffcond1,3}\mu(F^{(r)})^{1/2},\\
      \g{c}{12q2}=&1+\frac{1}{c_l \left[ \frac{c_{\rm{LR}}(1\wedge \frac{1}{2}\diam F)^{\theta}}{1+(1\vee \frac{1}{\diam F})c_{\rm{LR}}(1\wedge \frac{1}{2}\diam F)^{\theta}} \right]^{1+s_0}}\lesssim 1,\\
      \g{c}{sg2,n}=&\g{c}{sg1,n}+(D_n^{(r)})^{2}\g{c}{estres,n}(1+\mu_n(F_n^{(r)})\g{c}{continuity,1.n})+(D_n^{(r)})^{4}\Big[\mu(F^{(r)})\g{c}{continuity,1}+\mu_n(F_n^{(r)})\g{c}{continuity,1.n}\Big],\\
      \g{c}{sg1,n}=&\g{c}{estres,n}(t+(D_n^{(r)})^{2}+(D_n^{(r)})^{2}c_1(t,\kappa,2)\mu(F^{(r)})^{1/2})+(D_n^{(r)})^{2}\mu(F^{(r)})^{1/2}( c_1(t,\kappa,2)+c_1([0,t],\kappa,2))\\
      \lesssim &\g{c}{estres,n}(1+(D_n^{(r)})^{2}+(D_n^{(r)})^{2}\g{c}{suffcond1,3}\mu(F^{(r)})^{1/2})+(D_n^{(r)})^{2}\mu(F^{(r)})^{1/2}\g{c}{suffcond1,3},\\
      \g{c}{estres,n}=&\g{c}{itkill}\left( 1+\frac{\mu(F^{(r)})\diam F^{(r)}}{\g{c}{sle,n}(\delta)} \right)\left(1+\frac{(D_n^{(r)})^2}{\g{c}{sle,n}(\delta)}\right),\\
      \g{c}{itkill}=&(\diam K_n^{(r)}+\mu_n(F_n^{(r)}))\diam K_n^{(r)}\mu_n(F_n^{(r)}),\\
      \g{c}{sle,n}(\delta)=&\int_{0}^{T}(Ae^{-Bs}-1)^2e^{-\frac{s}{(D_n^{(r)})^2}}ds,\quad T=\frac{1}{3}m(2/3)\delta\log(5/2),\quad A=2/5,\quad B=\frac{3}{m(2/3)\delta},\\
      m(2/3)=&\inf_{b\in B(\rho,r_{\ast})\cap F}\mu(B(b,\frac{2}{3}\delta))\wedge \inf_{n\in\N}\inf_{b\in B(\rho,r_{\ast})\cap F_n}\mu_n(B(b,\frac{2}{3}\delta))>0,\\
       \g{c}{suffcond1,3}=&\left\{ \g{c}{continuity,1}\left[ c_{\rm{A2}}\Gamma\left(1-\frac{s}{(1+s)\theta}\right) \right]^{-1/2} \right\}^{2\kappa-1},
 \end{align*}
 where
 \begin{itemize}
  \item $D_n^{(r)}=\mu(F^{(r)})\diam F^{(r)}\vee \mu_n(F^{(r)}_n)\diam F^{(r)}_n$,
  \item $K_n^{(r)}=F^{(r)}\cup F^{(r)}_n$,
  \item $c_{\rm{A1}}=1+c_l^{-1}$ (see \cref{summ}(ii)),
  \item the constant $\g{c}{continuity,1}= \g{c}{continuity,1}(r)$ is a positive number satisfying $\|g^{(r)}_{\alpha}(x,\cdot)\|_{\Lip^1}\leq  \g{c}{continuity,1}(1+\alpha^{-1})$ for $x\in F^{(r)},\alpha>0$
  (Here, $g^{(r)}_{\alpha}$ is the alpha resolvent density associated with $X^{(r)}$. The constant $\g{c}{continuity,1.n}= \g{c}{continuity,1.n}(r)$ is that for $F_n^{(r)}$.), 
  \item the constant $c_{\rm{A2}}=c_{\rm{A2}}(r)$ is a positive number satisfying $c_{\rm{A2}}t^{-\frac{s}{(1+s)\theta}}\lesssim p^{(r)}(t,x,x)$ for $x\in F^{(r)},\;t>0$
  (see \cref{summ}(iii)), where $p^{(r)}$ is the heat kernel associated with $X^{(r)}$,
  \item the time-dependence of the constant in $\lesssim$ of $\g{c}{sg1,n}$ is $t^{-[\kappa-\frac{s}{2(1+s)\theta}(2\kappa-1)]},\;t\in (0,1]$.
 \end{itemize}
 We write $C(t)$ for the time-dependent constant in $\lesssim$ of $\g{c}{sg1,n}$.
Note the following:
\begin{itemize}
  \item By Assumption 2-c and \eqref{rescont}, $c_{\rm{A2}}=c_{\rm{A2}}(r)$ and $\g{c}{continuity,1}=\g{c}{continuity,1}(r)$ can be chosen so that 
  $\inf_{r\geq r_{\infty}}c_{\rm{A2}}(r)\ge c_{\rm{LHK}}>0$ and $\sup_{r\ge r_{\infty}}\g{c}{continuity,1}(r)\le \g{c}{continuity,1}(r_{\infty})<\infty$.
  Thus $\g{c}{suffcond1,3}$ is bounded and we obtain 
  \begin{align*}
     \g{c}{sg1,n}
      \lesssim C(t)\g{c}{estres,n}(D_n^{(r)})^{2}\mu(F^{(r)})^{1/2}.
  \end{align*}
  \item By \cref{1_1.1} and its proof, $\g{c}{continuity,1}= \g{c}{continuity,1}(r)$ is bounded above as a function of $r\geq r_{\infty}$.
  \item Since $d:=\inf_{n\in\N,r\ge r_{\infty}}D_n^{(r)}>0,$ we have 
  \begin{align*}
    \inf_{n\in\N,r\ge r_{\infty}}\g{c}{sle,n}(\delta)\ge \int_{0}^{T}(Ae^{-Bs}-1)^2e^{-\frac{s}{d^2}}ds>0.
  \end{align*}
  Hence we deduce $\g{c}{estres,n}\lesssim\g{c}{itkill}\mu(F^{(r)})\diam F^{(r)}(D_n^{(r)})^2$.
\end{itemize}
Combining the above observations, 
we conclude that
\begin{equation}\label{5.4}
  \begin{aligned}
    \g{c}{sg2,n}
  \lesssim&\g{c}{sg1,n}+(D_n^{(r)})^{2}\g{c}{estres,n}\mu_n(F_n^{(r)})+(D_n^{(r)})^{4}\Big[\mu(F^{(r)})+\mu_n(F_n^{(r)})\Big]\\
  \lesssim&C(t)
  (D_n^{(r)})^2 \left[\g{c}{estres,n}(\mu(F^{(r)})^{1/2}+\mu_n(F_n^{(r)}))+(D_n^{(r)})^2 (\mu(F^{(r)})+\mu_n(F_n^{(r)}))\right]\\
  \lesssim&C(t)
  (D_n^{(r)})^4 \left[\g{c}{itkill}\mu(F^{(r)})\diam F^{(r)}(\mu(F^{(r)})^{1/2}+\mu_n(F_n^{(r)}))\right]\\
  =&C(t)
  (D_n^{(r)})^4\\
  &\quad \times \left[(\diam K_n^{(r)}+\mu_n(F_n^{(r)}))\diam K_n^{(r)}\mu_n(F_n^{(r)})\mu(F^{(r)})\diam F^{(r)}(\mu(F^{(r)})^{1/2}+\mu_n(F_n^{(r)}))\right],
  \end{aligned}
\end{equation}
and
\begin{align*}
  \g{c}{gen,n}
\lesssim&\g{c}{sg2,n}\mu(F^{(r)})^{1/2}(D_n^{(r)})^{2(1+\frac{s}{2(1+s)\theta})(2\kappa-1)}\\
\lesssim&C(t)
  (D_n^{(r)})^{4+2(1+\frac{s}{2(1+s)\theta})(2\kappa-1)}\\
  &\quad \times \left[(\diam K_n^{(r)}+\mu_n(F_n^{(r)}))\diam K_n^{(r)}\mu_n(F_n^{(r)})\mu(F^{(r)})^{3/2}\diam F^{(r)}(\mu(F^{(r)})^{1/2}+\mu_n(F_n^{(r)}))\right],
\end{align*}
which completes the proof.
\end{proof}

\section{Examples}\label{app}

\subsection{Real trees coded by functions}
In this section, we provide an estimates for $d_{\frK}^{\tau,\BL^{\kappa}}(\calT^f,\calT^g)$
in terms of the supremum of $f-g$, where $\calT^f$ is the real tree coded by a function $f$.
Recall from the definition of $d_{\frK}^{\tau,\BL^{\kappa}}$ from \cref{nrca}, and
see \cref{rtreewcode} and succeeding paragraph for the precise definition of $\calT^f$.

Real trees frequently arise in probability theory and remain a central topic of research. 
For more background on random trees, the reader is referred to \cite{le2006random}.

\begin{dfn}[Real tree coded by a function, {\cite[Section 3]{abraham2013note}}]\label{rtreewcode}
  Let $f:[0,\infty)\to [0,\infty)$ be a compactly supported continuous function and set 
  \[ \sigma_f=\sup\{t>0:f(t)>0\}<\infty.\]
  Define a  semi-metric $\bar{d}^f$ on $[0,\sigma_f]$ by setting
  \begin{align*}
    \bar{d}^f(s,t):=f(t)+f(s)-2\inf_{u\in [s,t]}f(u),\qquad 0\le s\le t\le\sigma_f.
  \end{align*}
  We define a real tree coded by $f$ by 
  \begin{align*}
    T^f:=[0,\sigma_f]/\sim,
  \end{align*}
  where $s\sim t$ if and only if $\bar{d}^f(s,t)=0$.
\end{dfn}
Let $T^f$ be the real tree coded by a function $f \in C_c^{+}(\R_{\ge 0})$,
where
\begin{align*}
  C_c^{+}(\R_{\ge 0}) := \{ f: [0,\infty) \to [0,\infty) : \text{$f$ is a compactly supported continuous function} \}.
\end{align*}
We define a metric $d^f$ on $T^f$ by
\begin{align*}
  d^f(p^f(s), p^f(t)) := \bar{d}^f(s,t), \qquad 0 \le s, t \le \sigma_f,
\end{align*} 
where $p^f: [0,\sigma_f] \to T^f$ is the canonical projection.
Thus, $T^f$ is equipped with the structure of a metric space.
Note that since $p^f$ is continuous, $(T^f, d^f)$ is compact.
Furthermore, we define a finite Borel measure $m^f$ on $T^f$ as the pushforward of the Lebesgue measure on $[0,\sigma_f]$ under $p^f$.
By abuse of notation, we also refer to the triplet $\calT^f:=(T^f, d^f, m^f)$ as the real tree coded by $f$.
It is known that a real tree coded by a function $(T^f,d^f)$ is a resistance metric space. See \cite[Proposition 5.1]{kigami1995harmonic} for its proof.

We now provide an estimate for $d_{\frK}^{\tau,\BL^{\kappa}}(\calT^f,\calT^g)$.
Note that similar results are shown in \cite[Proposition 3.3]{abraham2013note} and \cite[Proposition 8.3]{Ndloc}.

\begin{prp}
  For $f,\;g\in C_c^{+}(\R_{\ge 0})$, we have 
  \begin{align*}
    d_{\frK}^{\tau,\BL^{\kappa}}(\calT^f,\calT^g)\leq 2\|f-g\|_{[0,\sigma_f\wedge\sigma_g]}+2^{\kappa}(\sigma_f\wedge\sigma_g)\|f-g\|_{[0,\sigma_f\wedge\sigma_g]}^{\kappa}+|\sigma_f-\sigma_g|,
  \end{align*}
  where
  \begin{align*}
    \|f-g\|_{[0,\sigma_f\wedge\sigma_g]}:=\sup_{u\in [0,\sigma_f\wedge\sigma_g]}|f(u)-g(u)|.
  \end{align*}
\end{prp}
\begin{rmk}
  In some cases, we let $\rho^f:=p^f(0)$ be the root of $T^f$, 
  and define the real tree coded by $f$ as $\calT^f=(T^f,d^f,m^f,\rho^f)$. 
  In this case, by using the distance $d_{\frK_{\bullet}}^{\tau,\BL^{\kappa}}$ defined in \cref{rcop}, 
  we obtain the following:
  \begin{align*}
    d_{\frK}^{\tau,\BL^{\kappa}}(\calT^f,\calT^g)\leq 4\|f-g\|_{[0,\sigma_f\wedge\sigma_g]}+2^{\kappa}(\sigma_f\wedge\sigma_g)\|f-g\|_{[0,\sigma_f\wedge\sigma_g]}^{\kappa}+|\sigma_f-\sigma_g|.
  \end{align*}
\end{rmk}

\begin{proof}
  First we set 
  \begin{equation*}
    \calR:=\{(x^f,x^g)\in T^f\times T^g:\exists t\in [0,\infty),\;x^f=p^f(t),\;x^g=p^g(t)\},
  \end{equation*}
  and
  \begin{equation*}
    \rm{dis}(\calR):=\sup\{|d^f(x^f,y^f)-d^g(x^g,y^g)|: (x^f,x^g),\;(y^f,y^g)\in\calR\}.
  \end{equation*}
  The quantity $\rm{dis}(\calR)$ is called a distortion of $\calR$.
  Note that the canonical projections $\calR\to T^f$ and $\calR\to T^g$ are surjection.
  Then we have
  \begin{align}\label{bhy}
    \rm{dis}(\calR)\le 4\|f-g\|_{[0,\sigma_f\wedge\sigma_g]}.
  \end{align}
  In fact, if we take $(x^f,x^g),\;(y^f,y^g)\in \calR$, and 
  set $x^f=p^f(s),\;x^g=p^g(s),\;y^f=p^f(t),$ and $y^g=p^g(t)$ for 
  $0\le s\le t\le\sigma_f\wedge\sigma_g$,
  it holds that
  \begin{align*}
    |d^f(x^f,y^f)-d^g(x^g,y^g)|
    =&\left|f(s)-g(s)+f(t)-g(t)+2\inf_{u\in [s,t]}g(u)-2\inf_{u\in [s,t]}f(u)\right|\\
    \leq&2\|f-g\|_{[0,\sigma_f\wedge\sigma_g]}+2\sup_{u\in [s,t]}|f(u)-g(u)|\\
    \leq&4\|f-g\|_{[0,\sigma_f\wedge\sigma_g]}.
  \end{align*}
  For $Z:=T^f \sqcup T^g$ and $\varepsilon>0$,
  we define a function $d^\varepsilon:Z\times Z\to [0,\infty)$ by setting
  \begin{align*}
    d^{\varepsilon}(x,y):=
    \begin{dcases}
      d^f(x,y), & x,y\in T^f,\\
      \inf\left\{ d^f(x,x')+\frac{1}{2}\rm{dis}(\calR)+d^g(y,y')+\varepsilon:(x',y')\in\calR \right\}, & x\in T^f,\;y\in T^g,\\
      d^g(x,y), & x,y\in T^g.
    \end{dcases}
  \end{align*}
  Note that $d^{\varepsilon}(x,y)=\frac{1}{2}\rm{dis}(\calR)+\varepsilon$ for $(x,y)\in\calR$.
  It is not difficult to confirm that $(Z,d^{\varepsilon})$ is a compact metric space,
   and $T^f$ and $T^g$ are isometrically embedded into $(Z,d^{\varepsilon})$.
  We next evaluate the Hausdorff distance $d^{\varepsilon}_{\rm{H}}(T^f,T^g)$.
  Fix an arbitrary $x\in T^f$. Then we may find $y\in T^g$ such that $(x,y)\in \calR$,
  and we observe that
  \begin{align*}
    d^{\varepsilon}(x,y)=
      \inf\left\{ d^f(x,x')+\frac{1}{2}\rm{dis}(\calR)+d^g(y,y')+\varepsilon:(x',y')\in\calR \right\}
      =\frac{1}{2}\rm{dis}(\calR)+\varepsilon.
  \end{align*}
  Thus it holds that 
  \begin{align*}
    d^{\varepsilon}_{\rm{H}}(T^f,T^g)\le \frac{1}{2}\rm{dis}(\calR)+\varepsilon\le 2\|f-g\|_{[0,\sigma_f\wedge\sigma_g]}+\varepsilon
  \end{align*}
  by \eqref{bhy}.
  We finally bound $d_{\BL^{\kappa}}^{\varepsilon}(m^f,m^g)$.
  Take a function $h:Z\to \R$ satisfying $\|h\|_{\BL^{\kappa}(Z,d^{\varepsilon})}\le 1$.
  Assuming $\sigma_f\le\sigma_g$, we deduce from 
  $(p^f(t),p^g(t))\in\calR$ and \eqref{bhy} that
  \begin{align*}
    \left|\int_Z hdm^f-\int_Z hdm^g\right|
    =&\left|\int_{0}^{\sigma_f}h(p^f(t))dt-\int_{0}^{\sigma_g}h(p^g(t))dt\right|\\
    \leq&\left|\int_{0}^{\sigma_f}h(p^f(t))-h(p^g(t))dt\right|+\left| \int_{\sigma_f}^{\sigma_g} h(p^g(t))dt\right|\\
    \leq&\int_{0}^{\sigma_f} d^{\varepsilon}(p^f(t),p^g(t))^{\kappa} dt+|\sigma_f-\sigma_g|\\
    \leq&\sigma_f\left( \frac{1}{2}\mathrm{dis}(\mathcal{R})+\varepsilon \right)^{\kappa}+|\sigma_f-\sigma_g|\\
    \leq&\sigma_f\left( 2\|f-g\|_{[0,\sigma_f\wedge\sigma_g]}+\varepsilon \right)^{\kappa}+|\sigma_f-\sigma_g|.
  \end{align*}
Altogether, we obtain the desired result.
\end{proof}

Finally, we establish a sufficient condition for $m^f$ to be regular.

\begin{prp}
  Let $f\in C_c^{+}(\R_{\ge 0})$ and set 
  \begin{align*}
    \omega(h):=\sup_{\substack{s,t\in [0,\sigma_f]\\|s-t|\le h}}|f(s)-f(t)|,\;h\ge 0,\qquad
  \omega^{-1}(r):=\sup\{h>0:\omega(h)<r\},\;r>0.
  \end{align*}
  \begin{itemize}
    \item [$\rm{(i)}$]For any $x\in T^f$ and $r>0$, we have
    \begin{align*}
      \omega^{-1}\left(\frac{r}{2}\right)\wedge\frac{\sigma_f}{2}\le m^f(B(x,r)).
    \end{align*}
    \item [$\rm{(ii)}$]For any $t\in [0,\sigma_f]$ and $r>0$, let $I(t,r)$ be the connected component of $\{s\in[0,\sigma_f]:f(s)>f(t)-r\}$ containing $t$.
    Suppose that there exists a function $V:[0,\infty)\to[0,\infty)$ such that
    \begin{align*}
      \rm{Leb}(I(t,r)\cap \{|f-f(t)|<r\})\le V(r),\qquad\forall t\in [0,\sigma_f],\;r>0,
    \end{align*}
where $\rm{Leb}$ is the one dimensional Lebesgue measure.
    Then it holds that 
    \begin{align*}
      m^f(B(x,r))\le V(r),\qquad \forall x\in T^f,\;r>0.
    \end{align*}
  \end{itemize}
\end{prp}
\begin{proof}[Proof of $\rm{(i)}$]
  Fix $x=p^f(t_0)\in T^f,\;t_0\in [0,\sigma_f]$ and $r>0$ arbitrarily.
  If we set $\delta_{\varepsilon}:=\omega^{-1}(r/2)-\varepsilon>0$ for small $\varepsilon>0$,
  we have $\omega(\delta_{\varepsilon})<r/2$.
  Define an interval $J_{\varepsilon}\subseteq [0,\sigma_f]$ by setting
  \begin{align*}
    J_{\varepsilon}:= 
\begin{dcases} 
\left[t_0,t_0+\left(\delta_{\varepsilon}\wedge\frac{\sigma_f}{2}\right)\right], & t_0\le\frac{\sigma_f}{2}, \\ 
\left[t_0-\left(\delta_{\varepsilon}\wedge\frac{\sigma_f}{2}\right), t_0\right], &  t_0>\frac{\sigma_f}{2}.
\end{dcases}
  \end{align*}
  Then $J_{\varepsilon}$ satisfies
  \begin{align*}
  t_0 \in J_{\varepsilon}, \quad \mathrm{Leb}(J_{\varepsilon})=\delta_{\varepsilon}\wedge\frac{\sigma_f}{2},\quad |t-t_0|\le\delta_{\varepsilon},\;\forall t\in J_{\varepsilon}.
  \end{align*}
  We will prove 
  \begin{align}\label{olki}
    J_{\varepsilon}\subseteq\{t\in[0,\sigma_f]:d^f(p^f(t),x)<r\}.
  \end{align}
  Fix an arbitrary $t\in J_{\varepsilon}$.
  Since $f$ is continuous, there exists $u\in [t\wedge t_0,t\vee t_0]$
  such that $f(u)=\inf_{s\in[t\wedge t_0,t\vee t_0]}f(s)$ and $|t-u|,\;|t_0-u|\le \delta_{\varepsilon}$.
  Thus we may conclude that
  \begin{align*}
    d^f(p^f(t),x) 
   =&f(t)+f(t_0)-2f(u)\\
   \le&\omega(|t-u|)+\omega(|t_0-u|)\\
   \le&2\omega(\delta_{\varepsilon})\\
   <&r,
\end{align*}
which implies \eqref{olki}.
Therefore we deduce that 
\begin{align*}
m^f(B(x,r)) 
=&\mathrm{Leb}\Big(\{t\in[0,\sigma_f]:p^f(t)\in B(x,r)\}\Big)\\
=&\mathrm{Leb}\Big(\{t\in[0,\sigma_f]:d^f(p^f(t),x)<r\}\Big)\\
\ge&\mathrm{Leb}(J_{\varepsilon})\\
=&\delta_{\varepsilon}\wedge \frac{\sigma_f}{2}.
\end{align*}
Letting $\varepsilon\to 0$, we obtain the result.
\end{proof}
\begin{proof}[Proof of $\rm{(ii)}$]
  Fix arbitrary $x=p^f(t_0)\in T^f,\;t_0\in [0,\sigma_f]$ and $r>0$.
 We will prove
 \begin{align}\label{kjuh}
  A_r\subseteq I(t_0,r)\cap\{|f-f(t_0)|<r\},\quad A_r:=\{s\in [0,\sigma_f]:d^f(p^f(s),x)<r\}.
 \end{align}
 Take $t\in A_r$ arbitrary. Then we have 
 $f(t)>f(t_0)-r$ from \[ |f(t)-f(t_0)|=|d^f(p^f(0),p^f(t))-d^f(p^f(0),p^f(t_0))|\le d^f(p^f(t),p^f(t_0))<r.\]
 Thus we obtain from $d^f(p^f(t),p^f(t_0))<r$ that
 \begin{align*}
   2\inf_{u\in[t\wedge t_0,t\vee t_0]}f(u)>f(t)+f(t_0)-r>2(f(t_0)-r),
 \end{align*}
 which implies $f(u)>f(t_0)-r$ for any $u\in [t\wedge t_0,t\vee t_0]$.
 In particular, it holds that $t\in [t\wedge t_0,t\vee t_0]\subseteq I(t_0,r)$.
 Therefore we conclude \eqref{kjuh}.
 Now, the desired results follows from the following:
\begin{align*}
  m^f(B(x,r))
  =\mathrm{Leb}(A_r)
  \le\mathrm{Leb}\Big(I(t_0,r)\cap\{|f-f(t_0)|<r\}\Big)\le V(r).
\end{align*}
\end{proof}

\begin{rmk}
Previous studies \cite{aldous1993continuum,duquesne2003limit} have demonstrated that, 
for Galton-Watson trees $T^{S_n}$ whose offspring distributions have finite or possibly infinite variance, 
the sequence of scaled contour functions $S_n$ converges in a suitable sense to twice the Brownian excursion $2W$. 
Furthermore, convergence of random walks on $T_n$ to the natural stochastic process on $T^{2W}$ 
has been established in \cite{Crtree}, 
and weak convergence of $\calT^{S_n}$ to $\calT^{2W}$ in the Gromov-Hausdorff-Prokhorov topology was demonstrated in \cite[Corollary 8.7]{Ndloc}. 
Consequently, as suggested by our main results and the arguments presented in this section, 
constructing a suitable coupling between $S_n$ and $2W$ would allow us 
to quantify the convergence obtained in \cite{Crtree}.
\end{rmk}

\subsection{Sierpinski gasket}
In this section, we derive estimates for the semigroup and heat kernel on the Sierpinski gasket.
The study of Brownian motion on the Sierpinski gasket, including its rigorous construction, 
began about forty years ago. In \cite{BP88, Gol87, Kus87}, 
this process was constructed as a scaling limit of random walks on approximating finite graphs.
Also, rate estimates for the generators and resolvents were demonstrated in \cite{PS18}.
Although this approximation has been known for several decades, 
quantitative convergence rates for semigroups and heat kernels have remained unknown. 
In \cref{sgrate}, we establish explicit estimates for these rates.

To deal with the Sierpinski gasket, we begin slightly more general setting.
Let $f_i:\R^k\to\R^k,i=1,\dots, n$ be similitudes (with respect to the Euclidean metric).
Here, we say a function $f:\R^k\to\R^k$ is a similitude or an $r-$similitude if there exists a constant $r\in (0,1)$
such that $|f(x)-f(y)|=r|x-y|$ for any $x,\;y\in\R^k$.
Then there exists a unique non-empty compact set $K\subseteq \R^k$ satisfying $K=\cup_{i=1}^n f_i(K)$, and 
the set $K$ is called the self-similar set with respect to $(f_i)_{i=1}^n$ (see \cite[Theorem 1.1.4]{Kig01}).
If we set $f_i(x)=\frac{1}{2}(x+p_i),\;x\in\R^2,\;i=1,2,3$ for vertices $p_i$ of a regular triangle,
then $K$ will be the Sierpinski gasket.
We return to the general setting. 
Suppose that each $f_i$ is an $r_i$-similitude with respect to the Euclidean metric $d_E$.
We say $(f_i)_{i=1}^n$ satisfies the open set condition if there exists a non-empty bounded open set $O\subseteq \R^k$
such that $\cup_{i=1}^n f_i(O)\subseteq O$ and $f_i(O)\cap f_j(O)=\emptyset$ for $i\neq j$.
It is known that then the Hausdorff dimension $\dim_{\rm{H}} K$ of $K$ with respect to $d_E$ is
given by a unique positive number $\alpha$ satisfying $\sum_{i=1}^n r_i^{\alpha}=1$ 
if the open set condition holds
(see \cite[Corollary 1.5.9]{Kig01}). 
In the setting above, we may confirm the regularity of the Hausdorff measure on $K$.
For a proof of the following lemma, see a proof of \cite[Section 5.3 Theorem 1]{MR625600}.

\begin{lem}\label{ahl}
Let $(f_i)_{i=1}^n$ be a family of similitudes, and $\alpha$ be the Hausdorff dimension of 
the self-similar set $K\subseteq \R^k$ with respect to $d_E$.
If $(f_i)_{i=1}^{n}$ satisfies the open set condition, then the 
$\alpha$-dimensional Hausdorff measure $\calH^{\alpha}$ is $\alpha$-Ahlfors regular i.e.
it satisfies the following for some constants $c,C\in (0,\infty)$.
\begin{align*}
  cr^{\alpha}\leq \calH^{\alpha}(B(x,r))\leq Cr^{\alpha},\qquad x\in K,r\in[0,\diam K).
\end{align*}
\end{lem}

\begin{figure}[h]
  \centering
  \begin{minipage}[b]{0.4\textwidth}
    \centering
    \includegraphics[width=40mm]{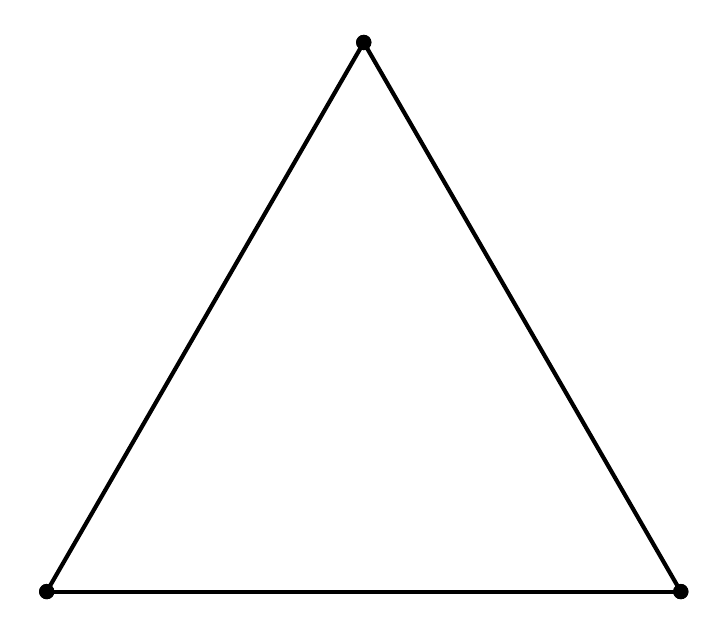}
    \par \vspace{1pt} $V_0$
  \end{minipage}
  \hspace{10mm}
  \begin{minipage}[b]{0.4\textwidth}
    \centering
    \includegraphics[width=40mm]{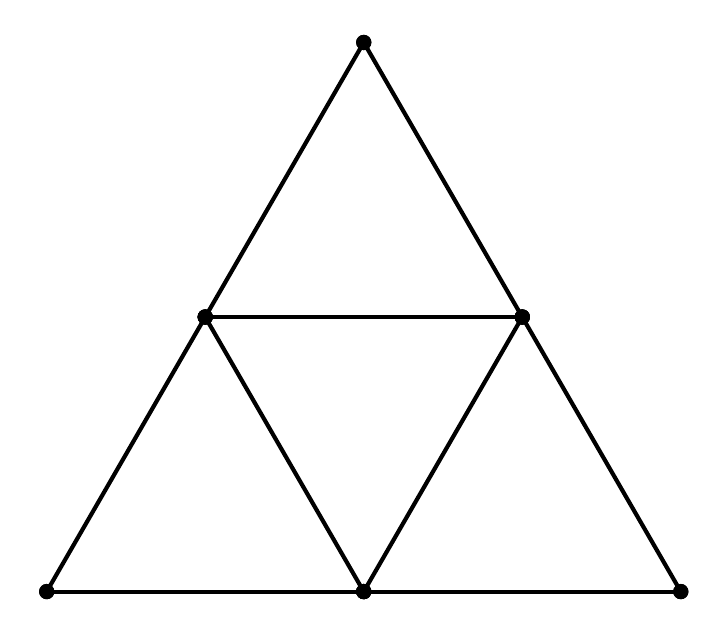}
    \par \vspace{1pt} $V_1$
  \end{minipage}

  \vspace{5mm}
  \begin{minipage}[b]{0.4\textwidth}
    \centering
    \includegraphics[width=40mm]{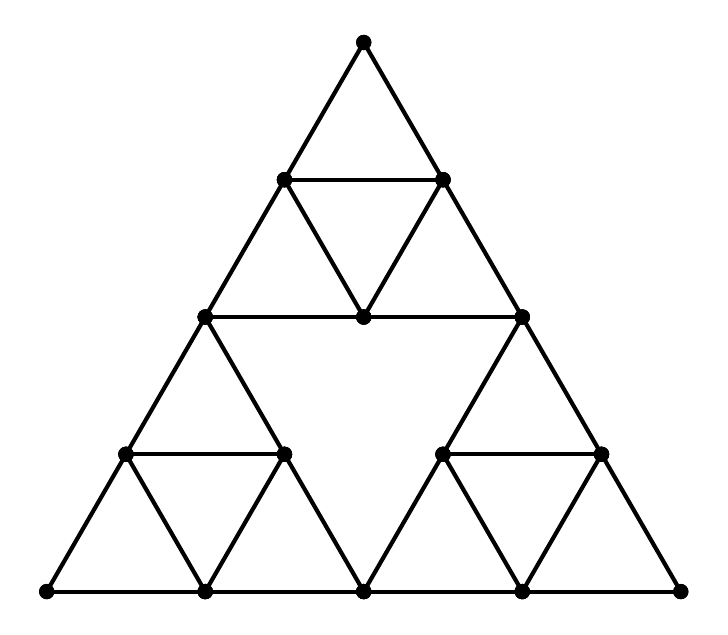}
    \par \vspace{1pt} $V_2$
  \end{minipage}
  \hspace{10mm}
  \begin{minipage}[b]{0.4\textwidth}
    \centering
    \includegraphics[width=40mm]{pdf_image/SG6_filled.pdf}
    \par \vspace{1pt} Sierpinski Gasket
  \end{minipage} 

  \vspace{5mm} 
  \caption{Let $V_0$ be the set of vertices of a regular triangle. 
  We define $V_n$ as $V_{n}:=\cup_{i=1}^{3}f_i(V_{n-1})$, 
  where the maps $f_i,\;i=1,2,3$ are defined at the beginning of this section.} 
  \label{fig:sg} 
  
\end{figure}

We now concentrate on the Sierpinski gasket.
Recall the definitions of $V_n$ and the Sierpinski gasket
from the beginning of this section (also, see \Cref{fig:sg}).
We write $K$ for the Sierpinski gasket.
We attach the nearest-neighbor network structure on each $V_n$ and regard the Sierpinski gasket
as a resistance metric space, assigning $(5/3)^n$ to every edge of $V_n$ as a conductance.

If we write
$R(\cdot,\cdot)$ for the resistance metric on the Sierpinski gasket, 
it is known (see \cite[Remark 3.7]{MR1301625}) that
\begin{align}\label{maa3}
  c d_E(x,y)^{\log_2\frac{5}{3}}\leq R(x,y)\leq C d_E(x,y)^{\log_2\frac{5}{3}},\qquad x,y\in K.
\end{align} 
The following is an immediate consequence of \cref{ahl} and \eqref{maa3}.

\begin{cor}\label{sgreg}
The $\log_2 3$-dimensional Hausdorff measure on $K$ is 
$\log_{\frac{5}{3}}3$-Ahlfors regular with respect to the resistance metric $R$.
\end{cor}

We define notation by
\begin{align*}
  W_n=\{1,2,3\}^n,\qquad f_I=f_{i_1}\circ\cdots\circ f_{i_n},\qquad K_I=f_I(K),\qquad V_I=f_I(V_0),\;\;\;I=i_1\cdots i_n\in W_n. 
\end{align*}
Note that $V_I$ is the set corners of $K_I$. 
Then the invariant measure $\mu_n$ of the simple random walk on $V_n$ is
given by 
\begin{align*}
  \mu_n=\frac{1}{3^{n+1}}\sum_{I\in W_n}\sum_{x\in V_I}\delta_x.
\end{align*}
Note that $\mu_n$ can be expressed as 
\begin{align*}
  \mu_n(\{x\})=\frac{\deg_{V_n}(x)}{\displaystyle\sum_{y\in V_n}\deg_{V_n}(y)},\qquad x\in V_n.
\end{align*}
Here, we denote the degree of a vertex $x\in V_n$ in the graph $V_n$ by $\deg_{V_n}(x)$.
We denote the $\log_2 3$-dimensional probability Hausdorff measure on $K$ by $\mu$.
We write $R_{\rm{H}}$ and $R_{\BL^{\kappa}}$ for the Hausdorff distance 
and $\BL^{\kappa}$ metric on $K$ 
with respect to $R$, respectively.
\begin{lem}
  It holds that
  \begin{align*}
    R_H(K,V_n)\lesssim \left(\frac{3}{5}\right)^{n},\qquad R_{\BL^{\kappa}}(\mu,\mu_n)\lesssim \left(\frac{3}{5}\right)^{\kappa n}.
  \end{align*}
\end{lem}
\begin{proof}
We write $d_{\rm{H}}$ for the Hausdorff distance on $\R^2$ 
with respect to the Euclidean metric.
Then it is clear that $d_{\rm{H}}(K,V_n)\lesssim 2^{-n}$.
The inequality for $R_H$ follows from this and \eqref{maa3}.
We next show the inequality for the $\BL^{\kappa}$ metric.
We write $V_I=\{x_I^1,x_I^2,x_I^3\}$.
Let $f$ be a $\kappa$-H\"older continuous function with respect to $R$ satisfying
$\|f\|_{\BL^{\kappa}}\leq 1$.
Then, since $\mu(K_I)=3^{-n}$ for $I\in W_n$, we have that
\begin{align*}
  \left| \int_K f d(\mu-\mu_n) \right|
  \leq&\sum_{I\in W_n}\left|\int_{K_I} f(x)-\frac{1}{3}(f(x_I^1)+f(x_I^2)+f(x_I^3))\mu(dx)\right|\\
  \leq&\sum_{I\in W_n}\int_{K_I}\frac{1}{3}\sum_{j=1}^3 |f(x)-f(x_I^j)|\mu(dx)\\
  \leq&\sum_{I\in W_n}\int_{K_I}(\diam_R K_I)^{\kappa}\mu(dx)\\
  \lesssim&\sum_{I\in W_n}\int_{K_I}(\diam_{d_E} K_I)^{\kappa \log_2\frac{5}{3}}\mu(dx)\\
  \lesssim&\sum_{I\in W_n}\int_{K_I}\left(\frac{3}{5}\right)^{\kappa n}\mu(dx)\\
  =&\left(\frac{3}{5}\right)^{\kappa n},
\end{align*}
which completes the proof.
\end{proof}

Since the Sierpinski gasket is connected, and the corresponding process is continuous (see \cite[Theorems 1.6.2 and 3.4.6]{Kig01}),
it satisfies Assumption (A3).
Combining the lemma above, \cref{gen}, and \cref{hkest}, we conclude the following.
\begin{thm}\label{sgrate}
Take $x,y\in K$ and $x_n,y_n\in V_n$ satisfying $R(y,y_n),R(x,x_n)\lesssim R_H(K,V_n)\lesssim (3/5)^n$ 
and set
\begin{equation*}
  E_{\rm{sg},1}=\frac{\kappa^2}{\kappa+(2+\kappa)\log_{\frac{5}{3}}5},
\qquad
  E_{\rm{sg},2}=\frac{\kappa (2+\kappa)\log_{\frac{5}{3}}3}{\kappa+(2+\kappa)\log_{\frac{5}{3}}5},
\end{equation*}
\begin{equation*}
  E_{\rm{hk},1}=\frac{\kappa}{(2+\kappa)\log_{\frac{5}{3}}5+\frac{1}{2}\log_{\frac{5}{3}}3+1},
  \;\;\text{and}\;\; E_{\rm{hk},2}=(2+\kappa)\log_{\frac{5}{3}}5
\end{equation*}
Then, for any $f\in\BL^{\kappa}$, it holds that
\begin{align*}
  |P_tf(x)-P^n_tf(x_n)|\lesssim \|f\|_{\BL^{\kappa}} \left(\frac{3}{5}\right)^{nE_{\rm{sg},1}}n^{E_{\rm{sg},2}},\;\;\rm{and}\;\;
  |p(t,x,y)-p_n(t,x_n,y_n)|\lesssim \left(\frac{3}{5}\right)^{nE_{\rm{hk},1}}n^{E_{\rm{hk},2}},
\end{align*}
for any $t>0$ and $\kappa\in [1/2,1]$.
\end{thm}

\begin{rmk}
  The exponents $E_{\rm{sg},1}$ and $E_{\rm{hk},1}$ are maximized by taking $\kappa=1$, and 
  in that case, each exponent will be
  \begin{align*}
    E_{\rm{sg},1}=0.0956756\cdots\;\;\rm{and}\;\;
    E_{\rm{hk},1}=0.0867505\cdots .
  \end{align*}
\end{rmk}  

\begin{rmk}
  Suppose that $f:\R^2\to\R$ is an Lipschitz continuous function with respect to the Euclidean distance $d_E$.
  Then, by \eqref{maa3}, it holds that
  \begin{align*}
    |f(x)-f(y)|\lesssim d_E(x,y)\lesssim R(x,y)^{\log_{\frac{5}{3}}2},
  \end{align*}
  and thus $f|_K$ is $\log_{\frac{5}{3}}2$-H\"older continuous with respect to $R$.
  Noting $\log_{\frac{5}{3}}2>1$, we obtain \eqref{ohsugo}.
  In particular, we may take any Lipschitz continuous function on $(\R^2,d_E)$ as 
  a test function in \cref{sgrate}.
\end{rmk}

\subsection{One dimensional Bouchaud trap model}\label{btm}
In physics, aging refers to a phenomenon in which a system's dynamics progressively slow down 
without reaching equilibrium on laboratory time scales, 
so that its current state carries information about its age. 
Such behavior is typical in disordered media such as spin glasses at low temperatures.
Originally, the model that is nowadays called the Bouchaud trap model (BTM), was introduced in \cite{Bou92} to understand aging phenomena,
and the one-dimensional version considered in this section was firstly studied in \cite{MR1725402}.
More mathematical details on the BTM can be found in \cite{AC06}.

To describe the (scaled) one-dimensional BTM precisely, we introduce
a family of i.i.d. random variables $(\tau_x)_{x\in\Z}$ on a probability space
$(\Omega,\calF,\bbP)$. 
For a given $\alpha>0$, we suppose that $\tau_x$ satisfies $\bbP(\tau_0\geq u)=u^{-\alpha},\;u>1$.
In this section, however, we focus on the regime where the mean trapping time is finite, namely $\alpha>1$.
The (scaled) dynamics of the BTM are given by an $n^{-1}\Z$-valued Markov process
$X^n=(X^n_t)$ whose generator is given by
\begin{align*}
  \Delta_n f(x/n)=\frac{n^2}{2\tau_x}\sum_{y:|y-x|=1}(f(y/n)-f(x/n)).
\end{align*}
In the $n$-th model, the random variable $\tau_x$ can be interpreted as the depth of a trap at location $x/n$.
As the value of $\tau_x$ increases, the process $X^n$ is more likely to stay trapped in the state $x/n$ for an extended period.
In terms of the resistance metric, this can be described as follows.
Let $F=\R$ and $F_n=n^{-1}\Z$ and equip these spaces with the usual Euclidean metric.
Note that the one-dimensional Euclidean metric is a resistance metric.
We take $0$ as a root of $F$ and $F_n$.
Set $\mu=\bbE[\tau_0]dx$ and 
\begin{align*}
  \mu_n=\frac{1}{n}\sum_{x\in \Z}\tau_x \delta_{\frac{x}{n}},
\end{align*}
where $\bbE[\cdot]$ denotes the expectation with respect to $\bbP$.
Then $X^n$ is the process associated with $(F_n,\mu_n)$. 
We write $X$ for the process corresponding to $(F,\mu)$.
Note that $X$ is a time-scaled Brownian motion.
Under this setting, the following quantified estimates were known.
Let $p$ and $p_n$ be the heat kernels associated with $(\R,\mu)$ and $(F_n,\mu_n)$, respectively.
\begin{thm}[{\cite[Theorems 1.5(ii) and 1.7]{BTM24}}]
Let $\alpha>2$ and $E\in \left(0,\frac{\alpha-2}{10(2\alpha-1)}\right)$.
Then it holds for any $K\in (0,\infty)$ and $0<T_1<T_2<\infty$ that
\begin{equation*}
  \sup_{x\in \R}|P_0(X_1\le x)-P^n_0(X^n_1\le x)|=O(n^{-\theta}),
\end{equation*}
and
\begin{equation*}
  \lim_{n\to\infty}n^{\frac{2}{3}E}\sup_{\substack{|x|\le K\\t\in [T_1,T_2]}}\left| p(t,0,x)-p_n\left( t,0,\frac{\floor{nx}}{n} \right) \right|=0,
\end{equation*}
$P$-almost surely.
\end{thm} 
Since the argument of \cite{BTM24} used an argument involving second moment, the parameter $\alpha$ needed to be strictly more than $2$.
In \cref{btmrate}, we extend the regime of $\alpha$ to $(1,\infty)$ and give better quenched estimates for the distribution function and heat kernel.
Also, we prove quenched and annealed estimates for semigroup in \cref{btmsgQ,btmsgA}, 
and an annealed estimate for the heat kernel in \cref{annhk}.
\begin{rmk}
  Among the six possible annealed and quenched estimates for semigroups, distribution functions, and heat kernels,
  we do not include the annealed estimate for the distribution function in this paper.
  Although this estimate can be derived using our method, 
  the resulting rate $n^{-\frac{1}{11}+\varepsilon}$, is worse than the rate $n^{-\frac{1}{10}}(\log n)^{\frac{1}{5}}$ obtained by \cite[Theorem 1.5(i)]{BTM24}. 
\end{rmk}
Let $X$ be the process corresponding to $(\R,\mu)$.
Note that it is a Brownian motion, up to a deterministic constant time change.
We write $p$ and $p_n$ for the heat kernels of $X$ and $X^n$ respectively.

\subsubsection{Quenched estimates}
Our aim in this section is to prove the following quenched estimates.
\begin{thm}\label{btmrate}
Let $\alpha>1$. Then, for any $\displaystyle E\in \left(0,\frac{2}{11}\left[(1-\alpha^{-1})\wedge \frac{1}{2}\right]\right),K\in (0,\infty)$, and $0<T_1<T_2<\infty$,
it holds that
\begin{equation}\label{distfunc}
  \sup_{x\in\R}\Big|P_0(X_t\leq x)-P_0^n(X_t^n\le x)\Big|=O(n^{-E}),
\end{equation}
and
\begin{equation}\label{hkbtm}
  \lim_{n\to\infty}n^{\frac{2}{3}E}\sup_{\substack{|x|\le K\\t\in [T_1,T_2]}}\left| p(t,0,x)-p_n\left( t,0,\frac{\floor{nx}}{n} \right) \right|=0,
\end{equation}
$\bbP$-almost surely.
\end{thm}
\begin{thm}\label{btmsgQ}\quad\par
  For any $\alpha>1,\;\kappa\in (1/2,1],\;t>0,\;f\in\BL^{\kappa},\;K\in (0,\infty)$ and $E<\frac{\kappa}{3\kappa+4}\Big[ (1-\alpha^{-1})\wedge\frac{1}{2} \Big]$, 
  we have
  \begin{align*}
    \sup_{|x|\le K}\Big|E_{x}[f(X_t)]-E_{\frac{\floor{nx}}{n}}^n[f(X_t^n)]\Big|\lesssim \|f\|_{\BL^{\kappa}} n^{-E},\qquad \bbP\rm{-a.s}.
  \end{align*}
\end{thm}

For the reader's convenience, we summarize some preliminary results used in the proof of \cref{dblbtm}.
For a metric space $(X,d)$ and $\varepsilon>0$, let $N(X,d,\varepsilon)$ be the smallest number of balls od radius $\varepsilon$
in the metric $d$ needed to cover $X$.
\begin{thm}[{\cite[Theorem 2.7.1]{VWJ96}}]\label{metent}
  Let $I\subseteq \R$ be a bounded interval with Lebesgue measure $|I|>0$ and $\kappa\in (0,1]$.
  Then there exists a constant $C=C(\kappa)$ such that
  \begin{align*}
    \log N(C^{\kappa}_1(I),\|\cdot\|_{\sup},\varepsilon)\leq C(|I|+2)\varepsilon^{-\frac{1}{\kappa}},\qquad\varepsilon>0,
  \end{align*} 
  where 
  \begin{align*}
    C^{\kappa}_1(I)=\{f:I\to\R:\|f\|_{\BL^{\kappa}}\le 1\}.
  \end{align*}
\end{thm}
\begin{dfn}[{\cite[Definitions 2.5.6 and 8.1.1]{HDP18}}]\label{SGincr}\quad\par
  \begin{itemize}
    \item [$\rm{(i)}$]For a real-valued random variable $X$, we define its sub-Gaussian norm 
  $\|X\|_{\psi_2}$ by setting
  \begin{align*}
    \|X\|_{\psi_2}:=\inf\{t>0:E[\exp(X^2/t^2)]\le 2\}.
  \end{align*}
   \item [$\rm{(ii)}$]Let $(X_t)_{t\in T}$ be a real-valued stochastic process on a metric space $(T,d)$.
   We say the process has a sub-Gaussian increments if there exists a constant $K\in (0,\infty)$
   such that
   \begin{align}\label{SGinc}
    \|X_t-X_s\|_{\psi_2}\le Kd(s,t),\qquad s,\;t\in T.
   \end{align}
  \end{itemize}
\end{dfn}
The following theorem, known as the Dudley's integral inequality gives a bound on the expectation of a process with sub-Gaussian increments.
\begin{thm}[Dudley's integral inequality, {\cite[Theorem 8.13 and Remark 8.1.9]{HDP18}}]\label{dudint}
  Let $(X_t)_{t\in T}$ be a mean zero stochastic process on a metric space $(X,d)$ 
  with sub-Gaussian increments as in \eqref{SGinc}.
  Then
  \begin{align*}
    E\left[\sup_{t\in T}X_t\right]\leq CK\int_{0}^{\diam (T,d)}\sqrt{\log N(T,d,\varepsilon)}d\varepsilon.
  \end{align*}
\end{thm}

We begin by considering a quenched estimate for the $\BL^{\kappa}$ metric.
\begin{prp}\label{dblbtm}
Set $r_n=n^{\theta},\;\theta\geq 0$ and suppose that $\kappa\in (1/2,1]$ and $\alpha>1$.
Then, for any $\zeta>\alpha^{-1}$, it holds that
\begin{align*}
  d_{\BL^{\kappa}}(\mu^{(r_n)},\mu_n^{(r_n)})
  \lesssim
  \begin{dcases}
    n^{-\frac{\alpha-1}{\alpha}}(\log n)^{1+\zeta}, & \theta=0,\;\alpha\in (1,2),\\
    n^{-\frac{\alpha-1}{\alpha}+\frac{2+\alpha}{2\alpha}\theta}(\log n)^{\zeta(1-\alpha/2)}, & \theta>0,\;\alpha\in (1,2),\\
    n^{-\frac{1}{2}}(\log n)^{1+\zeta},& \theta=0,\;\alpha=2,\\
    n^{-\frac{1}{2}+\theta}(\log n)^{\frac{1}{2}}, & \theta>0,\;\alpha=2,\\
    n^{-\frac{1}{2}}(\log n)^{\frac{1}{2}}, & \theta=0,\;\alpha\in (2,\infty),\\
    n^{-\frac{1}{2}+\theta}, & \theta>0,\;\alpha\in (2,\infty),
  \end{dcases}
  \qquad\qquad
  \bbP\mathrm{-a.s.}
\end{align*}
\end{prp}

\begin{proof}
  We define $\scrF_n:=\{f:[-r_n,r_n]\to\R:\|f\|_{\BL^{\kappa}}\leq 1\}$.
  Firstly, we have 
  \begin{equation}\label{iti}
    \begin{aligned}
      \MoveEqLeft d_{\BL^{\kappa}}(\mu^{(r_n)},\mu_n^{(r_n)})\\
    =&\sup_{f\in \scrF_n}\left| E[\tau_0]\int_{-r_n}^{r_n}f(y)dy-\frac{1}{n}\sum_{\substack{x\in\Z\\|\frac{x}{n}|\leq r_n}}\tau_x f(x/n) \right|\\
    \leq&\bbE[\tau_0]\sup_{f\in \scrF_n}\left| \int_{-r_n}^{r_n}f(y)dy-\frac{1}{n}\sum_{\substack{x\in\Z\\|\frac{x}{n}|\leq r_n}}f(x/n) \right|
    +\sup_{f\in \scrF_n}\left| \frac{1}{n}\sum_{\substack{x\in\Z\\|\frac{x}{n}|\leq r_n}}(\tau_x-E[\tau_0])f(x/n) \right|\\
    =&:\bbE[\tau_0]I+II.
    \end{aligned}
  \end{equation}
  The first term $I$ is bounded as follows.
  \begin{equation}\label{Iest}
    \begin{aligned}
      I
    =&\sup_{f\in \scrF_n}\left| \int_{-r_n}^{-\frac{\floor{nr_n}-1/2}{n}}f(y)dy+\sum_{x=-\floor{nr_n}+1}^{\floor{nr_n}-1}\int_{\frac{x-1/2}{n}}^{\frac{x+1/2}{n}}f(y)dy\right.\\
    &\left.\qquad\qquad\qquad\qquad\qquad\;\;\;\;\; + \int_{\frac{\floor{nr_n}-1/2}{n}}^{r_n}f(y)dy-\frac{1}{n}\sum_{x=-\floor{nr_n}}^{\floor{nr_n}} f(x/n) \right|\\
    \leq&\sup_{f\in \scrF_n}\left| \int_{-r_n}^{-\frac{\floor{nr_n}-1/2}{n}}f(y)dy+ \int_{\frac{\floor{nr_n}-1/2}{n}}^{r_n}f(y)dy-\frac{1}{n}f(\floor{nr_n}/n)-\frac{1}{n}f(-\floor{nr_n}/n) \right|\\
    &+\sup_{f\in \scrF_n}\sum_{x=-\floor{nr_n}+1}^{\floor{nr_n}-1}\int_{\frac{x-1/2}{n}}^{\frac{x+1/2}{n}}|f(y)-f(x/n)|dy\\
    \lesssim& n^{-1}+n^{-\kappa}r_n
    \end{aligned}
  \end{equation}
  We next deal with $II$. Take a sequence $b_n>1,\;n\in\N\cup\{0\}$ that is increasing and divergent satisfying $\sum_{n} b_n^{-\alpha}<\infty$, and define
  $\tau_{x,n}^{+}=\tau_x1_{\{\tau_x>b_{\floor{nr_n}}\}},\;\tau_{x,n}^{-}=\tau_x1_{\{\tau_x\le b_{\floor{nr_n}}\}}$,
  \begin{align}\label{anpm}
    A_n^{+}=\sup_{f\in \scrF_n}\left|\frac{1}{n}\sum_{\substack{x\in\Z\\|\frac{x}{n}|\leq r_n}}(\tau_{x,n}^{+}-\bbE[\tau_{0,n}^{+}])f(x/n)\right|,
    \;\;\rm{and}\;\;
    A_n^{-}=\sup_{f\in \scrF_n}\left|\frac{1}{n}\sum_{\substack{x\in\Z\\|\frac{x}{n}|\leq r_n}}(\tau_{x,n}^{-}-\bbE[\tau_{0,n}^{-}])f(x/n)\right|.
  \end{align}
  Then we have 
  \begin{align*}
    II\le A_n^{+}+A_n^{-}.
  \end{align*}

  We first show that 
  $A_n^{+}=O(r_n b_{\floor{nr_n}}^{-\alpha+1})$ almost surely.
  By $\sum_{n\ge 0}\bbP(\tau_n\vee \tau_{-n}>b_n)\leq 2\sum_{n}b_n^{-\alpha}<\infty$ and the 
  Borel-Cantelli lemma,
  it holds that $\bbP(\liminf \{\tau_n,\tau_{-n}\le b_n\})=1$. Thus, for almost all $\omega\in\Omega$,
  there exists $N(\omega)\in\N$ such that 
  \begin{align}\label{eb0}
    \tau_n(\omega)\le b_n,\qquad \forall |n|>N(\omega).
  \end{align}
  On the other hand, since $b_n$ diverges to $\infty$, there exists $N'(\omega)\in\N$ such that
  \begin{align}\label{eb1}
    \sup_{|n|\le N(\omega)}\tau_n(\omega)<b_{N'(\omega)}.
  \end{align}
  Thus, if $\floor{nr_n}>N(\omega)\vee N'(\omega)$, then it is the case that
  $\tau_x(\omega)\le b_{\floor{nr_n}}$ for $|x|=0,1,\dots,\floor{nr_n}$.
  In fact, if $|x|\leq N(\omega)$, we have 
  $\tau_x(\omega)\leq \sup_{|j|\le N(\omega)}<b_{N'(\omega)}\le b_{\floor{nr_n}}$
  from \eqref{eb1} and the monotonicity of $(b_n)$, and if $N(\omega)<|x|\le \floor{nr_n}$,
  it holds that $\tau_x(\omega)\le b_x\le b_{\floor{nr_n}}$ from \eqref{eb0}.
  Therefore, for sufficiently large $n$, $\tau_{x,n}^{+}=0$ for $|x|\le \floor{nr_n}$,
  and so
  we deduce that
  \begin{align}\label{anplus}
    A_n^{+}
    =\sup_{f\in \scrF_n}\left|\frac{1}{n}\sum_{\substack{x\in\Z\\|\frac{x}{n}|\leq r_n}}\bbE[\tau_{0,n}^{+}]f(x/n)\right|
    \lesssim r_n\int_{b_{\floor{nr_n}}}^{\infty}u^{-\alpha}du\lesssim r_n b_{\floor{nr_n}}^{-\alpha+1},
  \end{align}
  which completes the estimates for $A_n^{+}$.

  We next bound $A_n^{-}$.
  Set $T_{x,n}=\tau_{x,n}^{-}-\bbE[\tau^{-}_{0,n}]$ and $Z_n=nA_n^{-}=\sup_{f\in \scrF_n}\left|\sum_{\substack{x\in\Z\\|\frac{x}{n}|\leq r_n}}T_{x,n}f(x/n)\right|$.
  We begin by estimating $\bbE[Z_n]$.
  Let $(\varepsilon_x)_x$ be i.i.d random variables on $\Omega$ that satisfy $\bbP(\varepsilon_x=1)=\bbP(\varepsilon_x=-1)=1/2$, and are independent of $(\tau_{x})_x$.
  Then, by taking an independent copy $(T'_{x,n})_x$ of $(T_{x,n})_x$ on $\Omega$, we observe that  
 \begin{align*}
  \bbE[Z_n]
  =\bbE\left[ \sup_{f\in\scrF_n}\left|\sum_{|x/n|\leq r_n}(T_{x,n}-E[T_{x,n}'])f(x/n)\right| \right]
  \leq \bbE\left[ \sup_{f\in\scrF_n}\left|\sum_{|x/n|\leq r_n}(T_{x,n}-T_{x,n}')f(x/n)\right| \right].
 \end{align*}
 Since $(T_{x,n}-T'_{x,n})_x$ is a sequence of independent random variables with a symmetric distribution, we deduce that
 \begin{equation}\label{zn1}
  \begin{aligned}
    \bbE[Z_n]
  \leq &\bbE\left[ \sup_{f\in\scrF_n}\left|\sum_{|x/n|\leq r_n}(T_{x,n}-T_{x,n}')f(x/n)\right| \right]\\
  =&\bbE\left[ \sup_{f\in\scrF_n}\left|\sum_{|x/n|\leq r_n}\varepsilon_x(T_{x,n}-T_{x,n}')f(x/n)\right| \right]\\
  \leq& \bbE\left[ \sup_{f\in\scrF_n}\left|\sum_{|x/n|\leq r_n}\varepsilon_xT_{x,n}f(x/n)\right| \right]
  +\bbE\left[ \sup_{f\in\scrF_n}\left|\sum_{|x/n|\leq r_n}\varepsilon_xT_{x,n}'f(x/n)\right| \right]\\
  \lesssim&\bbE\left[ \sup_{f\in\scrF_n}\left|\sum_{|x/n|\leq r_n}\varepsilon_xT_{x,n}f(x/n)\right| \right]\\
  =&\bbE\left[ \bbE\left.\left[\sup_{f\in\scrF_n}\left|\sum_{|x/n|\leq r_n}\varepsilon_xt_{x,n}f(x/n)\right| \right]\right|_{t_{x,n}=T_{x,n}}\right].
  \end{aligned}
 \end{equation}
 Note that it follows from $f\in\scrF_n$ if and only if $-f\in\scrF_n$ that
 \begin{align}\label{abs}
   \bbE\left[\sup_{f\in\scrF_n}\left|\sum_{|x/n|\leq r_n}\varepsilon_xt_{x,n}f(x/n)\right| \right]
   = \bbE\left[\sup_{f\in\scrF_n}\sum_{|x/n|\leq r_n}\varepsilon_xt_{x,n}f(x/n)\right].
 \end{align}
 We define $Z^{\varepsilon}_f=\sum_{|x/n|\leq r_n}\varepsilon_xt_{x,n}f(x/n)$.
 Note that $Z^{\varepsilon}_f$ is a mean-zero random variable.
 Then it follows from
the Hoeffding inequality 
(see \cite[Theorem 2.8]{Con13})
 that
\begin{align*}
  \bbP(|Z^{\varepsilon}_f-Z^{\varepsilon}_g|>t)
  =\bbP\left( \left| \sum_{|x/n|\leq r_n}\varepsilon_xt_{x,n}(f(x/n)-g(x/n)) \right|>t \right)
  \leq 2\exp\left(-\frac{2t^2}{\sigma^2}\right),
\end{align*}
where, $\sigma^2:=\sum_{|x/n|\leq r_n}[t_{x,n}(f(x/n)-g(x/n))]^2$.
Therefore we deduce that
  \begin{align*}
    \bbE\left[\exp\left(\frac{|Z^{\varepsilon}_f-Z^{\varepsilon}_g|^2}{4\sigma^2}\right)\right]
  \leq& 1+\int_{1}^{\infty}\bbP\left(\exp\left(\frac{|Z^{\varepsilon}_f-Z^{\varepsilon}_g|^2}{4\sigma^2}\right)>u\right)du\\
  =&1+\int_1^{\infty}\bbP(|Z^{\varepsilon}_f-Z^{\varepsilon}_g|>2\sigma\sqrt{\log u})du\\
  \leq&1+2\int_{1}^{\infty}u^{-8}du\\
  =&1+\frac{2}{7}\\
  <&2.
  \end{align*}
This implies
\begin{align*}
  \inf\left\{K>0:E\left[\exp\left(\frac{|Z^{\varepsilon}_f-Z^{\varepsilon}_g|^2}{K^2}\right)\right]\le 2\right\}\le 2\sigma \le 2\left( \sum_{|x/n|\le r_n}t_{x,n}^2 \right)^{1/2}\|f-g\|_{\sup},
\end{align*}
and thus $(Z^{\varepsilon}_f)_{f\in\scrF_n}$ has sub-Gaussian increments as a stochastic process on $(\scrF_n,\|\cdot\|_{\sup})$ (recall the definition from \cref{SGincr}).
As a consequence of \cref{metent} and Dudley's integral inequality (recall the statement from \cref{dudint}),
we conclude that
\begin{equation*}\label{zn2}
  \begin{aligned}
     \bbE\left[\sup_{f\in\scrF_n}Z^{\varepsilon}_f\right]
  \lesssim &\left( \sum_{|x/n|\le r_n}t_{x,n}^2 \right)^{1/2}\int_{0}^{2} \sqrt{\log N(\scrF_n,\|\cdot\|_{\sup},a)}da\\
  \lesssim &\left( \sum_{|x/n|\le r_n}t_{x,n}^2 \right)^{1/2}r_n^{1/2}\int_{0}^{2} a^{-\frac{1}{2\kappa}}da\\
  \lesssim &\left( \sum_{|x/n|\le r_n}t_{x,n}^2 \right)^{1/2}r_n^{1/2}.
  \end{aligned}
\end{equation*}
Here we used $\kappa>1/2$.
Therefore, it follows from \eqref{zn1},\;\eqref{abs}, and \eqref{zn2} that
\begin{align}\label{EZn}
  \bbE[Z_n]\lesssim r_n^{1/2}\bbE\left[ \left( \sum_{|x/n|\le r_n}T_{x,n}^2 \right)^{1/2} \right]
  \leq r_n^{1/2} \bbE\left[\sum_{|x/n|\le r_n}T_{x,n}^2\right]^{1/2}\lesssim n^{1/2}r_n c_n^{1/2},
\end{align}
where $c_n$ is defined by 
\begin{align*}
  \bbE[T_{0,n}^2]
  \leq \bbE[(\tau_{0,n}^{-})^2]
  \lesssim \int_{1}^{b_{\floor{nr_n}}}u^{-\alpha+1}
  \lesssim c_n:=
  \begin{dcases}
    b_{\floor{nr_n}}^{2-\alpha}, & \alpha<2,\\
    \log b_{\floor{nr_n}}, & \alpha=2,\\
    1, & \alpha>2. 
  \end{dcases}
 \end{align*}
We next give a tail bound for $Z_n$ using \cite[Theorem 2.1]{Bou02}.
Set
\begin{align*}
  T^{\ast}_{x,n}=\frac{T_{x,n}}{b_{\floor{nr_n}}},\quad S=\frac{Z_n}{b_{\floor{nr_n}}},\quad S^{(-x)}=\sup_{f\in\scrF_n}\sum_{\substack{|y/n|\le r_n\\y\neq x}}T^{\ast}_{y,n}f(y/n),\;\;\rm{and}\;\; S_x=S-S^{(-x)}.
\end{align*}
for $|\frac{x}{n}|\le r_n$.
Note that $|T^{\ast}_{x,n}|\le 1$.
Our aim is to show the following:
\begin{itemize}
  \item [$\rm{(a)}$]$S_x=S-S^{(-x)}\le 1$,
  \item [$\rm{(b)}$]$\bbE[S_x|T_{y,n},y\neq x,|\frac{y}{n}\le r_n|]\ge 0$,
  \item [$\rm{(c)}$]$\sum_{|\frac{x}{n}|\le r_n}S_x\le S$,
  \item [$\rm{(d)}$]$\sum_{|\frac{x}{n}|\le r_n} \bbE[S_x^2|T_{y,n},y\neq x,|\frac{y}{n}|\le r_n]\leq \frac{1}{b_{\floor{nr_n}}^2}v_0$, where $v_0:=\sum_{|\frac{x}{n}|\le r_n} \bbE[T_{x,n}^2]$.
\end{itemize}
Note that 
\begin{align}\label{v0}
  v_0\lesssim nr_n c_n.
\end{align}

\begin{enumerate}[align=left, labelwidth=\widthof{Proof of $\rm{(a)}$}, 
    labelsep=0.5em, 
    leftmargin=!]
  \item [Proof of $\rm{(a)}$:]
  For any $\varepsilon>0$, there exists $f_{\varepsilon}\in\scrF_n$ satisfying 
$\sum_{|y/n|\le r_n}T^{\ast}_{y,n}f_{\varepsilon}(y/n)\ge S-\varepsilon$ and $S^{(-x)}\ge \sum_{y\neq x,|y/n|\le r_n}T^{\ast}_{y,n}f_{\varepsilon}(y/n)$.
Thus we deduce that
\begin{align}\label{fep}
  S-S^{(-x)}\le \varepsilon+T^{\ast}_{x,n}f_{\varepsilon}(x/n)\le \varepsilon+1.
\end{align}
Letting $\varepsilon\to 0$, we obtain the result.
\item [Proof of $\rm{(b)}$:]
Since $S=\sup_{f\in\scrF_n}\left[ \sum_{y\neq x}T^{\ast}_{y,n}f(y/n)+T^{\ast}_{x,n}f(x/n)\right]$ and $\bbE[T^{\ast}_{x,n}]=0$,
we deduce that
\begin{align*}
  \bbE[S|T_{y,n},y\neq x,|\frac{y}{n}|\le r_n]
  \ge& \sup_{f\in \scrF_n}\bbE\left[\left.  \sum_{y\neq x}T^{\ast}_{y,n}f(y/n)+T^{\ast}_{x,n}f(x/n)\right|T_{y,n},y\neq x,\left|\frac{y}{n}\right|\le r_n\right]\\
  =& \sup_{f\in \scrF_n}\left[\sum_{y\neq x}T^{\ast}_{y,n}f(y/n)+\bbE[T^{\ast}_{x,n}f(x/n)]\right]\\
  =&S^{(-x)},
\end{align*}
which shows $\rm{(b)}$.
\item[Proof of $\rm{(c)}$:]
Take the function $f_{\varepsilon}$ appearing in the proof of $\rm{(a)}$.
Then we have that
\begin{align*}
S^{(-x)}
=\sup_{f\in\scrF_n} \sum_{y\neq x}T^{\ast}_{y,n}f(y/n)
\ge\sum_{y\neq x}T^{\ast}_{y,n}f_{\varepsilon}(y/n)
\end{align*} 
and
\begin{align*}
\sum_{|y/n| \le r_n}T^{\ast}_{y,n}f_{\varepsilon}(y/n)
\ge S-\varepsilon.
\end{align*}
Now it follows from these observations that
\begin{align*}
S_x\le \left(\sum_{y}T^{\ast}_{y,n}f_{\varepsilon}(y/n) + \varepsilon \right)-\sum_{y\neq x} T^{\ast}_{y,n}f_{\varepsilon}(y/n)
=T^{\ast}_{x,n}f_{\varepsilon}(x/n)+\varepsilon.
\end{align*}
Therefore we deduce that 
\begin{align*}
  \sum_{x}S_x
  \le&\sum_{x}\left(T^{\ast}_{x,n}f_{\varepsilon}(x/n)+\varepsilon \right)\\
  =&\sum_{x}T^{\ast}_{x,n}f_{\varepsilon}(x/n)+ (2\floor{nr_n}+1)\varepsilon\\
  \le& S+(2\floor{nr_n}+1)\varepsilon,
\end{align*}
and letting $\varepsilon\to 0$, we complete the proof.
\item [Proof of $\rm{(d)}$:]
Since $T^{\ast}_{x,n}$ is independent of $(T_{y,n})_{y\neq x}$, it is enough to prove $|S_x|\leq |T^{\ast}_{x,n}|$.
The upper bound $S_x\le |T^{\ast}_{x,n}|$ is shown by the first inequality in \eqref{fep} 
and $\|f_{\varepsilon}\|_{\sup}\le 1$. We next show the lower bound $S_x\ge -|T^{\ast}_{x,n}|$.
For an arbitrary $\varepsilon>0$, there exists $g_{\varepsilon}\in\scrF_n$ satisfying
$S^{(-x)} \le \sum_{y \neq x} T^{\ast}_{y,n} g_{\varepsilon}(y/n)  + \varepsilon$
and
$S \ge  \sum_y T^{\ast}_{y,n} g_{\varepsilon}(y/n)$.
Thus we deduce that
\begin{align*}
S-S^{(-x)}\ge \sum_y T^{\ast}_{y,n}g_{\varepsilon}(y/n)-\left(\sum_{y \neq x}T^{\ast}_{y,n}g_{\varepsilon}(y/n)+\varepsilon\right)\ge-|T^{\ast}_{x,n} g_\varepsilon(x/n)|-\varepsilon\ge-|T^{\ast}_{x,n}|-\varepsilon, 
\end{align*}
and letting $\varepsilon\to 0$, we obtain the lower bound.
\end{enumerate}
By (a),\;(b),\;(c),\;(d) and \cite[Theorem 2.1]{Bou02}, it holds for $x\ge 0$ that
\begin{align*}
  \MoveEqLeft \bbP\left(Z_n\geq \bbE[Z_n]+\sqrt{(4b_{\floor{nr_n}}\bbE[Z_n]+2v_0)x}+\frac{1}{3}b_{\floor{nr_n}} x\right)\\
  =&\bbP\left(S\geq \bbE[S] +\sqrt{2(2\bbE[S]+b_{\floor{nr_n}}^{-2}v_0)x}+\frac{1}{3}x\right)\\
  \le&e^{-x}.
\end{align*}
Using $\sqrt{A+B}\le \sqrt{A}+\sqrt{B}$ and taking $x=2\log n$, we deduce that
\begin{align*}
  \bbP\left(Z_n\ge \bbE[Z_n]+2\sqrt{2b_{\floor{nr_n}} \bbE[Z_n]\log n}+2\sqrt{v_0\log n}+\frac{2}{3}b_{\floor{nr_n}}\log n\right)\le n^{-2}.
\end{align*}
Thus it follows from $2\sqrt{2b_{\floor{nr_n}} \bbE[Z_n]\log n}\le 2\bbE[Z_n]+b_{\floor{nr_n}}\log n$
that
\begin{align*}
  \bbP\left(Z_n\ge 3\bbE[Z_n]+2\sqrt{v_0\log n}+\frac{5}{3}b_{\floor{nr_n}}\log n\right)\le n^{-2}.
\end{align*}
By the Borel-Cantelli lemma, \eqref{EZn}, and \eqref{v0},
we conclude that
\begin{equation}\label{An-}
  \begin{aligned}
   A_n^{-}
  =&\frac{1}{n}Z_n\\
  \lesssim&\frac{1}{n}\left[\bbE[Z_n]+\sqrt{v_0\log n}+b_{\floor{nr_n}}\log n\right]
  \lesssim r_n\sqrt{\frac{c_n}{n}}+\sqrt{\frac{r_n c_n\log n}{n}}+\frac{b_{\floor{nr_n}}\log n}{n} 
  \end{aligned}
\end{equation}
for large $n$, almost surely.
 Combining \eqref{Iest},\;\eqref{anplus} and \eqref{An-}, we obtain
 \begin{align*}
  d_{\BL^{\kappa}}(\mu^{(r_n)},\mu_n^{(r_n)})\lesssim n^{-1}+n^{-\kappa}r_n+r_n b_{\floor{nr_n}}^{1-\alpha}+r_n\sqrt{\frac{c_n}{n}}+\sqrt{\frac{r_n c_n\log n}{n}}+\frac{b_{\floor{nr_n}}\log n}{n}
 \end{align*}
 for large $n$, almost surely. Taking $b_n=n^{\frac{1}{\alpha}}(\log n)^{\zeta}$ for an arbitrary $\zeta>\alpha^{-1}$, we complete the proof.
 \end{proof}

Define a measure $\nu_n$ on $n^{-1}\Z$ by $\nu_n=\frac{1}{n}\sum_{x\in\Z}\delta_{x/n}$ 
 and let $Y^n=(Y^n_t)$ be the process associated with the triplet $(n^{-1}\Z,d_E,\nu_n)$.
 Here, we write $d_E$ for the Euclidean metric.
 Also, we define $\sigma_A(f)$ for a function $f:[0,\infty)\to \R$ and a set $A\subseteq\R$ by setting
 \begin{align}\label{hitt}
  \sigma_A(f):=\inf\{t>0:f(t)\in A\}\in [0,\infty].
 \end{align}
 Similar to the hitting time, we abbreviate $\sigma_{\{x\}}(f)$ to $\sigma_x(f)$.
 The following estimates are essential to bound $\Psi_1$ in \cref{noncpt}.
 \begin{lem}
  It holds that
 \begin{align}\label{srwexp}
  \sup_{n\in\N}\sup_{\substack{x\in n^{-1}\Z\\r\leq 4nt}}P_x(\sigma_{B(x,r)^c}(Y^n)\le t)\le 4e^{-\frac{r^2}{8t}},
 \end{align}
 and 
  \begin{align}\label{btmrw}
  \sup_{n\in\N}\sup_{\substack{x\in n^{-1}\Z\\r\leq 4nt}}P_x^n(\sigma_{B(x,r)^c}(X^n)\le t)\le 4e^{-\frac{r^2}{8t}},\qquad \bbP\rm{-a.s.}
 \end{align}
 \end{lem}
 \begin{rmk}
Although an inequality similar to \eqref{srwexp} is established for each fixed $n$ in \cite[Lemma 5.21]{MR3616731},
it cannot be applied directly here because it does not provide uniform estimates with respect to $n$, 
and the ranges of $t$ and $r$ are limited. 
This lemma resolves these issues.
\end{rmk}
\begin{proof}
  We first prove \eqref{srwexp}.
  We may assume $x=0$.
  For a Poisson process $(N_t)_t$ with an intensity $2n^2$ and a discrete-time simple random walk $(S_j)_j$ on $\Z$,
  the process $Z^n_t:=\frac{1}{n}S_{N_t}$ has the same distribution as $Y^n$. Thus, using the reflection principle,
  we observe that 
  \begin{align*}
   P_0(\sigma_{B(0,r)^c}(Y^n)\le t)
   =P_0(\sigma_{B(0,r)^c}(Z^n)\le t)
   \leq 2 P_0\left( \sup_{s\leq t}Z^n_s\ge r \right)
   \le 4P_0(Z^n_t\ge r).
  \end{align*}
  Set $\theta=r/4t$. Then it follows from $\theta/n\le 1$ that $\cosh(\theta/n)-1=\frac{e^{\theta/n}+e^{-\theta/n}}{2}-1\le (\theta/n)^2$.
  By virtue of this inequality, we deduce that
  \begin{align*}
    P_0(Z^n_t\ge r)
    =&\sum_{k\ge 0}P_0\left(\frac{1}{n}S_k\ge r\right)P(N_t=k)\\
    =&\sum_{k\ge 0}P_0\left(e^{\frac{\theta}{n}S_k}\ge e^{\theta r}\right)\frac{1}{k!}(2n^2t)^ke^{-2n^2 t}\\
    \le&\sum_{k\ge 0}e^{-\theta r}E_0[e^{\frac{\theta}{n}S_k}]\frac{1}{k!}(2n^2t)^ke^{-2n^2 t}\\
    =&\sum_{k\ge 0}e^{-\theta r}  (\cosh(\theta/n))^k  \frac{1}{k!}(2n^2t)^ke^{-2n^2 t}\\
    =&\exp\left(  -\theta r+2n^2 t(\cosh(\theta/n)-1)  \right)\\
    \leq&\exp\left(  -\theta r+2n^2 t(\theta/n)^2\right)\\
    =&\exp\left(-\frac{r^2}{8t}\right),
  \end{align*}
  which completes the proof of \eqref{srwexp}.

  We next show \eqref{btmrw}.
  From \eqref{srwexp}, it suffices to prove $\sigma_{B(x,r)^c}(Y^n)\le \sigma_{B(x,r)^c}(X^n)$ $\bbP$-almost surely.
  We may assume $x=0$.
  Let $(L(y,t))_{y\in\R,t\ge 0}$ be the local time of a Brownian motion $B=(B_t)_t$ on $\R$ and set 
  $A^{\nu_n}_t=\int_{\R}L(y,t)\nu_n(dy)$ and $A^{\mu_n}_t=\int_{\R}L(y,t)\mu_n(dy)$.
  Then we have $A^{\nu_n}_t\le A^{\mu_n}_t$ since $\tau_y\ge 1$ $\bbP$-almost surely.
  Moreover, we may construct the processes $X^n$ and $Y^n$ by $B_{( A^{\mu_n})^{-1}_t}$ and $B_{(A^{\nu_n})_t^{-1}}$, respectively
  (see \cite[Theorem 3.34]{Nd25}).
  Here, we write $(A^{\mu_n})^{-1}_t$ for the left inverse.
  Now, if we set $r_n:=\min\{a\in n^{-1}\Z:a\ge r\}$, then it holds that
  $\sigma_{B(0,r)^c}(X^n)=\sigma_{\{\pm r_n\}}(X^n)=\sigma_{-r_n}(X^n)\wedge \sigma_{r_n}(X^n)$.
  And, it is not difficult to check that $\sigma_{r_n}(X^n)=A^{\mu_n}_{\sigma_{r_n}(B)}$.
  Combining these observations, we deduce that $\sigma_{B(0,r)^c}(Y^n)\le \sigma_{B(0,r)^c}(X^n)$,
  and the result follows from \eqref{srwexp}.
\end{proof}

To complete the proof of \cref{btmrate}, 
we finally need the following uniform estimate for the measures of short intervals.
\begin{lem}\label{taihenn}
  Set $r_n=n^{\theta},\theta>0$.
  Let $\varepsilon_n>0$ be a sequence satisfying $\varepsilon_n\to 0$ and $n\varepsilon_n\to \infty$.
  Then, for an arbitrary $\delta>0$, we have
  \begin{align*}
    \sup_{x\in \R}\mu_n^{(r_n)}((x-\varepsilon_n,x+\varepsilon_n))\lesssim \varepsilon_n+\varepsilon_n n^{\left(\frac{1+\theta}{\alpha}+\delta\right)(1-\alpha)}+\sqrt{\frac{\varepsilon_n v_n\log n}{n}}+n^{\frac{1+\theta}{\alpha}-1+\delta}\log n,\qquad \bbP\rm{-a.s.}
  \end{align*}
  where
  \begin{align*}
    v_n=
    \begin{dcases}
      n^{\left(\frac{1+\theta}{\alpha}+\delta\right)(2-\alpha)}, & \alpha\in (1,2),\\
      \log n, & \alpha=2,\\
      1, & \alpha>2. 
    \end{dcases}
  \end{align*}
\end{lem}
\begin{proof}
  It suffices to consider
  $\sup_{|x|\le r_n}\mu_n^{(r_n)}((x-\varepsilon_n,x+\varepsilon_n))$.
  Then we have
  \begin{align*}
    \sup_{|x|\le r_n}\mu_n^{(r_n)}((x-\varepsilon_n,x+\varepsilon_n))
    \leq \max_{|j|\le \ceil{nr_n}}\frac{1}{n}\sum_{i=j}^{j+\ceil{2n\varepsilon_n}}\tau_i.
  \end{align*}
  Take $h>0$ arbitrarily. 
  By $\sum_{i\ge 0}\bbP(\tau_i\vee\tau_{-i}>i^{\frac{1}{\alpha}+h})<\infty$ and the Borel-Cantelli lemma,
  there exists a random variable $I_0:\Omega\to \N$ such that $\tau_i,\tau_{-i}\le i^{\frac{1}{\alpha}+h}$ for $i>I_0$.
  Thus, if $n$ is large enough so that $\ceil{nr_n}>I_0$, we have 
  \begin{align}\label{sore}
    \max_{|i|\le \ceil{2nr_n}}\tau_i= 
    \left(\max_{|i|\le I_0}\tau_i\right)\vee \ceil{2nr_n}^{\frac{1}{\alpha}+h}\lesssim n^{\frac{1+\theta}{\alpha}+(1+\theta)h}.
  \end{align} 
  In what follows, we consider such large $n$.
  We define $\delta=(1+\theta)h$. Note that $\delta$ can be arbitrarily small.
  Therefore, 
it holds that
\begin{align}\label{tosore}
  \max_{|j|\le \ceil{nr_n}}\frac{1}{n}\sum_{i=j}^{j+\ceil{2n\varepsilon_n}}\tau_i
  =\max_{|j|\le \ceil{nr_n}}\frac{1}{n}\sum_{i=j}^{j+\ceil{2n\varepsilon_n}}\tau_i1_{\{\tau_i\le n^{\frac{1+\theta}{\alpha}+\delta}\}}.
\end{align}
We write $a_n$ for $n^{\frac{1+\theta}{\alpha}+\delta}$.
Since 
\begin{align*}
  0\le \bbE[\tau_0]-\bbE\left[\tau_01_{\{\tau_i\le a_n\}}\right]
 =\int_{a_n}^{\infty}u^{-\alpha}du+a_n\bbP(\tau_0>a_n)
=\frac{\alpha}{\alpha-1}a_n^{1-\alpha},
\end{align*}
it holds that
\begin{align*}
  \frac{1}{n}\sum_{i=j}^{j+\ceil{2n\varepsilon_n}}\bbE\left[\tau_i1_{\{\tau_i\le a_n\}}\right]
  =\frac{\ceil{2n\varepsilon_n}+1}{n}\bbE[\tau_0]-\frac{\ceil{2n\varepsilon_n}+1}{n}\frac{\alpha}{\alpha-1}a_n^{1-\alpha}
  \lesssim \varepsilon_n+\varepsilon_n a_n^{1-\alpha},
\end{align*}
and thus
\begin{equation}\label{kore}
  \begin{aligned}
    \MoveEqLeft \max_{|j|\le \ceil{nr_n}}\frac{1}{n}\sum_{i=j}^{j+\ceil{2n\varepsilon_n}}\tau_i1_{\{\tau_i\le n^{\frac{1+\theta}{\alpha}+\delta}\}}\\
  \lesssim&\frac{1}{n}\max_{|j|\le \ceil{nr_n}}\left|\sum_{i=j}^{j+\ceil{2n\varepsilon_n}}\tau_i1_{\{\tau_i\le a_n\}}-E\left[\tau_i1_{\{\tau_i\le a_n\}}\right]\right| +\varepsilon_n+\varepsilon_n a_n^{1-\alpha}.
  \end{aligned}
\end{equation}
Define $Y_{i,n}:=\tau_i1_{\{\tau_i\le a_n\}}-\bbE\left[\tau_i1_{\{\tau_i\le a_n\}}\right]$.
Note that $(Y_{i,n})_i$ is a family of independent random variables satisfying $|Y_{i,n}|\le a_n$.
By virtue of Bernstein's inequality (see \cite[Lemma 2.2.9]{VWJ96}),
we deduce that
\begin{align*}
  \bbP\left( \max_{|j|\le \ceil{nr_n}}\left|\sum_{i=j}^{j+\ceil{2n\varepsilon_n}}Y_{i,n}\right| >t\right)
  \leq&\sum_{|j|\le \ceil{nr_n}}\bbP\left(\left|\sum_{i=j}^{j+\ceil{2n\varepsilon_n}}Y_{i,n}\right| >t\right)\\
  \lesssim&nr_n \max_{|j|\le \ceil{nr_n}}\bbP\left(\left|\sum_{i=j}^{j+\ceil{2n\varepsilon_n}}Y_{i,n}\right| >t\right)\\
  \lesssim&n^{1+\theta}\exp\left( -c\frac{t^2}{n\varepsilon_n v_n+a_nt} \right),
\end{align*}
where $v_n$ is defined by
\begin{align*}
  \bbE[Y_{i,n}^2]\leq \bbE\left[ \left(\tau_01_{\{\tau_0\le a_n\}}\right)^2 \right]
  \lesssim v_n:=
  \begin{dcases}
    a_n^{2-\alpha}, & \alpha<2,\\
    \log n, & \alpha=2,\\
    1, & \alpha>2. 
  \end{dcases}
\end{align*}
Taking $t=t_n:=M(\sqrt{n\varepsilon_n v_n\log n}+a_n \log n)$ for sufficiently large $M>0$,
it follows that
\begin{align*}
  \bbP\left( \max_{|j|\le \ceil{nr_n}}\left|\sum_{i=j}^{j+\ceil{2n\varepsilon_n}}Y_{i,n}\right| >t_n\right)\le n^{-2}.
\end{align*}
Therefore we deduce that
\begin{align}\label{tokore}
  \max_{|j|\le \ceil{nr_n}}\left|\sum_{i=j}^{j+\ceil{2n\varepsilon_n}}Y_{i,n}\right|
  \leq \sqrt{n\varepsilon_n v_n\log n}+a_n \log n
\end{align}
for sufficiently large $n$ by the Borel-Cantelli lemma.
Combining \eqref{sore},\;\eqref{tosore},\;\eqref{kore} and \eqref{tokore}, we obtain the result.
\end{proof}

We are now ready to prove \cref{btmrate} and \cref{btmsgQ}.

\begin{proof}[Proof of \cref{btmrate}]
  We first show \eqref{distfunc}, and
  denote $X^{(r)}$ and $X^{n,(r)}$ for the process associated with
  $(F^{(r)},\mu^{(r)})$ and $(F^{(r)}_n,\mu_n^{(r)})$, respectively.
  We may assume $t=1$. In the following proof, the constant in $\lesssim$ does not depend on $x$.
  Note that we may define $X^{(r)}$ and $X$ (resp. $X^{n,(r)}$ and $X^n$)
  on a same probability space. 
  More precisely, we define $X^{(r)}$ to be the time change of $X$ by $\mu^{(r)}$, 
  and define $X^{n,(r)}$ analogously.
  Take sufficiently small $\theta\in (0,1)$ and set $r_n=n^{\theta}$.
  To begin with, we observe
  \begin{equation}\label{zxc}
    \begin{aligned}
  \MoveEqLeft |P_0(X_1\leq x)-P_0^n(X_1^n\le x)|\\
    \leq&|P_0(X_1\leq x)-P_0(X_1^{(r_n)}\le x)|
    +|P_0(X_1^{(r_n)}\leq x)-P_0^n(X_1^{n,(r_n)}\le x)|\\
    &+|P_0^n(X_1^{(r_n)}\leq x)-P_0^n(X_1^n\le x)|.    
    \end{aligned}
  \end{equation}
  Since $X$ is a scaled Brownian motion, and we have \eqref{btmrw},
  we deduce that 
  \begin{align*}
    |P_0(X_1\leq x)-P_0(X_1^{(r_n)}\le x)|
    =&|P_0(X_1\leq x,\sigma_{B(0,r_n)^c}(X)>1)+P_0(X_1\leq x,\sigma_{B(0,r_n)^c}(X)\le 1)\\
    &-P_0(X_1^{(r_n)}\le x,\sigma_{B(0,r_n)^c}(X)>1)-P_0(X_1^{(r_n)}\le x,\sigma_{B(0,r_n)^c}(X)\le 1)|\\
    =&|P_0(X_1\leq x,\sigma_{B(0,r_n)^c}(X)\le 1)-P_0(X_1^{(r_n)}\le x,\sigma_{B(0,r_n)^c}(X)\le 1)|\\
    \le&P_0(\sigma_{B(0,r_n)^c}(X)\le 1)\\
    \lesssim&\exp\left(-cn^{2\theta}\right),
  \end{align*}
  and 
  \begin{align*}
    |P_0^n(X_1^{n,(r_n)}\leq x)-P_0^n(X_1^n\le x)|
    \leq P_0(\sigma_{B(0,r_n)^c}(X^n)\le 1)\lesssim \exp\left(-cn^{2\theta}\right).
  \end{align*}
  It remains to estimate the second term of \eqref{zxc}, $|P_0(X_1^{(r_n)}\leq x)-P_0^n(X_1^{n,(r_n)}\le x)|$.
  Let $\varepsilon_n>0$ be a sequence satisfying $\varepsilon_n\to 0$ and $n\varepsilon_n\to\infty$.
  Take a smooth function $\varphi:\R\to [0,1]$ with $\varphi|_{(-\infty,-1)}=1$ and $\varphi|_{(1,\infty)}=0$ 
  and set $g_n(y)=\varphi(\frac{y-x}{\varepsilon_n}),\;y\in\R$. 
  Note that $g_n$ satisfies $0\le g_n\le 1,\;g_n|_{(-\infty,x-\varepsilon_n)}=1,$ 
  and $g_n|_{(x+\varepsilon_n,\infty)}=0$.
  Then we have 
  \begin{align*}
    \MoveEqLeft |P_0(X_1^{(r_n)}\leq x)-P_0^n(X_1^{n,(r_n)}\le x)|\\
    \leq&|P_1^{(r_n)}1_{(-\infty,x]}(0)-P_1^{(r_n)}g_n(0)|
    +|P_1^{(r_n)}g_n(0)-P_1^{n,(r_n)}g_n(0)|
    +|P_1^{n,(r_n)}g_n(0)-P_1^{(r_n)}1_{(-\infty,x]}(0)|,
  \end{align*} 
 and the terms, excluding the middle one, are bounded as follows:
 \begin{equation*}
  |P_1^{(r_n)}1_{(-\infty,x]}(0)-P_1^{(r_n)}g_n(0)|
  =\left|\int_{(x-\varepsilon_n,x+\varepsilon_n)} (1_{(-\infty,x]}-g_n)(y)p^{(r_n)}(1,0,y)\mu(dy) \right| 
  \lesssim \varepsilon_n, 
 \end{equation*}
 \begin{equation*}
  |P_1^{n,(r_n)}g_n(0)-P_1^{(r_n)}1_{(-\infty,x]}(0)|
  \lesssim \mu_n((x-\varepsilon_n,x+\varepsilon_n)).
 \end{equation*}
 Here we used the bound $\sup_n\sup_y p_n^{(r_n)}(1,0,y)<\infty$, 
 which follows from $\frac{1}{4}a\le \mu^{(r_n)}_n ((y-a,y+a))$ and \cref{summ}(ii).
 For $q>\alpha^{-1},\;\alpha>1,$ and $a>0$, set
 \begin{equation}\label{mnalpha}
  \begin{aligned}
    M_n(\alpha,a)=
\begin{dcases}
a+a^{1/\alpha}n^{1/\alpha-1}(\log(na))^{1/\alpha}(\log\log(na))^q, & \alpha\in (1,2),\\
a+a^{1/2}n^{-1/2}(\log(na))^{1/2}(\log\log(na))^q, & \alpha=2,\\
a+a^{1/2}n^{-1/2}(\log\log(na))^{1/2}, & \alpha>2.
\end{dcases}
  \end{aligned}
 \end{equation}
Then it follows from \cite[Theorem 2.1]{BTM24} and \cref{taihenn} that 
\begin{align*}
  \mu_n(F_n^{(r_n)})\lesssim M_n(\alpha,r_n),\qquad \mu_n((x-\varepsilon_n,x+\varepsilon_n))\lesssim \varepsilon_n+a_n,
\end{align*}
where $a_n=\varepsilon_n n^{\left(\frac{1+\theta}{\alpha}+\delta\right)(1-\alpha)}+\sqrt{\frac{\varepsilon_n v_n\log n}{n}}+n^{\frac{1+\theta}{\alpha}-1+\delta}\log n$.
For the middle term, by \cref{sg2}, it holds that
\begin{align*}
  |P_1^{(r_n)}g_n(0)-P_1^{n,(r_n)}g_n(0)|
  \lesssim c(n)\left( \varepsilon^{-4}n^{-\kappa}+\varepsilon^{-(4+\kappa)}d_{\BL^{\kappa}}(\mu^{(r_n)},\mu_n^{(r_n)})\right),
\end{align*}
where $c(n)=(r_nM_n(\alpha,r_n))^{\frac{9}{2}+5\kappa}r_n^{1/2}=n^{5(1+\kappa)\theta}M_n(\alpha,r_n)^{\frac{9}{2}+5\kappa}$ (see \eqref{5.4}).
Now, since $M_n(\alpha,r_n)\lesssim r_n$, we have $c(n)\lesssim n^{10(1+\kappa)\theta}$.
Therefore, we conclude that
\begin{align*}
  \MoveEqLeft \sup_{x\in\R}|P_0(X_1\leq x)-P_0^n(X_1^n\le x)|\\
  \lesssim& 
  \exp(-cn^{2\theta})+\varepsilon_n+a_n+n^{10(1+\kappa)\theta}\left( \varepsilon_n^{-4}n^{-\kappa}+\varepsilon_n^{-(4+\kappa)}d_{\BL^{\kappa}}(\mu^{(r_n)},\mu_n^{(r_n)})\right)
\end{align*}
Substituting $\varepsilon_n=(n^{10(1+\kappa)\theta}d_{\BL^{\kappa}}(\mu^{(r_n),\mu_n^{(r_n)}}))^{\frac{1}{5+\kappa}}$
 and taking sufficiently small $\theta,\delta$, we may verify that
 $\exp(-cn^{2\theta}),\;a_n,\;n^{10(1+\kappa)\theta-\kappa}\varepsilon_n^{-4}=o(\varepsilon_n)$.
 Thus it holds that
 \begin{align*}
  \sup_{x\in\R}|P_0(X_1\leq x)-P_0^n(X_1^n\le x)|
  \lesssim& (n^{10(1+\kappa)\theta}d_{\BL^{\kappa}}(\mu^{(r_n),\mu_n^{(r_n)}}))^{\frac{1}{5+\kappa}}\\
  \lesssim&
  \begin{dcases}
    n^{-\frac{\alpha-1}{\alpha}\frac{1}{5+\kappa}+\left[ \frac{2+\alpha}{2\alpha}+10(1+\kappa) \right]\frac{\theta}{5+\kappa}}(\log n)^{\frac{p(2-\alpha)}{2(5+\kappa)}}, & \alpha\in (1,2),\\
    n^{-\frac{1}{2(5+\kappa)}+\left[11+10\kappa\right]\frac{\theta}{5+\kappa}}(\log n)^{\frac{1}{2(5+\kappa)}}, & \alpha=2,\\
    n^{-\frac{1}{2(5+\kappa)}+\left[11+10\kappa\right]\frac{\theta}{5+\kappa}}, & \alpha\in (2,\infty),
  \end{dcases}
 \end{align*}
 for an arbitrary $p>1/\alpha$.
 Since $\theta>0$ can be taken sufficiently small, we deduce that 
 \begin{align*}
  \sup_{x\in\R}|P_0(X_t\leq x)-P_0^n(X_t^n\le x)|
  \lesssim n^{-E}
 \end{align*}
 for $E<\frac{1-\alpha^{-1}}{5+\kappa}\wedge \frac{1}{2(5+\kappa)}$.
 Taking $\kappa$ arbitrarily close to $1/2$, we complete the proof of \eqref{distfunc}.
 The second assertion \eqref{hkbtm} is shown by \cite[Theorem 4.8]{BTM24}
 and an argument similar to the proof of \cite[Theorem 1.7]{BTM24}.
\end{proof}

\begin{proof}[Proof of \cref{btmsgQ}]
  Set $r_n=n^{\theta}$ for sufficiently small $\theta\in (0,1)$.
  Note that, since $(F^{(r)},d_E,\mu^{(r)})$ satisfies the assumption $\rm{(A3)}$,
  we may use \cref{gen}(ii) instead of (i) in the proof of \cref{noncpt}.
  Then, by the argument in \cref{noncpt} and \eqref{btmrw},
  we have that
  \begin{align*}
    \frac{1}{\|f\|_{\BL^{\kappa}}}\sup_{|x|\le K}\Big|E_{x}[f(X_t)]-E_{\frac{\floor{nx}}{n}}^n[f(X_t^n)]\Big|\lesssim \exp(-cn^{2\theta}/t)+C_n(r_n)d_{\BL^{\kappa}}(\mu^{(r_n)},\mu_n^{(r_n)})^{E_1}(\log n)^{E_2},
  \end{align*}
  where $E_1=\frac{\kappa}{3\kappa+4}$ and $E_2=\frac{\kappa (2+\kappa)}{3\kappa+4}$.
  Here, the constant $C_n(r_n)$ is bounded above by $n^{A\theta}$ for some $A>0$,
  by the formula for $\g{C}{gen,n}$ in \cref{gen}(ii) and the fact $M_n(\alpha,r_n)\lesssim r_n$ (see \eqref{mnalpha} for the definition).
  Therefore, we may take $\theta>0$ arbitrarily small, and obtain the result from 
  \cref{dblbtm}.
\end{proof}

\subsubsection{Annealed estimates}
In this section, we prove the following annealed estimates.
\begin{thm}\label{btmsgA}\quad\par
  Suppose that $\alpha$ and $\kappa\in (1/2,1]$ satisfy $\alpha>\frac{\kappa}{3\kappa+4}+5\kappa+\frac{7}{2}$.
  Then it holds for any $f\in\BL^{\kappa},K\in(0,\infty),K\in (0,\infty)$ and $E<\frac{\kappa}{2(3\kappa+4)}$ that
  \begin{align*}
    \bbE\left[\sup_{|x|\le K}\Big|E_{x}[f(X_t)]-E_{\frac{\floor{nx}}{n}}^n[f(X_t^n)]\Big|\right]\lesssim \|f\|_{\BL^{\kappa}}n^{-E}.
  \end{align*}
\end{thm}
\begin{thm}\label{annhk}
  Suppose that $\alpha$ and $\kappa\in (1/2,1]$ satisfy $\alpha>\frac{\kappa}{3\kappa+4}+5\kappa+\frac{7}{2}$.
  Then it holds for any $K\in (0,\infty)$ and $E<\frac{\kappa}{2(3\kappa+4)(\kappa+2)}$ that
  \begin{align*}
    \bbE\left[\sup_{|x|\le K}\Big|p(t,0,x)-p_n(t,0,\floor{nx}/n)\Big|\right]\lesssim n^{-E}.
  \end{align*}
  In particular, if $\alpha>\frac{121}{14}$, we have 
  \begin{align*}
     \bbE\left[\sup_{|x|\le K}\Big|p(t,0,x)-p_n(t,0,\floor{nx}/n)\Big|\right]\lesssim n^{-E},
  \end{align*}
  for any $K\in (0,\infty)$ and $E<\frac{1}{42}$.
 \end{thm}
 
\begin{rmk}
  The restriction on the range of $\alpha$ arises from the integrability condition of $C_n(r)$ in \cref{noncpt}.
Therefore, it is worthwhile to weaken the integrability requirement of $C_n(r)$.
\end{rmk}

Similarly as in quenched estimates, we begin by 
considering an annealed estimate for $\BL^{\kappa}$ metric.
 \begin{prp}\label{blp}
Set $r_n=n^{\theta},\;\theta\geq 0$ and suppose that $\kappa\in (1/2,1]$ and $\alpha>1$.
Then, for any $\zeta\in (0,\alpha)$,
it holds that
\begin{align*}
  \bbE[d_{\BL^{\kappa}}(\mu^{(r_n)},\mu_n^{(r_n)})^{\zeta}]
  \lesssim  
  \begin{dcases}
    n^{-\zeta(1-\frac{1}{\alpha}-\theta)}, & \alpha\in (1,2),\\
    n^{-\zeta(\frac{1}{2}-\theta)}(\log n)^{\zeta/2}, & \alpha=2,\\
    n^{-\zeta(\frac{1}{2}-\theta)}, & \alpha\in (2,\infty).
    \end{dcases}  
 \end{align*}
 \end{prp}

\begin{proof}
  Take a sequence $d_n>1$ and set 
  $\tau_{x,n}^{+}=\tau_x1_{\{\tau_x>d_n\}},\;\tau_{x,n}^{-}=\tau_x1_{\{\tau_x\le d_n\}},$
  \begin{equation*}
  I:=\sup_{f\in \scrF_n}\left| \int_{-r_n}^{r_n}f(y)dy-\frac{1}{n}\sum_{\substack{x\in\Z\\|\frac{x}{n}|\leq r_n}}f(x/n) \right|,
  \end{equation*}
  \begin{equation*}
    A_n^{+}:=\sup_{f\in \scrF_n}\left|\frac{1}{n}\sum_{\substack{x\in\Z\\|\frac{x}{n}|\leq r_n}}(\tau_{x,n}^{+}-\bbE[\tau_{0,n}^{+}])f(x/n)\right|,
    \;\;\rm{and}\;\;
    A_n^{-}:=\sup_{f\in \scrF_n}\left|\frac{1}{n}\sum_{\substack{x\in\Z\\|\frac{x}{n}|\leq r_n}}(\tau_{x,n}^{-}-\bbE[\tau_{0,n}^{-}])f(x/n)\right|.
  \end{equation*}
  In the following proof, we write $A\lesssim B$ if there exists a deterministic constant $C\in (0,\infty)$ such that 
  $A\leq CB$.
  Then we have 
$d_{\BL^{\kappa}}(\mu^{(r_n)},\mu_n^{(r_n)})^{\zeta}\lesssim I^{\zeta} +(A_n^{+})^{\zeta}+(A_n^{-})^{\zeta}$
  and $I^{\zeta}\lesssim n^{-\zeta(\kappa-\theta)}$ from \eqref{Iest}.
  We first estimate $\bbE[(A_n^{+})^{\zeta}]$.
  Note that it holds that
  \begin{align*}
    0\le A_n^{+} 
    \lesssim\frac{1}{n}\sum_{\substack{x\in\Z\\|\frac{x}{n}|\le r_n}}(\tau_{x,n}^{+}-\bbE[\tau_{0,n}^{+}]+\bbE[\tau_{0,n}^{+}])+r_n\bbE[\tau_{0,n}^{+}]
    \lesssim\frac{1}{n}\sum_{\substack{x\in\Z\\|\frac{x}{n}|\le r_n}}(\tau_{x,n}^{+}-\bbE[\tau_{0,n}^{+}])+r_n\bbE[\tau_{0,n}^{+}].
  \end{align*} 
  Since $(\tau_{x,n}^{+}-\bbE[\tau_{0,n}^{+}])_x$ is i.i.d. and with mean $0$,
  it follows from
  a corollary of the Marcinkiewicz-Zygmund inequality given at
  \cite[Corollary 8.2]{Gut} that
  \begin{align*}
    \bbE[(A_n^{+})^{\zeta}]
    \lesssim&n^{-\zeta}\bbE\left[  \left| \sum_{\substack{x\in\Z\\|\frac{x}{n}|\le r_n}}(\tau_{x,n}^{+}-\bbE[\tau_{0,n}^{+}]) \right|^\zeta  \right]+r_n^{\zeta}\bbE[\tau_{0,n}^{+}]^{\zeta}\\
    \lesssim&
    n^{-\zeta}(nr_n)^{1\vee\frac{\zeta}{2}} \bbE[|\tau_{0,n}^{+}|^{\zeta}]+r_n^{\zeta}\bbE[\tau_{0,n}^{+}]^{\zeta}\\
    \lesssim&
      n^{-\zeta+(1+\theta)(1\vee\frac{p}{2})}d_n^{-\alpha+\zeta}+n^{\theta \zeta}d_n^{-\zeta(\alpha-1)},
  \end{align*}
  for $\zeta\ge 1$.
  Here we used
  \begin{align*}
    \bbE[|\tau_{0,n}^{+}|^{\zeta}]\lesssim
    \int_{d_n}^{\infty}u^{-\alpha+\zeta-1}\lesssim d_n^{-\alpha+\zeta}.
  \end{align*}
  If $\zeta\in (0,1)$, using Jensen's inequality and the estimate for $p=1$, we obtain that
  \begin{align*}
    \bbE[(A_n^{+})^{\zeta}]\le \bbE[A_n^{+}]^{\zeta}\le n^{\theta \zeta}d_n^{-\zeta(\alpha-1)}.
  \end{align*}
  We next bound $\bbE[(A_n^{-})^{\zeta}]$.
  For $t_{x,n}\in\R$, we define $T_{x,n}$ and $Z_{f}^{\varepsilon}$ by setting
  \begin{align*}
    T_{x,n}=\tau_{x,n}^{-}-\bbE[\tau^{-}_{0,n}],\qquad Z^{\varepsilon}_f=\sum_{|x/n|\leq r_n}\varepsilon_xt_{x,n}f(x/n),
  \end{align*}
where $(\varepsilon_x)_x$ are i.i.d random variables on $\Omega$, which satisfy $\bbP(\varepsilon_x=1)=\bbP(\varepsilon_x=-1)=1/2$, and are independent of $(\tau_{x})_x$.
 Then arguing as in \eqref{zn1}, we deduce that
  \begin{align*}
    \bbE[(A_n^{-})^{\zeta}]
    =&n^{-\zeta}\bbE[Z_n^{\zeta}]\\
    \lesssim&n^{-\zeta} \bbE\left[ \bbE\left.\left[\sup_{f\in\scrF_n}\left|\sum_{|x/n|\leq r_n}\varepsilon_xt_{x,n}f(x/n)\right|^{\zeta} \right]\right|_{t_{x,n}=T_{x,n}}\right]\\
    =&n^{-\zeta} \bbE\left[ \bbE\left.\left[\left(\sup_{f\in\scrF_n}Z_f^{\varepsilon}\right)^{\zeta} \right]\right|_{t_{x,n}=T_{x,n}}\right].
  \end{align*}
   Set 
\begin{align*}
  M:=\int_{0}^{\infty}\sqrt{\log N\left(\scrF_n,t_{\ast}\|\cdot\|_{\rm{sup}},\varepsilon\right)}d\varepsilon,
  \qquad t_{\ast}:=\left(\sum_{|\frac{x}{n}|\le r_n}t_{x,n}^2\right)^{1/2}.
\end{align*}
Then, by  \cite[Corollary 3.2]{subgauss}, there exists a universal constant $C\in (0,\infty)$ such that
\begin{align*}
  \bbP\left(\sup_{f\in\scrF_n}Z_f^{\varepsilon}\ge CMs\right)\leq 2e^{-\frac{s^2}{2}},\qquad s\ge 1.
\end{align*} 
Therefore we conclude that
\begin{align*}
  \bbE\left[\left(\sup_{f\in\scrF_n}\sum_{|x/n|\leq r_n}Z_f^{\varepsilon}\right)^{\zeta} \right]
  =&\int_{0}^{CM}\zeta u^{\zeta-1}\bbP\left(\sup_{f\in\scrF_n}Z_f^{\varepsilon}\ge u\right)du
  +\int_{CM}^{\infty}\zeta u^{\zeta-1}\bbP\left(\sup_{f\in\scrF_n}Z_f^{\varepsilon}\ge u\right)du\\
  \lesssim&M^{\zeta}+\int_{1}^{\infty}(Ms)^{\zeta-1}\bbP\left(\sup_{f\in\scrF_n}Z_f^{\varepsilon}\ge CMs\right)Mds\\
  \lesssim&M^{\zeta}+M^{\zeta}\int_{0}^{\infty}s^{\zeta-1}e^{-\frac{s^2}{2}}ds\\
  \lesssim&M^{\zeta}.
\end{align*}
Moreover, by virtue of \cref{metent}, it holds that
\begin{align*}
  M=&\int_{0}^{2t_{\ast}}\sqrt{\log N\left(\scrF_n,t_{\ast}\|\cdot\|_{\rm{sup}},\varepsilon\right)}d\varepsilon\\
  =&\int_{0}^{2t_{\ast}}\sqrt{\log N\left(\scrF_n,\|\cdot\|_{\rm{sup}},\frac{\varepsilon}{t_{\ast}}\right)}d\varepsilon\\
  =&t_{\ast}\int_{0}^{2}\sqrt{\log N\left(\scrF_n,\|\cdot\|_{\rm{sup}},\delta\right)}d\delta\\
  \lesssim&t_{\ast}r_n^{1/2}\int_{0}^{2}\delta^{-\frac{1}{2\kappa}}d\delta\\
  \lesssim&r_n^{1/2}\left(\sum_{|\frac{x}{n}|\le r_n}t_{x,n}^2\right)^{1/2}.
\end{align*}
Combining arguments above, we deduce for $\zeta\le 2$ that
\begin{align}\label{anp}
  \bbE[(A_n^{-})^{\zeta}]\lesssim 
  n^{-\zeta}r_n^{\zeta/2}\bbE\left[\left(\sum_{|\frac{x}{n}|\le r_n}T_{x,n}^2\right)^{\zeta/2}\right]
  \leq n^{-\zeta}r_n^{\zeta/2}\bbE\left[\sum_{|\frac{x}{n}|\le r_n}T_{x,n}^2\right]^{\zeta/2}
  \lesssim n^{\zeta(\theta-1/2)}c_n^{\zeta/2}.
\end{align}
Here we used $\zeta\le 2$ and Jensen's inequality, and 
$c_n$ is defined by 
\begin{align*}
  \bbE[T_{0,n}^2]
  \leq \bbE[(\tau_{0,n}^{-})^2]
  \lesssim \int_{1}^{d_n}u^{-\alpha+1}
  \lesssim c_n:=
  \begin{dcases}
    d_{n}^{2-\alpha}, & \alpha<2,\\
    \log d_{n}, & \alpha=2,\\
    1, & \alpha>2. 
  \end{dcases}
 \end{align*}
 Assuming $\alpha>2$, we finally consider the case where $\zeta\in (2,\alpha)$.
 By \cite[Corollary 1]{latala}, we obtain
 \begin{align*}
 \left\| \sum_{|\frac{x}{n}|\le r_n}T_{x,n}^2 \right\|_{\frac{\zeta}{2}}\lesssim \sup_{1\le s\le \frac{\zeta}{2}}n^{\frac{1+\theta}{s}}\|T_{0,n}^2\|_{s}
 \lesssim\sup_{1\le s\le \frac{\zeta}{2}}n^{\frac{1+\theta}{s}}\bbE[(\tau_{0,n}^{-})^{2s}]^{1/s}.
\end{align*}
Since $2s\le \zeta<\alpha$, we have 
$\bbE[(\tau_{0,n}^{-})^{2s}]^{1/s}\le \bbE[\tau_{0}^{\zeta}]^{2/\zeta}<\infty$, and thus 
$\left\| \sum_{|\frac{x}{n}|\le r_n}T_{x,n}^2 \right\|_{\frac{\zeta}{2}}\lesssim n^{1+\theta}$.
Noting that the first inequality in \eqref{anp} holds regardless of $\zeta$,
we conclude that
\begin{align*}
  \bbE[(A_n^{-})^{\zeta}]\lesssim 
  n^{-\zeta}r_n^{\zeta/2}\bbE\left[\left(\sum_{|\frac{x}{n}|\le r_n}T_{x,n}^2\right)^{\zeta/2}\right]
  \lesssim n^{-\zeta}r_n^{\zeta/2}n^{\frac{\zeta}{2}(1+\theta)}=n^{\zeta(\theta-1/2)},
\end{align*}
for $\zeta\in (2,\alpha)$.
Since $c_n=1$ if $\alpha>2$, we may establish the inequality  
$\bbE[(A_n^{-})^{\zeta}]\lesssim n^{\zeta(\theta-1/2)}c_n^{\zeta/2}$
for any $\alpha>1$ and $\zeta\in (0,\alpha)$.
 Altogether, we deduce 
 \begin{align*}
  \MoveEqLeft \bbE[d_{\BL^{\kappa}}(\mu^{(r_n)},\mu_n^{(r_n)})^{\zeta}]\\
  \lesssim&
  \begin{dcases}
     n^{-\zeta(\kappa-\theta)}+n^{\theta \zeta}d_n^{-\zeta(\alpha-1)}+n^{\zeta(\theta-1/2)}c_n^{\zeta/2}, & \zeta\in (0,1),\\
   n^{-\zeta(\kappa-\theta)}+n^{-\zeta+(1+\theta)(1\vee\frac{\zeta}{2})}d_n^{-\alpha+\zeta}+n^{\theta \zeta}d_n^{-\zeta(\alpha-1)}+n^{\zeta(\theta-1/2)}c_n^{\zeta/2}, & \zeta\in [1,\alpha).
  \end{dcases}
 \end{align*}
 Taking $d_n=n^{1/\alpha}$,
 we obtain the result.
\end{proof}

By virtue of \cref{blp}, we 
obtain the following estimate for the joint moment.

\begin{cor}\label{8.9}
  Let $\alpha>1,\;\kappa\in (1/2,1]$ and set $r_n=n^{\theta},\;\theta\ge 0$.
  For $\zeta,\;\eta\ge 0$ satisfying $\zeta+\eta<\alpha$, we have that
  \begin{align*}
    \bbE[d_{\BL^{\kappa}}(\mu^{(r_n)},\mu_n^{(r_n)})^{\zeta}\mu_n([-r_n,r_n])^{\eta}]
    \lesssim
     \begin{dcases}
    n^{-\zeta(1-\frac{1}{\alpha}-\theta)+\eta\theta}, & \alpha\in (1,2),\\
    n^{-\zeta(\frac{1}{2}-\theta)+\eta\theta}(\log n)^{\zeta/2}, & \alpha=2,\\
    n^{-\zeta(\frac{1}{2}-\theta)+\eta\theta}, & \alpha\in (2,\infty).
    \end{dcases}  
  \end{align*}
\end{cor}

\begin{proof}
  By $\mu_n([-r_n,r_n])=\int 1d\mu^{(r_n)}_n$, it holds that
  \begin{align*}
  \mu_n([-r_n,r_n])^{\eta}
  \leq\Big[\mu([-r_n,r_n])+d_{\BL^{\kappa}}(\mu^{(r_n)},\mu_n^{(r_n)})\Big]^{\eta}
  \lesssim r_n^{\eta}+d_{\BL^{\kappa}}(\mu^{(r_n)},\mu_n^{(r_n)})^{\eta}.
  \end{align*}
  Thus we deduce
  \begin{align*}
    \bbE[d_{\BL^{\kappa}}(\mu^{(r_n)},\mu_n^{(r_n)})^{\zeta}\mu_n([-r_n,r_n])^{\eta}]
    \lesssim r_n^{\eta} \bbE[d_{\BL^{\kappa}}(\mu^{(r_n)},\mu_n^{(r_n)})^{\zeta}]+
    \bbE[d_{\BL^{\kappa}}(\mu^{(r_n)},\mu_n^{(r_n)})^{\zeta+\eta}].
  \end{align*}
  Now the result is immediate from \cref{blp}.
\end{proof}

Now we are able to prove \cref{btmsgA}.
\begin{proof}[Proof of \cref{btmsgA}]
  Similarly as in the proof of \cref{btmsgQ},
   we have
  \begin{align*}
    \MoveEqLeft \bbE\left[\sup_{|x|\le K}\Big|E_{x}[f(X_t)]-E_{\frac{\floor{nx}}{n}}^n[f(X_t^n)]\Big|\right]\\
    \leq&\|f\|_{\BL^{\kappa}}\exp(-cn^{2\theta}/t)
    +\bbE\left[\sup_{|x|\le K}\Big|E_{x}[f(X_t^{(r_n)})]-E_{\frac{\ceil{nx}}{n}}^n[f(X_t^{(n,r_n)})]\Big|\right],
  \end{align*}
  where $r_n=n^{\theta}$, and $X^{(r_n)}$ and $X^{n,(r_n)}$ are the processes defined in the proof of \cref{btmrate}.
  Set $h=n^{-\kappa}+d_{\BL^{\kappa}}(\mu^{(r_n)},\mu_n^{(r_n)})$.
  Then, by 
  \cref{gen,su},
   there exists a deterministic number $h_0\in (0,1)$ such that 
  \begin{align*}
    \MoveEqLeft \bbE\left[\sup_{|x|\le K}\Big|E_{x}[f(X_t^{(r_n)})]-E_{\frac{\ceil{nx}}{n}}^n[f(X_t^{(n,r_n)})]\Big|\right]\\
    \leq &
    \bbE\left[\sup_{|x|\le K}\Big|E_{x}[f(X_t^{(r_n)})]-E_{\frac{\ceil{nx}}{n}}^n[f(X_t^{(n,r_n)})]\Big|:h\le h_0\right]\\
    &+\bbE\left[\sup_{|x|\le K}\Big|E_{x}[f(X_t^{(r_n)})]-E_{\frac{\ceil{nx}}{n}}^n[f(X_t^{(n,r_n)})]\Big|:h>h_0\right]\\
    \lesssim&\|f\|_{\BL^{\kappa}}\bbE[C(n)h^{\frac{\kappa}{3\kappa+4}}(\log (1/h))^{\frac{\kappa(2+\kappa)}{3\kappa+4}}]+\|f\|_{\BL^{\kappa}}\bbP(h>h_0)\\
    \lesssim&\|f\|_{\BL^{\kappa}}\bbE[C(n)h^{\frac{\kappa}{3\kappa+4}-\varepsilon}]+\|f\|_{\BL^{\kappa}}\bbE[h^p]\\
    \lesssim&\|f\|_{\BL^{\kappa}}\bbE[C(n)h^{\frac{\kappa}{3\kappa+4}-\varepsilon}]+\|f\|_{\BL^{\kappa}}\bbE[d_{\BL^{\kappa}}(\mu^{(r_n)},\mu_n^{(r_n)})^p]+\|f\|_{\BL^{\kappa}}n^{-p\kappa },
  \end{align*} 
  for $p\in (0,\alpha)$ and small $\varepsilon>0$.
  Here, $C(n)$ is $C(n)=r_n^{10\kappa+\frac{17}{2}}+r_n^{5(1+\kappa)}\mu_n([-r_n,r_n])^{5\kappa+\frac{7}{2}}$. See \cref{gen} and \eqref{5.4}.
 Thus we obtain
 \begin{align*}
  \MoveEqLeft \bbE[C(n)h^{\frac{\kappa}{3\kappa+4}-\varepsilon}]\\
  \lesssim&n^{-\frac{\kappa^2}{3\kappa+4}+\kappa\varepsilon}r_n^{10\kappa+\frac{17}{2}}
  +n^{-\frac{\kappa^2}{3\kappa+4}+\kappa\varepsilon}r_n^{5(1+\kappa)}\bbE[\mu_n([-r_n,r_n])^{5\kappa+\frac{7}{2}}]\\
  &+r_n^{10\kappa+\frac{17}{2}}\bbE[d_{\BL^{\kappa}}(\mu^{(r_n)},\mu_n^{(r_n)})^{\frac{\kappa}{3\kappa+4}-\varepsilon}]
  +r_n^{5(1+\kappa)}\bbE[d_{\BL^{\kappa}}(\mu^{(r_n)},\mu_n^{(r_n)})^{\frac{\kappa}{3\kappa+4}-\varepsilon}\mu_n([-r_n,r_n])^{5\kappa+\frac{7}{2}}],
 \end{align*}
 Then it follows from \cref{8.9} and $\kappa>1/2$ that
 \begin{align*}
  \frac{1}{\|f\|_{\BL^{\kappa}}}\bbE\left[\sup_{|x|\le K}\Big|E_{x}[f(X_t)]-E_{\frac{\floor{nx}}{n}}^n[f(X_t^n)]\Big|\right]
    \lesssim \exp(-cn^{2\theta}/t)+n^{A\theta}( n^{-\frac{\kappa}{2(3\kappa+4)}+\varepsilon}+n^{-\frac{p}{2}}),
 \end{align*}
 for some $A>0$. Taking $\theta,\;\varepsilon>0$ arbitrarily close to $0$ and $p=\frac{\kappa}{3\kappa+4}$, we obtain the result.
 \end{proof}

 To obtain the annealed estimate for the heat kernel, we need
 the following inequality, which is an annealed version of \cref{taihenn}.
 \begin{lem}\label{annmeas}
  For any $\varepsilon,K>0$ and $\alpha>2$, it holds that
  \begin{align*}
    \bbE\left[\sup_{|x|\le K}\mu_n([x-\varepsilon,x+\varepsilon])\right]\lesssim \varepsilon+(\varepsilon+K)^{1/2}n^{-1/2}.
  \end{align*}
 \end{lem}
 \begin{rmk}
  We can prove similar estimates for $\alpha\in (1,2]$.
  See the proof for the details.
 \end{rmk}

\begin{proof}
  Set $A_n=\floor{2n\varepsilon}+1,\;B_n=\floor{-n(K+\varepsilon)},$ and $C_n=\ceil{n(K-\varepsilon)}$.
  Then we have
  \begin{align*}
    \sup_{|x|\le K}\mu_n([x-\varepsilon,x+\varepsilon])
    \lesssim \max_{B_n\le j\le C_n}\frac{1}{n}\sum_{z=j}^{j+A_n}\tau_z
    \lesssim \frac{A_n}{n}+\frac{1}{n}\max_{B_n\le j\le C_n}\left|\sum_{z=j}^{j+A_n}T_z\right|,
  \end{align*}
  where $T_z=\tau_z-\bbE[\tau_z]$.
  If we define $S_k=\sum_{i=0}^{k}T_{B_n+i}$, 
  then 
  $(S_k)$ forms a martingale, and satisfies
  $\sum_{z=j}^{j+A_n}T_z=S_{A_n+j-B_n}-S_{j-B_n}$ for $j\in [B_n,C_n]$.
  Thus we may deduce that
  \begin{align*}
    \max_{B_n\le j\le C_n}\left|\sum_{z=j}^{j+A_n}T_z\right|\lesssim \max_{k\le D_n}|S_k|,
  \end{align*}
  where $D_n=A_n-B_n+C_n$.
  Using Doob's maximal inequality, we obtain
  \begin{align*}
    \bbE\left[ \max_{B_n\le j\le C_n}\left|\sum_{z=j}^{j+A_n}T_z\right| \right]
    \leq \left\| \max_{B_n\le j\le C_n}\left|\sum_{z=j}^{j+A_n}T_z\right|\right\|_2
    \lesssim \left\| \max_{k\le D_n}|S_k|\right\|_2
    \lesssim \|S_{D_n}\|_2
    \lesssim D_n^{\frac{1}{2}}.
  \end{align*} 
  Here, we used \cite[Corollary 8.2]{Gut}.
  Therefore, we conclude that
  \begin{align*}
    \bbE\left[\sup_{|x|\le K}\mu_n([x-\varepsilon,x+\varepsilon])\right]\lesssim
    \frac{A_n}{n}+D_n^{\frac{1}{2}}
    \lesssim \varepsilon+(\varepsilon+K)^{1/2}n^{-1/2},
  \end{align*}
  which completes the proof.
\end{proof}

 Finally we derive an annealed estimate for the heat kernel, namely, \cref{annhk}.
 
\begin{proof}[Proof of \cref{annhk}]
  We use the same method as in \cref{totyuu,hkest}.
  Set $x_n=\floor{nx}/n,\;f_r=1_{(x-r,x+r)},\;\psi(y)=1_{(-1,1)}\exp(-\frac{1}{1-y^2}),$ and 
  $g_{\varepsilon}(y)=\psi(\frac{y}{\varepsilon})$.
  Note that it holds that
  \begin{align}\label{dfgh}
    \sup_y|g_{\varepsilon}(y)|\le 1,\quad \|g_{\varepsilon}\|_{\Lip^{\kappa}}=\varepsilon^{-\kappa}\sup_{x\neq y}\frac{|\varphi(\frac{x}{\varepsilon})-\varphi(\frac{y}{\varepsilon})|}{|\frac{x}{\varepsilon}-\frac{x}{\varepsilon}|^{\kappa}}\le \|\varphi\|_{\Lip^{\kappa}}\varepsilon^{\kappa}, 
  \end{align}
  and
  \begin{align}\label{dfg}
    \int g_{\varepsilon}(y)\mu(dy)
    =\bbE[\tau_0]\int \varphi\left(\frac{y}{\varepsilon}\right)dy
    =\varepsilon\bbE[\tau_0]\int \varphi\left(y\right)dy.
  \end{align}
First, we have 
\begin{equation}\label{1q2}
  \begin{aligned}
    |p(t,0,x)-p_n(t,0,x_n)|
  \leq&
  \frac{1}{\int g_{\varepsilon}d\mu}\left|\left(\int g_{\varepsilon}d\mu\right)p(t,0,x)-\left(\int g_{\varepsilon}d\mu_n\right)p_n(t,0,x_n)\right|\\
  &+\frac{1}{\int g_{\varepsilon}d\mu}\left|  \int g_{\varepsilon}d(\mu-\mu_n)\right|p_n(t,0,x_n)\\
  =&:\frac{1}{\int g_{\varepsilon}d\mu}I+II,
  \end{aligned}
\end{equation}
By \cref{summ}(ii), there exists a deterministic constant $C=C(t)>0$ such that
$\sup_{\substack{n\in\N\\z\in n^{-1}\Z}}p_n(t,0,z)\le C$.
Thus it holds for large $n$ that
\begin{equation}\label{1q2w}
  \begin{aligned}
  II
  =&\frac{1}{\int g_{\varepsilon}d\mu}\left|  \int g_{\varepsilon}d(\mu^{(r_n)}-\mu_n^{(r_n)})\right|p_n(t,0,x_n)\\
  \lesssim&\varepsilon^{-1}\|g_{\varepsilon}\|_{\BL^{\kappa}}d_{\BL^{\kappa}}(\mu^{(r_n)},\mu^{(r_n)}_n)\\
  \lesssim&\varepsilon^{-(1+\kappa)}d_{\BL^{\kappa}}(\mu^{(r_n)},\mu^{(r_n)}_n),
  \end{aligned}
\end{equation}
where $r_n=n^{\theta},\theta>0$. Here we used \eqref{dfgh} and \eqref{dfg}.

We next bound $I$. We have the following:
\begin{align*}
  I
  \leq&\left|\left(\int g_{\varepsilon}d\mu\right)p(t,0,x)-P_tg_{\varepsilon}(x)\right|
  +|P_tg_{\varepsilon}(x)-P^n_tg_{\varepsilon}(x_n)|
  +\left|P^n_tg_{\varepsilon}(x_n)-\left(\int g_{\varepsilon}d\mu_n\right)p_n(t,0,x_n)\right|\\
  =&:III+IV+V.
\end{align*}
The first term $III$ is deterministic, and satisfies
\begin{align}\label{1q2w3}
  III
  =\left|\int_{-\varepsilon}^{\varepsilon} g_{\varepsilon}(y)(p(t,0,x)-p(t,y,x))\mu(dy)\right|\lesssim \varepsilon^2,
\end{align} 
since $p(t,\cdot,z)$ is Lipschitz continuous uniformly in $z\in\R$.
By \cref{btmsgA}, we have
\begin{align}\label{1q2w3e}
  \bbE\left[\sup_{|x|\le K}IV\right]\lesssim \|g_{\varepsilon}\|_{\BL^{\kappa}}n^{-E}\lesssim \varepsilon^{-\kappa}n^{-E},
\end{align}
for $E<\frac{\kappa}{2(3\kappa+4)}$.
We finally consider $V$.
By \cref{summ}(ii),\cref{suffcond2}(i), and \cref{arigato}, there exists a deterministic
constant $C=C(t)>0$ such that $\sup_{\substack{n\in\N\\z\in n^{-1}\Z}}\|p_n(t,z,\cdot)\|_{\Lip^{1/2}}\leq C$.
Thus it follows from \cref{annmeas} that
\begin{equation}\label{1qw232}
\begin{aligned}
   \bbE\left[\sup_{|x|\le K}V\right]
  =&\bbE\left[\sup_{|x|\le K}\left|\int_{-\varepsilon}^{\varepsilon}g_{\varepsilon}(y)(p_n(t,y,x_n)-p_n(t,0,x_n))\mu_n(dy)\right|\right]\\
  \lesssim&(\varepsilon+n^{-1})^{1/2}\bbE\left[ \sup_{|x|\le K}\mu_n([x-\varepsilon,x+\varepsilon])\right]\\
  \lesssim&(\varepsilon^{1/2}+n^{-1/2})(\varepsilon+\varepsilon^{1/2}n^{-1/2}+n^{-1/2}).
\end{aligned}  
\end{equation}
Combining 
\eqref{1q2},\;\eqref{1q2w},\;\eqref{1q2w3},\;\eqref{1q2w3e},\;\eqref{1qw232}, and \cref{blp},
we conclude that
\begin{align*}
   \MoveEqLeft E\left[\sup_{|x|\le K}\Big|p(t,0,x)-p_n(t,0,\floor{nx}/n)\Big|\right]\\
   \lesssim &\varepsilon+\varepsilon^{-(1+\kappa)}n^{-E}+\varepsilon^{-1}(\varepsilon^{1/2}+n^{-1/2})(\varepsilon+\varepsilon^{1/2}n^{-1/2}+n^{-1/2})+\varepsilon^{-(1+\kappa)}n^{-(\frac{1}{2}-\theta)}.
\end{align*}
Taking $\varepsilon=n^{-\frac{E}{2+\kappa}}$ and $\theta$ arbitrarily small,
we complete the proof.
\end{proof}

\appendix
\crefalias{section}{appendix}

\section{Proof of \cref{summ}}\label{appB}
In this appendix, we give a proof of \cref{summ}.
Recall the conditions and assumptions
$\rm{Reg}_{s_0,s_1},\;\rm{LRES}(\theta),$ and $\rm{UP}$
from the beginning of \Cref{cont}.
The argument in this section are mainly informed by 
\cite{Cr07,Kum04}.

We first prove \cref{summ}(ii)
\begin{proof}[Proof of \cref{summ}$\rm{(ii)}$]
  It is known that we have 
  \begin{align*}
    p(t,x,x)\leq \frac{2\sup_{y\in A}R(x,y)}{t}+\frac{\sqrt{2}}{\mu(A)},
  \end{align*}
  for any $t>0,x\in F,$ and a Borel subset $A$ of $X$ satisfying $0<\mu(A)<\infty$ (see \cite[Theorem 10.4]{Kig12}).
  Taking $A=B(x,t^{\frac{1}{1+s_0}})$, we obtain the result.
\end{proof}

We further consider the following for a resistance metric space $(F,R)$:
\begin{enumerate}[align=left, labelwidth=\widthof{$\rm{URES}$ }, 
    labelsep=0.5em, 
    leftmargin=!]
  \item [$\rm{URES}$]There exists $c_{\rm{UR}}>0$ such that $R(y,B(x,r)^c)\leq c_{\rm{UR}}r$ for all $y\in F$ and $B(x,r)\neq F$.
\end{enumerate}

\begin{lem}\label{B1}
$\rm{UP}$ is equivalent to the following.
\begin{itemize}
  \item $\exists a_{\rm{UP}}>1,\;\forall B(x,r)\neq F,\;B(x,a_{\rm{UP}}r)\setminus B(x,r)\neq \emptyset$
\end{itemize}
Furthermore, we may take the constants to satisfy $a_{\rm{UP}}=c_{\rm{UP}}^{-1}$.
\end{lem}
\begin{proof}
  Suppose that $\rm{UP}$ holds and take $B(x,r)\neq F$.
  If $B(x,c_{\rm{UP}}^{-1}r)=F$, then obviously we have $B(x,c_{\rm{UP}}^{-1}r)\setminus B(x,r)\neq \emptyset$.
  If $B(x,a_{\rm{UP}}^{-1}r)\neq F$, then we may use our assumption to deduce that 
  $B(x,c_{\rm{UP}}^{-1}r)\setminus B(x,r)=B(x,c_{\rm{UP}}^{-1}r)\setminus B(x,c_{\rm{UP}}c_{\rm{UP}}^{-1}r)\neq \emptyset$.
  The converse can be shown easily.
\end{proof}

\cref{summ}(i) is a corollary of the following proposition.

\begin{prp}\label{A3}
  Suppose that there exists a Borel measure $\nu$ satisfying $\rm{Reg}_{s_0,s_1}$.
Then $\rm{UP}$ implies $\rm{LRES}(1+s_0-s_1)$ and $\rm{URES}$ with $c_{\rm{LR}}=\left[ 4\frac{c_u}{c_l}\frac{(\frac{1}{4}c_{\rm{UP}}+1)^{s_1}}{(\frac{1}{16}c_{\rm{UP}})^{s_0}}\frac{1}{c_{\rm{UP}}} \right]^{-1}$ and $c_{\rm{UR}}=1+c_{\rm{UP}}^{-1}$.
\end{prp}
\begin{proof}
  The following proof is based on the argument in \cite[Lemma 4.1]{Kum04}.
  We first prove $\rm{URES}$. Take $y\in F$ and $B(x,r)\neq F$. 
  Since our goal is to obtain $R(y,B(x,r)^c)\leq (1+c_{\rm{UP}}^{-1})r$, we may assume $y\in B(x,r)$.
  The condition $\rm{UP}$ and \cref{B1} ensure the existence of $z\in F$ such that $r\leq R(x,z)\leq c_{\rm{UP}}^{-1}r$.
  Then it holds that $ R(y,z)\leq R(y,x)+R(x,z)\leq (1+c_{\rm{UP}}^{-1})r$. Therefore we obtain the following:
  \begin{align*}
    R(y,B(x,r)^c)
    =&\inf\{\calE(f,f):f\in\calF,f(y)=1,f|_{B(x,r)^c}=0\}^{-1}\\
    \leq&\inf\{\calE(f,f):f\in\calF,f(y)=1,f(z)=0\}^{-1}\\
    =&R(y,z)\\
    \leq&(1+c_{\rm{UP}}^{-1})r,
  \end{align*}
  which shows $\rm{URES}$.
  
  We next show $\rm{LRES}(1+s_0-s_1)$.
  Take $B(x,r)\neq F$ and $z\in B(x,r)\setminus B(x,c_{\rm{UP}}r)$. Let $h_z\in\calF$ be the harmonic extension
  of $1_{\{x\}}:\{x,z\}\to \R$. Then we have $\calE(h_z,h_z)=R(z,x)^{-1}\leq (c_{\rm{UP}}r)^{-1}$.
  Thus, if we put $\lambda_1=\frac{1}{4}c_{\rm{UP}}$,
  it follows for $y\in B(z,\lambda_1 r)$ that
  \begin{align*}
    |h_z(y)|^2=|h_z(y)-h_z(z)|^2\leq \frac{R(y,z)}{R(x,z)}\leq \frac{1}{4},
  \end{align*}
  which implies $|h_z|\leq \frac{1}{2}$ on the ball.
  Set $A:=B(x,r)\setminus B(x,c_{\rm{UP}}r)$ and take $z_0\in A$.
  When $z_0,\dots,z_n\in A$ are chosen and $A\setminus \cup_{i=0}^n B(z_i,\lambda_1 r)$ is not empty,
  we choose $z_{n+1}$ from $A\setminus \cup_{i=0}^n B(z_i,\lambda_1 r)$.
  We prove that
  this procedure terminates in at most $N$ steps, and it holds that 
  $N+1\leq \frac{c_u}{c_l}\frac{(\lambda_1+1)^{s_1}}{(\frac{1}{4}\lambda_1)^{s_0}}r^{s_1-s_0}$.
    Note that $(B(z_i,\frac{1}{4}\lambda_1 r))_i$ are disjoint subsets of $B(x,(\lambda_1+1)r)$.
    Thus we obtain
    \begin{align*}
      (n+1)c_l\left(\frac{1}{4}\lambda_1 r\right)^{s_0}\leq \sum_{i=0}^{n}\nu(B(z_i,\frac{1}{4}\lambda_1 r))
      \leq \nu(B(x,(\lambda_1+1)r))\leq c_u(\lambda_1+1)^{s_1}r^{s_1},
    \end{align*}
    and deduce $n+1\leq \frac{c_u}{c_l}\frac{(\lambda_1+1)^{s_1}}{(\frac{1}{4}\lambda_1)^{s_0}}r^{s_1-s_0}$.
  This shows the desired bound for $N+1$.

  By this procedure, we obtain $N+1$ points $(z_i)_{i=0}^N$
  which satisfy $B(x,r)\setminus B(x,c_{\rm{UP}}r)\subseteq\cup_{i=0}^N B(z_i,\lambda_1 r)$ and $c_{\rm{UP}}r\leq R(x,z_i)\leq r$.
  Set $g=\min_{i\leq N}h_{z_i}$ and $h=2(g-\frac{1}{2})^{+}1_{B(x,r)}\in\calF$.
  Then we deduce that
  \begin{align*}
    R(x,B(x,r)^c)^{-1}\leq \calE(h,h)\leq 4\sum_{i=0}^N \calE(h_{z_i},h_{z_i})
    \leq 4(N+1)(\min_i R(z_i,x))^{-1}.
  \end{align*}
  Combining this and the estimate for $N+1$, we complete the proof.
\end{proof}

Under the assumption $\rm{(A2)}$, set $D_F=\frac{1}{2}\diam F\wedge (2c_{\rm{LR}}^{-1}\diam F)^{1/\theta}$.
\begin{lem}\label{A2}
Assume $\rm{(A2)}$. If $r<D_F$, then 
\[
P_x(\sigma_{B(x,r)^c}\leq t)\leq \frac{4}{5}+c_{\rm{Exit}}t r^{-(1+s_0)\theta},
\]
where $c_{\rm{Exit}}=\frac{16}{5}\frac{1}{c_l c_{\rm{LR}} (\frac{1}{4}c_{\rm{LR}})^{s_0}}$.
\end{lem}
\begin{proof}
  Set $A:=B(x,r)^c$. Note that $r,\;\frac{1}{4}R(x,A)<\frac{1}{2}\diam F$.
  Thus it follows from \cref{4.2-a}(i) that
  \begin{align*}
    P_x(\sigma_A\leq t)
    \leq& 2\left[ 1-\frac{1-\frac{1}{4}}{1+\frac{1}{4}}\exp\left( -\frac{2t}{\mu(B(x,\frac{1}{4}R(x,A)))(1-\frac{1}{4})R(x,A)} \right) \right]\\
    \leq& 2\left[ 1-\frac{3}{5}\exp\left( -\frac{2t}{\mu(B(x, \frac{1}{4}c_{\rm{LR}}r^{\theta} ))\frac{3}{4}c_{\rm{LR}}r^{\theta}} \right) \right]\\
    \leq&2\left[ 1-\frac{3}{5}\left( 1-\frac{2t}{c_l(\frac{1}{4}c_{\rm{LR}})^{s_0}\frac{3}{4}c_{\rm{LR}} r^{(1+s_0)\theta} }\right)\right]\\
    =&\frac{4}{5}+\frac{16}{5}\frac{t}{c_l c_{\rm{LR}} (\frac{1}{4}c_{\rm{LR}})^{s_0}r^{(1+s_0)\theta}},
  \end{align*}
which shows the desired inequality.
\end{proof}

The following proves \cref{summ}(iii).
\begin{prp}\label{007}
If Assumption $\rm{(A2)}$ holds,
then it holds that 
\[
\frac{1}{100}\left[c_u(5c_{\rm{Exit}})^{\frac{s_1}{(1+s_0)\theta}} \right]^{-1}t^{-\frac{s_1}{(1+s_0)\theta}}\leq p(t,x,x),\qquad t<\frac{D_F^{(1+s_0)\theta}}{5c_{\rm{Exit}}}.
\]
\end{prp}
\begin{proof}
  The following argument builds on \cite[Proposition 4.3]{Kum04}.
  First, we observe $P_x(X_t\not\in B(x,r))\leq P_x(\sigma_{B(x,r)^c}\leq t)\leq\frac{4}{5}+c_{\rm{Exit}}t r^{-(1+s_0)\theta}$ for $r<D_F$
  from \cref{A2}.
  Set $t_1=\frac{D_F^{(1+s_0)\theta}}{10c_{\rm{Exit}}}$ and fix $t\in (0,t_1)$.
  Define $r_t=(10 c_{\rm{Exit}}t)^{\frac{1}{(1+s_0)\theta}}$. Then, since the map $0<a\mapsto \frac{a^{(1+s_0)\theta}}{10c_{\rm{Exit}}}$ is an increasing bijection,
  we can verify $r_t<D_F$ and $P_x(X_t\not\in B(x,r_t))\leq 9/10$.
  Therefore we conclude that
  \begin{align*}
    \frac{1}{100}
    =&(1-9/10)^2\\
    \leq& P_x(X_t\in B(x,r_t))^2\\
    =&\left(\int_{B(x,r_t)}p(t,x,y)\mu(dy)\right)^2\\
    \leq&\mu(B(x,r_t))p(2t,x,x)\\
    \leq&c_u (10 c_{\rm{Exit}}t)^{\frac{s_1}{(1+s_0)\theta}}p(2t,x,x).
  \end{align*}
  This completes the proof.
\end{proof}

Next we prove estimate for the expectation of the exist time
under Assumption $\rm{(A1)}$.
This estimate is essential to prove \cref{summ}(v).

\begin{lem}\label{exittime}
  Suppose that $\mu$ satisfies $\rm{(A1)}$ and
take $x\in F$ and $B(x_0,r)\neq F$. Then the following hold.
\begin{itemize}
  \item Under $\rm{LRES}(\theta)$, we have $E_x[\sigma_{B(x_0,r)^c}]\geq \left(\frac{1}{4}c_{\rm{LR}}\right)^{1+s_0}c_l r^{(1+s_0)\theta}.$
  \item Under $\rm{URES}$, we have $E_x[\sigma_{B(x_0,r)^c}]\leq  c_u c_{\rm{UR}}r^{1+s_1}$.
\end{itemize}
In particular, if both of those hold, we have $(1+s_0)\theta\geq 1+s_1$.
\end{lem}
\begin{proof}
  In the following argument, we follow \cite[Proposition 4.2]{Kum04}.
  Set $B:=B(x_0,r)$ and let $g_B$ be the Green function with the boundary $B^c$ (i.e. the Green function corresponding to the process killed upon hitting $B^c$).
  Then we have $g_B(x,y)\leq g_B(x,x)\leq R(x,B^c)\leq c_{\rm{UR}}r$ (cf. \cite[Theorem 4.1]{Kig12}).
  Thus, we deduce
  \begin{align*}
    E_x[\sigma_B]=\int_F g_B(x,y)\mu(dy)=\int_B g_B(x,y)\mu(dy)\leq c_u c_{\rm{UR}}r^{1+s_1}
  \end{align*}
  from \cite[Corollary 10.11]{Kig12}, which shows the first assertion.

  We next show the lower bound.
  Since it holds that 
  \begin{align*}
    |g_B(x_0,x_0)-g_B(x_0,y)|^2\leq g_B(x_0,x_0)R(x_0,y),\qquad y\in F,
  \end{align*}
  we have
  \begin{align*}
    \left|1-\frac{g_B(x_0,y)}{g_B(x_0,x_0)}\right|^2\leq \frac{R(x_0,y)}{g_B(x_0,x_0)}=\frac{R(x_0,y)}{R(x_0,B^c)}\leq \frac{R(x_0,y)}{c_{\rm{LR}}r^{\theta}}.
  \end{align*}
  This implies $\frac{g_B(x_0,y)}{g_B(x_0,x_0)}\geq 1/2$ for $y\in B(x_0,\frac{1}{4}c_{\rm{LR}}r^{\theta})$.
  Therefore we deduce that
  \begin{align*}
    E_{x_0}[\sigma_B]=\int_F g_B(x_0,y)\mu(dy)
    \geq \frac{1}{2}g_B(x_0,x_0)\mu( B(x_0,\frac{1}{4}c_{\rm{LR}}r^{\theta}) )
    \geq \left(\frac{1}{4}c_{\rm{LR}}\right)^{1+s_0}c_l r^{(1+s_0)\theta}.
  \end{align*}
  This gives us the desired inequality.
  For the final assertion, consider the limit $r\to 0$.
\end{proof}

We are now ready to prove \cref{summ}(v).

\begin{proof}[Proof of \cref{summ}$\rm{(v)}$]
  The following argument is based on \cite[Lemma 12]{Cr07}.
   Note that the assumption $\rm{(A3)}$ implies $\rm{LRES}(1+s_0-s_1)$ and $\rm{URES}$ by \cref{A3}. 
  Set $\theta=1+s_0-s_1$. 
For $r<\frac{1}{2}\diam F$, we have from \cref{exittime} that 
  \begin{align*}
    c_l(c_{\rm{LR}}/4)^{1+s_0}r^{(1+s_0)\theta}
    \leq&E_x[\sigma_{B(x,r)^c}]\\
    \leq&t+E_x[ E_{X_t}[\sigma_{B(x,r)^c}] : \sigma_{B(x,r)^c}>t]\\
    \leq&t+(1-P_x(\sigma_{B(x,r)^c}\leq t))c_u c_{\rm{UR}}r^{1+s_1},
  \end{align*}
  and thus, it holds that
  \begin{align}\label{haa}
    P_x(\sigma_{B(x,r)^c}\leq t)\leq 1-\frac{c_l(c_{\rm{LR}}/4)^{1+s_0}r^{(1+s_0)\theta}}{c_u c_{\rm{UR}}r^{1+s_1}}+\frac{t}{c_u c_{\rm{UR}}r^{1+s_1}}.
  \end{align}
  Note that, by \cref{exittime}, $\frac{c_l(c_{\rm{LR}}/4)^{1+s_0}r^{(1+s_0)\theta}}{c_u c_{\rm{UR}}r^{1+s_1}}\in [0,1]$ for $r<\frac{1}{2}\diam F$.
  For $n\in\N$, define
  \begin{align*}
    \sigma_0=0,\qquad\sigma_{i+1}=\inf\{t\geq \sigma_i:R(X_{\sigma_i},X_t)\geq r/n\}.
  \end{align*}
  Since the Dirichlet form is local under Assumption 
  $\rm{(A3)}$, the process is a diffusion (see \cite[Theorem 7.2.2]{FOT94}), 
  which implies 
  $\sum_{i=1}^n (\sigma_i-\sigma_{i-1})\leq \sigma_{B(x,r)^c}$.
  Also, it follows from the strong Markov property and \eqref{haa} that
  \begin{align*}
    P_x(\sigma_{i+1}-\sigma_i\leq t|X_s,s\leq \sigma_i)
    =&P_{X_{\sigma_i}}(\sigma_{B(X_0,r/n)^c}\leq t)\\
    \leq& 1-\frac{1}{2}\frac{c_l(c_{\rm{LR}}/4)^{1+s_0}(r/n)^{(1+s_0)\theta}}{c_u c_{\rm{UR}}(r/n)^{1+s_1}}+\frac{t}{c_u c_{\rm{UR}}(r/n)^{1+s_1}}\\
    =&1-\frac{1}{2}\frac{c_l(c_{\rm{LR}}/4)^{1+s_0}}{c_u c_{\rm{UR}}}\left(\frac{r}{n}\right)^{(2+s_0)(s_0-s_1)}+\frac{t}{c_u c_{\rm{UR}}(r/n)^{1+s_1}}.
  \end{align*}
  This observation, \cite[Lemma 1.1]{BarlowBass}, and $\log(1-x)\leq-x$ imply that
  \begin{align*}
    \log\Psi_t(x,r)
    \leq&2\left(  \frac{n \frac{t}{c_u c_{\rm{UR}}(r/n)^{1+s_1}} }{1-\frac{1}{2}\frac{c_l(c_{\rm{LR}}/4)^{1+s_0}}{c_u c_{\rm{UR}}}\left(\frac{r}{n}\right)^{(2+s_0)(s_0-s_1)}}  \right)^{1/2}
    +n\log\left( 1-\frac{1}{2}\frac{c_l(c_{\rm{LR}}/4)^{1+s_0}}{c_u c_{\rm{UR}}}\left(\frac{r}{n}\right)^{(2+s_0)(s_0-s_1)} \right)\\
    \leq&4\left(\frac{nt}{c_u c_{\rm{UR}}(r/n)^{1+s_1}}\right)^{1/2}
    -\frac{n}{2}\frac{c_l(c_{\rm{LR}}/4)^{1+s_0}}{c_u c_{\rm{UR}}}\left(\frac{r}{n}\right)^{(2+s_0)(s_0-s_1)}.
  \end{align*}
Now, we consider the following inequality:
\begin{align*}
  4\left(\frac{nt}{c_u c_{\rm{UR}}(r/n)^{1+s_1}}\right)^{1/2}
    \leq \frac{n}{4}\frac{c_l(c_{\rm{LR}}/4)^{1+s_0}}{c_u c_{\rm{UR}}}\left(\frac{r}{n}\right)^{(2+s_0)(s_0-s_1)},
\end{align*}
which is equivalent to $na(t/r)\leq r$, where $a(\xi)=\left\{ \frac{1}{256 \xi}\frac{[c_l(c_{\rm{LR}}/4)^{1+s_0}]^2}{c_u c_{\rm{UR}}} \right\}^{-\frac{1}{s_1+2(2+s_0)(s_0-s_1)}}$.
We first consider the case $a(t/r)\leq r$. Then
\begin{align*}
  n_0=&\sup\left\{ n\in\N: 4\left(\frac{nt}{c_u c_{\rm{UR}}(r/n)^{1+s_1}}\right)^{1/2}
    \leq \frac{n}{4}\frac{c_l(c_{\rm{LR}}/4)^{1+s_0}}{c_u c_{\rm{UR}}}\left(\frac{r}{n}\right)^{(2+s_0)(s_0-s_1)}\right\}\\
    =&\sup\{n\in\N:na(t/r)\leq r\}
\end{align*}
is well-defined, and lies in $\N$.
By definition, $n_0$ satisfies $n_0\leq r/a(t/r)<n_0+1$. Thus we deduce by above argument and $a(t/r)\leq r$ that
\begin{align*}
  \log\Psi_t(x,r)
  \leq&-\frac{1}{4}n_0\frac{c_l(c_{\rm{LR}}/4)^{1+s_0}}{c_u c_{\rm{UR}}}\left(\frac{r}{n_0}\right)^{(2+s_0)(s_0-s_1)}\\
  \leq&-\frac{1}{4}\left(\frac{r}{a(t/r)}-1\right)\frac{c_l(c_{\rm{LR}}/4)^{1+s_0}}{c_u c_{\rm{UR}}}a(t/r)^{(2+s_0)(s_0-s_1)}\\
  \leq&-\frac{1}{4}\left[\frac{c_l(c_{\rm{LR}}/4)^{1+s_0}}{c_u c_{\rm{UR}}}r a(t/r)^{(2+s_0)(s_0-s_1)-1}+\frac{c_l(c_{\rm{LR}}/4)^{1+s_0}}{c_u c_{\rm{UR}}} r^{(2+s_0)(s_0-s_1)}\right]\\
  \leq&-\frac{1}{4}\frac{c_l(c_{\rm{LR}}/4)^{1+s_0}}{c_u c_{\rm{UR}}}r a(t/r)^{(2+s_0)(s_0-s_1)-1}+\frac{1}{4}.
\end{align*}
If $r<a(t/r)$, then by $(2+s_0)(s_0-s_1)-1<0$, we have
\begin{align*}
  \frac{1}{4}\frac{c_l(c_{\rm{LR}}/4)^{1+s_0}}{c_u c_{\rm{UR}}}r a(t/r)^{(2+s_0)(s_0-s_1)-1}
  \leq \frac{1}{4}\frac{c_l(c_{\rm{LR}}/4)^{1+s_0}}{c_u c_{\rm{UR}}}r^{(2+s_0)(s_0-s_1)}
  \leq \frac{1}{4}.
\end{align*}
Combining above arguments, we conclude that
\begin{align*}
  \Psi_t(x,r)
  \leq&e^{1/4} \exp\left[ -\frac{1}{4}\frac{c_l(c_{\rm{LR}}/4)^{1+s_0}}{c_u c_{\rm{UR}}}r a(t/r)^{(2+s_0)(s_0-s_1)-1} \right]\\
  \leq&2 \exp\left[ -\frac{1}{4}\frac{c_l(c_{\rm{LR}}/4)^{1+s_0}}{c_u c_{\rm{UR}}} \left( \frac{1}{256}\frac{[c_l(c_{\rm{LR}}/4)^{1+s_0}]^2}{c_u c_{\rm{UR}}} \right)^{\beta}\frac{r^{1+\beta}}{t^{\beta}}\right],
\end{align*}
which is what we desired.
\end{proof}

Finally, we are able to complete the proof of \cref{summ}.
\begin{proof}[Proof of \cref{summ}$\rm{(iv)}$]
 The following argument builds on \cite[Proposition 4.3]{Kum04}.
   Write $C$ for $\frac{1}{4}\frac{c_l(c_{\rm{LR}}/4)^{1+s_0}}{c_u c_{\rm{UR}}} \left( \frac{1}{256}\frac{[c_l(c_{\rm{LR}}/4)^{1+s_0}]^2}{c_u c_{\rm{UR}}} \right)^{\beta}$.
  Then we have $P_x(X_t\not\in B(x,r))\leq P_x(\sigma_{B(x,r)^c}\leq t)\leq 2\exp[-Cr^{1+\beta}/t^{\beta}]$ for $r<\frac{1}{2}\diam F$.
  Set $t_1=\left[ \frac{C(\frac{1}{2}\diam F)^{1+\beta}}{\log 4} \right]^{1/\beta}$.
  For $t\in (0,t_1)$, define $r_t=\left[ \frac{t^{\beta}\log 4}{C} \right]^{1/(1+\beta)}$.
  Then we may verify that $r_t<\frac{1}{2}\diam F$ and 
  $P_x(X_t\not\in B(x,r_t))\leq 1/2$.
  Therefore we conclude that 
  \begin{align*}
    \frac{1}{4}
    =&(1-1/2)^2\\
    \leq&P_x(X_t\in B(x,r_t))^2\\
    =&\left(\int_{B(x,r_t)}p(t,x,y)\mu(dy)\right)^2\\
    \leq&\mu(B(x,r_t))p(2t,x,x)\\
    \leq&c_u \left[ \frac{t^{\beta}\log 4}{C} \right]^{\frac{s_1}{1+\beta}}p(2t,x,x)
  \end{align*} 
  for $t<t_1$.
  It is then an immediate consequence that
  \begin{align*}
    \frac{1}{4c_u}\left[ \frac{\log 4}{2^{\beta}C} \right]^{-\frac{s_1}{1+\beta}} t^{-\frac{\beta s_1}{1+\beta}} \leq p(t,x,x),\qquad t<2t_1.
  \end{align*}
\end{proof}

\section{Proof of \cref{hitwkconv}}\label{appA}
In this appendix, we give a proof of \cref{hitwkconv}.
Let $(F,R,\mu,\rho),\;(F_n,R_n,\mu_n,\rho_n),\;A,\;A_n,$ and $(M,d^M)$ be as in 
\cref{hitwkconv}.
We denote the processes corresponding to $(F,R,\mu)$ and $(F_n,R_n,\mu_n)$ 
by $(X=(X_t)_t,(P_x)_{x\in F})$ and $(X^n=(X^n_t)_t,(P_x^n)_{x\in F_n})$, respectively.

We regard the hitting time of $A\subseteq M$ as a function on $D:=D(\R_{\geq 0},M):=\{x:\R_{\geq 0}\to M:\textrm{$x$ is c\`adl\`ag.}\}$.
Precisely, we define $\sigma_A:D\to[0,\infty]$ for $A\subseteq M$ by setting
\begin{align*}
  \sigma_A(x)=\inf\{t>0:x(t)\in A\},\quad x\in D.
\end{align*}
We equip $D$ with the $J_1$ topology.
We refer the reader to  \cite[Chapter 3]{EK86}
for details on the $J_1$ topology.

We denote the set of the continuity points of $x\in D$ by $C_x$.
\begin{lem}\label{upsemicont}
For any open subset $U\subseteq M$, $\sigma_U$ is upper semicontinuous.
In particular, it is measurable.
\end{lem}
\begin{proof}
  It suffices to show that $\limsup_{n\to\infty}\sigma_U(x_n)\leq \sigma_U(x)$ for any convergent sequence $(x_n)\subseteq D$ and its limit $x\in D$.
  Thus, we may assume that $\sigma_U(x)<\infty$.
  Fix an arbitrary $\varepsilon>0$. Then we can find $t_0\in [\tau_U(x),\tau_U(x)+\varepsilon)$
  such that $x(t_0)\in U$. Since $U$ is open and $x$ lies in $D$,
  we can verify that for any sufficiently large $j\in \N$, there exists $t_j\in [t_0,t_0+1/j]\cap C_x$ such that $x(t_j)\in U$.
  The $J_1$ convergence yields that $\lim_{n\to\infty}x_n(t_j)=x(t_j)\in U$ for each $j$, since $t_j$ is in $C_x$.
  Therefore, for each $j$, there exists $N_j \in \N$ such that $x_n(t_j) \in U$ whenever $n>N_j$.
  This implies that $\limsup_{n\to\infty}\sigma_U(x_n)\leq t_j$ for any $j$. 
  Letting $j\to\infty$ and then $\varepsilon\to 0$, we obtain the result.
\end{proof}
Note that since the space is recurrent, $P_{\rho}(\sigma_A(X)<\infty)=1$ by \cite[Theorem 5.2.16(i)]{ChF12}  and \cite[Theorem 4.7.1(iii)]{FOT94}.
\begin{lem}\label{quasi}
Let $U_k:=\cup_{a\in A}B(a,1/k)$. Then $\sigma_{U_k}(X)$ converges to $\sigma_A(X)$ almost surely with respect to $P_{\rho}$.
\end{lem}
\begin{proof}
  Since $\sigma_{U_k}(X)$ is increasing in $k$, it has a pointwise limit $\sigma_{\infty}$.
  By $\sigma_{U_k}(X)\leq \sigma_{A}(X)$, we have $\sigma_{\infty}\leq \sigma_A(X)<\infty$ $P_{\rho}$-a.s.
  The quasi-left-continuity of $X$ (see \cite[(A.2.4)]{FOT94} for the definition of quasi-left-continuity) 
  yields that
  \begin{align*}
    P_{\rho}\left(\lim_{k\to\infty}X_{\sigma_{U_k}(X)}=X_{\sigma_{\infty}},\sigma_{\infty}<\infty\right)
    =P_{\rho}\left(\sigma_{\infty}<\infty\right)=1.
  \end{align*}
  Since $X_{\sigma_{U_k}(X)}\in\overline{U_k}$, and $A$ is closed, we deduce $\sigma_{\infty}\geq \sigma_A(X)$ $P_{\rho}$-a.s.
\end{proof}
Now we can prove the result in the compact case.
\begin{lem}\label{cptcase}
If $F_n$ and $F$ are compact, then $P^n_{\rho_n}(\sigma_{A_n}^n\in \cdot)$ converges weakly to  $P_{\rho}(\sigma_{A}\in \cdot)$.
\end{lem}
\begin{proof}
  We denote the set of continuity points of $\R\ni t\mapsto P_{\rho}(\sigma_A(X)\leq t)\in\R$ by $C$.
  It suffices to show the following inequalities.
  \begin{equation}\label{low}
    \liminf_{n\to\infty}P^n_{\rho_n}(\sigma_{A_n}(X^n)\leq t)\geq P_{\rho}(\sigma_{A}(X)\leq t),\qquad t\in C,
  \end{equation}
\begin{equation}\label{upp}
    \limsup_{n\to\infty}P^n_{\rho_n}(\sigma_{A_n}(X^n)\leq t)\leq P_{\rho}(\sigma_{A}(X)\leq t),\qquad t\in C
\end{equation}
We first prove \eqref{low}. We may assume that $t>0$.
Take sufficiently small $\varepsilon\in (0,t)$ and set $A^{\varepsilon}=\cup_{a\in A}B(a,\varepsilon)$.
By our assumption, it holds that $A_n\subseteq A^{\varepsilon}$ and $d^M_{\mathrm{H}}(A,A_n)<\varepsilon$ for large $n\in\N$.
We consider such $n\in\N$.
Fix $y\in \overline{A^{\varepsilon}}\cap F$. Then we may find $a_n\in A_n$ satisfying $d^M(y,a_n)\leq 2\varepsilon$.
Now, the commute time identity (see \cref{cti}) yields that
\begin{align*}
  P_y^n(\sigma_{A_n}> \sqrt{\varepsilon})
  \leq P_y^n(\sigma_{a_n}(X^n)> \sqrt{\varepsilon})
  \leq \varepsilon^{-1/2}E_y^n[\sigma_{a_n}(X^n)]
  \leq c\sqrt{\varepsilon}.
\end{align*}
Define $X^{n,\sigma_{A^{\varepsilon}}}=(X^{n,\sigma_{A^{\varepsilon}}}_t)_t$ by setting $X^{n,\sigma_{A^{\varepsilon}}}_t=X^{n}_{\sigma_{A^{\varepsilon}}(X^n)+t}$.
Then we deduce that 
\begin{align*}
  P^n_{\rho_n}(\sigma_{A_n}(X^n)\leq t)
  \geq& P^n_{\rho_n}\left( \sigma_{A_n}\left( X^{n,\sigma_{A^{\varepsilon}}} \right)\leq \sqrt{\varepsilon},\sigma_{A^{\varepsilon}}(X^n)\leq t-\sqrt{\varepsilon} \right)\\
  =&E^n_{\rho_n}\left[ P^n_{\sigma_{A^{\varepsilon}}(X^n)}\left( \sigma_{A_n}(X^n)\leq \sqrt{\varepsilon} \right) :\sigma_{A^{\varepsilon}}(X^n)\leq t-\sqrt{\varepsilon} \right]\\
  \geq&(1-c\sqrt{\varepsilon})P^n_{\rho_n}\left(\sigma_{A^{\varepsilon}}(X^n)< t-\sqrt{\varepsilon} \right).
\end{align*}
Since $\{\sigma_{A^{\varepsilon}}<s\}$ is an open subset of $D$ for any $s\in\R_{\geq 0}$ by \cref{upsemicont},
and $P_{\rho_n}^n(X^n\in\cdot)$ converges weakly to $P_{\rho}(X\in\cdot)$ as a probability measure on $D$
(see \cite[Theorem 1.2]{Cr18}),
we conclude that
\begin{align*}
  \liminf_{n\to\infty}P^n_{\rho_n}(\sigma_{A_n}(X^n)\leq t)
  \geq(1-c\sqrt{\varepsilon})\liminf_{n\to\infty}P^n_{\rho_n}\left(\sigma_{A^{\varepsilon}}(X^n)< t-\sqrt{\varepsilon} \right)
  \geq(1-c\sqrt{\varepsilon})P_{\rho}(\sigma_{A^{\varepsilon}}(X)<t-\sqrt{\varepsilon}).
\end{align*}
Letting $\varepsilon\to 0$, we obtain \eqref{low} by $t\in C$ and \cref{quasi}.

We next prove \eqref{upp}. We may assume that $t>0$.
Using Skorokhod's representation theorem, we may construct 
a probability space $(\Omega,\calF,P)$ and $D$-valued random elements $Y^n=(Y_t^n)_t,Y=(Y_t)_t$ on $\Omega$
satisfying the following:
\begin{itemize}
  \item $P^n_{\rho_n}(X^n\in\cdot)=P(Y^n\in\cdot)$
  \item $P_{\rho}(X\in\cdot)=P(Y\in\cdot)$
  \item $Y^n$ converges to $Y$, $P$-a.s.\ in $D$ with respect to $J_1$ topology.
\end{itemize} 
We will show that
\begin{align}\label{eq4}
  \limsup_{n\to\infty}\{\sigma_{\widetilde{A}_n}(Y^n)\leq t\}\subseteq\{ \sigma_{A^{\varepsilon}}\leq t \},\qquad t,\varepsilon>0,
\end{align}
where $A^{\varepsilon}:=\cup_{a\in A}B(a,\varepsilon)$ and $\widetilde{A}_n:=\cup_{a\in A_n}B(a,1/n)$.
Fix $\omega$ in the left-hand side of \eqref{eq4}.
Then we may find a sequence $n_k\uparrow \infty$ such that $s_k:=\sigma_{\widetilde{A}_{n_k}}(Y^{n_k}(\omega))\leq t$.
Hereafter, we omit $\omega$ from the notation.
If necessary, by taking a further subsequence, we may assume that $s_k$ converges to $s\in [0,t]$.
The $J_1$ convergence implies the existence of an increasing homeomorphism $\lambda_k:\R_{\geq 0}\to\R_{\geq 0}$ 
satisfying
\begin{equation*}
  \lim_{k\to\infty}\sup_{u\geq 0}|\lambda_k(u)-u|=0,
\end{equation*}
and
\begin{equation*}
\lim_{k\to\infty}\sup_{u\leq T}d^M(Y_{\lambda_k(u)},Y^{n_k}_u)=0,\qquad T>0.
\end{equation*}
We deduce that
\begin{align*}
  \lim_{k\to\infty}d^M(Y_{\lambda_k(s_k)},Y^{n_k}_{s_k})=0.
\end{align*}
Combining $d^M(Y^{n_k}_{s_k},A_{n_k})\leq 1/{n_k},\lim_{k\to\infty}d^M_{\mathrm{H}}(A,A_{n_k})=0,$ and $\lim_{k\to\infty}\lambda_k(s_k)=s$,
we have either $Y_s\in A$ or $Y_{s-}\in A$.
In both cases, it holds that $\sigma_{A^{\varepsilon}}(Y)\leq s\leq t$.
This completes the proof of \eqref{eq4}.

Note that , we have $P_{\rho_n}^n(X^n\in\cdot)=P(Y^n\in\cdot)$ and, 
by \cref{upsemicont},
$P_{\rho_n}^n(\sigma_{\widetilde{A}_n}(X^n)\in\cdot)=P(\sigma_{\widetilde{A}_n}(Y^n)\in\cdot)$.
Similarly, we obtain $P_{\rho}(\sigma_{A^{\varepsilon}}(X)\in\cdot)=P(\sigma_{A^{\varepsilon}}(Y)\in\cdot)$.
Combining these observations and \eqref{eq4}, we conclude that 
\begin{equation}\label{last}
  \begin{aligned}
     \limsup_{n\to\infty} P^n_{\rho_n}(\sigma_{A_n}(X^n)\leq t)
     \leq&\limsup_{n\to\infty} P^n_{\rho_n}(\sigma_{\widetilde{A}_n}(X^n)\leq t)\\
=&\limsup_{n\to\infty} P(\sigma_{\widetilde{A}_n}(Y^n)\leq t)\\
\leq&P\left(\limsup_{n\to\infty}\{\sigma_{\widetilde{A}_n}(Y^n)\leq t\}\right)\\
\leq&P(\sigma_{A^{\varepsilon}}(Y)\leq t)\\
=&P_{\rho}(\sigma_{A^{\varepsilon}}(X)\leq t).
  \end{aligned}
\end{equation}
Letting $\varepsilon\to 0$, we obtain \eqref{upp} by $t\in C$ and \cref{quasi}.
\end{proof}

\begin{proof}[Proof of \cref{hitwkconv}]
  Set $A_n^{(r)}:=\overline{B(\rho_n,r)}\cap A_n$ and $A^{(r)}:=\overline{B(\rho,r)}\cap A$ for $r>0$.
  Then we may find a sequence $r_k\uparrow\infty$ such that it holds that 
  $A_n^{(r_k)},A^{(r_k)}\neq\emptyset$ and $\lim_{n\to\infty}d^M_{\mathrm{H}}(A^{(r_k)},A^{(r_k)}_n)=0$
  for all $k\in\N$.
  Define sets $C,\;S\subseteq \R$ as follows:
  \begin{equation*}
    C:=\{t: \textrm{$t$ is a continuity point of $s\mapsto P_{\rho}(\sigma_A(X)\leq s)$}\},
  \end{equation*}
  \begin{equation*}
     S:=C\setminus\bigcup_{k\in\N}\{t: \textrm{$t$ is a discontinuity point of $s\mapsto P_{\rho}(\sigma_{A}(X^{(r_k)})\leq s)$}\}.
  \end{equation*}
  Since $S$ is a dense subset of $C$, it is enough to show that 
  $\lim_{n\to\infty}P_{\rho_n}^n (\sigma_{A_n}(X^n)\leq t)=P_{\rho}(\sigma_{A}(X)\leq t)$ for any $t\in S$.
  We fix $t\in S$.
  Note that $X^{(r)}$ and $X^{n,(r)}$ are the time changes of $X$ and $X^{n}$ with respect to $\mu^{(r)}$ and $\mu_n^{(r)}$, respectively
  (see \cite[Lemma 2.6]{CHK17}, and recall the definitions of above notation from \eqref{rversion}). 
  In particular, we may define the original process and the $r$-version process on the common probability space.
  Since $\sigma_{A}(X^{(r)})=\sigma_{A^{(r)}}(X^{(r)})\leq \sigma_{A^{(r)}}(X)$, and $F$ is recurrent,
  it follows from \cref{4.2-a} that
  \begin{align*}
    0
    \leq&P_{\rho}(\sigma_{A}(X^{(r)})\leq t)-P_{\rho}(\sigma_{A^{(r)}}(X)\leq t)\\
    =&P_{\rho}(\sigma_{A^{(r)}}(X^{(r)})\leq t,\sigma_{A^{(r)}}(X)> t)\\
    \leq&P_{\rho}(\sigma_{B(\rho,r)^c}(X)\leq t)\\
    \leq&4\left[ \frac{1}{R(\rho,B(\rho,r)^c)}+\frac{t}{\mu(B(\rho,1))(R(\rho,B(\rho,r)^c)-1)} \right]
  \end{align*} 
  for any $t\geq 0$ and large $r>0$.
  Similarly, we have that
  \begin{align*}
    0
    \leq&P_{\rho}(\sigma_{A^{(r)}}(X)\leq t)-P_{\rho}(\sigma_{A}(X)\leq t)\\
    =&P_{\rho}(\sigma_{A^{(r)}}(X)\leq t,\sigma_{A}(X)>t)\\
    \leq&P_{\rho}(\sigma_{B(\rho,r)^c}(X)\leq t)\\
    \leq&4\left[ \frac{1}{R(\rho,B(\rho,r)^c)}+\frac{t}{\mu(B(\rho,1))(R(\rho,B(\rho,r)^c)-1)} \right].
  \end{align*}
  Thus we obtain 
  \begin{align*}
    \MoveEqLeft |P_{\rho}(\sigma_{A}(X)\leq t)-P_{\rho_n}^n(\sigma_{A_n}(X^n)\leq t)|\\
    \leq& |P_{\rho}(\sigma_{A}(X)\leq t)-P_{\rho}(\sigma_{A^{(r_k)}}(X)\leq t)|\\
   &+|P_{\rho}(\sigma_{A^{(r_k)}}(X)\leq t)-P_{\rho}(\sigma_{A}(X^{(r_k)})\leq t)|\\
   &+|P_{\rho}(\sigma_{A}(X^{(r_k)})\leq t)-P_{\rho_n}^n(\sigma_{A_n}(X^{n,(r_k)})\leq t)|\\
   &+|P_{\rho_n}^n(\sigma_{A_n}(X^{n,(r_k)})\leq t)-P_{\rho_n}^n(\sigma_{A_n^{(r_k)}}(X^n)\leq t)|\\
   &+ |P_{\rho_n}^n(\sigma_{A_n^{(r_k)}}(X^n)\leq t)-P_{\rho_n}^n(\sigma_{A_n}(X^n)\leq t)|\\
   \leq&8\left[ \frac{1}{R(\rho,B(\rho,r_k)^c)}+\frac{t}{\mu(B(\rho,1))(R(\rho,B(\rho,r_k)^c)-1)} \right]\\
   &+|P_{\rho}(\sigma_{A}(X^{(r_k)})\leq t)-P_{\rho_n}^n(\sigma_{A_n}(X^{n,(r_k)})\leq t)|\\
   &+8\left[ \frac{1}{R_n(\rho,B(\rho_n,r_k)^c)}+\frac{t}{\mu_n(B(\rho_n,1))(R_n(\rho_n,B(\rho_n,r_k)^c)-1)} \right].
  \end{align*}
  Using $t\in S$, $\liminf_{n\to\infty}\mu_n(B(\rho_n,1))\geq \mu(B(\rho,1))>0$, \eqref{reccc}, the non-explosion condition \eqref{non-exp},
  and the compact case \cref{cptcase},
  we complete the proof.
\end{proof}

\section*{Acknowledgements}
The author would like to express his sincere gratitude to  Dr David Croydon for his constant
guidance, encouragement, and many fruitful discussions throughout this work.

\nocite{*}
\bibliographystyle{abbrv}
\bibliography{rate2}
\end{document}